\documentclass[a4paper]{amsart}
\usepackage{amsmath, amsthm, amsfonts, amscd, amssymb, MnSymbol}
\usepackage{xcolor}
\usepackage{enumitem}
\usepackage{tikz-cd}
\usepackage{hyperref}
\usepackage{mathtools}
\usepackage[normalem]{ulem}

\DeclareMathAlphabet{\mathdutchcal}{U}{dutchcal}{m}{n}
\SetMathAlphabet{\mathdutchcal}{bold}{U}{dutchcal}{b}{n}
\DeclareMathAlphabet{\mathdutchbcal}{U}{dutchcal}{b}{n}

\theoremstyle{plain}
\newtheorem{thm}{Theorem}[section]
\newtheorem{cor}[thm]{Corollary}
\newtheorem{lem}[thm]{Lemma}
\newtheorem{prop}[thm]{Proposition}

\theoremstyle{definition}
\newtheorem{defn}[thm]{Definition}
\newtheorem{ex}[thm]{Example}
\newtheorem{ques}[thm]{Question}

\newtheorem{rmk}[thm]{Remark}

\numberwithin{thm}{subsection}
\numberwithin{equation}{subsection}

\usepackage{biblatex}
\usepackage{xurl}
\hypersetup{breaklinks=true}
\newcommand{\Ker}{\operatorname{Ker}} 

\newcommand{\Irr}{\operatorname{Irr}}
\newcommand{\Tr}{\operatorname{Tr}} 
\newcommand{\id}{\operatorname{id}}
\newcommand{\ad}{\operatorname{ad}}
\newcommand{\Ad}{\operatorname{Ad}}
\newcommand{\Span}{\operatorname{Span}}

\newcommand{\End}{\operatorname{End}}

\newcommand{\inv}{{}_{\operatorname{inv}}}
\newcommand{\weights}{\mathbf P}

\newcommand{\ev}{\operatorname{ev}}
\newcommand{\diag}{\operatorname{diag}}

\newcommand{\Nbb}{\mathbb N}
\newcommand{\Cbb}{\mathbb C}
\newcommand{\Zbb}{\mathbb Z}
\newcommand{\Rbb}{\mathbb R}

\newcommand{\Lbb}{\mathbb{L}}
\newcommand{\Qbb}{\mathbb{Q}}

\newcommand{\Pbf}{\mathbf P}
\newcommand{\Qbf}{\mathbf Q}
\newcommand{\qbf}{\mathbf q}
\newcommand{\Vbf}{\mathbf V}
\newcommand{\Ebf}{\mathbf E}
\newcommand{\Fbf}{\mathbf F}
\newcommand{\Kbf}{\mathbf K}
\newcommand{\Ubf}{\mathbf U}
\newcommand{\vbf}{\mathbf v}
\newcommand{\wbf}{\mathbf w}
\newcommand{\pibf}{\boldsymbol{\pi}}

\newcommand{\Xbf}{\mathbf X}

\newcommand{\gbf}{\mathbf g}
\newcommand{\abf}{\boldsymbol{a}}

\newcommand{\Acal}{\mathcal{A}}
\newcommand{\Bcal}{\mathcal{B}}
\newcommand{\Ccal}{\mathcal{C}}
\newcommand{\Dcal}{\mathcal{D}}

\newcommand{\Hcal}{\mathcal{H}}
\newcommand{\Rcal}{\mathcal{R}}

\newcommand{\Kcal}{\mathcal{K}}

\newcommand{\Ocal}{\mathcal{O}}

\newcommand{\Xcal}{\mathcal{X}}
\newcommand{\Ucal}{\mathcal{U}}
\newcommand{\Tcal}{\mathcal{T}}
\newcommand{\Zcal}{\mathcal{Z}}
\newcommand{\hcal}{\mathdutchcal{h}}

\newcommand{\gf}{\mathfrak{g}}
\newcommand{\hf}{\mathfrak{h}}
\newcommand{\kf}{\mathfrak{k}}

\newcommand{\Cf}{\mathfrak{C}}
\newcommand{\tf}{\mathfrak{t}}
\newcommand{\nf}{\mathfrak{n}}
\newcommand{\bfrak}{\mathfrak{b}}

\newcommand{\nfrak}{\mathfrak{n}}

\newcommand{\eps}{\varepsilon}

\newcommand{\la}{\langle}
\newcommand{\ra}{\rangle}
\newcommand{\La}{\boldsymbol{\la}}
\newcommand{\Ra}{\boldsymbol{\ra}}

\title[Laplacians on \lowercase{q}-deformations]{Laplacians on $q$-deformations of compact semisimple Lie groups}

\allowdisplaybreaks

\begin{document}

\author{Heon Lee}
\address{Heon Lee, Institute for Advanced Study in Mathematics, Harbin Institute of Technology, Harbin 150001, China}
\email{heonlee@hit.edu.cn}

\begin{abstract}
The problem of formulating a correct notion of Laplacian on compact quantum groups (CQGs) has long been recognized as both fundamental and nontrivial. Existing constructions typically rely on selecting a specific first-order differential calculus (FODC), but the absence of a canonical choice in the noncommutative setting renders these approaches inherently non-canonical. In this work, we propose a simple set of conditions under which a linear operator on a CQG can be recognized as a Laplacian---specifically, as the formal modulus square of the differential associated with a bicovariant FODC. A key feature of our framework is its generality: it applies to arbitrary finite-dimensional bicovariant $*$-FODCs on \( K_q \), the \( q \)-deformation of a compact semisimple Lie group \( K \). To each such calculus, we associate a Laplacian defined via the formal modulus square of its differential. Under mild additional assumptions, we demonstrate that these operators converge to classical Laplacians on \( K \) in the classical limit \( q \to 1 \), thereby justifying their interpretation as ``$q$-deformed Laplacians." Furthermore, we prove that the spectra of the $q$-deformed Laplacians are discrete, real, bounded from below, and diverge to infinity, much like those of their classical counterparts. However, in contrast to the classical case, the associated heat semigroups do not define quantum Markov semigroups.
\end{abstract}

\maketitle

\section{Introduction}\label{sec:Introduction}

\renewcommand{\theequation}{\thesection.\arabic{equation}}
\renewcommand{\thethm}{\thesection.\arabic{thm}}

The Laplacian $\square : C^\infty(M) \rightarrow C^\infty(M) $ on a $d$-dimensional Riemannian manifold $(M,g)$ gives rise to the partial differential equation
\[
\frac{\partial}{\partial t} u + \square u = 0 , \quad u \in C^\infty ( \Rbb_{\geq 0} \times M ),
\]
called \textit{the heat equation on $M$}, which describes the diffusion of heat across the manifold. As such, solutions to this equation are deeply influenced by the geometry of $M$, and consequently, they can provide significant insight into its structure \cite{Gauduchon}.

Motivated by the classical picture, considerable effort has gone into defining a suitable notion of Laplacian on noncommutative spaces. Since compact quantum groups (CQGs), as noncommutative generalizations of compact groups \cite{Woronowicz1987b}, offer a rich source of noncommutative examples, they provide a natural framework in which to explore such operators. Accordingly, this question has been actively pursued in the setting of CQGs \cite{Heckenberger_Laplace-Beltrami_2000, Heckenberger_Spin_2003, KustermanMurphyTuset_Laplacian_2005, Landi2009, Landi2011, Zampini_Hodge-duality_2012, Majid_Hodge-Fourier_2017}.

When \( K \) is a compact Lie group, the Laplacian is defined as the unique linear map \( \square : C^\infty(K) \rightarrow C^\infty(K) \) satisfying
\begin{equation}\label{eq:classical Laplacian-introduction}
\int_K \overline{f(x)}\, (\square g)(x)\, dx = \int_K \langle df_x , dg_x \rangle\, dx, \quad f, g \in C^\infty(K),
\end{equation}
where \( \int_K dx \) denotes integration with respect to the Haar measure on \( K \) and \( \langle \cdot, \cdot \rangle \) is a positive definite \( C^\infty(K) \)-sesquilinear form on the space of differential 1-forms \( \Omega^1(K) \). This form, together with the exterior derivative \( d: C^\infty(K) \rightarrow \Omega^1(K) \) provides the necessary structure to define the ``modulus square'' of \( d \), which yields \( \square \).

Following this classical principle, the aforementioned works considered a \textit{fixed} first-order differential calculus (FODC) on a CQG, and introduced additional structures on it---such as nondegenerate bilinear forms with specific properties, Hodge operators, and others---that yield suitable analogues of the right-hand side of \eqref{eq:classical Laplacian-introduction}. The ``Laplacian'' is then defined as the unique operator satisfying an analogue of the left-hand side of \eqref{eq:classical Laplacian-introduction}. Laplacians constructed in this manner have been employed to uncover deep noncommutative geometric structures of certain CQGs.

Despite these successes, all of these approaches share a common drawback: the choice of a particular FODC, which may need to satisfy certain conditions, must be made before the Laplacian can even be defined. However, any such choice, no matter how natural it may appear (even Woronowicz's \( 4D_\pm \)-calculus on \( SU_q(2) \)), is inherently non-canonical. This lack of canonicity was already remarked in Woronowicz’s observations in \cite{Woronowicz1989}, where he emphasized the absence of a universally preferred differential calculus for general CQGs. Theorem~\ref{thm:the q->1 behavior of FODCs}, one of the main results of this paper, provides a conceptual explanation for this phenomenon in the case of the \( q \)-deformation \( K_q \) of a compact semisimple Lie group \( K \). Specifically, it shows that for each matrix realization of \( K \) that does not have any multiple irreducible components (i.e., each matrix realization that is multiplicity-free), there exists a finite-dimensional bicovariant FODC on \( K_q \) that converges to the classical FODC on \( K \) as \( q \to 1 \) (see Remark~\ref{rmk:q->1 limit of FODCs}), with the resulting FODCs being inequivalent for distinct matrix realizations.

This motivates the need for an alternative definition of the Laplacian on CQGs---one that (1) encompasses all finite-dimensional bicovariant $*$-FODCs on \( K_q \), (2) recovers the classical Laplacians in the case of compact Lie groups, and (3) yields well-behaved Laplacians to which standard operator-theoretic tools can be applied for investigating the noncommutative geometry of \( K_q \).

The purpose of this paper is to introduce and study a construction that satisfies these requirements. Our approach contrasts with previous studies in that we first select a linear operator on a CQG, which is intended to serve as the Laplacian, and then use it to induce an FODC equipped with a nondegenerate sesquilinear form, with respect to which the chosen linear operator becomes the unique operator that satisfies the quantum analogue of \eqref{eq:classical Laplacian-introduction}.

More precisely, let \( \Cf^\infty(\Kcal) \) be a CQG equipped with its Haar state \( \hcal : \Cf^\infty(\Kcal) \rightarrow \Cbb \), and let \( L : \Cf^\infty(\Kcal) \rightarrow \Cf^\infty(\Kcal) \) be a linear operator. Then, Theorem~\ref{thm:linear functionals as Laplacians} asserts that if \( \square \) diagonalizes over the Peter-Weyl decomposition of \( \Cf^\infty(\Kcal) \) with real eigenvalues, commutes with the antipode of $\Kcal$, and vanishes at the unit \( 1_{\Cf^\infty(\Kcal)} \), it induces a bicovariant $*$-FODC \( (\Omega, d) \) along with a nondegenerate right \( \Cf^\infty(\Kcal) \)-sesquilinear form
\(
\langle \cdot, \cdot \rangle : \Omega \times \Omega \rightarrow \Cf^\infty(\Kcal),
\)
with respect to which the operator \( \square \) serves as a \emph{Laplacian} in the sense that it satisfies
\begin{equation}\label{eq:quantum Laplacian-introduction}
\hcal ( f^* \square g) = \hcal \big( \langle df, dg \rangle \big), \quad f,g \in \Cf^\infty(\Kcal),
\end{equation}
which is the quantum analogue of~\eqref{eq:classical Laplacian-introduction}. It is worth noting, however, that although the induced sesquilinear form
\[
\hcal\big( \langle \cdot , \cdot \rangle \big) : \Omega \times \overline{\Omega} \rightarrow \Cbb
\]
is nondegenerate, it need not be positive definite, which marks a key difference from the classical case.

Moreover, if \( \square \) is taken to be a classical Laplacian on a compact Lie group \( K \), then the induced FODC coincide with the classical FODC on \( K \), and \eqref{eq:quantum Laplacian-introduction} reduces to \eqref{eq:classical Laplacian-introduction}.

However, the greatest advantages of this construction become most evident when applied to the $q$-deformation $K_q$ of a compact semisimplie Lie group $K$. In this setting, we classify the linear operators on \( \Cf^\infty(K_q) \) that satisfy the three assumptions of the main construction and additionally induce finite-dimensional FODCs in the construction. A necessary step in this classification yields a new algebraic result:
\begin{itemize}
    \item The first explicit description of the center of the dual Hopf algebra \( \Cf^\infty(K_q)^\circ \), which complements Joseph's description of \( \Cf^\infty(K_q)^\circ \) \cite[Proposition~9.4.9]{Joseph}.
\end{itemize}
Building on this classification, we establish the following:
\begin{itemize}
    \item Every finite-dimensional bicovariant $*$-FODC on \( K_q \) arises from this construction. Thus, the framework developed here allows us to study Laplacians on \( K_q \) associated with any such FODC.

    \item For each such FODC satisfying a mild additional condition, we construct a corresponding Laplacian, called a \emph{\( q \)-deformed Laplacian}, which converges to a classical Laplacian on \( K \) as \( q \to 1 \).

    \item These \( q \)-deformed Laplacians have a simple and explicit form: they are certain positive linear combinations of the \textit{quantum Casimir elements} \( \{ z_\mu \mid \mu \in \weights^+ \} \) \cite[Section~3.13]{VoigtYuncken} acting on \( \Cf^\infty(K_q) \) via convolution.

    \item The eigenvalues of \( q \)-deformed Laplacians can be expressed in terms of numerically computable algebraic invariants of \( K \).
\end{itemize}

The second result shows that \( q \)-deformed Laplacians are not merely formal analogues of classical Laplacians---both of which satisfy~\eqref{eq:quantum Laplacian-introduction}---but genuine \( q \)-deformations, converging to classical Laplacians in the classical limit.

Together with Theorem~\ref{thm:the q->1 behavior of FODCs}, which asserts that
\begin{itemize}
    \item As \( q \to 1 \), all finite-dimensional bicovariant FODCs on \( K_q \) that admit \( q \)-deformed Laplacians converge to the classical FODC on \( K \),
\end{itemize}
this can be rephrased heuristically as:
\begin{center}
\emph{``The \( q \)-deformation lifts the infinite degeneracy of the classical first-order differential calculus and classical Laplacians on \( K \)."}
\end{center}

We also establish that, when considered as unbounded operators on the GNS Hilbert space \( L^2(K_q) \),
\begin{itemize}
    \item The spectra of the closures of \( q \)-deformed Laplacians are discrete, real, lower-semibounded, and diverge to infinity,
\end{itemize}
just like the spectra of their classical counterparts \cite{Jost}.
This similarity enables us to explore the noncommutative geometry of \( K_q \) through their spectral properties, much as in classical spectral geometry.

These properties also ensure that for any \( q \)-deformed Laplacian \( \square \), the family \( (e^{-t \square})_{t \geq 0} \), defined via functional calculus, forms a well-defined semigroup of bounded operators on \( L^2(K_q) \), called \textit{the heat semigroup generated by \( \square \)}. These semigroups restrict to semigroups of operators on \( C(K_q) \), the universal \( C^* \)-algebra completion of \( \Cf^\infty(K_q) \). However, unlike the classical case,
\begin{itemize}
    \item The heat semigroups generated by \( q \)-deformed Laplacians do not form quantum Markov semigroups on $C(K_q)$---the most extensively studied class of stochastic processes on CQGs \cite{Franz2014}.
\end{itemize}

Thus, in addition to the immediate contribution to the noncommutative geometry of \( K_q \) through spectral methods, the discovery of \( q \)-deformed Laplacians also enriches the ongoing study of stochastic processes on CQGs by providing a wealth of previously unexplored stochastic processes that are deeply connected to the noncommutative geometry of \( K_q \).

We now provide a brief outline of the paper. Sections~\ref{sec:Compact quantum groups (CQGs)} and \ref{sec:FODC} recall basic facts about CQGs and FODCs, respectively, which will be used throughout the paper. In Section~\ref{sec:Compact Lie groups}, we apply these preliminaries to analyze a classical example—namely, a compact Lie group \( K \), which will serve to motivate the main construction of this paper.

Section~\ref{sec:Main construction} presents the main construction of the paper, through which we define the notion of a \textit{Laplacian on a CQG}.

The remainder of the paper focuses on the \( q \)-deformation \( K_q \) of a compact semisimple Lie group \( K \). Section~\ref{sec:The q-deformation} recalls basic facts about \( K_q \). In Section~\ref{sec:FODCs on Kq}, we provide an explicit description of all finite-dimensional bicovariant FODCs on \( K_q \), building on the classification result of \cite{Baumann1998}, and present the first classification of finite-dimensional bicovariant \( * \)-FODCs on \( K_q \). In Section~\ref{sec:Laplacians on Kq}, we apply the general construction from Section~\ref{sec:Main construction} to \( K_q \), classify all Laplacians on \( K_q \) arising from this construction, and show that all finite-dimensional bicovariant \( * \)-FODCs arise within this framework. Along the way, we provide the first explicit description of the center of \( \Cf^\infty(K_q)^\circ \). Moreover, we compute the \( q \to 1 \) limits of certain Laplacians on \( K_q \), leading to the definition of \( q \)-deformed Laplacians. Section~\ref{sec:The q-deformed Laplacians} explores several properties of \( q \)-deformed Laplacians, demonstrating that while their spectra resemble those of the classical Laplacians, their heat semigroups do not form quantum Markov semigroups.

Having established that \( q \)-deformed Laplacians on \( K_q \) converge to classical Laplacians on \( K \) as \( q \to 1 \), it is natural to ask what happens to the FODCs associated with these Laplacians in the classical limit. Section~\ref{sec:the q->1 limit of FODCs} addresses this question and shows that to each multiplicity-free matrix realization of \( K \), which is equivalent to the condition that it admits a $q$-deformed Laplacian, there corresponds a distinct finite-dimensional bicovariant FODC on \( K_q \) that converges to the classical FODC on \( K \) as \( q \to 1 \). Appendix~\ref{sec:proof of collecting all q} contains the proof of Theorem~\ref{thm:collecting all q}.

We do not attempt in this paper to extend the definition of \( q \)-deformed Laplacians to higher-order differential calculi. This is because, first, as emphasized in this Introduction, formulating \( q \)-deformed Laplacians via FODCs is already a nontrivial and important problem in the noncommutative geometry of \( K_q \). Second, preliminary explorations into such extensions revealed subtle difficulties that deserve a separate, more focused treatment, and incorporating these complexities would considerably increase the size of the paper, which is already quite long. To keep the scope and size reasonable, we defer the study of higher-order extensions to a sequel to this work.

We conclude this section with a few remarks on notation. All vector spaces and algebras in this paper are assumed to be over the field of complex numbers, denoted by \( \Cbb \), unless otherwise specified. The complex linear span of a subset \( S \) of a vector space will always be denoted by \( \Span_\Cbb S \). Given a vector space \( V \), we denote its dual space by \( V^* \) and the algebra of linear operators on \( V \) by \( \End(V) \). Notations such as \( \End_{\Qbb(s)} (V) \) will also be used, with their meaning being clear from the context. All tensor products are taken with respect to \( \Cbb \), unless otherwise indicated by a subscript, e.g., \( \otimes_\Rbb \).

The symbol \( \id \) denotes the identity map on any set. When multiple identity maps appear in a single expression, we distinguish them by subscripts, e.g., \( \id_V \), \( \id_W \), etc. Similarly, the symbol \( 1 \) denotes the identity element in any algebra.

Any sesquilinear pairing between two complex vector spaces (e.g., an inner product on a Hilbert space) is denoted by \( \langle \cdot, \cdot \rangle \), with the first argument being conjugate linear. Bilinear pairings between vector spaces will often be denoted by \( (\cdot, \cdot) \) when the context makes the pairing clear (e.g., the canonical pairing between a vector space and its dual).

If \( \Hcal \) is a Hilbert space, the algebra of bounded operators on \( \Hcal \) is denoted by \( \Lbb(\Hcal) \). When \( \Hcal \) is finite-dimensional and we wish to ignore the \( * \)-structure of \( \Lbb(\Hcal) \), we will often write it as \( \End(\Hcal) \). The algebra of \( n \times n \) matrices with entries in an algebra \( \Acal \) is denoted by \( M_n(\Acal) \).

\renewcommand{\theequation}{\thesubsection.\arabic{equation}}
\renewcommand{\thethm}{\thesubsection.\arabic{thm}}

\section{Compact quantum groups (CQGs)}\label{sec:Compact quantum groups (CQGs)}

In this section, we review basic facts about compact quantum groups, following, for example, \cite[Chapters~1--3]{Timmermann}. Any results not covered by this reference will be proved in full. Special emphasis is placed on linear functionals and operators on compact quantum groups with distinguished properties.

\subsection{Hopf \texorpdfstring{$*$}{TEXT}-algebras}\label{subsec:Hopf *-algebras}

Throughout this subsection, \((\Acal , \Delta, \epsilon , S)\) denotes a Hopf \((*)-\)algebra, where \(\Delta\) is the comultiplication, \(\epsilon\) is the counit, and \(S\) is an antipode, which is assumed to be invertible. All formulas that explicitly involve involutions pertain to Hopf \(*\)-algebras, while those that do not also hold for Hopf algebras. Statements for Hopf algebras that involve (\(*\)-) become the corresponding statements for Hopf \(*\)-algebras when we remove the parentheses. We will continue to follow this practice throughout the paper, whenever appropriate.

We also adopt the Sweedler notation for comultiplication. That is, for \( a \in \Acal \), we write \( \Delta(a) = a_{(1)} \otimes a_{(2)} \), with summation over a certain index set understood.

For $a \in \Acal$, define
    \[
    a \rightarrow b = a_{(1)} b S(a_{(2)}) \;\; \text{and} \;\; b \leftarrow a = S(a_{(1)}) b a_{(2)}, \quad b \in \Acal.
    \]
These define left and right $\Acal$-module actions on $\Acal$, called \textit{the left and right adjoint actions on $\Acal$}, respectively. Note that
\begin{equation}\label{eq:antipodes and adjoint actions}
    S^{\pm 1} ( a \rightarrow b) = S^{\pm 1} (b) \leftarrow S^{\pm 1}(a), \quad a, b \in \Acal.
\end{equation}

Let \( V \) be a vector space. A linear map \( \Phi_V: V \rightarrow \Acal \otimes V \) is called a \emph{left coaction of \( \Acal \) on \( V \)} if it satisfies
\[
(\Delta \otimes \id)\, \Phi_V = (\id \otimes \Phi_V)\, \Phi_V \quad \text{and} \quad (\epsilon \otimes \id)\, \Phi_V = \id.
\]
Similarly, a linear map \( {}_V\Phi: V \rightarrow V \otimes \Acal \) is called a \emph{right coaction} if it satisfies
\[
(\id \otimes \Delta)\, {}_V\Phi = ({}_V\Phi \otimes \id)\, {}_V\Phi \quad \text{and} \quad (\id \otimes \epsilon)\, {}_V\Phi = \id.
\]

The map \( \ad: \Acal \rightarrow \Acal \otimes \Acal \) defined by
\begin{equation}\label{eq:right adjoint coaction}
    \ad(a) = a_{(2)} \otimes S(a_{(1)}) a_{(3)}
\end{equation}
defines a right \( \Acal \)-coaction on \( \Acal \), known as the \emph{right adjoint coaction on $\Acal$}.

The following convention for skew-pairing is taken from \cite{VoigtYuncken}.

\begin{defn}\label{defn:skew-pairing}
Let \( (\Ucal , \hat{\Delta} , \hat{\epsilon}, \hat{S} ) \) and \( (\Acal , \Delta, \epsilon, S) \) be Hopf algebras. A bilinear map \( ( \cdot, \cdot): \Ucal \times \Acal \to \mathbb{C} \) is called a \textbf{skew-pairing of \( \Ucal \) and \( \Acal \)} if, for \( X, Y, Z \in \Ucal \) and \( f, g, h \in \Acal \), the following conditions hold:
\begin{enumerate}[label=P\arabic*., series=P]
\item \( ( XY , f ) = ( X \otimes Y, \Delta(f) ) , \quad ( X, fg ) = ( \hat{\Delta} (X) , g \otimes f ) \)
\item \( ( X , 1_\Acal ) = \hat{\epsilon} (X), \quad ( 1_\Ucal , f ) = \epsilon (f) \)
\item \( ( \hat{S} (X) , f ) = ( X , S^{-1} (f) ), \quad ( X , S(f) ) = ( \hat{S}^{-1} (X), f ) \)
\end{enumerate}
If \( \Ucal \) and \( \Acal \) are Hopf \( * \)-algebras, we also require that
\begin{enumerate}[label=P\arabic*., resume=P]
\item \( ( X^* , f ) = \overline{( X, S (f)^* )}, \quad ( X, f^* ) = \overline{( \hat{S}^{-1} (X)^* , f )} \).
\end{enumerate}
The pairing is called \textbf{nondegenerate} if \( ( X, f ) = 0 \) for all \( f \in \Acal \) implies \( X = 0 \) and \( ( X, f ) = 0 \) for all \( X \in \Ucal \) implies \( f = 0 \).
\end{defn}

Here is an example of a skew-pairing. Let \( \phi \in \Acal^* \) and define \( \hat{\Delta}(\phi) \in (\Acal \otimes \Acal)^* \) by
\begin{equation}\label{eq:comultiplication of dual Hopf algebra}
    ( \hat{\Delta} (\phi), a \otimes b ) = ( \phi , ba ), \quad a, b \in \Acal.
\end{equation}
Recall that there is a natural embedding \( \Acal^* \otimes \Acal^* \subseteq (\Acal \otimes \Acal)^* \). Define
\begin{equation*}
    \Acal^\circ = \{ \phi \in \Acal^* \mid \hat{\Delta}(\phi) \in \Acal^* \otimes \Acal^* \}.
\end{equation*}
Then, \( \Acal^\circ \) can be endowed with a unique Hopf (\(*\)-)algebra structure \( (\Acal^\circ, \hat{\Delta}, \hat{\epsilon}, \hat{S} ) \) that makes the canonical pairing
\[
    \Acal^\circ \times \Acal \ni (\phi, a ) \longmapsto \phi(a) \in \Cbb
\]
a skew-pairing of Hopf (\(*\)-)algebras, called \textit{the dual Hopf (\(*\)-)algebra of \( \Acal \)}, compare \cite[Section~1.2.8]{Klimyk}. If \( \Acal^\circ \) separates \( \Acal \), i.e., \( ( \phi, a ) = 0 \) for all \( \phi \in \Acal^\circ \) implies \( a = 0 \), then this skew-pairing is nondegenerate.

The unital (\(*\)-)algebra structure of \( \Acal^\circ \) can be extended to the dual space \( \Acal^* \). Specifically, if we define the multiplication (and the involution) on \( \Acal^* \) by, for \( \phi, \psi \in \Acal^* \),
\begin{align*}
    (\phi \psi, a) = (\phi, a_{(1)}) (\psi, a_{(2)}), \,\, \Big(\, \text{and } (\phi^* , a) = \overline{(\phi, S(a)^*)}, \, \Big) \quad a \in \Acal,
\end{align*}
then \( \Acal^* \) becomes a (\(*\)-)algebra with unit \( \epsilon \in \Acal^* \), containing \( \Acal^\circ \) as a unital (\(*\)-)subalgebra.

Note that \( \Acal \) becomes an \( \Acal^* \)-bimodule with respect to the following left and right multiplications:
\begin{equation}\label{eq:Acirc action on A}
    \phi \triangleright a = a_{(1)} ( \phi , a_{(2)} ), \quad a \triangleleft \phi = ( \phi, a_{(1)} ) a_{(2)}, \quad a \in \Acal, \, \phi \in \Acal^*.
\end{equation}

\begin{rmk}\label{rmk:embedding via skew-pairing}
If \( (\cdot, \cdot ) : \Ucal \times \Acal \rightarrow \Cbb \) is a nondegenerate skew-pairing, then the map
    \[
    \Ucal \ni X \longmapsto (X, \,\cdot\,) \in \Acal^\circ
    \]
is a well-defined injective algebra homomorphism, and \( \Acal \) becomes a left/right \( \Ucal \)-module via the pull-back of \eqref{eq:Acirc action on A}.
\end{rmk}

On the other hand, \( \Acal^* \) is an \( \Acal \)-bimodule with the multiplications given by, for \( \phi \in \Acal^* \) and \( a \in \Acal \),
\begin{equation*}
    (a \phi, b) = (\phi, ba) , \quad (\phi a , b) = (\phi, ab), \quad b \in \Acal.
\end{equation*}

\subsection{CQGs}\label{subsec:CQGs}

\begin{defn}\label{defn:CQG}
A Hopf \( * \)-algebra \( \Acal \) is called a \textbf{compact quantum group (CQG)} if there exists a non-zero linear map \( \hcal : \Acal \rightarrow \Cbb \) satisfying, for all \( a \in \Acal \),
\begin{enumerate}[label=A\arabic*., series=A]
    \item (Positivity) \( \hcal(a^*a) \geq 0 \)
    \item (Invariance) \( (\hcal \otimes \id)\Delta(a) = \hcal(a) 1 = (\id \otimes \hcal)\Delta(a) \).
\end{enumerate}
If \( \Acal \) is a CQG, then the linear map \( \hcal \) satisfying \( \hcal(1) = 1 \) and the above conditions is unique, and will be called \textbf{the Haar state of \( \Acal \)}. The Haar state always satisfies
\begin{enumerate}[label=A\arabic*., resume=A]
    \item (Faithfulness) For \( a \in \Acal \), \( \hcal(a^* a) = 0 \) implies \( a = 0 \).
\end{enumerate}
\end{defn}

Following the convention of \cite[Section~4.2.3]{VoigtYuncken}, we shall write \( \Acal = \Cf^\infty(\Kcal) \) when \( \Acal \) is a CQG, and we refer to \( \Kcal \) as a CQG. Throughout the rest of Section~\ref{sec:Compact quantum groups (CQGs)}, \( (\Cf^\infty(\Kcal), \Delta, \epsilon, S) \) is a CQG, \( \hcal \) is its Haar state, and \( (\Cf^\infty(\Kcal)^\circ, \hat{\Delta}, \hat{\epsilon}, \hat{S}) \) is its dual Hopf \( * \)-algebra.

An element \( u \in M_n(\Cf^\infty(\Kcal)) \cong M_n (\Cbb) \otimes \Cf^\infty(\Kcal) \) is called an \emph{\( n \)-dimensional corepresentation of \( \Cf^\infty(\Kcal) \)} if \( u \) is unitary and satisfies
\[
\Delta(u_{ij}) = \sum_{k=1}^n u_{ik} \otimes u_{kj}, \quad \text{for all } 1 \leq i,j \leq n.
\]
In addition, we have \( \epsilon(u_{ij}) = \delta_{ij} \) and \( S(u_{ij}) = u_{ji}^* \) for all \( 1 \leq i,j \leq n \).

Two corepresentations \( u, v \in M_n(\Cf^\infty(\Kcal)) \) are said to be \textit{equivalent} if there exists an invertible matrix \( T \in M_n(\Cbb) \) such that
\[
TuT^{-1} = v.
\]
A corepresentation \( u \) is called \textit{irreducible} if
\[
\left\{ T \in M_n(\Cbb) \mid Tu = uT \right\} = \Cbb \cdot 1_{M_n(\Cbb)}.
\]

Let \( \Irr(\Kcal) \) denote the set of equivalence classes of irreducible corepresentations of \( \Cf^\infty(\Kcal) \), and let \( \{ u^\mu \}_{\mu \in \Irr(\Kcal)} \) be a complete set of representatives. For each \( \mu \in \Irr(\Kcal) \), we denote the dimension of the corepresentation by \( n_\mu \in \mathbb{N} \), and write its matrix elements as \( u^\mu_{ij} \) for \( 1 \leq i,j \leq n_\mu \).

\begin{prop}\label{prop:Peter-Weyl}
The set \( \{ u^\mu_{ij} \mid \mu \in \Irr(\Kcal),\ 1 \leq i,j \leq n_\mu \} \) forms a linear basis of \( \Cf^\infty(\Kcal) \). In particular,
\[
\Cf^\infty(\Kcal) = \bigoplus_{\mu \in \Irr(\Kcal)} \Span_\Cbb \{ u^\mu_{ij} \mid 1 \leq i,j \leq n_\mu \}.
\]
\end{prop}

This decomposition is called the \textit{Peter–Weyl decomposition of \( \Cf^\infty(\Kcal) \)}, and any basis of the form \( \{ u^\mu_{ij} \mid 1 \leq i,j \leq n_\mu \} \) is referred to as a \textit{Peter–Weyl basis} for \( \Cf^\infty(\Kcal) \).

A linear operator \( L : \Cf^\infty(\Kcal) \to \Cf^\infty(\Kcal) \) is said to \textit{diagonalize over the Peter–Weyl decomposition} if, for each \( \mu \in \Irr(\Kcal) \), there exists a scalar \( C_L(\mu) \in \Cbb \) such that
\[
L u^\mu_{ij} = C_L(\mu) \, u^\mu_{ij}, \quad \text{for all } 1 \leq i,j \leq n_\mu,
\]
for any Peter–Weyl basis \( \{ u^\mu_{ij} \mid \mu \in \Irr(\Kcal),\ 1 \leq i,j \leq n_\mu \} \). In this case, \( C_L(\mu) \) is called the \textit{eigenvalue of \( L \) at \( \mu \in \Irr(\Kcal) \)}.

Thanks to Proposition~\ref{prop:Peter-Weyl}, we have an isomorphism of $*$-algebras
\begin{equation}\label{eq:dual of CQG}
\Cf^\infty(\Kcal)^* \cong \prod_{\mu \in \Irr(\Kcal)} M_{n_\mu}(\Cbb),
\end{equation}
given by
\(
\phi \mapsto \Big( (\phi, u^\mu_{ij})_{1 \leq i,j \leq n_\mu} \Big)_{\mu \in \Irr(\Kcal)}.
\)
Throughout this paper, we will identify \( \Cf^\infty(\Kcal)^* \) with this direct product of matrix algebras.

The $*$-subalgebra
\[
\Dcal(\Kcal) = \bigoplus_{\mu \in \Irr(\Kcal)} M_{n_\mu}(\Cbb) \subseteq \Cf^\infty(\Kcal)^*
\]
is called the \textbf{dual quantum group of \( \Cf^\infty(\Kcal) \)}. In fact, \( \Dcal(\Kcal) \) carries the structure of an \emph{algebraic discrete quantum group}, which is in duality with the compact quantum group \( \Cf^\infty(\Kcal) \) in a precise sense, see \cite[Chapter~2]{Timmermann}. In particular, this duality implies that the map
\begin{equation}\label{eq:dual quantum group and antipode}
\Dcal(\Kcal) \ni X \longmapsto X \circ S^{\pm 1} \in \Dcal(\Kcal)
\end{equation}
is well-defined and bijective.

For each \( \mu \in \Irr(\Kcal) \), there exists a positive invertible matrix \( F_\mu \in M_{n_\mu}(\Cbb) \) such that
\begin{equation}\label{eq:Woronowicz character}
    S^2(u^\mu) = F_\mu u^\mu F_\mu^{-1} \quad \text{in } M_{n_\mu}(\Cf^\infty(\Kcal)).
\end{equation}
The following Schur orthogonality relations then hold:
\begin{equation*}
    \hcal\big( (u^\mu_{ij})^* u^\nu_{kl} \big) = \delta_{\mu\nu} \, \frac{\delta_{jl}}{\Tr(F_\mu)} (F_\mu^{-1})_{ki}, \quad 
    \hcal\big( u^\mu_{ij} (u^\nu_{kl})^* \big) = \delta_{\mu\nu} \, \frac{\delta_{ik}}{\Tr(F_\mu)} (F_\mu)_{lj}
\end{equation*}
for all \( \mu, \nu \in \Irr(\Kcal) \) and respective indices.

\subsection{Linear functionals on a Hopf \texorpdfstring{$*$}{TEXT}-algebra and a CQG}\label{subsec:Linear functionals on a CQG}

In this subsection, \( \Acal \) and \( \Cf^\infty(\Kcal) \) denote a Hopf ($*$-)algebra and a CQG, respectively. We continue to follow the notations and conventions established in the preceding two subsections.

\begin{defn}\label{defn:special linear functionals}
Let \( \phi : \Acal \rightarrow \Cbb \) be a linear functional. We say that \( \phi \) is:
\begin{itemize}
    \item \textbf{\( \boldsymbol{\ad} \)-invariant} if
    \( (\phi \otimes \id)\, \ad(a) = \phi(a) \) for all $a \in \Acal$;
    \item \textbf{self-adjoint} if \( \phi = \phi^* \) in the $*$-algebra \( \Acal^* \);
    \item \textbf{Hermitian} if \( (\phi, a^*) = \overline{(\phi, a)} \) for all \(a \in \Acal\);
    \item \textbf{conditionally positive} if \( (\phi, a^* a) \geq 0 \) for all \(a \in \Ker \epsilon \).
\end{itemize}
\end{defn}

Recall that for any \( \phi \in \Acal^* \), the left multiplication by $\phi$ is given by
\begin{equation*}
    \phi \triangleright = (\id \otimes \phi)\, \Delta \in \End(\Acal).
\end{equation*}

\begin{prop}\label{prop:semigroups generated by linear functionals}
Let \( \phi \in \Acal^* \). Then, for any \( a \in \Acal \) and \( t \geq 0 \), the series
\begin{equation*}
    T_t^\phi(a) = \sum_{n=0}^\infty \frac{(-t \phi \triangleright)^n a}{n!} = \sum_{n=0}^\infty \frac{(-t \phi)^n \triangleright a}{n!}
\end{equation*}
converges in \( \Acal \) with respect to the weak topology induced by its dual space \( \Acal^* \).

Hence, \( (T_t^\phi)_{t \geq 0} \subseteq \End(\Acal) \) defines a well-defined semigroup of linear operators on \( \Acal \), referred to as \textbf{the semigroup generated by \( \phi \triangleright \)}. We also write \( T_t^\phi = e^{-t \phi \triangleright} \).
\end{prop}

\begin{proof}
By \cite[Lemma~1.6(a)]{Franz2006}, the exponential series
\[
\exp(-t \phi) = \sum_{n=0}^\infty \frac{(-t \phi)^n}{n!}
\]
converges in \( \Acal^* \) with respect to the weak* topology. Therefore, for any \( a \in \Acal \) and \( t \geq 0 \), the series
\begin{align*}
    T_t^\phi(a) = \sum_{n=0}^\infty \frac{(-t \phi)^n \triangleright a}{n!}
    = \sum_{n=0}^\infty \frac{(\id \otimes (-t \phi)^n)\, \Delta(a)}{n!} = a_{(1)} \sum_{n=0}^\infty \frac{(-t \phi)^n(a_{(2)})}{n!}
\end{align*}
converges in \( \Acal \) with respect to the weak topology defined by \( \Acal^* \).

The semigroup property \( T_{t+s}^\phi = T_t^\phi T_s^\phi \) for all \( t, s \geq 0 \) is immediate from the exponential form.
\end{proof}

\begin{rmk}\label{rmk:linear functional and operator correspondence}
Let \( \phi \in \Acal^* \). Since \( \epsilon \circ (\phi \triangleright) = \phi \), we may regard the linear operator \( \phi \triangleright \) and the functional \( \phi \) as representing the same object. We will adopt this convention whenever it is convenient and causes no confusion. For instance, we will refer to the semigroup \( (T_t^\phi)_{t \geq 0} \) as the semigroup generated by \( \phi \).
\end{rmk}

Recall that a linear map \( T : \Bcal \rightarrow \Ccal \) between two \( * \)-algebras is called \emph{completely positive} if
\[
\sum_{1 \leq i, j \leq n } a_i^* T(b_i^* b_j) a_j \in \Ccal_+ := \{ c^* c \mid c \in \Ccal \}
\]
for all choices of \( a_j, b_j \in \Bcal \) ($1 \leq j \leq n$).

\begin{defn}\label{defn:generating functional of QMS}
Let \( \phi \in \Acal^* \). The semigroup \( (T_t^\phi)_{t \geq 0} \) is called a \textbf{quantum Markov semigroup on \( \Acal \)} if each \( T_t^\phi \) is unital and completely positive.
\end{defn}

Note, however, that in \cite{Franz2014}, the term \emph{quantum Markov semigroup} refers to the continuous extension of \( (T_t^\phi)_{t \geq 0} \) to a suitable topological completion of \( \Acal \).

\begin{prop}\label{prop:complete positivity of semigroup}
Let \( \phi \in \Acal^* \) be a Hermitian linear functional such that \( \phi(1) = 0 \). Then the semigroup \( (T_t^\phi)_{t \geq 0} = (e^{-t \phi \triangleright})_{t \geq 0} \) is a quantum Markov semigroup if and only if \( -\phi \) is conditionally positive.
\end{prop}

\begin{proof}
The condition \( \phi(1) = 0 \) ensures that each \( T_t^\phi \) is unital.

If \( (T_t^\phi)_{t \geq 0} \) is a quantum Markov semigroup, then the functionals \( \varphi_t := \epsilon \circ T_t^\phi \) are completely positive for all \( t \geq 0 \). By the Schoenberg correspondence \cite[Proposition~1.7]{Franz2006}, this implies that \( -\phi \) is conditionally positive.

Conversely, suppose that \( -\phi \) is conditionally positive. Then the Schoenberg correspondence guarantees that the functionals \( \varphi_t = \epsilon \circ T_t^\phi \) are completely positive. Since \( T_t^\phi = (\id \otimes \varphi_t) \circ \Delta \), we have, for any \( a_j, b_j \in \Acal \) ($1 \leq j \leq n$),
\[
\sum_{i,j = 1}^n a_i^* T_t^\phi(b_i^* b_j) a_j 
= (\id \otimes \varphi_t) \left( \sum_{i,j = 1}^n (a_i^* \otimes 1) \Delta(b_i^* b_j) (a_j \otimes 1) \right).
\]
Letting \( \sum_{k=1}^m x_k \otimes y_k := \sum_{j=1}^n \Delta(b_j)(a_j \otimes 1) \), this expression becomes
\[
(\id \otimes \varphi_t) \left( \sum_{k,l = 1}^m x_k^* x_l \otimes y_k^* y_l \right)
= \sum_{k,l = 1}^m x_k^* \varphi_t(y_k^* y_l) x_l,
\]
which lies in \( \Acal_+ \), since \( (\varphi_t(y_k^* y_l))_{1 \leq k,l \leq m} \in M_m(\Cbb) \) is a positive matrix.
\end{proof}

Now, we focus on $\ad$-invariant linear functionals.

\begin{prop}\label{prop:ad-invariance criterion for a linear functional}
    Let $\Ucal \subseteq \Acal^*$ be a subset that separates $\Acal$. Then, for $\phi \in \Acal^*$,
    \begin{align*}
        \text{$\phi$ is $\ad$-invariant}
        \Longleftrightarrow X \phi = \phi X \quad \text{for all } X \in \Ucal.
    \end{align*}
\end{prop}

\begin{proof}
    This follows from the equivalences:
    \begin{align*}
        \text{$\phi$ is $\ad$-invariant} &\Longleftrightarrow (\phi, a_{(2)}) S(a_{(1)}) a_{(3)} = (\phi, a) \cdot 1 ,\; \forall a \in \Acal \\
        &\Longleftrightarrow (\phi, a_{(2)}) S(a_{(1)}) = (\phi, a_{(1)}) S(a_{(2)}) ,\; \forall a \in \Acal \\
        &\Longleftrightarrow (\phi, a_{(2)}) (X, a_{(1)}) = (\phi, a_{(1)}) (X, a_{(2)}) ,\; \forall a \in \Acal,\; \forall X \in \Ucal \\
        &\Longleftrightarrow \phi X = X \phi ,\; \forall X \in \Ucal.
    \end{align*}
\end{proof}

For the remainder of this subsection, we restrict our attention to linear functionals on CQGs.

\begin{cor}\label{cor:ad-invariance criterion for a linear functional}
    Let $\phi \in \Cf^\infty(\Kcal)^*$. Then $\phi$ is $\ad$-invariant if and only if it lies in the center of the algebra $\Cf^\infty(\Kcal)^*$. Via the identification in \eqref{eq:dual of CQG}, this is equivalent to the condition that for each $\mu \in \Irr(\Kcal)$, there exists $C_\phi(\mu) \in \Cbb$, called the \textbf{eigenvalue of $\phi$ at $\mu$}, such that
    \[
        (\phi, u^\mu_{ij}) = C_\phi(\mu)\, \delta_{ij}, \quad 1 \leq i,j \leq n_\mu.
    \]
    Moreover, $\phi$ is self-adjoint if and only if $C_\phi(\mu) \in \Rbb$ for all $\mu \in \Irr(\Kcal)$.
\end{cor}

\begin{proof}
    Consider the dual quantum group $\Dcal(\Kcal) \subseteq \Cf^\infty(\Kcal)^*$, which separates $\Cf^\infty(\Kcal)$. Applying Proposition~\ref{prop:ad-invariance criterion for a linear functional} with $\Ucal = \Dcal(\Kcal)$ and $\Acal = \Cf^\infty(\Kcal)$ yields
    \[
        \text{$\phi$ is $\ad$-invariant} \Longleftrightarrow X \phi = \phi X ,\; \forall X \in \Dcal(\Kcal).
    \]
    By the definition of $\Dcal(\Kcal)$ and \eqref{eq:dual of CQG}, this is equivalent to centrality of $\phi$ in $\Cf^\infty(\Kcal)^*$. The rest follow from \eqref{eq:dual of CQG}.
\end{proof}

\begin{cor}\label{cor:diagonalization of translation invariant linear operators}
    Let $\phi \in \Cf^\infty(\Kcal)^*$ be $\ad$-invariant. Then the operator $\phi \triangleright \in \End(\Cf^\infty(\Kcal))$ diagonalizes over the Peter–Weyl decomposition, with eigenvalues given by $(C_\phi(\mu))_{\mu \in \Irr(\Kcal)}$. In this case, we also have
    \begin{equation}\label{eq:right translation invariance}
        \phi \triangleright = \triangleleft \phi := (\phi \otimes \id) \circ \Delta.
    \end{equation}
    Conversely, if $L \in \End(\Cf^\infty(\Kcal))$ diagonalizes over the Peter–Weyl decomposition, then $\phi := \epsilon \circ L \in \Cf^\infty(\Kcal)^*$ is $\ad$-invariant and $L = \phi \triangleright$.
\end{cor}

\begin{proof}
    For the first part, let $\mu \in \Irr(\Kcal)$ and $1 \leq i,j \leq n_\mu$. Then
    \begin{align*}
        \phi \triangleright (u^\mu_{ij}) 
        &= (\id \otimes \phi) \Delta(u^\mu_{ij}) 
        = \sum_{k=1}^{n_\mu} u^\mu_{ik} \phi(u^\mu_{kj}) 
        = c_\phi(\mu) u^\mu_{ij}.
    \end{align*}

    For the converse, suppose $L$ diagonalizes with eigenvalues $(C_L(\mu))_{\mu \in \Irr(\Kcal)}$. Define $\phi := \epsilon \circ L$. Then for all $\mu$ and $i,j$,
    \[
        \phi(u^\mu_{ij}) = \epsilon(L(u^\mu_{ij})) = c_L(\mu) \epsilon(u^\mu_{ij}) = C_L(\mu) \delta_{ij}.
    \]
    Hence, by Corollary~\ref{cor:ad-invariance criterion for a linear functional}, $\phi$ is $\ad$-invariant, and a calculation like the above confirms that $L = \phi \triangleright$.
\end{proof}

Thus, for $\phi \in \Cf^\infty(\Kcal)^*$, $\ad$-invariance is equivalent to the operator $\phi \triangleright$ being diagonal with respect to the Peter–Weyl decomposition, with matching eigenvalues. These eigenvalues are real precisely when $\phi$ is self-adjoint.

We now seek a condition on a self-adjoint, $\ad$-invariant functional $\phi \in \Cf^\infty(\Kcal)^*$ under which the operator $\phi \triangleright$ commutes with the antipode.

\begin{prop}\label{prop:self-adjointness and ad-invariance are preserved under antipode}
Let $\phi$ be an $\ad$-invariant linear functional on $\Cf^\infty(\Kcal)$. Then $\phi \circ S^{\pm 1} \in \Cf^\infty(\Kcal)^*$ is also $\ad$-invariant. Moreover, if $\phi$ is also self-adjoint, then so is $\phi \circ S^{\pm 1}$.
\end{prop}

\begin{proof}
By Corollary~\ref{cor:ad-invariance criterion for a linear functional}, we have $\phi X = X \phi$ for all $X \in \Dcal(\Kcal)$. Hence, for any $f \in \Cf^\infty(\Kcal)$,
\begin{align*}
((\phi \circ S^{\pm 1})(X \circ S^{\pm 1}), f) 
= (X \phi, S^{\pm 1}(f)) = (\phi X, S^{\pm 1}(f)) \\
= ((X \circ S^{\pm 1})(\phi \circ S^{\pm 1}), f).
\end{align*}
Thus, $\phi \circ S^{\pm 1}$ commutes with all $X \in \Dcal(\Kcal)$ (cf. \eqref{eq:dual quantum group and antipode}) and is therefore $\ad$-invariant.

Now suppose $\phi$ is self-adjoint. Then for all $f \in \Cf^\infty(\Kcal)$,
\begin{align*}
((\phi \circ S^{\pm 1})^*, f) 
= \overline{(\phi \circ S^{\pm 1}, S(f)^*)} = (\phi^*, S^{\mp 1}(f)) = (\phi, S^{\mp 1}(f)) \\
= (\phi \circ S^{\pm 1}, S^{\mp 2}(f)).
\end{align*}
Applying the Woronowicz character identity~\eqref{eq:Woronowicz character}, we obtain
\begin{align*}
((\phi \circ S^{\pm 1})^*, u^\mu_{ij}) 
= \sum_{k,l} (F_\mu^{\mp 1})_{ik} (\phi \circ S^{\pm 1}, u^\mu_{kl}) (F_\mu^{\pm 1})_{lj} = (\phi \circ S^{\pm 1}, u^\mu_{ij}),
\end{align*}
where the final equality follows from the fact that $(\phi \circ S^{\pm 1}, u^\mu_{kl}) = C_{\phi \circ S^{\pm 1}}(\mu) \delta_{kl}$. Hence, $\phi \circ S^{\pm 1}$ is self-adjoint.
\end{proof}

\begin{cor}\label{cor:self-adjointness and Hermiticity for ad-invariant functional}
Let $\phi$ be a self-adjoint, $\ad$-invariant linear functional on $\Cf^\infty(\Kcal)$. Then, for all $f \in \Cf^\infty(\Kcal)$,
\[
(\phi S, f) = \overline{(\phi, f^*)}.
\]
In particular, $\phi$ is Hermitian if and only if $\phi = \phi S$.
\end{cor}

\begin{proof}
By Proposition~\ref{prop:self-adjointness and ad-invariance are preserved under antipode},
\(
(\phi S, f) = ((\phi S)^*, f) = \overline{(\phi S, S(f)^*)} = \overline{(\phi, f^*)}.
\)
\end{proof}

\begin{prop}\label{prop:left multiplication operator commutes with antipode}
Let $\phi \in \Cf^\infty(\Kcal)^*$ be self-adjoint and $\ad$-invariant. Then the operator $\phi \triangleright$ commutes with the antipode if and only if $\phi$ is Hermitian.
\end{prop}

\begin{proof}
We compute:
\[
(\phi \triangleright) \circ S = (\id \otimes \phi)\Delta S = S \circ (\phi S \otimes \id)\Delta = S \circ (\phi S \triangleright),
\]
where the final equality follows from Proposition~\ref{prop:self-adjointness and ad-invariance are preserved under antipode} and the identity~\eqref{eq:right translation invariance}. Thus, $(\phi \triangleright) \circ S = S \circ (\phi \triangleright)$ if and only if $\phi = \phi S$, which, by Corollary~\ref{cor:self-adjointness and Hermiticity for ad-invariant functional}, is equivalent to $\phi$ being Hermitian.
\end{proof}

In Section~\ref{subsec:Laplacians on a CQG}, we will see that those operators on $\Cf^\infty(\Kcal)$ which diagonalize with real eigenvalues over the Peter–Weyl decomposition, commute with the antipode, and vanish at the unit are precisely those we interpret as \emph{Laplacians} on a CQG. Proposition~\ref{prop:left multiplication operator commutes with antipode} shows that every such operator arises as $\phi \triangleright$, for a self-adjoint, $\ad$-invariant, Hermitian functional $\phi$ that vanishes at the unit.

\section{First-order differential calculi (FODCs) over a Hopf (\texorpdfstring{$*$}{TEXT}-)algebra}\label{sec:FODC}

In this section, we gather some results from the theory of first-order differential calculus (FODC) over Hopf ($*$-)algebras \cite{Woronowicz1989}. Throughout, we let \( (\Acal, \Delta, \epsilon, S) \) denote a Hopf ($*$-)algebra, and \( (\Acal^\circ, \hat{\Delta}, \hat{\epsilon}, \hat{S}) \) its dual Hopf ($*$-)algebra.

\subsection{Bicovariant (\texorpdfstring{$*$}{TEXT}-)Bimodules}\label{subsec:Bicovariant bimodule}

Differential calculi that reflect the Hopf algebra structure of $\Acal$ are referred to as \textit{bicovariant}. To introduce this notion, we first define the concept of a \textit{bicovariant bimodule}.

Let $\Omega$ be an $\Acal$-bimodule. We equip $\Acal \otimes \Omega$ with the $\Acal$-bimodule structure given by
\[
a \cdot ( b \otimes \omega) = a_{(1)} b \otimes a_{(2)} \omega, \quad (b \otimes \omega) \cdot a = b a_{(1)} \otimes \omega a_{(2)},
\]
for $a, b \in \Acal$ and $\omega \in \Omega$. Similarly, $\Omega \otimes \Acal$ is endowed with the analogous bimodule structure.

\begin{defn}\label{defn:bicovariant bimodule}
    Let $\Omega$ be an $\Acal$-bimodule. Then, $(\Omega , \Phi_\Omega, {}_\Omega \Phi)$ is called a \textbf{bicovariant bimodule over $\Acal$} if:
    \begin{enumerate}[label=B\arabic*., series=B]
        \item $\Phi_\Omega : \Omega \rightarrow  \Acal \otimes \Omega$ is a left $\Acal$-coaction and an $\Acal$-bilinear map;
        \item ${}_\Omega \Phi : \Omega \rightarrow \Omega \otimes \Acal$ is a right $\Acal$-coaction and an $\Acal$-bilinear map;
        \item The coactions are compatible: $(\Phi_\Omega \otimes \id) \circ {}_\Omega \Phi = (\id \otimes {}_\Omega \Phi) \circ \Phi_\Omega$.
    \end{enumerate}
\end{defn}

Given a bicovariant bimodule $(\Omega, \Phi_\Omega , {}_\Omega \Phi)$ over $\Acal$, define
\[
\inv \Omega = \{ \omega \in \Omega \mid \Phi_{\Omega} (\omega) = 1 \otimes \omega \},
\]
the subspace of \textit{left-invariant elements}. This space becomes a right $\Acal$-module under the multiplication
\[
\omega \cdot a = S(a_{(1)} ) \omega a_{(2)}, \quad \omega \in \inv \Omega, \, a \in \Acal,
\]
which is well-defined due to the $\Acal$-bilinearity of $\Phi_\Omega$. To distinguish this action from the original right $\Acal$-action on $\Omega$, we will always use the dot $\cdot$ notation.

For all $\omega \in \inv \Omega$, we compute
\[
(\Phi_{\Omega} \otimes \id) {}_\Omega \Phi ( \omega) = (\id \otimes {}_\Omega \Phi ) \Phi_\Omega (\omega) = 1 \otimes {}_\Omega \Phi (\omega),
\]
which shows that ${}_\Omega \Phi (\inv \Omega) \subseteq \inv \Omega \otimes \Acal$. Thus, ${}_\Omega \Phi$ restricts to a right $\Acal$-coaction on $\inv \Omega$, denoted by ${}_{\inv \Omega} \Phi$.

The following two results are \cite[Theorems~2.4--2.5]{Woronowicz1989}. Note that the first identity in \eqref{eq:properties of f_ij} follows from our convention \eqref{eq:comultiplication of dual Hopf algebra}.

\begin{prop}\label{prop:the structure representations of bicovariant bimodule}
    Let $(\Omega, \Phi_\Omega , {}_\Omega \Phi )$ be a bicovariant bimodule over $\Acal$. Then:
    \begin{enumerate}[series=STBCBM]
        \item $\Omega$ is a free left and right $\Acal$-module. Specifically, the multiplication maps restrict to isomorphisms
        \[
        \Acal \otimes \inv \Omega \cong \Omega ,\quad \inv \Omega \otimes \Acal \cong \Omega.
        \]
    \end{enumerate}
    Fix a $\Cbb$-linear basis $\{ \omega_i \mid i \in I \}$ of $\inv \Omega$.
    \begin{enumerate}[resume=STBCBM]
        \item There exists a unique family $(f_{ij})_{i,j \in I} \subseteq \Acal^\circ$ such that for each $i$, only finitely many $f_{ij}$ are nonzero, and
        \begin{equation}\label{eq:definition of f_ij}
            \omega_i a = \sum_{j \in I} (f_{ij} \triangleright a) \omega_j, \quad a \omega_i = \sum_{j \in I} \omega_j ( \hat{S}(f_{ij}) \triangleright a),
        \end{equation}
        for all $a \in \Acal$. These elements satisfy
        \begin{equation}\label{eq:properties of f_ij}
            \hat{\Delta}(f_{ij}) = \sum_{k \in I} f_{kj} \otimes f_{ik}, \quad \hat{\epsilon}(f_{ij}) = \delta_{ij}.
        \end{equation}

        \item There exists a unique family $(t_{ij})_{i,j \in I} \subseteq \Acal$ such that
        \begin{equation}\label{eq:definition of R_ij}
            {}_{\inv \Omega} \Phi (\omega_j) = \sum_{i \in I} \omega_i \otimes t_{ij}.
        \end{equation}
        These satisfy, for all $i,j \in I$:
        \[
        \Delta(t_{ij}) = \sum_{k \in I} t_{ik} \otimes t_{kj}, \quad \epsilon(t_{ij}) = \delta_{ij},
        \]
        and for all $a \in \Acal$:
        \begin{equation*}
            \sum_{k \in I} t_{ki} (a \triangleleft f_{kj}) = \sum_{k \in I} (f_{ik} \triangleright a) t_{jk}.
        \end{equation*}
    \end{enumerate}
    The families $(f_{ij})$ and $(t_{ij})$ are called the \textbf{structure representations of $\Omega$ with respect to the invariant basis $\{\omega_i\}$}.
\end{prop}

\begin{prop}\label{prop:the structure theorem of bicovariant bimodule}
    Suppose an index set $I$ and families $(f_{ij})_{i,j \in I} \subseteq \Acal^\circ$, $(t_{ij})_{i,j \in I} \subseteq \Acal$ satisfy:
    \begin{enumerate}[label=S\arabic*.]
        \item $\hat{\Delta}(f_{ij}) = \sum_{k \in I} f_{kj} \otimes f_{ik}$, \; $\hat{\epsilon}(f_{ij}) = \delta_{ij}$;
        \item $\Delta(t_{ij}) = \sum_{k \in I} t_{ik} \otimes t_{kj}$, \; $\epsilon(t_{ij}) = \delta_{ij}$;
        \item $\sum_{k \in I} t_{ki} (a \triangleleft f_{kj}) = \sum_{k \in I} (f_{ik} \triangleright a) t_{jk}$ for all $a \in \Acal$.
    \end{enumerate}
    Then the free left $\Acal$-module with basis $\{ \omega_i \}$, equipped with the right action from \eqref{eq:definition of f_ij} and coactions
    \[
    \Phi_\Omega \left( \sum_{j \in I} a_j \omega_j \right) = \sum_{j \in I} \Delta(a_j)(1 \otimes \omega_j), \quad {}_\Omega \Phi \left( \sum_{j \in I} a_j \omega_j \right) = \sum_{i,j \in I} \Delta(a_j)(\omega_i \otimes t_{ij}),
    \]
    is a bicovariant bimodule over $\Acal$ with structure representations $(f_{ij})_{i,j \in I}$, $(t_{ij})_{i,j \in I}$.
\end{prop}

\begin{prop}\label{prop:tensor product of structure maps}
    Let $I$ and $J$ be index sets. Suppose
    \[
        \big((f_{ij})_{i,j \in I}, (t_{ij})_{i,j \in I} \big) \text{ and } \big((g_{kl})_{k,l \in J}, (s_{kl})_{k,l \in J} \big)
    \]
    define bicovariant bimodules over $\Acal$ with respect to invariant bases as in Proposition~\ref{prop:the structure theorem of bicovariant bimodule}. Then the families
    \[
    ( f_{ij} g_{kl})_{(i,k), (j,l)} \subseteq \Acal^\circ, \quad (t_{ij} s_{kl})_{(i,k), (j,l)} \subseteq \Acal
    \]
    satisfy conditions S1--S3 from Proposition~\ref{prop:the structure theorem of bicovariant bimodule}. Therefore, the free left $\Acal$-module $\Omega$ with basis $\{ \omega_{ik} \mid (i,k) \in I \times J \}$, equipped with
    \[
    \omega_{ik} a = \sum_{j,l} (f_{ij} g_{kl} \triangleright a) \omega_{jl}, \quad
    \Phi_\Omega(\omega_{jl}) = 1 \otimes \omega_{jl}, \quad
    {}_\Omega \Phi(\omega_{jl}) = \sum_{i,k} \omega_{ik} \otimes t_{ij} s_{kl},
    \]
    defines a bicovariant bimodule over $\Acal$.
\end{prop}
\begin{proof}
For S1, observe that for $(i,k), (j,l) \in I \times J$,
\[
\hat{\epsilon}(f_{ij} g_{kl}) = \hat{\epsilon}(f_{ij}) \hat{\epsilon}(g_{kl}) = \delta_{ij} \delta_{kl} = \delta_{(i,k), (j,l)},
\]
and
\begin{align*}
\hat{\Delta}(f_{ij} g_{kl}) = \hat{\Delta}(f_{ij}) \hat{\Delta}(g_{kl}) = \left( \sum_{p \in I} f_{pj} \otimes f_{ip} \right) \left( \sum_{q \in J} g_{ql} \otimes g_{kq} \right) \\
= \sum_{(p,q) \in I \times J} (f_{pj} g_{ql}) \otimes (f_{ip} g_{kq}).
\end{align*}
The verification of S2 proceeds analogously.

For S3, note that for $a \in \Acal$ and $(i,k), (j,l) \in I \times J$,
\begin{align*}
\sum_{(p,q) \in I \times J} t_{pi} s_{qk} \big( a \triangleleft (f_{pj} g_{ql}) \big)
&= \sum_{(p,q) \in I \times J} t_{pi} s_{qk} \big( (a \triangleleft f_{pj}) \triangleleft g_{ql} \big) \\
&= \sum_{(p,q) \in I \times J} t_{pi} \big( g_{kq} \triangleright (a \triangleleft f_{pj}) \big) s_{lq} \\
&= \sum_{(p,q) \in I \times J} t_{pi} (a_{(1)} \triangleleft f_{pj}) (g_{kq}, a_{(2)}) s_{lq} \\
&= \sum_{(p,q) \in I \times J} (f_{ip} \triangleright a_{(1)}) t_{jp} (g_{kq}, a_{(2)}) s_{lq} \\
&= \sum_{(p,q) \in I \times J} \big( (f_{ip} g_{kq}) \triangleright a \big) t_{jp} s_{lq}.
\end{align*}
\end{proof}

We now incorporate $*$-structures into this framework.

\begin{defn}\label{defn:bicovariant *-bimodule}
    Let $(\Omega, \Phi_\Omega, {}_\Omega \Phi)$ be a bicovariant bimodule over $\Acal$, and let $* : \Omega \rightarrow \Omega$ be a conjugate linear involution, i.e., $*^2 = \id$. Equip $\Acal \otimes \Omega$ and $\Omega \otimes \Acal$ with the tensor product $*$-structures. Then, $(\Omega, *, \Phi_\Omega, {}_\Omega \Phi)$ is called a \textbf{bicovariant $*$-bimodule over $\Acal$} if:
    \begin{enumerate}[label=B\arabic*., resume=B]
        \item For all $a \in \Acal$ and $\omega \in \Omega$, one has $(a \omega)^* = \omega^* a^*$ and $(\omega a)^* = a^* \omega^*$;
        \item Both $\Phi_\Omega$ and ${}_\Omega \Phi$ are $*$-preserving.
    \end{enumerate}
\end{defn}

\subsection{Bicovariant (\texorpdfstring{$*$}{TEXT}-)FODCs}\label{subsec:Bicovariant FODC}

\begin{defn}\label{defn:FODC}
Let $\Omega$ be an $\Acal$-bimodule and let $d: \Acal \rightarrow \Omega$ be a complex-linear map. Then $(\Omega, d)$ is called a \textbf{first order differential calculus (FODC) over $\Acal$} if the following conditions hold:
    \begin{enumerate}[label=F\arabic*., series=F]
        \item \textbf{(Leibniz rule)} For all $a, b \in \Acal$, we have
        \(
        d(ab) = (da)b + a (db).
        \)
        \item \textbf{(Standard form)} Every element $\omega \in \Omega$ can be written as
        \(
        \omega = \sum_{j=1}^k a_j \, db_j.
        \)
    \end{enumerate}
If, in addition, $(\Omega, \Phi_\Omega, {}_\Omega \Phi)$ is a bicovariant bimodule such that
    \begin{enumerate}[label=F\arabic*., resume=F]
        \item $\Phi_{\Omega} \circ d = (\id \otimes d)\Delta$ and ${}_\Omega \Phi \circ d = (d \otimes \id)\Delta$,
    \end{enumerate}
then $(\Omega, d, \Phi_\Omega, {}_\Omega \Phi)$ is called a \textbf{bicovariant FODC over $\Acal$}.

Let $*: \Omega \rightarrow \Omega$ be an involution. Then, $(\Omega, *, d)$ is called a \textbf{$*$-FODC over $\Acal$} if
    \begin{enumerate}[label=F\arabic*., resume=F]
        \item $d(a^*) = (da)^*$ for all $a \in \Acal$.
    \end{enumerate}
If $(\Omega, *, \Phi_\Omega, {}_\Omega \Phi)$ is a bicovariant $*$-bimodule and $d : \Acal \rightarrow \Omega$ is any linear map satisfying F1--F4, then $(\Omega, *, d, \Phi_\Omega, {}_\Omega \Phi)$ is called a \textbf{bicovariant $*$-FODC over $\Acal$}.
\end{defn}

We will often write $(\Omega, d)$ to denote a ($*$-)bicovariant FODC when the other structure maps are understood. When $\Acal = \Cf^\infty(\Kcal)$ for a CQG $\Kcal$, we say that \textit{$(\Omega,d)$ is an FODC on $\Kcal$}, meaning that it is an FODC over $\Cf^\infty(\Kcal)$.

Let $(\Omega, d)$ be a bicovariant FODC over $\Acal$. Then, the \textit{dimension of $(\Omega,d)$} is defined as the dimension of the complex vector space $\inv \Omega$.

Let $(\Omega, d)$ and $(\Omega', d')$ be bicovariant ($*$-)FODCs over $\Acal$. We say that they are \textit{isomorphic} if there exists an ($*$-preserving) $\Acal$-bimodule isomorphism $\varphi: \Omega \rightarrow \Omega'$ that intertwines the left and right coactions and satisfies $\varphi \circ d = d'$.

\begin{rmk}\label{rmk:constant multiple of FODC}
Let $(\Omega, d)$ be a bicovariant FODC and let $0 \neq \lambda \in \Cbb$. Define $d': \Acal \rightarrow \Omega$ by $d'(a) := \lambda \, da$. Then it is easily checked that $(\Omega, d')$ is also a bicovariant FODC. In fact, $(\Omega,d)$ and $(\Omega,d')$ are isomorphic via the map
\[
\varphi: \Omega \ni \omega \longmapsto \lambda \omega \in \Omega.
\]
The same conclusion holds for $*$-FODCs when $\lambda \in \Rbb \setminus \{0\}$.
\end{rmk}

In \cite{Woronowicz1989}, Woronowicz introduced a construction that generates all bicovariant ($*$-)FODCs up to isomorphism. Recall that the right adjoint coaction \eqref{eq:right adjoint coaction} defines a right $\Acal$-coaction on $\Acal$. A subspace $R \subseteq \Acal$ is called \textit{$\ad$-invariant} if
\[
\ad(R) \subseteq R \otimes \Acal.
\]
For example, $\Ker \epsilon$ is $\ad$-invariant.

\begin{prop}[Model bicovariant ($*$-)FODC]\label{prop:model BCFODC}
Let $R$ be an $\ad$-invariant right ideal in $\Acal$ such that $R \subseteq \Ker \epsilon$. Define
\[
\Omega_R := \Acal \otimes (\Ker \epsilon / R),
\]
and let $\pi_R: \Ker \epsilon \rightarrow \Ker \epsilon / R$ be the canonical projection. Then $\Omega_R$ becomes a bicovariant FODC with the following structure maps:

\begin{enumerate}[label=M\arabic*., series=M]
    \item \textbf{($\Acal$-bimodule structure)} For $a, b \in \Acal$ and $c \in \Ker \epsilon$,
    \begin{equation}\label{eq:BCFODC-bimodule}
    a (b \otimes \pi_R(c)) = ab \otimes \pi_R(c), \quad (b \otimes \pi_R(c)) a = b a_{(1)} \otimes \pi_R(c a_{(2)}).
    \end{equation}
    
    \item \textbf{(Differential)} For $a \in \Acal$,
    \[
    d_R a = a_{(1)} \otimes \pi_R \big( a_{(2)} - \epsilon(a_{(2)})1 \big).
    \]
    
    \item \textbf{(Left coaction)}
    \[
    \Phi_{\Omega_R} := \Delta \otimes \id : \Omega_R = \Acal \otimes (\Ker \epsilon / R) \rightarrow \Acal \otimes \Omega_R.
    \]
    
    \item \textbf{(Right coaction)} For $a \in \Acal$ and $b \in \Ker \epsilon$,
    \[
    {}_{\Omega_R} \Phi(a \otimes \pi_R(b)) = a_{(1)} \otimes \pi_R(b_{(2)}) \otimes a_{(2)} S(b_{(1)}) b_{(3)}.
    \]
\end{enumerate}
We have $\inv \Omega_R = 1 \otimes (\Ker \epsilon / R)$.

If, in addition, $*S(R) \subseteq R$, then $(\Omega_R, d_R)$ becomes a bicovariant $*$-FODC with involution given by
\begin{enumerate}[label=M\arabic*., resume=M]
    \item \textbf{(Involution)} For $a \in \Acal$ and $b \in \Ker \epsilon$,
    \begin{equation}\label{eq:LCFODC-involution}
    (a \otimes \pi_R(b))^* = -a_{(1)}^* \otimes \pi_R \big(S(b)^* a_{(2)}^* \big).
    \end{equation}
\end{enumerate}
\end{prop}
\begin{proof}
See \cite[Sections~6.1--6.3 and 11.2]{Sontz}.
\end{proof}

The next theorem asserts that the construction in Proposition~\ref{prop:model BCFODC} establishes a bijective correspondence between $\ad$-invariant right ideals in $\Ker \epsilon$ and isomorphism classes of bicovariant FODCs.

\begin{thm}\label{thm:BCFODC}
Every bicovariant FODC over $\Acal$ is isomorphic to $(\Omega_R, d_R)$ for some $\ad$-invariant right ideal $R \subseteq \Ker \epsilon$.

Moreover, $(\Omega_R, d_R)$ and $(\Omega_{R'}, d_{R'})$ are isomorphic if and only if $R = R'$.

The same holds for bicovariant $*$-FODCs, provided we require $*S(R) \subseteq R$.
\end{thm}
\begin{proof}
See \cite[Theorems~6.10 and 11.5]{Sontz}. For the uniqueness, note that if $\varphi: \Omega_R \rightarrow \Omega_{R'}$ is an isomorphism, then for all $a \in \Ker \epsilon$,
\begin{align*}
\varphi(1 \otimes \pi_R(a)) = \varphi(S(a_{(1)}) d_R a_{(2)}) = S(a_{(1)}) \varphi(d_R a_{(2)}) = S(a_{(1)}) d_{R'} a_{(2)} \\
= 1 \otimes \pi_{R'}(a),
\end{align*}
which, in light of the fact that $\varphi$ is an isomorphism, implies $R = R'$.
\end{proof}

Accordingly, when a bicovariant FODC $(\Omega,d)$ is isomorphic to some $(\Omega_R, d_R)$ as above, we call $R$ \textit{the right ideal corresponding to $\Omega$}.

\subsection{Quantum germs map and left-invariant vector fields (LIVFs)}\label{subsec:Quantum germs map and left-invariant vector fields}

Let $R$ be an $\ad$-invariant right ideal contained in $\Ker \epsilon$ and consider the FODC $(\Omega_R, d_R)$ of Proposition~\ref{prop:model BCFODC}. Note that
\[
S(a_{(1)})\, d_R a_{(2)} = 1 \otimes (\epsilon \otimes \id)\, d_R(a) = 1 \otimes \pi_R ( a - \epsilon(a) )
\]
for $a \in \Acal$. We now generalize this map to arbitrary bicovariant FODCs.

\begin{prop}\label{prop:quantum germs map}
Let $(\Omega, d )$ be a bicovariant FODC over $\Acal$. Then,
\[
Q : \Acal \ni a \longmapsto S(a_{(1)})\, d a_{(2)} \in \inv \Omega
\]
is well-defined and surjective; this map is called \textbf{the quantum germs map of $(\Omega,d)$}. It satisfies
\begin{equation}\label{eq:basic identity of quantum germs map}
Q(ab) = Q(a) \cdot b + \epsilon(a) Q(b), \quad a, b \in \Acal.
\end{equation}
Moreover, if $(\Omega,d)$ is a bicovariant $*$-FODC, then
\begin{equation}\label{eq:involution and quantum germs map}
Q(a)^* = -Q(S(a)^*), \quad a \in \Acal.
\end{equation}
\end{prop}

\begin{proof}
In the case of $(\Omega_R,d_R)$ from Proposition~\ref{prop:model BCFODC}, we compute:
\begin{equation}\label{eq:quantum germs map for model BCFODC}
Q(a) = 1 \otimes \pi_R (a - \epsilon(a)) \in \inv (\Omega_R) \cong \Ker \epsilon / R, \quad a \in \Acal,
\end{equation}
which is clearly surjective. The general statement follows from Theorem~\ref{thm:BCFODC}.

The identities \eqref{eq:basic identity of quantum germs map} and \eqref{eq:involution and quantum germs map} can be verified by direct computation, see \cite[Proposition~6.7 and Proposition~11.6]{Sontz}.
\end{proof}

Let $(\Omega, d)$ be a bicovariant FODC with quantum germs map $Q$. Then,
\[
d a = a_{(1)} Q(a_{(2)}), \quad a \in \Acal,
\]
and by \eqref{eq:quantum germs map for model BCFODC}, the right ideal $R$ corresponding to $(\Omega,d)$ is given by
\begin{equation}\label{eq:right ideal corresponding to BCFODC}
R = \Ker \epsilon \cap \Ker Q.
\end{equation}

\begin{defn}\label{defn:LIVF}
Let $(\Omega, d)$ be a bicovariant FODC over $\Acal$ and let $R$ be the right ideal corresponding to it. Then, elements of the subspace
\begin{equation}\label{eq:LIVF}
\Xcal_R = \{ X \in \Acal^* \mid \forall a \in R + \Cbb 1, \, (X, a) = 0 \}
\end{equation}
are called \textbf{left-invariant vector fields (LIVFs) for the FODC $(\Omega, d)$}.
\end{defn}

Note that $\Xcal_R \subseteq \Acal^\circ$ by \cite[Proposition~6.10]{Sontz}. Also, since
\begin{equation}\label{eq:R from XcalR}
R = \Ker \epsilon \cap (R + \Cbb 1) = \{ a \in \Ker \epsilon \mid \forall X \in \Xcal_R, \, (X, a) = 0 \},
\end{equation}
we see that $\Xcal_R = \Xcal_{R'}$ implies $R = R'$.

\begin{prop}\label{prop:LIVF basis}
Let $(\Omega, d)$ be a bicovariant FODC over $\Acal$ and let $\{ \omega_i \mid i \in I \}$ be a linear basis of $\inv \Omega$. Then, the functionals $X_i \in \Acal^*$ defined by
\begin{equation}\label{eq:LIVF and quantum germs map}
Q(a) = \sum_i (X_i , a)\, \omega_i, \quad a \in \Acal,
\end{equation}
where $Q$ is the quantum germs map of $(\Omega,d)$, are LIVFs for $(\Omega, d)$. These are linearly independent and, if $I$ is finite, form a linear basis of the space of LIVFs.

Moreover, for all $a \in \Acal$, we have
\begin{equation}\label{eq:differential and LIVF}
d a = \sum_{i \in I} (X_i \triangleright a)\, \omega_i.
\end{equation}
\end{prop}

\begin{proof}
Let $R$ be the right ideal corresponding to $(\Omega,d)$. Since $Q(1) = 0$, we have $(X_i, 1) = 0$ for all $i \in I$. By \eqref{eq:right ideal corresponding to BCFODC}, each $X_i$ annihilates $R$, so $X_i \in \Xcal_R$.

Because $Q$ is surjective, for each $j \in I$ we can find $a_j \in \Acal$ such that $Q(a_j) = \omega_j$, i.e.\ $(X_i, a_j) = \delta_{ij}$. Thus, $\{ X_i \mid i \in I \}$ is linearly independent.

Now, suppose $I$ is finite and let $X \in \Xcal_R$. If $a \in \Acal$ satisfies $(X_i, a) = 0$ for all $i \in I$, then $Q(a) = 0$, so $a = \big( a - \epsilon(a) \big) + \epsilon(a) \in R + \Cbb 1$ by \eqref{eq:right ideal corresponding to BCFODC}. Hence $\bigcap_{i \in I} \Ker X_i \subseteq \Ker X$, implying $X \in \Span_\Cbb \{ X_i \mid i \in I \}$. Therefore, $\Xcal_R = \Span_\Cbb \{ X_i \mid i \in I \}$.

To verify \eqref{eq:differential and LIVF}, let $a \in \Acal$ and compute:
\[
d a = a_{(1)} Q(a_{(2)}) = \sum_i a_{(1)} (X_i, a_{(2)})\, \omega_i = \sum_i (X_i \triangleright a)\, \omega_i.
\]
\end{proof}

\begin{cor}\label{cor:dimension encoded in the dual space}
Let $(\Omega, d)$ be a bicovariant FODC over $\Acal$ and let $\Xcal$ be the space of LIVFs for it. Then, $(\Omega,d)$ is finite-dimensional if and only if $\Xcal$ is a finite-dimensional $\Cbb$-vector space, in which case $\dim \Xcal = \dim \inv \Omega$.
\end{cor}

\begin{prop}\label{prop:*-structure encoded in the dual space}
Let $(\Omega,d)$ be a bicovariant FODC over $\Acal$ and let $\Xcal$ be the space of LIVFs for it. Then, $(\Omega, d)$ can be made a bicovariant $*$-FODC if and only if $*(\Xcal) \subseteq \Xcal$.
\end{prop}

\begin{proof}
Without loss of generality, assume $(\Omega,d) = (\Omega_R, d_R)$ where $R$ is the right ideal corresponding to $(\Omega, d)$.

By Proposition~\ref{prop:model BCFODC} and Theorem~\ref{thm:BCFODC}, the $*$-structure on $\Omega_R$ exists if and only if $*S(R) \subseteq R$. Since $(X^*, a) = \overline{(X, S(a)^*)}$ for all $X \in \Acal^*$ and $a \in \Acal$, this condition is equivalent to $*(\Xcal_R) \subseteq \Xcal_R$ by \eqref{eq:LIVF}--\eqref{eq:R from XcalR}.
\end{proof}

\subsection{Direct sum of bicovariant (\texorpdfstring{$*$}{TEXT}-)FODCs}\label{subsec:Direct sum of bicovariant FODCs}

Recall that, given left $\Acal$-coactions $\Phi_l : V_l \rightarrow \Acal \otimes V_l$ ($1\leq l \leq m$), the map
\[
\Phi : V_1 \oplus \cdots \oplus V_m \xrightarrow{\Phi_1 \oplus \cdots \oplus \Phi_m } (\Acal \otimes V_1) \oplus \cdots \oplus (\Acal \otimes V_m) \cong \Acal \otimes (V_1 \oplus \cdots \oplus V_m)
\]
is a left $\Acal$-coaction called the \emph{direct sum of $\Phi_1, \ldots, \Phi_m$}, and is denoted by $\Phi = \Phi_1 \oplus \cdots \oplus \Phi_m$. Direct sums of right $\Acal$-coactions are defined analogously.

\begin{prop}\label{prop:direct sum of structure maps}
    Let $\Omega_1 , \ldots , \Omega_m$ be bicovariant bimodules over $\Acal$. Then the product $\Acal$-bimodule
    \[
    \Omega = \Omega_1 \oplus \cdots \oplus \Omega_m
    \]
    equipped with the product $\Acal$-coactions
    \[
    \Phi_\Omega = \Phi_{\Omega_1} \oplus \cdots \oplus \Phi_{\Omega_m}, \quad {}_\Omega \Phi = {}_{\Omega_1} \Phi \oplus \cdots \oplus {}_{\Omega_m} \Phi
    \]
    is a bicovariant bimodule over $\Acal$, called the \textbf{direct sum of $\Omega_1, \ldots , \Omega_m$}. Moreover,
    \[
    \inv \Omega = \inv \Omega_1 \oplus \cdots \oplus \inv \Omega_m.
    \]

    When $\Omega_1 , \ldots, \Omega_m$ are bicovariant $*$-bimodules, then $\Omega$ equipped with the product $*$-structure becomes a bicovariant $*$-bimodule over $\Acal$.
\end{prop}

\begin{proof}
    That $(\Omega, \Phi_\Omega, {}_\Omega \Phi)$ satisfies conditions B1--B2 of Definition~\ref{defn:bicovariant bimodule} is a straightforward verification. Condition B3 follows by evaluating both sides on elements of the form $( 0 , \ldots, 0, \omega_l , 0, \ldots , 0)$ with $\omega_l \in \Omega_l$.

    If $\Omega_1, \ldots , \Omega_m$ are bicovariant $*$-bimodules, then $\Omega$ satisfies conditions B4--B5 of Definition~\ref{defn:bicovariant *-bimodule} by direct inspection from the definitions of the structure maps.
\end{proof}

\begin{prop}\label{prop:direct sum of FODCs}
    Let $(\Omega_1 , d_1), \ldots , (\Omega_m , d_m)$ be bicovariant ($*$-)FODCs with corresponding quantum germs maps $Q_1, \ldots , Q_m$, respectively. Suppose that the map
    \[
    Q: \Acal \ni a \longmapsto \big( Q_1 (a) , \ldots , Q_m (a) \big) \in \inv \Omega_1 \oplus \cdots \oplus \inv \Omega_m
    \]
    is surjective. Let $\Omega = \Omega_1 \oplus \cdots \oplus \Omega_m$ be the direct sum of the bicovariant ($*$-)bimodules. Then, with the differential
    \begin{equation}\label{eq:direct sum of differentials}
        d: \Acal \ni a \longmapsto (d_1 a, \ldots , d_m a) \in \Omega,
    \end{equation}
    the pair $(\Omega,d)$ becomes a bicovariant ($*$-)FODC over $\Acal$, called the \textbf{direct sum of $(\Omega_1, d_1), \ldots, (\Omega_m, d_m)$}, and denoted
    \[
    (\Omega, d) = (\Omega_1, d_1) \oplus \cdots \oplus (\Omega_m , d_m).
    \]
    Its quantum germs map is given by $Q$, and the corresponding right ideal is
    \begin{equation}\label{eq:direct sum right ideal}
        R = \bigcap_{1\leq l \leq m} R_l,
    \end{equation}
    where $R_1, \ldots , R_m$ are the right ideals corresponding to $(\Omega_1 , d_1), \ldots , (\Omega_m, d_m)$, respectively.
    
    When $(\Omega_1, d_1), \ldots , (\Omega_m, d_m)$ are finite-dimensional and $\Xcal_1, \ldots , \Xcal_m$ are the corresponding spaces of left-invariant vector fields, then
    \begin{equation}\label{eq:direct sum LIVF}
        \Xcal_R = \Xcal_1 \oplus \cdots \oplus \Xcal_m.
    \end{equation}
\end{prop}

\begin{proof}
    First, we show that $(\Omega,d)$ defines an FODC. The Leibniz rule is immediate. To verify the standard form property, fix $a \in \Acal$ and write $\Delta(a) = \sum_{i} e_i \otimes f_i$. Then,
    \begin{align}\label{eq:quantum germs map for direct sum}
        S(a_{(1)})da_{(2)} &= \sum_i S(e_i) d f_i = \sum_i S(e_i) (d_1 f_i , \ldots , d_m f_i ) \nonumber \\
        &= \sum_i \big( S(e_i) d_1 f_i , \ldots , S(e_i) d_m f_i \big) \nonumber \\
        &= \big( \sum_i S(e_i) d_1 f_i , \ldots , \sum_i S(e_i) d_m f_i \big) = \big( Q_1 (a) , \ldots , Q_m (a) \big) = Q(a).
    \end{align}
    Since $Q$ is surjective, we conclude
    \[
    \Span_\Cbb \{ a \, db \mid a,b \in \Acal \} \supseteq \inv \Omega.
    \]
    The left-hand side is a left $\Acal$-submodule of $\Omega$, and the right-hand side is a left $\Acal$-basis of $\Omega$, so every element in $\Omega$ can be written in standard form. Hence $(\Omega,d)$ is indeed an FODC.

    For bicovariance, let $a \in \Acal$. Then,
    \begin{align*}
        \Phi_\Omega (da) 
        &= \Phi_\Omega \sum_{1 \leq l \leq m} ( 0, \ldots , 0, \rlap{$\overbrace{\phantom{ d_l a }}^{l\text{-th}} $}d_l a , 0, \ldots , 0 ) \\
        &\cong \sum_{1 \leq l \leq m} (0, \ldots, 0 , \Phi_{\Omega_l} (d_l a), 0 , \ldots , 0 ) \\
        &\cong \sum_{1 \leq l \leq m} a_{(1)} \otimes (0, \ldots , 0 , d_l a_{(2)}, 0, \ldots , 0) \\
        &= a_{(1)} \otimes (d_1 a_{(2)} , \ldots , d_m a_{(2)}) = (\id \otimes d) \Delta (a),
    \end{align*}
    verifying the first identity in F3 of Definition~\ref{defn:FODC}; the second follows similarly. Thus, $(\Omega, d)$ is bicovariant, and the quantum germs map is given by \eqref{eq:quantum germs map for direct sum}. If each $(\Omega_l, d_l)$ is a $*$-FODC, then each $d_l$ is $*$-preserving, so $d$ is also $*$-preserving.

    Equation~\eqref{eq:direct sum right ideal} follows from \eqref{eq:right ideal corresponding to BCFODC} and the identity
    \[
    \Ker \epsilon \cap \Ker Q = \Ker \epsilon \cap \bigcap_{1 \leq l \leq m} \Ker Q_l = \bigcap_{1 \leq l \leq m} \left( \Ker \epsilon \cap \Ker Q_l \right).
    \]

    Finally, if $(\Omega_1, d_1), \ldots, (\Omega_m, d_m)$ are finite-dimensional, then expanding $Q(a) = (Q_1(a), \ldots, Q_m(a))$ as in \eqref{eq:LIVF and quantum germs map}, it follows from Proposition~\ref{prop:LIVF basis} that the union of bases of $\Xcal_1, \ldots, \Xcal_m$ gives a basis of $\Xcal_R$, proving \eqref{eq:direct sum LIVF}.
\end{proof}

\section{Compact Lie groups}\label{sec:Compact Lie groups}

In this section, we apply the general framework developed in Sections~\ref{sec:Compact quantum groups (CQGs)}--\ref{sec:FODC} to a classical setting, namely that of a compact Lie group. While the results presented here are classical and well-known, we revisit them in detail to carefully illustrate how Laplacians on a compact Lie group naturally give rise to its classical FODC. This perspective serves to motivate the construction introduced in Section~\ref{sec:Main construction}.

Throughout this section, let $K$ be a compact Lie group, and let $\int_{K} dx$ denote the integral with respect to the normalized Haar measure on $K$. We write $\kf$ for its Lie algebra, $\gf = \Cbb \otimes_\Rbb \kf$ for its complexification, and $\exp: \kf \rightarrow K$ for the exponential map.

\subsection{Adjoint representations}\label{subsec:Adjoint representations}

In this paper, a representation $\pi : K \rightarrow \Lbb(V)$ is called a \textit{unitary representation of $K$} if $V$ is a finite-dimensional Hilbert space, $\pi$ is continuous, and each $\pi(x)$ is unitary. It is called \textit{irreducible} if $V$ contains no proper subspace invariant under all elements of $\pi(K)$. The \textit{induced Lie algebra representation} $\pi : \kf \rightarrow \Lbb(V)$ is defined by
\begin{equation*}
\pi(X) = \left. \frac{d}{dt} \right|_{t=0} \pi(\exp(tX)), \quad X \in \kf.
\end{equation*}
These operators act on $V$ as \textit{skew-adjoint operators}, i.e., for $X \in \kf$ and $v,w \in V$,
\begin{equation*}
\la \pi(X) v , w \ra = - \la v , \pi(X) w \ra.
\end{equation*}
We extend the induced Lie algebra representation complex linearly to $\pi : \gf \rightarrow \End(V)$.

Let $x \in K$ and denote by $c_x : K \rightarrow K$ the conjugation map $y \mapsto x y x^{-1}$. Then, the representation $\Ad: K \rightarrow \End(\kf)$ defined for $x \in K$ by
\[
\Ad(x) X = \left. \frac{d}{dt} \right|_{t=0} c_x (\exp(tX)), \quad X \in \kf,
\]
is called \textit{the adjoint representation of $K$}. Extending each $\Ad(x)$ complex linearly yields a map $\Ad: K \rightarrow \End(\gf)$.

Throught this section, we fix an inner product $\la \cdot , \cdot \ra$ on $\kf$ that is \textit{$\Ad$-invariant}, i.e., for any $x \in K$ and $X, Y \in \kf$,
\[
\la \Ad(x) X , \Ad(x) Y \ra = \la X , Y \ra.
\]
Choose an orthonormal basis $\{X_1, \cdots, X_d \} \subseteq \kf$ with respect to this inner product, and let $\{ \eps^1 , \cdots , \eps^d \} \subseteq \kf^*$ be the corresponding dual basis. We also extend $\la \cdot , \cdot \ra$ to $\gf$ by requiring it to be conjugate linear in the first argument and complex linear in the second. With respect to this, $\Ad : K \rightarrow \Lbb(\gf)$ becomes a unitary representation. Its induced Lie algebra representation is the map $\ad : \gf \rightarrow \End(\gf)$, called \textit{the adjoint representation of $\kf$}, and defined by
\[
\ad X (Y) = [X, Y], \quad X, Y \in \gf.
\]
For $X \in \kf$, we have
\begin{equation}\label{eq:exp ad is Ad}
e^{\ad X} = \Ad(\exp X) \in \End(\gf),
\end{equation}
where $e^{(\,\cdot\,)}$ denotes the matrix exponential. Moreover,
\begin{equation}\label{eq:ad invariance of inner product}
\la \ad X (Y), Z \ra = - \la Y, \ad X (Z) \ra, \quad Y, Z \in \gf.
\end{equation}
If $K$ is connected, this last identity is equivalent to $\Ad$-invariance of the inner product $\la \cdot , \cdot \ra$.

\subsection{Classical FODC on \texorpdfstring{$K$}{TEXT}}\label{subsec:Classical FODC on K}

\begin{defn}\label{defn:matrix coefficients of K}
Functions of the form
    \[
    K \ni x \longmapsto \la v , \pi(x) w \ra \in \Cbb,
    \]
    where $\pi: K \rightarrow \Lbb(V)$ is a finite-dimensional unitary representation of $K$ and $v, w \in V$, are called \textbf{matrix coefficients of $K$}. The set of matrix coefficients of $K$ will be denoted by $\Cf^\infty(K)$.
\end{defn}

The set $\Cf^\infty(K)$ should be distinguished from $C^\infty(K)$, the algebra of smooth functions on $K$, which properly contains $\Cf^\infty (K)$, see \cite[Problem~20-11]{JohnLee}. For each point $x \in K$, we denote the evaluation homomorphism at the point $x$ by $\ev_x$, i.e.,
\[
\ev_x (f) = f(x) \quad f \in \Cf^\infty(K).
\]

\begin{prop}\label{prop:matrix coefficients of K}
    The set $\Cf^\infty(K)$ equipped with the pointwise operations is a $*$-algebra, which becomes a CQG with the following:
    \begin{enumerate}[label=K\arabic*., series=K]
\item (Comultiplication) For $f \in \Cf^\infty(K)$ and $x,y \in K$,
    \[
    (\ev_x \otimes \ev_y) \Delta(f) = f(xy)
    \]
\item (Counit) $\epsilon = \ev_e$ where $e$ is the identity of $K$
\item (Antipode) For $f \in \Cf^\infty(K)$ and $x \in K$,
\[
S(f)(x) = f(x^{-1} )
\]
\item (Haar state) For $f \in \Cf^\infty(K)$,
\[
\hcal(f) = \int_K f(x) dx.
\]
\end{enumerate}
\end{prop}
\begin{proof}
    See \cite[Examples~1.2.5, 1.3.26, and 2.2.3]{Timmermann}.
\end{proof}

Let $\pi : K \rightarrow \Lbb(V)$ be a unitary representation and $\{ e_1 , \cdots, e_n \}$ be an orthonormal basis of $V$. Then,
\[
\big( \la e_i , \pi ( \,\cdot\, ) e_j \ra \big)_{1 \leq i,j \leq n} \in M_n ( \Cf^\infty (K) )
\]
is a unitary corepresentation of $\Cf^\infty(K)$. According to \cite[Example~3.1.5]{Timmermann}, this sets up a one-to-one correspondence between the unitary representations of $K$ and the corepresentations of $\Cf^\infty(K)$. So, $\Irr(K)$, the set of all equivalence classes of irreducible corepresentations of $\Cf^\infty(K)$, can be identified with the set of all equivalence classes of irreducible unitary representations of $K$. For each $\mu \in \Irr(K)$, let $\pi_\mu : K \rightarrow \Lbb(V(\mu))$ be an irreducible unitary representation corresponding to $\mu$. Let $n_\mu = \dim V(\mu)$ and fix an orthonormal basis $\{ e_1 ^\mu , \cdots , e_{n_\mu} ^\mu \}$ for $V(\mu)$. For $ 1\leq i,j \leq n_\mu$, let
\begin{equation*}
u^\mu _{ij} = \la e_i ^\mu , \pi_\mu (\,\cdot\,) e_j ^\mu \ra \in M_{n_\mu} (\Cf^\infty (K)).
\end{equation*}
Then, $\{ u^\mu _{ij} \mid \mu \in \Irr(K), \, 1 \leq i,j \leq n_\mu \}$ is a Peter-Weyl basis of $\Cf^\infty(K)$ and the correspondence
\begin{equation}\label{eq:classical matrix coefficients identification}
\Cf^\infty(K) \ni u^\mu _{ij} \longmapsto \la e_i ^\mu , (\,\cdot\,) e_j ^\mu \ra \in \bigoplus_{\mu \in \Irr(K)} \Lbb(V(\mu))^*
\end{equation}
sets up an isomorphism between the two spaces.

For $X \in \kf$ and $f \in C^\infty(K)$, define a smooth function $Xf: K \rightarrow \Cbb$ by
\begin{equation}\label{eq:Lie algebra element differential}
Xf(x) = \left. \frac{d}{dt} \right|_{t=0} f( x \exp(tX) ) , \quad x \in K.
\end{equation}
It satisfies Leibniz's rule, i.e., $X(fg) = (Xf) g + f (Xg)$.

Recall that, since $K$ is a Lie group, the space of $1$-forms on $K$ can be identified with the space $C^\infty(K) \otimes_\Rbb \kf^* \cong C^\infty(K) \otimes \gf^*$, see \cite{JohnLee}. Denote temporarily the exterior derivative on $0$-forms of the manifold $K$ by $D: C^\infty(K) \rightarrow C^\infty(K) \otimes \gf^*$, which is defined by
\begin{equation}\label{eq:classical differential}
Df = X_1 f \otimes \varepsilon^1 + \cdots + X_d f \otimes \varepsilon^d.
\end{equation}
Write $Df_x = (\ev_x \otimes \id) Df $ for $x \in K$. Then, for $f \in \Cf^\infty(K)$ and $X \in \kf$
\[
Df_x (X) = Xf (x), \quad x \in K.
\]

\begin{prop}\label{prop:classical FODC}
    Let $R_K = \{ f \in \Ker \epsilon \mid Df_e = 0 \} \subseteq \Ker \epsilon$. Then, $R_K$ is an $\ad$-invariant ideal of $\Cf^\infty(K)$ satisfying $*S(R_K) \subseteq R_K$.
    
    Thus, the construction of Proposition~\ref{prop:model BCFODC} applies to the space
    \[
    \Omega_K = \Cf^\infty(K) \otimes \Ker \epsilon/R_K
    \]
    to yield a bicovariant $*$-FODC structure on it, which, under the identification $ \Omega_K = \Cf^\infty(K) \otimes (\Ker \epsilon/R_K) \cong \Cf^\infty(K) \otimes \gf^*$ via the isomorphism
    \begin{equation}\label{eq:cotangent space identification}
    \Ker \epsilon /R_K \ni \pi_{R_K}(f) \longmapsto Df_e \in \gf^*,
    \end{equation}
    has the following as its structure maps: 
    \begin{enumerate}[label=D\arabic*., series=D]
        \item (Differential) $d: \Cf^\infty(K) \rightarrow \Omega_K \cong \Cf^\infty(K) \otimes \gf^*$ is equal to $D|_{\Cf^\infty(K)}$
        \item ($\Cf^\infty(K)$-module actions) For $f ,g \in \Cf^\infty(K)$ and $\omega \in \gf^*$,
        \[
        f( g \otimes \omega) = fg \otimes \omega = (g \otimes \omega)f
        \]
        \item ($\Cf^\infty(K)$-coactions) For $f \in \Cf^\infty(K)$, $\omega \in \gf^*$, and $x \in K$,
        \begin{gather*}
        (\ev_x \otimes \id_{\Omega_K}) \Phi_{\Omega_K} (f \otimes \omega) = f\big( x (\,\cdot\,) \big) \otimes \omega \\
        (\id_{\Omega_K} \otimes \ev_x) {}_{\Omega_K} \Phi (f \otimes \omega) = f\big( (\,\cdot\,) x \big) \otimes \omega \circ \Ad(x)^{-1}
        \end{gather*}
        \item (Involution) $\Omega_K \cong \Cf^\infty (K) \otimes_\Rbb \kf^* \xrightarrow{\overline{(\cdot)} \otimes \id} \Cf^\infty (K) \otimes_\Rbb \kf^* \cong \Omega_K$.
    \end{enumerate}
    Therefore, we call $(\Omega_K, d)$ \textbf{the classical FODC on $K$}.
\end{prop}
\begin{proof}
    Throughout the proof, we write $R = R_K$ and $\Omega = \Omega_K$.
    
    Leibniz's rule implies that, for $f \in R$ and $g \in \Cf^\infty(K)$,
\[
D(fg)_e = (Df_e) g(e) + f(e) Dg_e = 0 g(e) + 0 Dg_e = 0.
\]
Thus, $R$ is an ideal of $\Cf^\infty(K)$. Observe that, for all $f \in R$ and $x \in K$,
\[
(\id \otimes \ev_x) \ad(f) = f_{(2)} f_{(1)} (x^{-1}) f_{(3)} (x) = f \circ c_{x^{-1}}
\]
and hence
\[
(\ev_e \otimes \ev_x) (D \otimes \id) \ad(f) = D \big(f \circ c_{x^{-1}} \big)_e = Df_e \circ \Ad (x^{-1}) = 0,
\]
which shows that $\ad (R) \subseteq R \otimes \Cf^\infty(K)$, i.e., $R$ is $\ad$-invariant. Finally, note that, for all $ f \in \Cf^\infty(K)$ and $x \in K$,
\[
\ev_x S(f)^* = \overline{ f(x^{-1} )}.
\]
Therefore, if $f \in R$ and hence $X_j f(e) = 0$ for all $ 1 \leq j \leq d$, we have
\[
D ( S(f)^* )_e = - \overline{X_1 f (e)} \epsilon^1 - \cdots - \overline{X_d f(e)} \epsilon^d = 0,
\]
proving that $*S (R) \subseteq R$.

The map \eqref{eq:cotangent space identification} is by definition well-defined and injective. To see that it is surjective, we consider a unitary representation $\pi :K \rightarrow \Lbb(V)$ whose induced Lie algebra representation $\pi : \kf \rightarrow \Lbb(V)$ is faithful. Let $ \{ e_i \mid 1 \leq i \leq n \} $ be an orthonormal basis of $V$ and $\{ e_{ij} \mid 1 \leq i,j \leq n\}$ be the associated matrix units. Then, by definition, for each $X \in \kf$,
\[
\pi(X) = \sum_{1 \leq i,j \leq n} \left. \frac{d}{dt} \right|_{t=0} \la e_i , \pi(\exp(tX)) e_j \ra e_{ij} = \sum_{1 \leq i,j \leq n} (X u_{ij}) (e) e_{ij}
\]
where $u_{ij} = \la e_i , \pi (\cdot) e_j \ra \in \Cf^\infty(K)$ for $1 \leq i,j \leq n$. Since $\pi:\kf \rightarrow \Lbb(V)$ is faithful, the matrices $\{\pi(X_k) \mid 1 \leq k \leq d \}$ are linearly independent and hence the map
\[
M_n (\Cbb) \ni (a_{ij} )_{ 1 \leq i,j \leq n} \longmapsto \Big( \sum_{1 \leq i,j \leq n} a_{ij} (X_k u_{ij}) (e) \Big)_{ 1\leq k\leq d} \in \Cbb^d
\]
is surjective, which implies
\begin{align*}
\gf^* \subseteq \Span_{\Cbb} \Big\{ \sum_{1 \leq k \leq d} (X_k u_{ij})(e) \otimes \varepsilon^k \; \Big| \; 1\leq i,j \leq n \Big\} = \Span_{\Cbb} \{ D (u_{ij})_e \mid 1 \leq i,j \leq n \} \\
\subseteq \{ Df_e \mid f \in \Cf^\infty(K) \} = \{ Df_e \mid f \in \Ker \epsilon \},
\end{align*}
the last equality being a consequence of the identity $Df_e = D(f - f(e) )_e$.

Now, using this identification, we will check if the structure maps of Proposition~\ref{prop:model BCFODC} indeed translate into the formulae of the proposition.
    
    First, the differential. Let $f \in \Cf^\infty(K)$. Then,
        \begin{align*}
        df &\cong f_{(1)} \otimes D(f_{(2)} - \epsilon (f_{(2)}) 1 )_e = f_{(1)} \otimes (D f_{(2)}) _e \\
        &= f_{(1)} \otimes \Big( X_1 f_{(2)} (e) \varepsilon^1 + \cdots + X_d f_{(2)} (e) \varepsilon^d \Big) = Df \in \Cf^\infty(K) \otimes \gf^*
        \end{align*}
since, for any $X \in \kf$ and $x \in K$, we have
        \begin{align}\label{eq:differential as convolution operators}
        (Xf) (x) = \left. \frac{d}{dt} \right|_{t=0} f \big( x \exp (tX) \big) = \left. \frac{d}{dt}\right|_{t=0} f_{(1)} (x) f_{(2)} \big( \exp(tX) \big) \nonumber\\
        = f_{(1)} (x) (X f_{(2)})(e).
        \end{align}
Therefore, $d= D$. Hence, from now on, we will no longer use the notation $D$ to denote the exterior derivative on $0$-forms.

Next, the $\Cf^\infty(K)$-bimodule structure. Let $f, g \in \Cf^\infty(K)$ and $h \in \Ker \epsilon$, i.e., $h(e) = 0$. Then, by \eqref{eq:BCFODC-bimodule},
    \[
    f (g \otimes dh_e) = fg \otimes dh_e
    \]
and by Leibniz's rule,
    \[
    (g \otimes dh_e) f = g f_{(1)} \otimes d(h f_{(2)})_e = g f_{(1)} \otimes f_{(2)}(e) (dh_e) =  g f \otimes dh_e.
    \]

Now, let us look at the coactions. Note that $\inv \Omega = \gf^*$ under the identification \eqref{eq:cotangent space identification}. Thus,
the left coaction $\Phi_\Omega : \Omega \rightarrow \Cf^\infty(K) \otimes \Omega$ is given by, for $f \in \Cf^\infty(K)$, $\omega \in \gf^*$, and $x \in K$,
    \begin{equation*}
        (\ev_x \otimes \id_{\Omega} ) \Phi_\Omega (f \omega ) = (\ev_x \otimes \id_\Omega) \big(\Delta(f) (1 \otimes \omega)\big) = f\big(x (\,\cdot\,) \big) \otimes \omega.
    \end{equation*}
The restricted right coaction ${}_{\inv \Omega} \Phi$ is given by, for $g \in \Ker \epsilon$ and $x \in K$,
    \begin{align*}
        (\id_\Omega \otimes \ev_x) \inv &\Phi(1 \otimes dg_e) = \big( (\id_\Omega \otimes \ev_x) \big(1 \otimes (d g_{(2)})_e \otimes S(g_{(1)}) g_{(3)} \big) \\
        &= (1 \otimes g_{(1)} (x^{-1}) g_{(3)} (x) (dg_{(2)})_e ) = ( 1 \otimes d (g \circ c_{x^{-1}} )_e ) \\
        &= (1 \otimes dg_e \circ \Ad (x^{-1}) ).
    \end{align*}
By requiring $\Cf^\infty(K)$-linearity, one gets the formula for the right coaction in D3.

Finally, the involution. By \eqref{eq:LCFODC-involution}, for any $f \in \Cf^\infty(K)$ and $g \in \Ker \epsilon$, we have
\[
( f \otimes dg_e )^* = - \overline{f_{(1)}} \otimes d ( \overline{S(g)} \, \overline{ f_{(2)}} )_e = - \overline{f_{(1)}} \, \overline{f_{(2)} (e)} \otimes d ( S(\overline{g}) )_e = \overline{f} \otimes d \overline{g}_e
\]
since $X \big(S(g)\big) (e) = \left. \frac{d}{dt} g( \exp(-tX)) \right|_{t=0} = - Xg (e) $ for all $X \in \kf$. Thus, we see that the involution is given by the complex conjugation on the $\Cf^\infty(K)$-part of $\Omega = \Cf^\infty(K) \otimes_\Rbb \kf^*$.
\end{proof}

\begin{rmk}\label{rmk:classical FODC}
If we identify $\gf \cong \gf^*$ using the complex bilinear extension of the $\Ad$-invariant inner product fixed in Section~\ref{subsec:Adjoint representations}, then for all $X \in \gf \cong \gf^* = \inv \Omega_K$,
    \begin{equation}\label{eq:right coaction on Omega(K)-invariant part}
        (\id \otimes \ev_x) {}_{\inv \Omega_K} \Phi ( X) = \Ad (x) X, \quad x \in K
    \end{equation}
by the second identity in D3 of Proposition~\ref{prop:classical FODC}.

Now, suppose that the compact Lie group $K$ is embedded into $M_n (\Cbb)$ for some $n \in \Nbb$, which induces an embedding of $\kf$ into $M_n (\Cbb)$. Let $\{e_{ik} \mid 1 \leq i,k \leq n \}$ be the matrix units of $M_n (\Cbb)$ and $\{u_{ik} \mid 1 \leq i,k \leq n \} \subseteq \Cf^\infty(K)$ be the matrix coefficients of $K$ defined by, for $ 1 \leq i, k \leq n$,
\[
u_{ik} : K \ni \sum_{1 \leq j,l \leq n} X_{jl} e_{jl} \longmapsto X_{ik} \in \Cbb.
\]
Then, \eqref{eq:right coaction on Omega(K)-invariant part} becomes, for $X = \displaystyle \sum_{1 \leq j,l \leq n} X_{jl} e_{jl} \in \kf \leq M_n (\Cbb)$,
    \begin{equation}\label{eq:right coaction on Omega(K)-invariant part-matrix group}
        {}_{\inv \Omega_K} \Phi(X) = \sum_{1 \leq i,j,k,l \leq n} X_{jl} \Big( e_{ik} \otimes \big(u _{ij} S(u _{lk})\big) \Big).
    \end{equation}
\end{rmk}

\subsection{Classical Laplacians on \texorpdfstring{$K$}{TEXT}}\label{subsec:classical Laplacians on K}

Identify $\kf^* \cong \kf$ using the fixed $\Ad$-invariant inner product on $\kf$. Then, $\{\eps^1, \dots, \eps^d \}$ becomes an orthonormal basis in the inner product space $\kf^* \cong \kf$, and thus
\begin{equation}\label{eq:inner product on the invariant part-classical}
    \la df_e , dg_e \ra = \Big\la \sum_{i=1}^{d} X_i f(e) \eps^i , \sum_{j=1}^{d} X_j g(e) \eps^j \Big\ra = \sum_{j=1}^{d} X_j f(e) X_j g(e)
\end{equation}
on $\kf^*$ for all real-valued $f, g \in C^\infty(K)$.

Extend this inner product to a $\Cf^\infty(K)$-valued $\Cf^\infty(K)$-sesquilinear form on the classical FODC $\Omega_K$ by
\begin{equation}\label{eq:Cinfty(K)-valued sesquilinear form on GammaK}
    \la \cdot , \cdot \ra : \Omega_K \times \Omega_K \ni (f \omega , g \eta) \longmapsto \overline{f} g \la \omega , \eta \ra \in \Cf^\infty(K)
\end{equation}
for $f, g \in \Cf^\infty(K)$ and $\omega, \eta \in \kf^*$.

Recall that $\hcal(f) = \int_K f(x)\,dx$ for $f \in \Cf^\infty(K)$, and consider the form
\begin{equation*}
    \Omega_K \times \Omega_K \ni (\omega, \eta) \longmapsto \hcal\big( \la \omega , \eta \ra \big) \in \Cbb.
\end{equation*}
One can check that this defines an inner product on $\Omega_K$.

\begin{defn}\label{defn:classical Laplacian}
A linear operator $ \square : \Cf^\infty(K) \rightarrow \Cf^\infty(K) $ is called the \textbf{classical Laplacian on $K$ associated with the $\Ad$-invariant inner product $\la \cdot , \cdot \ra$} if, for all $f , g \in \Cf^\infty(K)$,
\begin{equation*}
    \hcal( \overline{f} \square g ) = \hcal\big( \la df , dg \ra \big).
\end{equation*}
\end{defn}

Note that if such an operator exists, then it is unique due to the faithfulness of $\hcal$. In this subsection, we will construct such an operator.

\begin{defn}\label{defn:the classical universal enveloping algebra}
The \textbf{universal enveloping algebra of $\gf$} is a unital algebra $U(\gf)$ equipped with a linear map $\iota : \gf \rightarrow U(\gf)$ satisfying the following universal property: Given any linear map $\varphi : \gf \rightarrow A$ into a unital algebra $A$ such that $\varphi([X,Y]) = \varphi(X)\varphi(Y) - \varphi(Y)\varphi(X)$ for all $X, Y \in \gf$, there exists a unique unital algebra homomorphism $\tilde{\varphi}: U(\gf) \rightarrow A$ such that $\tilde{\varphi} \circ \iota = \varphi$.
\end{defn}

The map $\iota : \gf \rightarrow U(\gf)$ is injective, allowing us to regard $\gf$ as a subspace of $U(\gf)$. The space $U(\gf)$ is linearly spanned by monomials of the form $Y_1 \cdots Y_n \in U(\gf)$ with $Y_1, \dots, Y_n \in \gf$, see \cite[Chapter~III]{Knapp}.

\begin{prop}\label{prop:the classical universal enveloping algebra}
When equipped with the following maps, $U(\gf)$ becomes a Hopf algebra:
\begin{enumerate}[label=k\arabic*., series=k]
    \item (Comultiplication) $\Delta : U(\gf) \rightarrow U(\gf) \otimes U(\gf)$ given by
    \[
    \Delta(X) = X \otimes 1 + 1 \otimes X, \quad X \in \gf;
    \]
    \item (Counit) $\epsilon: U(\gf) \rightarrow \Cbb$ given by $\epsilon (X)= 0$ for $X \in \gf$;
    \item (Antipode) $S: U(\gf) \rightarrow U(\gf)$ given by $S(X) = -X$ for $X \in \gf$.
\end{enumerate}
Moreover, when endowed with the following involution, it becomes a Hopf $*$-algebra, denoted by $U ^\Rbb (\kf)$ to emphasize its dependence on $\kf$.
\begin{enumerate}[label=k\arabic*., resume=k]
    \item (Involution) $* : U(\gf) \rightarrow U(\gf)$ given by $X^* = -X$ for $X \in \kf$.
\end{enumerate}
\end{prop}

\begin{proof}
See \cite[Examples~I.6 and I.10]{Klimyk}.
\end{proof}

\begin{prop}\label{prop:classical skew-pairing}
The bilinear map $( \cdot , \cdot ) : \kf \times \Cf^\infty(K) \rightarrow \Cbb$ defined by
\begin{equation}\label{eq:classical-skew-pairing}
    ( X, f ) = \left. \frac{d}{dt} \right|_{t=0} f( \exp(tX)) = (X f)(e), \quad X \in \kf, \, f \in \Cf^\infty(K)
\end{equation}
extends to a nondegenerate skew-pairing $( \cdot , \cdot ) : U ^\Rbb (\kf ) \times \Cf^\infty (K) \rightarrow \Cbb$.
\end{prop}

\begin{proof}
See \cite[Example~1.4.7]{Timmermann}.
\end{proof}

Using this pairing, we embed the Hopf $*$-algebra $U ^\Rbb (\kf)$ into $\Cf^\infty(K)^\circ$, allowing us to regard elements of $U ^\Rbb (\kf)$ as linear functionals on $\Cf^\infty(K)$. This yields a $U ^\Rbb (\kf)$-bimodule structure on $\Cf^\infty(K)$ (see Remark~\ref{rmk:embedding via skew-pairing}).

\begin{prop}\label{prop:differential operators from linear functionals}
For $X \in \kf$ and $f \in \Cf^\infty(K)$, we have
\begin{equation*}
X \triangleright f = Xf.
\end{equation*}
Hence, for any $A = Y_1 \cdots Y_n \in U ^\Rbb (\kf)$ with $Y_1, \dots, Y_n \in \kf$, we obtain
\begin{equation}\label{eq:differential operators from linear functionals}
    A \triangleright f = Y_1 \cdots Y_n f, \quad f \in \Cf^\infty(K),
\end{equation}
where the right-hand side denotes successive applications of \eqref{eq:Lie algebra element differential}.
\end{prop}

\begin{proof}
By definition, $X \triangleright f = (\id \otimes X) \Delta(f)$. Then, \eqref{eq:classical-skew-pairing} and \eqref{eq:differential as convolution operators} verify the identity. The second statement follows since $\triangleright$ is a left $U^\Rbb(\kf)$-module action.
\end{proof}

\begin{defn}\label{defn:classical Casimir elements}
The element
\[
Z= -\left( X_1 ^2 + \cdots + X_d ^2 \right) \in U^\Rbb(\kf)
\]
is called the \textbf{classical Casimir element of $U^\Rbb (\kf)$ associated with the $\Ad$-invariant inner product $\la\cdot, \cdot\ra$}.
\end{defn}

The element $Z$ is independent of the orthonormal basis chosen and is central in $U^\Rbb (\kf)$, see \cite[Proposition~5.24]{Knapp}.

By \eqref{eq:differential operators from linear functionals}, we also have
\begin{equation}\label{eq:classical Casimir}
Z \triangleright f = - \left( X_1 ^2 f + \cdots + X_d ^2 f \right), \quad f \in \Cf^\infty(K).
\end{equation}

\begin{prop}\label{prop:classical Laplacian characterization}
The operator $Z \triangleright : \Cf^\infty(K) \rightarrow \Cf^\infty(K)$ is the classical Laplacian associated with the $\Ad$-invariant inner product $\la \cdot , \cdot \ra$, i.e.,
\begin{equation}\label{eq:classical Laplacian}
    \hcal\big( \overline{f} \, (Z \triangleright g) \big) = \hcal\big( \la df, dg \ra \big).
\end{equation}
\end{prop}

\begin{defn}\label{defn:classical heat semigroup}
    Therefore, every classical Laplacian on $K$ is given by a linear operator of the form $\phi \triangleright$ for some $\phi \in \Cf^\infty(K)^*$. Thus, any classical Laplacian $\square : \Cf^\infty(K) \rightarrow \Cf^\infty(K)$ on $K$ generates a semigroup $(e^{-t \square})_{t \geq 0}$ on $\Cf^\infty(K)$ (cf. Proposition~\ref{prop:semigroups generated by linear functionals}), called \textbf{the heat semigroup on $K$ generated by $\square$}.
\end{defn}

To prove Proposition~\ref{prop:classical Laplacian characterization}, we need two lemmas.

\begin{lem}\label{lem:classical Laplacian characterization}
    For $f, g \in \Cf^\infty(K)$, the following identity holds:
\begin{equation}\label{eq:classical Laplacian Leibniz rule}
    Z \triangleright (\overline{f} g) = \overline{ (Z \triangleright f )} g - 2 \la df, dg \ra + \overline{f} (Z \triangleright g)
\end{equation}
\end{lem}
\begin{proof}
    By \eqref{eq:inner product on the invariant part-classical} and \eqref{eq:differential as convolution operators}, we have
    \begin{align*}
    \la df,dg \ra = \overline{f_{(1)}} g_{(1)} \la d(f_{(2)})_e , d(g_{(2)})_e \ra = \overline{f_{(1)}} g_{(1)} \sum_{1 \leq j \leq d} \overline{X_j f_{(2)} (e)} X_j g_{(2)} (e) \\
    = \sum_{1\leq j \leq d} \overline{ (X_j f)} (X_j g).
    \end{align*}
    Now, apply \eqref{eq:classical Casimir} on the function $\overline{f}g$, apply Leibniz's rule twice for each $X_j$ ($j=1, \cdots , d$), and use the preceding identity to arrive at \eqref{eq:classical Laplacian Leibniz rule}.
\end{proof}

\begin{lem}\label{lem:L2-adjoints of differential operators}
    For $f \in \Cf^\infty(K)$ and $X \in \kf$, we have
    \begin{equation*}
    \hcal(Xf) = 0.
    \end{equation*}
    Thus, for $f, g \in \Cf^\infty(K)$,
    \begin{equation*}
    \hcal\big( \overline{Xf} \, g \big) = - \hcal\big( \overline{f} \, Xg \big).
    \end{equation*}
\end{lem}
\begin{proof}
    By the right invariance of the Haar measure, we have
    \begin{align*}
\int_{K} (X f)(x) dx = \int_K \left. \frac{d}{dt} \right|_{t=0} f(x \exp(tX)) dx = \left. \frac{d}{dt} \right|_{t=0} \int_K f(x \exp(tX) ) dx \\
= \left. \frac{d}{dt} \right|_{t=0} \int_K f(x) dx = 0,
\end{align*}
    proving the first identity. The second identity follows from Leibniz's rule for $X$ and the fact that $X \overline{f} = \overline{Xf}$.
\end{proof}

\begin{proof}[Proof of Proposition~\ref{prop:classical Laplacian characterization}]
By \eqref{eq:classical Casimir} and Lemma~\ref{lem:L2-adjoints of differential operators}, the following two equalities hold for $f,g \in \Cf^\infty(K)$:
\begin{gather*}
    \hcal\big(Z \triangleright (\overline{f} g)\big) = 0 \\
    \hcal\big( \overline{(Z \triangleright f)} \, g \big) = \hcal\big( \overline{f} \, (Z \triangleright g) \big)
\end{gather*}

Taking $\hcal$ on both sides of \eqref{eq:classical Laplacian Leibniz rule} and using these equalities, we get
\[
0 = \hcal\big( \overline{f} \, (Z \triangleright g) \big) - 2 \hcal\big( \la df, dg \ra \big) + \hcal\big( \overline{f} \, (Z \triangleright g) \big),
\]
from which \eqref{eq:classical Laplacian} follows.
\end{proof}

The following proposition summarizes the properties of $Z$ that made the proof of Proposition~\ref{prop:classical Laplacian characterization} possible, see also Theorem~\ref{thm:linear functionals as Laplacians}.

\begin{prop}\label{prop:classical Casimir element is self-adjoint, ad-invariant, and Hermitian}
    When considered as a linear functional on $\Cf^\infty(K)$, $Z$ is self-adjoint, $\ad$-invariant, Hermitian, and vanishes at the unit, which is equivalent to the condition that the classical Laplacian $Z \triangleright : \Cf^\infty(K) \rightarrow \Cf^\infty(K)$ diagonalizes over the Peter-Weyl decomposition of $\Cf^\infty(K)$ with real eigenvalues, commutes with the antipode, and vanishes at the unit.
\end{prop}
\begin{proof}
    The self-adjointness follows from
    \[
    Z^* = - \big( (X_1 ^*)^2 + \cdots + (X_d ^*)^2 \big) = - \big( (-X_1)^2 + \cdots + (-X_d)^2 \big) = Z.
    \]
    
    Being a central element of $U^\Rbb (\kf) \subseteq \Cf^\infty(K)^\circ$ that separates $\Cf^\infty(K)$, $Z$ is $\ad$-invariant by Proposition~\ref{prop:ad-invariance criterion for a linear functional}.

The Hermiticity follows from 
\[
Z S^{-1} = \hat{S} (Z) = - \big( \hat{S}(X_1)^2 + \cdots + \hat{S}(X_d)^2 \big) = - \big( (-X_1)^2 + \cdots + (-X_d)^2 \big) = Z
\]
and Corollary~\ref{cor:self-adjointness and Hermiticity for ad-invariant functional}.

Finally, note that $Z(1) = \hat{\epsilon}(Z) = - \big( \hat{\epsilon}(X_1)^2 + \cdots + \hat{\epsilon}(X_d)^2 \big) = 0$.

The final assertion follows from Corollary~\ref{cor:diagonalization of translation invariant linear operators} and Proposition~\ref{prop:left multiplication operator commutes with antipode}.
\end{proof}

In the case when $K$ is simply connected and semisimple, we will be able to calculate explicitly the eigenvalues of certain classical Laplacians, see Proposition~\ref{prop:eigenvalues of classical Casimir elements}.

\subsection{Classical Laplacians induce the classical FODC}\label{subsec:Classical Laplacians induce the classical FODC}

In Section~\ref{subsec:classical Laplacians on K}, we saw that the classical Laplacian $Z \triangleright$ was defined in terms of the classical FODC on $K$ and an $\Ad$-invariant inner product on $\kf$. In this subsection, we show that, conversely, the classical Laplacian $Z \triangleright$ can be used to recover both the classical FODC and the $\Ad$-invariant inner product on $\kf$ that were used to define $Z$.

\begin{thm}\label{thm:classical Casimir element induces classical FODC}
For all $f, g \in \Ker \epsilon$, we have
\begin{equation}\label{eq:classical Laplacian at the identity}
\la df_e , dg_e \ra = -\frac{1}{2} (Z, \overline{f} g).
\end{equation}
Thus,
\begin{equation*}
\{ f \in \Ker \epsilon \mid \forall g \in \Ker \epsilon,\, ( Z, f g ) = 0 \} = \{ f \in \Ker \epsilon \mid df_e = 0 \}.
\end{equation*}
Together with Proposition~\ref{prop:classical FODC}, this implies that the classical FODC on $K$ and the invariant inner product $\la \cdot, \cdot \ra$ on $\kf$ are determined by the classical Laplacian $Z \triangleright : \Cf^\infty(K) \rightarrow \Cf^\infty(K)$.
\end{thm}

\begin{proof}
Taking $\epsilon$ on both sides of \eqref{eq:classical Laplacian Leibniz rule}, we obtain
\begin{align*}
(Z, \overline{f} g) = (Z, \overline{f}) g(e) - 2 \la df_e , dg_e \ra + \overline{f}(e) (Z, g) = -2 \la df_e , dg_e \ra
\end{align*}
for all $f, g \in \Ker \epsilon$, which proves \eqref{eq:classical Laplacian at the identity}.
\end{proof}

\begin{cor}\label{cor:heat semigroups are Markov semigroups}
The heat semigroup $(e^{-t Z \triangleright})_{t \geq 0}$ is a quantum Markov semigroup.
\end{cor}

\begin{proof}
Equation~\eqref{eq:classical Laplacian at the identity} implies that $-Z$ is conditionally positive. The conclusion now follows from Proposition~\ref{prop:complete positivity of semigroup}.
\end{proof}

Theorem~\ref{thm:classical Casimir element induces classical FODC} motivates the construction that we now introduce.

\section{Main construction}\label{sec:Main construction}

\subsection{FODCs induced by linear functionals}\label{subsec:FODCs on a CQG induced by linear fucntionals}

Throughout this subsection, $(\Acal, \Delta, \epsilon, S)$ denotes a Hopf ($*$-)algebra and $(\Acal^\circ, \hat{\Delta}, \hat{\epsilon}, \hat{S})$ its dual Hopf ($*$-)algebra.

\begin{thm}\label{thm:FODC-linear-functional}
Let $\phi$ be a linear functional on $\Acal$. Then,
\[
R_\phi= \{ a \in \Ker \epsilon \mid \forall b \in \Ker \epsilon,\, ( \phi, a b ) = 0 \} \subseteq \Ker \epsilon
\]
is a right ideal in $\Acal$. Moreover, if $\phi$ is $\ad$-invariant (resp. self-adjoint), then $R_\phi$ is $\ad$-invariant (resp. $*S(R_\phi) \subseteq R_\phi$).
\end{thm}

\begin{proof}
Throughout the proof, we write $R = R_\phi$.

That $R$ is a right ideal follows directly from its definition and from the fact that $\Ker \epsilon$ is an ideal in $\Acal$.

Assume now that $\phi$ is $\ad$-invariant and define the map $\varphi : \Ker \epsilon \rightarrow (\Ker \epsilon)^*$ by
\[
( \varphi(a) , b ) = ( \phi , a b ) \quad \text{for all } b \in \Ker \epsilon.
\]
Then, $\Ker \varphi = R$. Consider the map $\varphi \otimes \id_\Acal : \Ker \epsilon \otimes \Acal \rightarrow (\Ker \epsilon)^* \otimes \Acal$. We claim that $(\varphi \otimes \id_\Acal)(\ad(a)) = 0$ for all $a \in R$ (this is well-defined since $\ad(R) \subseteq \ad(\Ker \epsilon) \subseteq \Ker \epsilon \otimes \Acal$), which would imply
\[
\ad (R) \subseteq \Ker (\varphi \otimes \id_\Acal) = R \otimes \Acal,
\]
establishing the $\ad$-invariance of $R$.

Fix $a \in R$ and pick an arbitrary $b \in \Ker \epsilon$. Since the map
\[
\Acal \otimes \Acal \otimes \Acal \ni x \otimes y \otimes z \mapsto y_{(1)} x \otimes y_{(2)} \otimes y_{(3)} z \in \Acal \otimes \Acal \otimes \Acal
\]
is a bijection, there exist elements $x_j, y_j, z_j \in \Acal$ ($j = 1, \dots, k$) for some $k \in \Nbb$ such that
\begin{equation}\label{eq:proof-FODC-linear-functional}
1 \otimes b \otimes 1 = \sum_{j=1}^k y_{j,(1)} x_j \otimes y_{j,(2)} \otimes y_{j,(3)} z_j.
\end{equation}

Denote by $m$ the multiplication on $\Acal$. Then:
\begin{align}\label{eq:proof-FODC-linear-functional-2}
\left( (\varphi \otimes \id)\big(\ad(a)\big),\, b \otimes \id_\Acal \right)
&= ( \phi, a_{(2)} b )\, S(a_{(1)}) a_{(3)} \nonumber \\
&= m \circ ( S \otimes \phi \otimes \id )(a_{(1)} \otimes a_{(2)} b \otimes a_{(3)} ) \nonumber \\
&= \sum_{j=1}^k m \circ (S \otimes \phi \otimes \id) \left( a_{(1)} y_{j,(1)} x_j \otimes a_{(2)} y_{j,(2)} \otimes a_{(3)} y_{j,(3)} z_j \right) \nonumber \\
&= \sum_{j=1}^k m \circ (S \otimes \phi \otimes \id) \left( (a y_j)_{(1)} x_j \otimes (a y_j)_{(2)} \otimes (a y_j)_{(3)} z_j \right) \nonumber \\
&= \sum_{j=1}^k S(x_j) \left[ ( \phi , (a y_j)_{(2)} )\, S( (a y_j)_{(1)} )\, (a y_j)_{(3)} \right] z_j \nonumber \\
&= \sum_{j=1}^k S(x_j) \left( \phi \otimes \id , \ad(a y_j) \right) z_j = \sum_{j=1}^k ( \phi, a y_j )\, S(x_j) z_j \\
&= \sum_{j=1}^k \left[ \left( \phi, a (y_j - \epsilon(y_j)1) \right) + \phi(a)\, \epsilon(y_j) \right] S(x_j) z_j \nonumber \\
&= \phi(a) \sum_{j=1}^k S(x_j)\, \epsilon(y_j)\, z_j, \nonumber
\end{align}
since $\varphi(a)$ vanishes on $\Ker \epsilon$.

It remains to evaluate the sum:
\begin{align*}
\sum_{j=1}^k S(x_j)\, \epsilon(y_j)\, z_j
&= \sum_{j=1}^k S(x_j) S(y_{j,(1)})\, \epsilon(y_{j,(2)})\, y_{j,(3)} z_j \\
&= \sum_{j=1}^k S(y_{j,(1)} x_j)\, \epsilon(y_{j,(2)})\, y_{j,(3)} z_j \\
&= m ( S \otimes \epsilon \otimes \id ) \left( \sum_{j=1}^k y_{j,(1)} x_j \otimes y_{j,(2)} \otimes y_{j,(3)} z_j \right) \\
&= m ( S \otimes \epsilon \otimes \id )(1 \otimes b \otimes 1) = \epsilon(b)\, 1 = 0,
\end{align*}
by \eqref{eq:proof-FODC-linear-functional}. Hence,
\[
\left( (\varphi \otimes \id)\big(\ad(a)\big),\, b \otimes \id_\Acal \right) = 0.
\]
Since $b \in \Ker \epsilon$ was arbitrary, we conclude $(\varphi \otimes \id)(\ad(a)) = 0$ in $(\Ker \epsilon)^* \otimes \Acal$, completing the proof of $\ad$-invariance.

Now suppose $\phi$ is self-adjoint. To show $*S(R) \subseteq R$, let $a \in R$. Since $*S : \Acal \rightarrow \Acal$ is an antilinear algebra isomorphism satisfying $(*S)^2 = \id_\Acal$ and $\Ker \epsilon$ is stable under $*$ and $S$, we compute:
\[
( \phi, S(a)^* b ) = ( \phi , S(a S(b)^*)^* ) = \overline{( \phi^*, a S(b)^* )} = \overline{( \phi, a S(b)^* )} = 0
\]
for all $b \in \Ker \epsilon$. Hence, $S(a)^* \in R$, and so $*S(R) \subseteq R$.
\end{proof}

By Proposition~\ref{prop:model BCFODC}, we conclude:

\begin{cor}\label{cor:BCFODC induced by linear functionals}
An $\ad$-invariant linear functional $\phi$ on $\Acal$ induces a bicovariant FODC $(\Omega_\phi, d_\phi)$ over $\Acal$ via the $\ad$-invariant right ideal
\begin{equation}\label{eq:right-ideal-linear-functional}
R_\phi = \{ a \in \Ker \epsilon \mid \forall b \in \Ker \epsilon,\, (\phi, ab) = 0 \}.
\end{equation}
If, in addition, $\phi$ is self-adjoint, then $(\Omega_\phi, d_\phi)$ is a bicovariant $*$-FODC.

We call $(\Omega_\phi, d_\phi)$ the \textbf{bicovariant ($*$-)FODC induced by $\phi$}.
\end{cor}

\begin{rmk}\label{rmk:allusion in the work of Majid}
An allusion to the construction \eqref{eq:right-ideal-linear-functional} appears in \cite{Majid1998} as a bridge between Woronowicz's construction (Proposition~\ref{prop:model BCFODC}) and the quantum tangent space approach to FODCs used therein, see \cite[Proposition~2.7 and Equation~(14)]{Majid1998}.
\end{rmk}

Note that, in Theorem~\ref{thm:classical Casimir element induces classical FODC}, the linear functional $Z : \Cf^\infty(K) \rightarrow \Cbb$ is both self-adjoint and $\ad$-invariant (Proposition~\ref{prop:classical Casimir element is self-adjoint, ad-invariant, and Hermitian}), thereby inducing the bicovariant $*$-FODC of Proposition~\ref{prop:classical FODC}.

The following proposition shows that, when restricting to finite-dimensional FODCs, it suffices to consider functionals in $\Acal^\circ$.

\begin{prop}\label{prop:finite-dimensional BCFODC induced by linear functionals}
    Let $\phi \in \Acal^*$ be $\ad$-invariant. Then, the induced FODC $(\Omega_\phi, d_\phi)$ is finite-dimensional if and only if $\phi \in \Acal^\circ$.
\end{prop}

\begin{proof}
    Note that
    \[
    R_\phi = \{ a \in \Acal \mid \forall b \in \Ker \epsilon,\, (\phi, ab) = 0 \} = \{ a \in \Acal \mid \forall b \in \Ker \epsilon,\, (b \phi, a) = 0 \}.
    \]
    Thus, if the space $\{ b \phi \mid b \in \Ker \epsilon \}$ is finite-dimensional, say with basis $\{ b_1 \phi, \dots, b_n \phi \}$, then the linear map
    \[
    \inv \Omega_\phi \ni a + R_\phi \longmapsto \big( b_1 \phi(a), \dots, b_n \phi(a) \big) \in \Cbb^n
    \]
    is well-defined and injective, which implies that $(\Omega_\phi, d_\phi)$ is finite-dimensional.

    Conversely, observe that
    \begin{align*}
        R_\phi + \Cbb 1 &= \left\{ a \in \Acal \mid \forall b \in \Ker \epsilon,\, \big( \phi, (a - \epsilon(a)) b \big) = 0 \right\} \\
        &= \left\{ a \in \Acal \mid \forall b \in \Ker \epsilon,\, \big( b \phi, a - \epsilon(a) \big) = 0 \right\} \\
        &= \left\{ a \in \Acal \mid \forall b \in \Ker \epsilon,\, \big( b \phi - \phi(b) \epsilon, a \big) = 0 \right\},
    \end{align*}
    which implies
    \[
    \Xcal_{R_\phi} = \left\{ X \in \Acal^* \mid \forall a \in R_\phi + \Cbb 1,\, (X, a) = 0 \right\} \supseteq \left\{ b \phi - \phi(b) \epsilon \mid b \in \Ker \epsilon \right\}.
    \]
    Therefore, if $(\Omega_\phi, d_\phi)$ is finite-dimensional, then $\Xcal_{R_\phi}$ is also finite-dimensional by Corollary~\ref{cor:dimension encoded in the dual space}, and hence the subspace
    \[
    \{ b \phi \mid b \in \Ker \epsilon \} \subseteq \{ b \phi - \phi(b) \epsilon \mid b \in \Ker \epsilon \} + \Cbb \phi
    \]
    must also be finite-dimensional.

    Consequently, we conclude that
    \[
    (\Omega_\phi, d_\phi) \text{ is finite-dimensional } \iff \{ b \phi \mid b \in \Ker \epsilon \} \text{ is finite-dimensional},
    \]
    which is equivalent to
    \[
    \Acal \phi = \{ b \phi \mid b \in \Ker \epsilon \} + \Cbb \phi
    \]
    being finite-dimensional. By \cite[Corollary~1.4.5]{Joseph}, this last condition holds if and only if $\phi \in \Acal^\circ$.
\end{proof}

\begin{rmk}\label{rmk:affine invariance of induced FODC}
For any $\ad$-invariant linear functional $\phi \in \Acal^*$ and scalars $a, b \in \Cbb$ with $a \neq 0$,
\[
R_\phi = R_{a \phi + b},
\]
and thus $\phi$ and $a \phi + b$ induce the same FODC.

Later, we will find an equivalent condition under which two $\ad$-invariant functionals in $\Cf^\infty(K_q)^\circ$ induce the same FODC, see Corollary~\ref{cor:when two ad-invariant functionals induce the same FODC}.
\end{rmk}

We close this subsection with some results that will be used to compute the center of $\Cf^\infty(K_q)^\circ$.

\begin{lem}\label{lem:a variation of the induction construction}
    Let $\phi \in \Acal^*$ be $\ad$-invariant. Then,
    \[
    R_\phi' = \{ a \in \Acal \mid \forall b \in \Acal, \, (\phi, ab) = 0 \}
    \]
    is an $\ad$-invariant right ideal in $\Acal$.
\end{lem}

\begin{proof}
    The proof proceeds similarly to that of Theorem~\ref{thm:FODC-linear-functional}, and is in fact simpler, since we no longer need to assume $a, b \in \Ker \epsilon$. For instance, we define the map $\varphi : \Acal \to \Acal^*$ by
    \[
    ( \varphi(a), b ) = (\phi, ab), \quad b \in \Acal,
    \]
    so that $\Ker \varphi = R_\phi'$. Then, the proof concludes at \eqref{eq:proof-FODC-linear-functional-2}.
\end{proof}

\begin{lem}\label{lem:cofactorization of linear functionals}
    Let $\phi \in \Acal^\circ$ and suppose $\hat{\Delta}(\phi) = \sum_{i=1}^n X_i \otimes Y_i$,
    where $\{ X_i \}_{1 \leq i \leq n}$ is linearly independent and $Y_i \neq 0$ for all $1 \leq i \leq n$. Then,
    \[
    R_\phi' = \{ a \in \Acal \mid \forall b \in \Acal, \, (\phi, ab) = 0 \} = \bigcap_{i=1}^n \Ker Y_i,
    \]
    and hence $R_\phi'$ is of finite codimension in $\Acal$.
\end{lem}

\begin{proof}
    By assumption, the linear map
    \[
    \Acal \ni b \longmapsto \big( X_i(b) \big)_{1 \leq i \leq n} \in \Cbb^n
    \]
    is surjective. The conclusion then follows from the observation that
    \[
    (\phi, ab) = \big( \hat{\Delta}(\phi), b \otimes a \big) = \left( \sum_{i=1}^n X_i(b) Y_i, a \right), \quad a, b \in \Acal.
    \]
\end{proof}

\subsection{Laplacian on CQGs}\label{subsec:Laplacians on a CQG}

Throughout this subsection, let $(\Cf^\infty(\Kcal), \Delta, \epsilon, S)$ be a CQG with Haar state $\hcal$. As before, $(\Cf^\infty(\Kcal)^\circ, \hat{\Delta}, \hat{\epsilon}, \hat{S})$ denotes its dual Hopf $*$-algebra.

\begin{defn}\label{defn:strong sesquilinear form}
Let $\Omega$ be a bicovariant bimodule. A sesquilinear map $\la \cdot , \cdot \ra : \Omega \times \Omega \rightarrow \Cf^\infty(\Kcal)$ is called a \textbf{strongly nondegenerate right $\Cf^\infty(\Kcal)$-sesquilinear form on $\Omega$} if it is right $\Cf^\infty(\Kcal)$-sesquilinear, i.e.,
\begin{equation*}
\la \omega f, \eta g \ra = f^* \la \omega, \eta \ra g, \quad f, g \in \Cf^\infty (\Kcal), \; \omega, \eta \in \Omega,
\end{equation*}
and restricts to a $\Cbb$-valued nondegenerate sesquilinear form
\[
\la \, \cdot \, , \, \cdot \, \ra : \inv \Omega \times \inv \Omega \rightarrow \Cbb.
\]
\end{defn}

Conversely, any nondegenerate sesquilinear form on $\inv \Omega$ extends to a unique strongly nondegenerate right $\Cf^\infty(\Kcal)$-sesquilinear form on $\Omega$ (cf. Proposition~\ref{prop:the structure representations of bicovariant bimodule}~(1)). One can easily check that this defines a one-to-one correspondence.

Note that \eqref{eq:Cinfty(K)-valued sesquilinear form on GammaK} provides an example of such a strongly nondegenerate right $\Cf^\infty(K)$-sesquilinear form on $\Omega_K$, arising via this correspondence.

The following proposition explains the term ``strong nondegeneracy.''

\begin{prop}\label{prop:strong nondegeneracy}
Let $(\Omega,d)$ be an FODC on $\Kcal$ equipped with a strongly nondegenerate right $\Cf^\infty(\Kcal)$-sesquilinear form. Then, the following sesquilinear form is nondegenerate:
\begin{equation}\label{eq:right sesquilinear extension integrated}
\Omega \times \Omega \ni (\omega , \eta ) \longmapsto \hcal\big(\la \omega, \eta \ra \big) \in \Cbb.
\end{equation}
\end{prop}
\begin{proof}
Let $f, g \in \Cf^\infty(\Kcal)$ and $\omega, \eta \in \inv \Omega$. Applying $\hcal$ to both sides of
\[
\la \omega f , \eta g \ra = f^* g \la \omega , \eta \ra,
\]
yields
\begin{align*}
\hcal\big( \la \omega f , \eta g \ra \big) = \hcal ( f^* g ) \la \omega , \eta \ra.
\end{align*}
Thus, under the identification $\Omega \cong \Cf^\infty(\Kcal) \otimes \inv \Omega$ (Proposition~\ref{prop:the structure representations of bicovariant bimodule}~(1)), the form \eqref{eq:right sesquilinear extension integrated} becomes the tensor product of two nondegenerate sesquilinear forms on $\Cf^\infty(\Kcal)$ and $\inv \Omega$, and is therefore nondegenerate.
\end{proof}

In order to define a classical Laplacian on a compact Lie group $K$, we first fixed an inner product on the invariant part $\inv \Omega_K \cong \gf^*$ of the classical FODC, and then extended it $\Cf^\infty(K)$-linearly to $\Omega_K$, see Definition~\ref{defn:classical Laplacian}. In this sense, the following definition of a Laplacian on a CQG---for which a canonical choice of FODC is unavailable \cite{Woronowicz1989}---may be viewed as a natural quantum analogue.

\begin{defn}\label{defn:quantum Laplacian}
Let $(\Omega,d)$ be an FODC on $\Kcal$ equipped with a strongly nondegenerate right $\Cf^\infty(\Kcal)$-sesquilinear form $\la \cdot , \cdot \ra$. A linear operator $\square: \Cf^\infty(\Kcal) \rightarrow \Cf^\infty(\Kcal)$ is called the \textbf{Laplacian on $\Kcal$ associated with $\big(\Omega, d, \la \cdot , \cdot \ra\big)$} if it satisfies, for all $f, g \in \Cf^\infty(\Kcal)$,
\begin{equation*}
\hcal( f^* \square g ) = \hcal\big( \la d f , d g \ra \big).
\end{equation*}
By the faithfulness of $\hcal$, the Laplacian is unique if it exists.
\end{defn}

Given an $\ad$-invariant linear functional $\phi$, denote by $Q_\phi$ the quantum germs map associated with the FODC $(\Omega_\phi, d_\phi)$.

\begin{thm}\label{thm:linear functionals as Laplacians}
    Let $L : \Cf^\infty(\Kcal) \rightarrow \Cf^\infty(\Kcal)$ be a linear operator that diagonalizes with real eigenvalues over the Peter–Weyl decomposition, commutes with the antipode, and vanishes at the unit, and let $\phi = \epsilon L$. Equivalently (cf. Corollary~\ref{cor:diagonalization of translation invariant linear operators} and Proposition~\ref{prop:left multiplication operator commutes with antipode}), let $L = \phi \triangleright$, where $\phi \in \Cf^\infty(\Kcal)^*$ is a self-adjoint, $\ad$-invariant, Hermitian linear functional vanishing at the unit. Then, the map $\la \cdot , \cdot \ra : \Omega_\phi \times \Omega_\phi \rightarrow \Cf^\infty(\Kcal)$ defined for $f,g \in \Cf^\infty(\Kcal)$ and $x,y \in \Ker \epsilon$ by
\begin{equation}\label{eq:nondegenerate pairing between two FODCs-Hermitian case}
    \big\la Q_{\phi} (x)f , Q_{\phi} (y)g \big\ra_\phi = - \frac{1}{2} f^* g \big( \phi , S(x)^* S(y) \big) 
\end{equation}
is a strongly nondegenerate right $\Cf^\infty(\Kcal)$-sesquilinear form on $\Omega_\phi$, called the \textbf{sesquilinear form induced by $\phi$}, such that the operator $\phi \triangleright : \Cf^\infty(\Kcal) \rightarrow \Cf^\infty(\Kcal)$ is the Laplacian associated with $\big(\Omega_\phi , d_\phi , \la\cdot,\cdot\ra_\phi \big)$. That is, for all $f,g \in \Cf^\infty(\Kcal)$,
\begin{equation}\label{eq:linear functionals as Laplacians-Hermitian case}
    \hcal\big( f^* (\phi \triangleright g) \big) = \hcal\big( \la d_\phi f, d_{\phi} g \ra_\phi \big).
\end{equation}
\end{thm}

\begin{ex}\label{ex:classical Laplacians in the light of quantum Laplacians}
Let $K$ be a compact Lie group with Lie algebra $\kf$, equipped with an $\Ad$-invariant inner product $\la \cdot , \cdot \ra$, and let $Z \in U^\Rbb (\kf)$ be the classical Casimir element associated with it.

Proposition~\ref{prop:classical Casimir element is self-adjoint, ad-invariant, and Hermitian} shows that the classical Laplacian $Z \triangleright :\Cf^\infty(K) \rightarrow \Cf^\infty(K)$ satisfies the assumptions of Theorem~\ref{thm:linear functionals as Laplacians}. Moreover, by Theorem~\ref{thm:classical Casimir element induces classical FODC}, the FODC and the strongly nondegenerate right $\Cf^\infty(K)$-sesquilinear form induced by $Z$ via \eqref{eq:right-ideal-linear-functional} and \eqref{eq:nondegenerate pairing between two FODCs-Hermitian case}, respectively, coincide with the classical FODC and the $\Cf^\infty(K)$-sesquilinear extension of the $\Ad$-invariant inner product $\la \cdot, \cdot \ra$ on $\Omega_K$. Indeed, using \eqref{eq:classical Laplacian at the identity}, we have, for $f, g \in \Ker \epsilon$,
\begin{equation}\label{eq:nondegenerate pairing between two invariant parts-comparison}
    - \frac{1}{2} \big(Z , \overline{S(f)} S(g) \big) = - \frac{1}{2} \big( Z, S (\overline{f} g) \big) = - \frac{1}{2} (Z, \overline{f} g) = \la df_e , dg_e \ra.
\end{equation}
\end{ex}

Before proceeding to the proof, we record an equivalent condition under which the sesquilinear form $\hcal \big(\la\cdot,\cdot\ra_\phi \big)$ is positive definite.

\begin{prop}\label{prop:positive definiteness and conditional positivity}
The sesquilinear form $\hcal\big(\la \cdot, \cdot \ra_\phi\big)$ on $\Omega_\phi$ is positive definite if and only if $-\phi$ is conditionally positive.
\end{prop}

\begin{proof}
    Suppose $\hcal\big(\la \cdot, \cdot \ra_\phi\big)$ is positive definite. Then, for all $f \in \Ker \epsilon$,
    \[
    \frac{1}{2} \big( - \phi, S(f)^* S(f) \big) = \hcal\Big( \big\la Q_\phi (f), Q_\phi (f) \big\ra_\phi \Big) \geq 0,
    \]
    proving that $-\phi$ is conditionally positive.

    Conversely, assume $- \phi$ is conditionally positive. Then for any $f_i \in \Cf^\infty(\Kcal)$ and $x_i \in \Ker \epsilon$ with $1 \leq i \leq n$, we compute:
    \begin{align*}
    \hcal\Bigg(\Big\la \sum_{1 \leq i \leq n} Q_\phi(x_i) f_i,\; \sum_{1 \leq j \leq n} Q_\phi (x_j) f_j \Big\ra_\phi \Bigg) 
    &= \frac{1}{2} \sum_{1 \leq i,j \leq n} \hcal(f_i ^* f_j)\, \big( - \phi, S(x_i)^* S(x_j) \big)\\
    &= \frac{1}{2} \sum_{1 \leq i,j \leq n} \hcal \Big( f_i ^* \big( - \phi, S(x_i)^* S(x_j) \big) f_j \Big) \geq 0,
    \end{align*}
    since the matrix
    \[
    \Big( - \phi, S(x_i)^* S(x_j) \Big)_{1 \leq i,j \leq n} \in M_{n} (\Cbb)
    \]
    is positive by assumption.

    Hence, $\hcal\big( \la \cdot, \cdot \ra_\phi \big)$ is positive semi-definite. Now, suppose $\omega \in \Omega_\phi$ satisfies $\hcal\big(\la \omega, \omega \ra_\phi \big) = 0$. Then, by the Cauchy–Schwarz inequality for the sesquilinear form $\hcal \big( \la \cdot , \cdot \ra_\phi \big)$, we have
    \[
    \hcal\big(\la \omega, \eta \ra_\phi \big) \leq \hcal\big(\la \omega, \omega \ra_\phi\big)^{\frac{1}{2}} \hcal\big(\la \eta, \eta \ra_\phi\big)^{\frac{1}{2}} = 0
    \]
    for all $\eta \in \Omega_\phi$, which implies $\omega = 0$ by the nondegeneracy of $\hcal\big(\la\cdot,\cdot\ra_\phi\big)$ (Proposition~\ref{prop:strong nondegeneracy}). Thus, $\hcal\big(\la \cdot, \cdot \ra_\phi\big)$ is positive definite.
\end{proof}

We now begin the proof of Theorem~\ref{thm:linear functionals as Laplacians}.

\begin{prop}\label{prop:nondegenerate pairing between two FODCs}
Let $\phi$ be a self-adjoint, $\ad$-invariant linear functional on $\Cf^\infty(\Kcal)$. Then the sesquilinear map
\[
\la \cdot, \cdot \ra_\phi : \inv \Omega_\phi \times \inv \Omega_{\phi S} \rightarrow \Cbb
\]
defined, for $x,y \in \Ker \epsilon$, by
\begin{equation}\label{eq:nondegenerate pairing between two invariant parts}
\big\la Q_{\phi} (x) , Q_{\phi S} (y) \big\ra_\phi = - \frac{1}{2} \big( \phi, S(x)^* S(y) \big)
\end{equation}
is well-defined and nondegenerate. In particular, if $\phi$ is Hermitian (and hence satisfies $\phi = \phi S$), then $\la \cdot, \cdot \ra_\phi$ extends uniquely to a strongly nondegenerate right $\Cf^\infty(\Kcal)$-sesquilinear form on $\Omega_\phi$, which we continue to denote by $\la \cdot,\cdot \ra_\phi$.
\end{prop}

\begin{proof}
    Since $*S(R_\phi) = R_\phi$ and $(*S)^2 = \id$, it follows that for any fixed $x \in \Ker \epsilon$,
    \[
    \big(\phi, S(x)^* S(y) \big) = 0 \text{ for all } y \in \Ker \epsilon \quad \Longleftrightarrow \quad x \in R_\phi.
    \]
    Moreover, using the identity $\big(\phi, S(x)^* S(y) \big) = \big(\phi S , y \, S^2(x)^* \big)$, we deduce that for any fixed $y \in \Ker \epsilon$,
    \[
    \big(\phi, S(x)^* S(y) \big) = 0 \text{ for all } x \in \Ker \epsilon \quad \Longleftrightarrow \quad y \in R_{\phi S}.
    \]
    Together with \eqref{eq:right ideal corresponding to BCFODC}, these two characterizations establish both the well-definedness and the nondegeneracy of the pairing in \eqref{eq:nondegenerate pairing between two invariant parts}.
\end{proof}

\begin{rmk}\label{rmk:left ideal and right ideal complication}
    Some might wonder why we choose to use the seemingly more complicated expression \eqref{eq:nondegenerate pairing between two invariant parts} rather than, for example,
    \begin{equation}\label{eq:nondegenerate pairing-alternative}
    \la Q_{\phi S} (a) , Q_{\phi} (b) \ra = - \frac{1}{2} (\phi, b a^*),
    \end{equation}
    which is also well-defined and nondegenerate. The reason lies in a clash between two different conventions employed in this setting.

    If one adopts \eqref{eq:nondegenerate pairing-alternative} and attempts to prove Lemma~\ref{lem:linear functionals as Laplacians-Leibniz} below, one soon discovers that the more conventional GNS inner product $(a,b) \mapsto \hcal(a^* b)$ on $\Cf^\infty(\Kcal)$ must instead be replaced with
    \[
    \Cf^\infty(\Kcal) \times \Cf^\infty(\Kcal) \ni (a, b) \longmapsto \hcal(ba^*) \in \Cbb
    \]
    in order to proceed. However, in the theory of operator algebras, the use of the GNS inner product is a firmly established convention.

    Had we constructed bicovariant FODCs in Proposition~\ref{prop:model BCFODC} from left ideals rather than right ideals (cf. \cite[Section~6.7]{Sontz}), then for Corollary~\ref{cor:BCFODC induced by linear functionals}, we could have taken
    \[
    L_\phi = \{ b \in \Ker \epsilon \mid \forall a \in \Ker \epsilon,\, ( \phi, a b ) = 0 \} \subseteq \Ker \epsilon,
    \]
    which would have allowed the use of the simpler sesquilinear map
    \[
    \inv \Omega_{\phi S} \times \inv \Omega_{\phi} \ni (Q_{\phi S} (x) , Q_{\phi} (y) ) \longmapsto - ( \phi, x^* y ) \in \Cbb
    \]
    in place of \eqref{eq:nondegenerate pairing between two invariant parts}, all while retaining compatibility with the GNS inner product.

    Nevertheless, the use of right ideals in the bicovariant differential calculus literature is also a well-established convention. In the face of this conflict between two prevailing conventions, the choice of sesquilinear form given in \eqref{eq:nondegenerate pairing between two invariant parts} seems to be the only way that avoids altering either of them.

    Importantly, the use of \eqref{eq:nondegenerate pairing between two invariant parts} remains fully consistent with the classical case, as illustrated by \eqref{eq:nondegenerate pairing between two invariant parts-comparison}.
\end{rmk}

\begin{lem}\label{lem:differential in terms of quantum germs map-right multiplication form}
    Let $(\Omega,d)$ be a bicovariant FODC with quantum germs map $Q$. Then, for $f \in \Cf^\infty(\Kcal)$, we have
    \begin{equation}\label{eq:differential in terms of quantum germs map-right multiplication form}
    df = Q \big( \epsilon(f_{(2)}) - S^{-1} (f_{(2)}) \big) f_{(1)}.
    \end{equation}
\end{lem}
\begin{proof}
    Recall that for $\omega \in \inv \Omega$ and $f \in \Cf^\infty(\Kcal)$, we have $\omega \cdot f = S(f_{(1)}) \omega f_{(2)}$. Using \eqref{eq:basic identity of quantum germs map}, we compute:
    \begin{align*}
        df = f_{(1)} Q \big( f_{(2)} - \epsilon(f_{(2)}) \big) &= f_{(1)} Q \Big( \big( \epsilon(f_{(3)}) - S^{-1} (f_{(3)}) \big) f_{(2)} \Big) \\
           &= f_{(1)} \Big[ Q \big( \epsilon(f_{(3)}) - S^{-1} (f_{(3)}) \big) \cdot f_{(2)} \Big] \\
           &= f_{(1)} \Big[ S(f_{(2)}) Q \big( \epsilon(f_{(4)}) - S^{-1}(f_{(4)}) \big) f_{(3)} \Big] \\
           &= Q \big( \epsilon(f_{(2)}) - S^{-1} (f_{(2)}) \big) f_{(1)}.
    \end{align*}
\end{proof}

Compare the following with Lemma~\ref{lem:classical Laplacian characterization}.

\begin{lem}\label{lem:linear functionals as Laplacians-Leibniz}
    Let $\phi$ be a self-adjoint, $\ad$-invariant linear functional on $\Cf^\infty(\Kcal)$ that vanishes at the unit. Then, for all $f, g \in \Cf^\infty(\Kcal)$,
    \begin{equation*}
        \phi \triangleright (f^* g) = ( \phi S \triangleright f )^* g - 2 \la d_{\phi} f , d_{ ( \phi S ) } g \ra_\phi + f^* (\phi \triangleright g).
    \end{equation*}
    In particular, if $\phi$ is moreover Hermitian, then
    \begin{equation}\label{eq:linear functionals as Laplacians-Leibniz-Hermitian}
        \phi \triangleright (f^* g) = ( \phi \triangleright f )^* g - 2 \la d_{\phi} f , d_{ \phi } g \ra_\phi + f^* (\phi \triangleright g).
    \end{equation}
\end{lem}

\begin{proof}
    Extend $\la\cdot,\cdot\ra_\phi$ right $\Cf^\infty(\Kcal)$-sesquilinearly to $\Omega_\phi \times \Omega_{\phi S}$, and use \eqref{eq:nondegenerate pairing between two invariant parts} and \eqref{eq:differential in terms of quantum germs map-right multiplication form} to compute, for $f, g \in \Cf^\infty(\Kcal)$,
    \begin{align*}
        2 \la d_{\phi} f , d_{ ( \phi S )} g \ra_\phi 
        &= - f_{(1)}^* g_{(1)} \Big( \phi, \big( \epsilon(f_{(2)}) - f_{(2)} \big)^* \big( \epsilon(g_{(2)}) - g_{(2)} \big) \Big) \\
        &= f_{(1)}^* g \big( \phi, f_{(2)}^* \big) + f^* g_{(1)} \big( \phi, g_{(2)} \big) - f_{(1)}^* g_{(1)} \big( \phi, f_{(2)}^* g_{(2)} \big) \\
        &= f_{(1)}^* g \, \overline{\big( (\phi S)^*, f_{(2)} \big)} + f^* (\phi \triangleright g) - \phi \triangleright (f^* g) \\
        &= \big( (\phi S)^* \triangleright f \big)^* g + f^* (\phi \triangleright g) - \phi \triangleright (f^* g) \\
        &= ( \phi S \triangleright f )^* g + f^* (\phi \triangleright g) - \phi \triangleright (f^* g),
    \end{align*}
    where the last equality follows from Proposition~\ref{prop:self-adjointness and ad-invariance are preserved under antipode}.

    The final statement follows from Corollary~\ref{cor:self-adjointness and Hermiticity for ad-invariant functional}.
\end{proof}

\begin{proof}[Proof of Theorem~\ref{thm:linear functionals as Laplacians}]
    The fact that \eqref{eq:nondegenerate pairing between two FODCs-Hermitian case} extends to a well-defined, strongly nondegenerate, right $\Cf^\infty(\Kcal)$-sesquilinear form was already verified in Proposition~\ref{prop:nondegenerate pairing between two FODCs}. Thus, it remains only to establish \eqref{eq:linear functionals as Laplacians-Hermitian case}.

    Let $f, g \in \Cf^\infty(\Kcal)$. Applying $\hcal$ to both sides of \eqref{eq:linear functionals as Laplacians-Leibniz-Hermitian} and using the right invariance of $\hcal$, we obtain
    \begin{equation}\label{eq:Leibniz integrated}
        \hcal(f^* g)\, \phi(1) = \hcal\big( (\phi \triangleright f)^* g \big) - 2 \hcal\big( \la d_\phi f , d_{\phi} g \ra_\phi \big) + \hcal\big( f^* (\phi \triangleright g) \big).
    \end{equation}
    However, $\phi(1) = 0$ by assumption. Moreover, by \cite[Lemma~2.3]{KustermanMurphyTuset_Laplacian_2005}, we have
    \[
        \hcal\big( (\phi \triangleright f)^* g \big) = \hcal\big( f^* (\phi^* \triangleright g) \big) = \hcal\big( f^* (\phi \triangleright g) \big),
    \]
    due to the self-adjointness of $\phi$. Thus, \eqref{eq:Leibniz integrated} simplifies to
    \[
        0 = 2 \hcal\big( f^* ( \phi \triangleright g ) \big) - 2 \hcal\big( \la d_\phi f , d_{\phi} g \ra_\phi \big),
    \]
    which implies \eqref{eq:linear functionals as Laplacians-Hermitian case}.
\end{proof}

\begin{rmk}\label{rmk:linear functionals as Laplacians-non-Hermitian case}
    Note that the proof shows that, without assuming $\phi$ is Hermitian, one obtains
    \[
        0 = \hcal \big( f^* ( \phi S \triangleright g) \big) - 2 \hcal \big( \la d_\phi f , d_{ ( \phi S ) } g \ra_\phi \big) + \hcal \big( f^* ( \phi \triangleright g ) \big),
    \]
    which implies, with respect to the strongly nondegenerate right $\Cf^\infty(\Kcal)$-sesquilinear form $\la \cdot, \cdot \ra_\phi : \Omega_\phi \times \Omega_{\phi S} \rightarrow \Cf^\infty(\Kcal)$,
    \[
        \hcal \big( \la d_\phi f , d_{( \phi S )} g \ra_\phi \big) = \hcal \bigg( f^* \Big( \frac{\phi + \phi S}{2} \triangleright g \Big) \bigg).
    \]

    Therefore, even if the assumption that $\phi$ is Hermitian is dropped from Theorem~\ref{thm:linear functionals as Laplacians}, the resulting Laplacian comes from the functional $\frac{\phi + \phi S}{2}$, which is Hermitian. Hence, no further generality is gained by relaxing this condition.
\end{rmk}

\section{The \texorpdfstring{$q$}{TEXT}-deformations of compact semisimple Lie groups}\label{sec:The q-deformation}

For the remainder of the paper, we focus on compact quantum groups arising from $q$-deformations of compact semisimple Lie groups. This section summarizes the notations and results concerning these examples as presented in \cite{VoigtYuncken}, which will be used in Sections~\ref{sec:FODCs on Kq}--\ref{sec:Laplacians on Kq}. Any statements not explicitly stated in \cite{VoigtYuncken} will be proved.

Throughout, we fix $0 < q < 1$, let $h \in \Rbb$ be such that $q = e^h$, and set $\hbar = \frac{h}{2\pi}$.

\subsection{Semisimple Lie algebras}\label{subsec:Semisimple Lie algebras}

Let $K$ be a simply connected compact semisimple Lie group with Lie algebra $\kf$. Let $\gf = \Cbb \otimes_\Rbb \kf$, and denote its Killing form by $(\cdot,\cdot)$. This form allows us to identify $\gf \cong \gf^*$, and we transfer the Killing form to $\gf^*$, still denoting it by $(\cdot, \cdot): \gf^* \times \gf^* \rightarrow \Cbb$.

Fix a maximal torus $T \subseteq K$ with Lie algebra $\tf \subseteq \kf$, and let $\boldsymbol{\Delta}$ be the set of roots associated with the Cartan subalgebra $\hf = \Cbb \otimes_\Rbb \tf$. Fix a set of positive roots $\boldsymbol{\Delta}^+ \subseteq \boldsymbol{\Delta}$, and let $\{ \alpha_1 , \cdots, \alpha_N \}$ be the associated simple roots. Let $\{\varpi_1 , \cdots , \varpi_N \}$ be the corresponding fundamental weights, and denote the associated Cartan matrix by $a = (a_{ij})_{1 \leq i,j \leq N}$, i.e.,
\[
a_{ij} = \frac{2 (\alpha_i , \alpha_j)}{(\alpha_i, \alpha_i)}, \quad 1 \leq i,j \leq N.
\]

Let $\Pbf$ and $\Qbf$ be the abelian subgroups of $(i \tf)^*$ generated by $\{\varpi_1, \cdots, \varpi_N\}$ and $\{\alpha_1 , \cdots , \alpha_N\}$, respectively, referred to as \textit{the weight lattice and the root lattice}. Let $\Pbf^+ \subseteq \Pbf$ and $\Qbf^+ \subseteq \Qbf$ denote the subsets consisting of nonnegative integral linear combinations of the respective generators. Define $d_j = \frac{(\alpha_j, \alpha_j)}{2} \in \Qbb$ and set, for $1 \leq j \leq N$,
\begin{equation*}
    \alpha_j^\lor = d_j^{-1} \alpha_j, \quad \varpi_j^\lor = d_j^{-1} \varpi_j.
\end{equation*}
Then we have
\[
(\alpha_i^\lor, \varpi_j) = \delta_{ij} = (\alpha_i, \varpi_j^\lor), \quad 1 \leq i, j \leq N.
\]
Let $\Qbf^\lor$ and $\Pbf^\lor$ be the abelian subgroups of $(i \tf)^*$ generated by $\{\alpha_1^\lor , \cdots , \alpha_N^\lor \}$ and $\{\varpi_1^\lor , \cdots , \varpi_N^\lor \}$, respectively. Note that $\Qbf^\lor$ (resp. $\Pbf^\lor$) is the $\Zbb$-dual of $\Pbf$ (resp. $\Qbf$) with respect to the Killing form. In particular, we have $\Qbf \subseteq \Pbf$ and $\Qbf^\lor \subseteq \Pbf^\lor$.

The Weyl group associated with the root system $\boldsymbol{\Delta}$, denoted by $W$, is the finite subgroup of $GL(\hf^*)$ generated by the reflections
\[
s_j : \hf^* \ni \zeta \longmapsto \zeta - 2 \frac{(\zeta, \alpha_j)}{(\alpha_j, \alpha_j)} \alpha_j \in \hf^*, \quad 1 \leq j \leq N.
\]
Note that elements of $W$ preserve the Killing form by definition. Let $w_0 \in W$ denote the longest element of the Weyl group, i.e., the length of a reduced expression
\begin{equation}\label{eq:the longest Weyl group element}
    w_0 = s_{i_1} \cdots s_{i_t}, \quad 1 \leq i_1, \cdots, i_t \leq N
\end{equation}
is maximal among all elements of $W$. This element is unique and satisfies $w_0^2 = \id$. We fix the reduced expression \eqref{eq:the longest Weyl group element} for $w_0$. Then $t$ equals the cardinality of $\boldsymbol{\Delta}^+$, and
\begin{equation}\label{eq:enumeration of positive roots}
    \beta_{i_r} = s_{i_1} \cdots s_{i_{r-1}} \alpha_{i_r}, \quad 1 \leq r \leq t
\end{equation}
gives an enumeration of the elements of $\boldsymbol{\Delta}^+$, see \cite[Section~5.6]{Humphreys_reflection}. Moreover, the map $-w_0 : \hf^* \rightarrow \hf^*$ permutes the elements of $\boldsymbol{\Delta}^+$, see \cite[Section~II.6]{Knapp}.

Let $\pi : \gf \rightarrow \End(V)$ be a finite-dimensional complex Lie algebra representation. Then, the set $\{ \pi(H) \mid H \in \hf \}$ is simultaneously diagonalizable, and there exists a subset $\Pbf(\pi) \subseteq \Pbf$, whose elements are called the \textit{weights of $\pi$}, such that for each $H \in \hf$, the eigenvalues of the operator $\pi(H)$ are given by $\{ \nu(H) \mid \nu \in \Pbf(\pi) \}$. A nonzero vector $v \in V$ is called a \textit{weight vector} if there exists $\nu \in \Pbf(\pi)$ such that $\pi(H)v = \nu(H)v$ for all $H \in \hf$, in which case $\nu$ is called the \textit{weight of $v$}. The representation $\pi$ is said to be \textit{irreducible} if $V$ has no proper subspace invariant under all elements of $\pi(\gf)$.

Given an irreducible representation of $\gf$, there exists a unique \textit{highest weight} $\mu \in \Pbf(\pi)$ such that every $\nu \in \Pbf(\pi)$ can be written as
\[
\nu = \mu - \sum_{j=1}^N m_j \alpha_j, \quad m_j \in \Nbb.
\]
The Weyl group $W$ maps $\Pbf(\pi)$ into itself. Since $-w_0$ permutes $\boldsymbol{\Delta}^+$, every $\nu \in \Pbf(\pi)$ can also be expressed as
\[
\nu = w_0 \mu + \sum_{j=1}^N m_j \alpha_j, \quad m_j \in \Nbb,
\]
so that $w_0 \mu$ is called the \textit{lowest weight} of $\pi$.

Each highest weight lies in $\Pbf^+$, and the correspondence that assigns to each irreducible finite-dimensional complex Lie algebra representation its highest weight defines a one-to-one correspondence between the set of equivalence classes of irreducible finite-dimensional complex Lie algebra representations of $\gf$ and the set $\Pbf^+$. For each $\mu \in \Pbf^+$, we denote the corresponding irreducible representation by $(\pi_\mu , V(\mu))$, with set of weights $\Pbf(\mu)$. By the universal property of $U(\gf)$ (cf. Definition~\ref{defn:the classical universal enveloping algebra}), $\pi_\mu$ extends to an irreducible algebra representation of $U(\gf)$, which we also denote by $\pi_\mu$.

As $K$ is simply connected, there is a one-to-one correspondence between the unitary representations of $K$ and the finite-dimensional complex Lie algebra representations of $\gf = \Cbb \otimes_\Rbb \kf$, see \cite[Theorem~20.19]{JohnLee}. In particular, $\Irr(K)$, the set of irreducible unitary representations of $K$, can be identified with $\weights^+$. Thus, via \eqref{eq:classical matrix coefficients identification}, we have the following identification.
\begin{equation}\label{eq:classical matrix coefficients identification-semisimple case}
    \Cf^\infty(K) \cong \bigoplus_{\mu \in \weights^+} \Lbb(V(\mu))^*
\end{equation}

\subsection{Quantized universal enveloping algebra}\label{subsec:Quantized universal enveloping algebra}

For $z \in \Cbb$, define
\begin{equation*}
    [z]_q = \frac{q^z - q^{-z}}{q - q^{-1}}.
\end{equation*}
For $n \in \Nbb$, also define $[n]_q ! = \prod_{1 \leq k \leq n} [k]_q$ and
\[
\left[ \begin{array}{c} n \\ k \end{array} \right]_q = \frac{[n]_q!}{[k]_q ! [n-k]_q !}, \quad 0 \leq k \leq n.
\]
Let $q_j = q^{d_j}$ for $j = 1, \dots, N$.

\begin{defn}\label{defn:the quantized universal enveloping algebra}
Let $U_q(\gf)$ be the unital $\Cbb$-algebra generated by $\{K_\lambda, E_j, F_j \mid \lambda \in \Pbf,\, 1 \leq j \leq N\}$, subject to the following relations for $1 \leq i,j \leq N$ and $\lambda, \mu \in \Pbf^+$:
\begin{enumerate}[label=U\arabic*., series=U]
\item $K_0 = 1,\quad K_\lambda K_\mu = K_{\lambda + \mu}$
\item $K_\lambda E_j K_{-\lambda} = q^{(\lambda, \alpha_j)} E_j,\quad K_\lambda F_j K_{-\lambda} = q^{-(\lambda, \alpha_j)} F_j$
\item $[E_i, F_j] = \delta_{ij} \dfrac{K_i - K_i^{-1}}{q_i - q_i^{-1}}$
\item If $i \neq j$, then
\[
\sum_{k=0}^{1 - a_{ij}} \left[ \begin{array}{c} 1 - a_{ij} \\ k \end{array} \right]_{q_i} E_i^{1 - a_{ij} - k} E_j E_i^k 
= \sum_{k=0}^{1 - a_{ij}} \left[ \begin{array}{c} 1 - a_{ij} \\ k \end{array} \right]_{q_i} F_i^{1 - a_{ij} - k} F_j F_i^k = 0
\]
\end{enumerate}
where $K_i = K_{\alpha_i}$.

Then $U_q(\gf)$ becomes a Hopf algebra with the following structure maps:
\begin{enumerate}[label=U\arabic*., resume=U]
\item (Comultiplication) $\hat{\Delta}: U_q(\gf) \to U_q(\gf) \otimes U_q(\gf)$ given by
\begin{gather*}
\hat{\Delta}(K_\lambda) = K_\lambda \otimes K_\lambda, \\
\hat{\Delta}(E_j) = 1 \otimes E_j + E_j \otimes K_j,\quad 
\hat{\Delta}(F_j) = K_j^{-1} \otimes F_j + F_j \otimes 1
\end{gather*}
\item (Counit) $\hat{\epsilon}: U_q(\gf) \to \Cbb$ given by
\[
\hat{\epsilon}(K_\lambda) = 1,\quad \hat{\epsilon}(E_j) = \hat{\epsilon}(F_j) = 0
\]
\item (Antipode) $\hat{S}: U_q(\gf) \to U_q(\gf)$ given by
\[
\hat{S}(K_\lambda) = K_\lambda^{-1},\quad \hat{S}(E_j) = -E_j K_j^{-1},\quad \hat{S}(F_j) = -K_j F_j
\]
\end{enumerate}

This Hopf algebra is called \textbf{the quantized universal enveloping algebra of $\gf$}. It becomes a Hopf $*$-algebra with the following involution, denoted $U_q^\Rbb(\kf)$:
\begin{enumerate}[label=U\arabic*., resume=U]
\item (Involution)
\[
K_\lambda^* = K_\lambda,\quad E_j^* = K_j F_j,\quad F_j^* = E_j K_j^{-1}
\]
\end{enumerate}
\end{defn}

When referring to properties of $U_q^\Rbb(\kf)$ independent of the $*$-structure, we will write $U_q(\gf)$. In particular, all statements in this subsection formulated for $U_q(\gf)$ hold for the quantized universal enveloping algebra over any field, such as $\Qbb(s)$.

Let $\rho = \varpi_1 + \cdots + \varpi_N \in \Pbf^+$. Then
\[
\hat{S}^2(X) = K_{2\rho} X K_{-2\rho}, \quad X \in U_q(\gf).
\]

Let $\Tcal_1, \dots, \Tcal_N : U_q(\gf) \to U_q(\gf)$ denote the algebra automorphisms defined in \cite[Theorems~3.58--3.59]{VoigtYuncken}, inducing an action of the braid group $B_\gf$ on $U_q(\gf)$. Let $\beta_1, \dots, \beta_t$ (as in \eqref{eq:enumeration of positive roots}) enumerate $\boldsymbol{\Delta}^+$. Define
\[
E_{\beta_r} = \Tcal_{i_1} \cdots \Tcal_{i_{r-1}} E_{i_r}, \quad 
F_{\beta_r} = \Tcal_{i_1} \cdots \Tcal_{i_{r-1}} F_{i_r} \in U_q^\Rbb(\kf), \quad 1 \leq r \leq t.
\]
By \cite[Lemma~3.61]{VoigtYuncken}, if $\beta_r = \alpha_j$ for some $j$, then $E_{\beta_r} = E_j$. Moreover, by \cite[Theorem~3.58]{VoigtYuncken},
\begin{align*}
K_\lambda E_{\beta_r} K_{-\lambda} 
= \Tcal_{i_1} \cdots \Tcal_{i_{r-1}} \Big( K_{s_{i_{r-1}}^{-1} \cdots s_{i_1}^{-1} \lambda} E_{i_r} K_{s_{i_{r-1}}^{-1} \cdots s_{i_1}^{-1} \lambda}^{-1} \Big) = q^{(s_{i_{r-1}}^{-1} \cdots s_{i_1}^{-1} \lambda, \alpha_{i_r})} E_{\beta_r} \\ 
= q^{(\lambda, \beta_r)} E_{\beta_r}
\end{align*}
for all $\lambda \in \weights$. The analogous identity holds for $F_{\beta_r}$. Also,
\[
[E_{\beta_r}, F_{\beta_r}] = \Tcal_{i_1} \cdots \Tcal_{i_{r-1}} [E_{i_r}, F_{i_r}] 
= \frac{K_{\beta_r} - K_{\beta_r}^{-1}}{q_{\beta_r} - q_{\beta_r}^{-1}},
\]
where $q_{\beta_r} := q_{i_r}$ for $1 \leq r \leq t$. Thus, for all $\alpha \in \boldsymbol{\Delta}^+$ and $\lambda \in \weights$,
\begin{align}\label{eq:relations for the root vectors}
K_\lambda E_\alpha K_{-\lambda} = q^{(\lambda, \alpha)} E_\alpha, \quad 
K_\lambda F_\alpha K_{-\lambda} = q^{-(\lambda, \alpha)} F_\alpha, \quad 
[E_\alpha, F_\alpha] = \frac{K_\alpha - K_\alpha^{-1}}{q_\alpha - q_\alpha^{-1}}.
\end{align}

The following elements form a \textit{PBW-basis} of $U_q(\gf)$:
\begin{equation}\label{eq:PBW-basis}
F_{\beta_1}^{b_1} \cdots F_{\beta_t}^{b_t} K_\lambda E_{\beta_1}^{a_1} \cdots E_{\beta_t}^{a_t}, \quad a_j, b_j \in \Nbb, \ \lambda \in \weights.
\end{equation}

Let $\pi: U_q(\gf) \to \End(V)$ be a representation on a finite-dimensional vector space $V$. We say that $\pi$ is \textit{integrable} if the operators $\{\pi(K_\lambda) \mid \lambda \in \Pbf^+\}$ are simultaneously diagonalizable and there exists a subset $\Pbf(\pi) \subseteq \Pbf$ (the set of \textit{weights} of $\pi$) such that for each $\lambda \in \weights$, the eigenvalues of $\pi(K_\lambda)$ are given by $\{ q^{(\lambda, \nu)} \mid \nu \in \Pbf(\pi) \}$.

A nonzero vector $v \in V$ is a \textit{weight vector} if $\pi(K_\lambda)v = q^{(\lambda, \nu)}v$ for all $\lambda \in \Pbf$, for some $\nu \in \Pbf(\pi)$ (the \textit{weight} of $v$). The root vectors $E_\alpha, F_\alpha$ ($\alpha \in \boldsymbol{\Delta}^+$) act on $v$ by ``raising'' or ``lowering'' its weight:
\begin{equation}\label{eq:raising and lowering actions of root vectors}
\pi(K_\lambda)\pi(E_\alpha)v = q^{(\lambda, \nu + \alpha)} \pi(E_\alpha)v,\quad 
\pi(K_\lambda)\pi(F_\alpha)v = q^{(\lambda, \nu - \alpha)} \pi(F_\alpha)v.
\end{equation}

The representation $\pi$ is said to be \textit{irreducible} if $V$ has no proper $U_q(\gf)$-submodules. In this case, the \textit{highest and lowest weights} of the representation are defined in terms of its weights, just as in the case of $\gf$. Each highest weight lies in $\Pbf^+$, and assigning to each irreducible, integrable, finite-dimensional representation its highest weight yields a bijection between the set of isomorphism classes of such representations and $\Pbf^+$. As in the case of $\gf$, for $\mu \in \Pbf^+$, let $(\pi_\mu, V(\mu))$ denote the corresponding representation, with set of weights $\Pbf(\mu)$. When irreducible representations of both $\gf$ and $U_q(\gf)$ must be considered simultaneously, we denote those associated with the latter using superscript notation, e.g., $({}^q \pi_\mu, {}^q V(\mu))$. The space $V(\mu)$ admits an inner product that makes $\pi_\mu$ a $*$-representation of $U_q^\Rbb(\kf)$, and we fix such a Hilbert space structure on each $V(\mu)$.

The representation associated with $0 \in \Pbf^+$ is the counit $\hat{\epsilon}$, referred to as the \textit{trivial representation}.

Define $U_q(\hf) := \Span_\Cbb\{ K_\lambda \mid \lambda \in \weights^+ \}$, a Hopf subalgebra of $U_q(\gf)$. Let $U_q(\nf_\pm)$ be the subalgebras generated by $\{E_j\}$ and $\{F_j\}$, respectively. Note that
\[
U_q(\nf_\pm) U_q(\hf) = U_q(\hf) U_q(\nf_\pm),
\]
which we denote by $U_q(\bfrak_\pm)$. These are Hopf subalgebras of $U_q(\gf)$. The multiplication map yields a vector space isomorphism
\begin{equation*}
U_q(\nfrak_-) \otimes U_q(\hf) \otimes U_q(\nfrak_+) \xrightarrow{\;\; \cong \;\;} U_q(\gf).
\end{equation*}

\subsection{Quantized Algebra of Functions}\label{subsec:Quantized algebra of functions}

\begin{defn}\label{defn:the quantized algebra of functions}
The space
    \begin{equation}\label{eq:quantized algebra of functions}
    \Cf^\infty (K_q) = \bigoplus_{\mu \in \Pbf^+} \Lbb(V(\mu))^*
    \end{equation}
admits a unique Hopf $*$-algebra structure for which the following pairing becomes a nondegenerate skew-pairing between $U_q^\Rbb(\kf)$ and $\Cf^\infty(K_q)$:
\begin{equation}\label{eq:skew-pairing between Uq(K) and Cinfty(Kq)}
    (\,\cdot\,, \,\cdot\,) : U_q^\Rbb(\kf) \times \Cf^\infty(K_q) \ni (X, f) \longmapsto \sum_{\mu \in \Pbf^+} f_\mu( \pi_\mu(X) ) \in \Cbb.
\end{equation}

With the Haar state $\hcal : \Cf^\infty(K_q) \to \Cbb$ defined as the projection onto the component $\Lbb(V(0))^* \cong \Cbb$, the algebra $\Cf^\infty(K_q)$ becomes a CQG, referred to as \textbf{the quantized algebra of functions on $K$}.
\end{defn}

Via the skew-pairing \eqref{eq:skew-pairing between Uq(K) and Cinfty(Kq)}, we will often identify elements of $U_q^\Rbb(\kf)$ with linear functionals on $\Cf^\infty(K_q)$.

For each $\mu \in \weights^+$, fix an orthonormal basis $\{e_1^\mu, \dots, e_{n_\mu}^\mu\}$ of $V(\mu)$ consisting of weight vectors; that is, for each $1 \leq j \leq n_\mu$, there exists $\epsilon_j^\mu \in \Pbf (\mu)$ such that
\begin{equation*}
    \pi_\mu(K_\lambda) e_j^\mu = q^{(\lambda, \epsilon_j^\mu)} e_j^\mu, \quad \lambda \in \Pbf.
\end{equation*}
For $v, w \in V(\mu)$, define the linear functional $\la v \mid \cdot \mid w \ra \in \Lbb(V(\mu))^*$ by
\begin{equation*}
    \la v \mid \cdot \mid w \ra (T) = \la v, T w \ra, \quad T \in \Lbb(V(\mu)).
\end{equation*}
Then, the elements
\begin{equation*}
    u^\mu_{ij} = \la e_i^\mu \mid \cdot \mid e_j^\mu \ra \in \Lbb(V(\mu))^*, \quad 1 \leq i,j \leq n_\mu,
\end{equation*}
form a unitary corepresentation of $\Cf^\infty(K_q)$, and the family $\{ u^\mu_{ij} \mid \mu \in \Pbf^+, \, 1 \leq i,j \leq n_\mu \}$ constitutes a Peter–Weyl basis of $\Cf^\infty(K_q)$.

Transposing the inclusion maps $U_q(\bfrak_\pm) \hookrightarrow U_q(\gf)$ yields Hopf algebra homomorphisms
\[
U_q(\gf)^\circ \longrightarrow U_q(\bfrak_\pm)^\circ.
\]
Since $\Cf^\infty(K_q) \subseteq U_q(\gf)^\circ$ via the nondegenerate pairing \eqref{eq:skew-pairing between Uq(K) and Cinfty(Kq)}, its images under these maps are Hopf subalgebras of $U_q(\bfrak_\pm)^\circ$, which we denote by $\Ocal(B_q^\pm)$. Endowing $\Ocal(B_q^\pm)$ with the multiplication and comultiplication structures opposite to those of $U_q(\bfrak_\pm)^\circ$, the restricted maps
\[
\Psi^\pm : \Cf^\infty(K_q) \longrightarrow \Ocal(B_q^\pm)
\]
become surjective Hopf algebra homomorphisms.

Let $\tau : U_q(\bfrak_+) \times U_q(\bfrak_-) \rightarrow \Cbb$ denote the Drinfeld pairing (cf. \cite[Definition~3.74]{VoigtYuncken}), a skew-pairing characterized by
\[
\tau(K_\lambda, K_\mu) = q^{-(\lambda, \mu)}, \quad \tau(E_i, K_\mu) = 0 = \tau(K_\lambda, F_j), \quad \tau(E_i, F_j) = \frac{-\delta_{ij}}{q_i - q_i^{-1}}
\]
for $\lambda, \mu \in \weights$ and $1 \leq i,j \leq N$. Then, the maps
\begin{align*}
&\iota_+ : U_q(\bfrak_+) \ni X \longmapsto \tau( \hat{S}(X), \, \cdot \, ) \in \Ocal(B_q^-) \\
&\iota_- : U_q(\bfrak_-) \ni Y \longmapsto \tau( \, \cdot \, , Y ) \in \Ocal(B_q^+)
\end{align*}
are well-defined Hopf algebra isomorphisms. Observe also that the projections $U_q(\bfrak_\pm) \cong U_q(\nfrak_\pm) \otimes U_q(\hf) \xrightarrow{ \hat{\epsilon} \otimes \id } U_q(\hf)$ are Hopf algebra homomorphisms. Combining all these, we obtain a surjective Hopf algebra homomorphism
\begin{equation}\label{eq:Cfinfty(Kq) -> Uq(h)}
    \Phi: \Cf^\infty(K_q) \xrightarrow{\Psi^+} \Ocal(B_q^+) \xrightarrow[\cong]{\iota_-^{-1}} U_q(\bfrak_-) \xrightarrow{ \hat{\epsilon} \otimes \id} U_q(\hf).
\end{equation}
Let $\mu \in \weights^+$ and $1 \leq i,j \leq n_\mu$. If $i \neq j$, then $\Phi(u^\mu_{ij}) = 0$ due to the last map in \eqref{eq:Cfinfty(Kq) -> Uq(h)}. If $i = j$, then a straightforward computation shows that $\Psi^+(u^\mu_{jj}) = \iota_-(K_{-\epsilon_j^\mu})$. Hence, we deduce:
\begin{equation}\label{eq:cfinifty(Kq) -> Uq(h) calculation}
    \Phi(u^\mu_{ij}) = K_{-\epsilon_j^\mu} \delta_{ij}.
\end{equation}

\subsection{Dual Hopf \texorpdfstring{$*$}{TEXT}-algebra of \texorpdfstring{$\Cf^\infty(K_q)$}{TEXT}}\label{subsec:Dual Hopf algebra of the quantized algebra of functions}

Let $\zeta \in \hf^*$. As in \cite[Section~6.1.1]{VoigtYuncken}, we define an element $K_\zeta \in \Cf^\infty (K_q)^*$ by
\begin{equation*}
    (K_\zeta , u^\mu _{ij}) = q^{(\zeta , \epsilon_j ^\mu)} \delta_{ij}, \quad \mu \in \Pbf^+, \, 1 \leq i,j \leq n_\mu.
\end{equation*}
Note that if $\zeta \in \Pbf$, then this definition coincides with the generator $K_\zeta \in U_q (\gf)$ from Definition~\ref{defn:the quantized universal enveloping algebra}, embedded into $\Cf^\infty(K_q)^*$ via \eqref{eq:skew-pairing between Uq(K) and Cinfty(Kq)}. One can verify that $\hat{\Delta}(K_\zeta) = K_\zeta \otimes K_\zeta$, and hence $K_\zeta \in \Cf^\infty(K_q)^\circ$, with relations
\[
K_{\zeta} K_{\xi} = K_{\zeta + \xi}, \quad \hat{S}^{\pm 1} (K_\zeta) = K_{- \zeta}, \quad K_\zeta ^* = K_{-\overline{\zeta}}
\]
for all $\zeta, \xi \in \hf^*$, where $\overline{(\cdot)} : \hf^* \rightarrow \hf^*$ denotes the conjugation with respect to the real form $\tf^* \subseteq \hf^*$. Moreover, $K_\zeta = 1$ if and only if $\zeta \in i \hbar^{-1 } \Qbf^\lor$. Accordingly, by abuse of notation, we will often write $\zeta \in \hf^* / i \hbar^{-1} \Qbf^\lor$ when referring to the parameter $\zeta$ defining $K_\zeta$. Finally, note that
\begin{equation}\label{eq:commutation for characters on torus}
K_\zeta E_\alpha = q^{(\zeta , \alpha)} E_\alpha K_\zeta , \;\; K_\zeta F_\alpha = q^{-(\zeta, \alpha)} F_\alpha K_\zeta , \quad \alpha \in \boldsymbol{\Delta}^+, \,\zeta \in \hf^*,
\end{equation}
which can be verified from the definition of $K_\zeta$ and the actions of the elements $\{ E_\alpha , \, F_\alpha \mid \alpha \in \boldsymbol{\Delta}^+ \}$ on weight vectors, see \eqref{eq:raising and lowering actions of root vectors}.

Recall that since $\hf$ is abelian, every $\lambda \in \hf^*$ extends to an algebra homomorphism $U(\hf) \rightarrow \Cbb$ via
\begin{equation*}
    \lambda( H_1 \cdots H_n ) = \lambda(H_1) \cdots \lambda(H_n), \quad H_1, \cdots , H_n \in \hf.
\end{equation*}
Given $X \in U(\hf)$, define $D_X \in \Cf^\infty (K_q)^*$ by
\[
    ( D_X , u^\mu _{ij} ) = ( - \epsilon_j ^\mu ) (X) \delta_{ij} , \quad \mu \in \Pbf^+, \, 1 \leq i,j \leq n_\mu.
\]
Note that $(- \epsilon_j ^\mu)(X)$ may not be equal to $- \epsilon_j ^\mu (X)$, depending on the degree of the monomials in $X$. For $X, Y \in U(\hf)$, we have
\begin{equation}\label{eq:multiplicativity of DX}
( D_{XY} , u^\mu _{ij} ) = (- \epsilon_j ^\mu) (XY) \delta_{ij} = (- \epsilon_j ^\mu)(X)( - \epsilon_j ^\mu)(Y) \delta_{ij} = ( D_X D_Y , u^\mu _{ij} ).
\end{equation}
For $H \in \hf \subseteq U(\hf)$, $\lambda, \mu \in \weights^+$, and appropriate indices,
\begin{align*}
    ( D_H , u^\lambda _{ij} u^\mu _{kl} ) &= - (\epsilon_j ^\lambda + \epsilon_l ^\mu ) (H) \delta_{ij} \delta_{kl} \\
    &= (D_H , u^\lambda _{ij} ) ( 1 , u^\mu _{kl} ) + ( 1 , u^\lambda _{ij} ) ( D_H , u^\mu _{kl} ) \\
    &= \big(D_H \otimes 1 + 1 \otimes D_H , u^\mu _{kl} \otimes u^\lambda _{ij} \big),
\end{align*}
and hence $D_H \in \Cf^\infty(K_q)^\circ$ with $\hat{\Delta} (D_H) = D_H \otimes 1 + 1 \otimes D_H$, which implies $D_X \in \Cf^\infty(K_q)^\circ$ for all $X \in U(\hf)$ by \eqref{eq:multiplicativity of DX}.

Translating \cite[Proposition~9.4.9]{Joseph} into our conventions yields:

\begin{prop}\label{prop:the Hopf dual of Cinfty(Kq)}
The multiplication map of $\Cf^\infty(K_q)^\circ$ induces an isomorphism
\begin{equation}\label{eq:the Hopf dual of Cinfty(Kq)}
    U_q (\nf_-) \otimes \{ D_X \mid X \in U(\hf) \} \otimes \Span_\Cbb \{ K_\zeta \mid \zeta \in \hf^* \} \otimes U_q (\nf_+) \xrightarrow[\cong]{} \Cf^\infty(K_q)^\circ.
\end{equation}
Moreover, the map $U(\hf) \ni X \mapsto D_X \in \Cf^\infty(K_q)^\circ$ is injective, and the elements $\{K_\zeta \mid \zeta \in \hf^*/ i \hbar^{-1} \Qbf^\lor \}$ are linearly independent.
\end{prop}
\begin{proof}
Note that $U_q (\hf)^\circ$ corresponds to $(\check{U}^0 )^\star$ in \cite{Joseph}. Thus, by the proof of \cite[Proposition~9.4.9]{Joseph}, it suffices to verify that the image of $U_q(\hf)^\circ$ under the transpose of the map \eqref{eq:Cfinfty(Kq) -> Uq(h)} is isomorphic to
\[
\{ D_X \mid X \in U(\hf) \} \otimes \Span_\Cbb \{ K_\zeta \mid \zeta \in \hf^* \} \subseteq \Cf^\infty(K_q)^\circ.
\]

First, recall that $U_q (\hf) \cong \Cf^\infty (T)$ as Hopf algebras via $K_\lambda \mapsto t^\lambda$ for $\lambda \in \weights$, where $t^\lambda :T \rightarrow \Cbb$ is defined by
\[
t^\lambda ( \exp (H) ) = e^{\lambda(H)}, \quad H \in \tf,
\]
which is well-defined by \cite[Theorem~5.107]{Knapp}, see also \cite[Proposition~4.58]{Knapp}. The Peter–Weyl decomposition of $\Cf^\infty(T)$ (cf.~\eqref{eq:classical matrix coefficients identification}) implies this is an isomorphism.

To compute $\Cf^\infty(T)^\circ$, recall that $U(\hf) \subseteq \Cf^\infty(T)^\circ$ via the non-degenerate pairing \eqref{eq:classical-skew-pairing}. Also, for each $\zeta \in \hf^*$, the map
\[
e_\zeta : \Cf^\infty(T) \ni t^\lambda \longmapsto q^{-(\zeta, \lambda)} = e^{( - h \zeta , \lambda)} \in \Cbb , \quad \lambda \in \weights,
\]
is an algebra homomorphism, so $e_\zeta \in \Cf^\infty(T)^\circ$. In fact, $\{ e_\zeta \mid \zeta \in \hf^* \} = \Cf^\infty(T)^\wedge$, the set of non-trivial algebra homomorphisms $\Cf^\infty(T) \to \Cbb$, see \cite[Section~III.8]{tomDieck}. Hence, by \cite[Theorem~2.1.8]{Joseph}, the multiplication map gives an isomorphism:
\[
U(\hf) \otimes \Span_\Cbb \{ e_\zeta \mid \zeta \in \hf^* \} \cong \Cf^\infty(T)^\circ.
\]

We now claim that under $\Phi^*$, the subsets $U(\hf)$ and $\Span_\Cbb \{ e_\zeta \}$ in $\Cf^\infty(T)^\circ \cong U_q(\hf)^\circ$ map to $\{ D_X \}$ and $\Span_\Cbb \{ K_\zeta \}$ in $\Cf^\infty(K_q)^\circ$, respectively. Let $\mu \in \weights^+$ and $1 \leq i,j \leq n_\mu$. Then, using \eqref{eq:cfinifty(Kq) -> Uq(h) calculation} and the identification $U_q (\hf) \cong \Cf^\infty(T)$:
\[
( \Phi^* ( H ) , u^\mu _{ij} ) = (H , t^{- \epsilon_j ^\mu} ) \delta_{ij} = \left. \frac{d}{dt} \right|_{t = 0} e^{- \epsilon_j ^\mu ( tH ) } \delta_{ij} = -\epsilon_j ^\mu (H) \delta_{ij} = ( D_{H} , u^\mu _{ij} )
\]
for all $H \in \hf$, so $\Phi^*(X) = D_X$ for all $X \in U(\hf)$. Similarly,
\[
( \Phi^* ( e_\zeta ) , u^\mu _{ij} ) = ( e_\zeta , t^{- \epsilon_j ^\mu} ) \delta_{ij} = q^{(\zeta , \epsilon_j ^\mu )} \delta_{ij} = (K_{\zeta} , u^\mu _{ij}),
\]
for all $\zeta \in \hf^*$, completing the claim.

The second statement follows from the injectivity of $\Phi^*$---a consequence of the surjectivity of $\Phi$---and from the linear independence of the characters $\{ e_\zeta \}$ on $\Cf^\infty(T)$, which follows from Artin's theorem on the linear independence of characters applied to the group $\Pbf$.
\end{proof}

The following corollary will be useful in subsequent arguments.

\begin{cor}\label{cor:Kzeta independence}
Let $\eta_1 , \cdots , \eta_m \in \tf^*/ i \hbar^{-1} \Qbf^\lor$ be distinct. Suppose $X_1, \cdots , X_m \in U_q (\gf)$ satisfy
\[
K_{\eta_1} X_1 + \cdots + K_{\eta_m} X_m =0 \quad \text{or} \quad X_1 K_{\eta_1} + \cdots + X_m K_{\eta_m} = 0.
\]
Then $X_1 = \cdots = X_m = 0$, i.e., the elements $K_{\eta_1}, \cdots , K_{\eta_m}$ are $U_q (\gf)$-independent.
\end{cor}
\begin{proof}
Via the multiplication map, we have
\begin{align*}
\Span_\Cbb \{ K_\zeta \mid \zeta \in \hf^*/ i \hbar^{-1} \Qbf^\lor \} \\
&\cong \Span_\Cbb \{ K_\xi \mid \xi \in (i \tf)^* \} \otimes \Span_\Cbb \{ K_\eta \mid \eta \in \tf^* / i \hbar^{-1} \Qbf^\lor \} \nonumber ,
\end{align*}
where the elements inside the span signs are linearly independent by Artin's theorem applied to the group $\weights$. Therefore, by \eqref{eq:PBW-basis} and Proposition~\ref{prop:the Hopf dual of Cinfty(Kq)}, the set
\[
\big \{ F_{\beta_1} ^{b_1} \cdots F_{\beta_t} ^{b_t} K_\xi K_\eta E_{\beta_1} ^{a_1} \cdots E_{\beta_t} ^{a_t} \mid a_i, b_j \in \Nbb, \, \xi \in (i \tf)^*, \, \eta \in \tf^* / i \hbar^{-1} \Qbf^\lor \big\}
\]
is linearly independent in $\Cf^\infty(K_q)^\circ$. But the commutation relations \eqref{eq:commutation for characters on torus} show that
\begin{align*}
q^{( b_1 \beta_1 + \cdots + b_t \beta_t , \eta) } K_\eta (F_{\beta_1} ^{b_1} \cdots F_{\beta_t} ^{b_t} K_\lambda E_{\beta_1} ^{a_1} \cdots E_{\beta_t} ^{a_t}) = F_{\beta_1} ^{b_1} \cdots F_{\beta_t} ^{b_t} K_\lambda K_\eta E_{\beta_1} ^{a_1} \cdots E_{\beta_t} ^{a_t} \\
= q^{( a_1 \beta_1 + \cdots + a_t \beta_t , \eta) } (F_{\beta_1} ^{b_1} \cdots F_{\beta_t} ^{b_t} K_\lambda E_{\beta_1} ^{a_1} \cdots E_{\beta_t} ^{a_t}) K_\eta,
\end{align*}
for all $a_i, b_j \in \Nbb$, $\lambda \in \weights$, and $\eta \in \tf^* / i \hbar^{-1} \Qbf^\lor$, which, together with \eqref{eq:PBW-basis}, completes the proof.
\end{proof}

\subsection{Universal \texorpdfstring{$R$}{TEXT}-matrix and related constructions}\label{subsec:Universal R-matrix}

The $*$-algebra $U_q^\Rbb(\kf)$ is embedded into the $*$-algebra $\prod_{\mu \in \Pbf^+} \Lbb(V(\mu))$ via
\begin{equation*}
    U_q^\Rbb(\kf) \ni X \longmapsto (\pi_\mu(X))_{\mu} \in \prod_{\mu \in \Pbf^+} \Lbb(V(\mu)).
\end{equation*}
With this identification, the skew-pairing \eqref{eq:skew-pairing between Uq(K) and Cinfty(Kq)} extends naturally to
\begin{equation*}
    \Big( \prod_{\mu \in \Pbf^+} \Lbb(V(\mu)) \Big) \times \Cf^\infty(K_q) \ni (x, f) \longmapsto \sum_{\mu \in \Pbf^+} f_\mu(x_\mu) \in \Cbb.
\end{equation*}
This extension also applies to the skew-pairing between the Hopf $*$-algebras $U_q^\Rbb(\kf) \otimes U_q^\Rbb(\kf)$ and $\Cf^\infty(K_q) \otimes \Cf^\infty(K_q)$, defined as the tensor product of \eqref{eq:skew-pairing between Uq(K) and Cinfty(Kq)}. It yields the canonical pairing between $ \prod_{\lambda, \mu \in \Pbf^+} \Lbb(V(\lambda)) \otimes \Lbb(V(\mu))$ and $\Cf^\infty(K_q) \otimes \Cf^\infty(K_q)$, which we also denote by $( \cdot , \cdot )$.

Let $\displaystyle \Rcal \in \prod_{\lambda, \mu \in \Pbf^+} \End(V(\lambda)) \otimes \End(V(\mu))$ be the universal $R$-matrix of $U_q(\gf)$ from \cite[Theorem~3.108]{VoigtYuncken}. Explicitly,
\begin{equation}\label{eq:the universal R-matrix}
    \Rcal = q^{\sum_{i,j=1}^N B_{ij}(H_i \otimes H_j)} \prod_{r=1}^t \exp_{q_{\beta_r}} \big( (q_{\beta_r} - q_{\beta_r}^{-1}) (E_{\beta_r} \otimes F_{\beta_r}) \big),
\end{equation}
where $q^{\sum_{i,j=1}^N B_{ij}(H_i \otimes H_j)}$ denotes a symbolic operator acting on $V(\lambda) \otimes V(\mu)$ via
\begin{equation*}
    q^{\sum_{i,j=1}^N B_{ij}(H_i \otimes H_j)} (e_k^\lambda \otimes e_l^\mu) = q^{(\epsilon_k^\lambda, \epsilon_l^\mu)} e_k^\lambda \otimes e_l^\mu
\end{equation*}
for $\lambda, \mu \in \weights^+$ and respective basis indices, and
\begin{equation*}
    \exp_{q_j}(X) = \sum_{n=0}^\infty \frac{q_j^{n(n-1)/2}}{[n]_{q_j}!} X^n
\end{equation*}
for $\displaystyle X \in \prod_{\lambda, \mu \in \Pbf^+} \End(V(\lambda)) \otimes \End(V(\mu))$ whenever the right-hand side converges in the product topology. Its inverse is given by
\begin{equation}\label{eq:inverse of q-exponential}
    \exp_{q_j^{-1}}(-X) = \sum_{n=0}^\infty \frac{q_j^{-n(n-1)/2}}{[n]_{q_j}!} (-X)^n.
\end{equation}

The element $\Rcal$ is invertible and satisfies
\begin{equation}\label{eq:antipode and universal R-matrix}
    (S \otimes \id)(\Rcal) = \Rcal^{-1} = (\id \otimes S^{-1})(\Rcal).
\end{equation}
Moreover, for $f, g \in \Cf^\infty(K_q)$,
\begin{align}   \label{eq:the universal R-matrix braiding}
    ( \Rcal, f_{(1)} \otimes g_{(1)} ) g_{(2)} f_{(2)} &= f_{(1)} g_{(1)} ( \Rcal, f_{(2)} \otimes g_{(2)} ), \\
    ( \Rcal^{-1}, f_{(1)} \otimes g_{(1)} ) f_{(2)} g_{(2)} &= g_{(1)} f_{(1)} ( \Rcal^{-1}, f_{(2)} \otimes g_{(2)} ). \nonumber
\end{align}

Define $l^\pm : \Cf^\infty(K_q) \rightarrow U_q(\gf)$ by
\begin{equation}\label{eq:l-functionals}
    (l^+(f), g) = (\Rcal, g \otimes f), \quad (l^-(f), g) = (\Rcal^{-1}, f \otimes g),
\end{equation}
for $f, g \in \Cf^\infty(K_q)$. These are Hopf algebra homomorphisms satisfying
\begin{equation*}
    l^\pm(f)^* = l^\mp(f^*).
\end{equation*}

The left and right adjoint actions on $U_q(\gf)$ are given by
\[
X \rightarrow Y = X_{(1)} Y \hat{S}(X_{(2)}), \quad Y \leftarrow X = \hat{S}(X_{(1)}) Y X_{(2)}, \quad X, Y \in U_q(\gf).
\]
These induce left and right $U_q(\gf)$-module structures on $\Cf^\infty(K_q)$ via transpose:
\[
(Y, X \rightarrow f) := (Y \leftarrow X, f), \quad (Y, f \leftarrow X) := (X \rightarrow Y, f), \quad Y \in U_q(\gf)
\]
for $X \in U_q(\gf)$ and $f \in \Cf^\infty(K_q)$.
These are the \emph{left and right coadjoint actions} of $U_q(\gf)$ on $\Cf^\infty(K_q)$. For $Y \in U_q(\gf)$, define:
\[
U_q(\gf) \rightarrow Y = \{ X \rightarrow Y \mid X \in U_q(\gf) \}, \quad Y \leftarrow U_q(\gf) = \{ Y \leftarrow X \mid X \in U_q(\gf) \}.
\]
Set
\begin{align*}
    F_l U_q(\gf) &= \{ Y \in U_q(\gf) \mid \text{$U_q(\gf) \rightarrow Y$ is finite-dimensional} \}, \\
    F_r U_q(\gf) &= \{ Y \in U_q(\gf) \mid \text{$Y \leftarrow U_q(\gf)$ is finite-dimensional} \},
\end{align*}
called the \emph{left and right locally finite parts} of $U_q(\gf)$, respectively. These are subalgebras of $U_q(\gf)$. By \eqref{eq:antipodes and adjoint actions},
\begin{equation}\label{eq:antipode and locally finite parts}
    \hat{S}^{\pm 1}(F_l U_q(\gf)) = F_r U_q(\gf).
\end{equation}

Define $I: \Cf^\infty(K_q) \rightarrow U_q(\gf)$ by
\begin{equation}\label{eq:I-isomorphism}
    I(f) = l^-(f_{(1)}) \hat{S}l^+(f_{(2)}).
\end{equation}
Then $I$ is a linear isomorphism onto $F_l U_q(\gf)$ satisfying
\begin{equation}\label{eq:I preservers adjoint actions}
    I(X \rightarrow f) = X \rightarrow I(f)
\end{equation}
for $X \in U_q(\gf)$. Define
\[
J(f) = \hat{S}^{-1} I(S(f)) = \hat{S}l^+(f_{(1)}) l^-(f_{(2)}) = \hat{S} I(S^{-1}(f)),
\]
which is a linear isomorphism onto $F_r U_q(\gf)$ and satisfies
\begin{equation}\label{eq:J preserves adjoint actions}
    J(f \leftarrow X) = J(f) \leftarrow X
\end{equation}
for $X \in U_q(\gf)$. Furthermore, using \eqref{eq:antipode and universal R-matrix}, one verifies that
\begin{equation}\label{eq:an identity regarding I}
    (I(f), g) = \big(I(S(g)), S(f)\big) = (J(g), f),
\end{equation}
for all $f, g \in \Cf^\infty(K_q)$.

Let $\mu \in \weights^+$ and $1 \leq i,j \leq n_\mu$. Then,
\begin{align}\label{eq:I and involution}
    I(u^\mu_{ij})^* &= \left( \sum_{k} l^-(u^\mu_{ik}) \hat{S}l^+(u^\mu_{kj}) \right)^* = \sum_k l^+(S(u^\mu_{kj}))^* l^-(u^\mu_{ik})^* \\
    &= \sum_k l^-(u^\mu_{jk}) l^+((u^\mu_{ik})^*) = \sum_k l^-(u^\mu_{jk}) \hat{S}l^+(u^\mu_{ki}) = I(u^\mu_{ji}) \nonumber.
\end{align}

Let $v_{w_0 \mu} \in V(\mu)$ be a unit vector of weight $w_0 \mu$. Then
\begin{equation}\label{eq:an evaluation of I}
    I\left( \langle v_{w_0 \mu} \mid \cdot \mid v_{w_0 \mu} \rangle \right) = K_{-2 w_0 \mu}.
\end{equation}

Let $\sigma$ denote the flip map on $\prod_{\lambda, \nu} \End(V(\lambda)) \otimes \End(V(\nu))$, and define $\Rcal_{21} = \sigma(\Rcal)$. Then
\[
\Rcal' = \Rcal_{21}^{-1} \in \prod_{\lambda, \nu} \End(V(\lambda)) \otimes \End(V(\nu))
\]
is also a universal $R$-matrix of $U_q(\gf)$. Noting that
\[
(\Rcal', g \otimes f) = (\Rcal^{-1}, f \otimes g) = (l^-(f), g), \quad ((\Rcal')^{-1}, f \otimes g) = (\Rcal, g \otimes f) = (l^+(f), g),
\]
the analogue of \eqref{eq:I-isomorphism} becomes
\[
I'(f) = l^+(f_{(1)}) \hat{S}l^-(f_{(2)}),
\]
and satisfies
\begin{equation}\label{eq:I' preserves adjoint actions}
    I'(X \rightarrow f) = X \rightarrow I'(f)
\end{equation}
for all $X \in U_q(\gf)$. Let $v_\mu \in V(\mu)$ be a unit vector of highest weight $\mu$. Then,
\begin{equation}\label{eq:an evaluation of I'}
    I'\left( \langle v_\mu \mid \cdot \mid v_\mu \rangle \right) = K_{2 \mu},
\end{equation}
see the first paragraph of \cite[Proposition~3.116]{VoigtYuncken} where this property is shown to be a consequence of the braiding property \eqref{eq:the universal R-matrix braiding} which is shared by all universal $R$-matrices.

Since $\la v_\mu \mid \cdot \mid v_\mu \ra \in \End(V(\mu))^*$ is cyclic in $\End(V(\mu))^*$ with respect to the left coadjoint action of $U_q(\gf)$ on $\Cf^\infty(K_q)$, \eqref{eq:I' preserves adjoint actions}--\eqref{eq:an evaluation of I'} imply that
\[
I' \big( \End(V(\mu))^* \big) = U_q (\gf) \rightarrow K_{2\mu}.
\]
However, the latter set is equal to $I\big( \End(V(-w_0\mu))^* \big)$ (cf.~the proof of \cite[Theorem~3.113]{VoigtYuncken}) and hence has dimension $n_{-w_0 \mu}^2 = n_\mu^2$. This implies that $I'$ is injective on $\End(V(\mu))^*$. Therefore, we conclude that
\[
I' : \Cf^\infty(K_q) = \bigoplus_{\mu \in \weights^+} \End(V(\mu))^* \longrightarrow \bigoplus_{\mu \in \weights^+} \big(U_q (\gf) \rightarrow K_{2\mu}\big) = F_l U_q (\gf)
\]
is an isomorphism.

Also, just as in the case of $I$, we can use \eqref{eq:antipode and universal R-matrix} to prove
\begin{equation}\label{eq:an identity regarding I'}
    (I'(f), g) = \big(I'(S(g)), S(f)\big), \quad f, g \in \Cf^\infty (K_q).
\end{equation}
Finally, we have $I'(u^\mu_{ij})^* = I'(u^\mu_{ji})$ for $\mu \in \weights^+$ and $1 \leq i,j \leq n_\mu$.

\section{Finite-dimensional bicovariant (\texorpdfstring{$*$}{TEXT}-)FODCs on \texorpdfstring{$K_q$}{TEXT}}\label{sec:FODCs on Kq}

In this section, we use the construction of \cite{Jurco1991} to describe a family of finite-dimensional bicovariant FODCs on $K_q$ and show that it yields all finite-dimensional bicovariant FODCs up to isomorphism. This classification was verified in the cases $K_q = SU_q(n+1)$ and $K_q = Sp_q(2n)$ in \cite{Heckenberger_classification}. Our proof relies on a result from \cite{Baumann1998}.

Throughout, we fix $0 < q < 1$.

\subsection{Construction of finite-dimensional bicovariant (\texorpdfstring{$*$}{TEXT}-)FODCs on \texorpdfstring{$K_q$}{TEXT}}\label{subsec:Construction of FODCs on Kq}

\begin{prop}\label{prop:one dimensional BCFODCs}
    Let $\zeta \in i \hbar^{-1} \Pbf^\lor / i \hbar^{-1} \Qbf^\lor$. Then, the two families
    \[
    \{ K_\zeta \} \subseteq \Cf^\infty (K_q)^\circ , \quad \{1 \} \subseteq \Cf^\infty(K_q)
    \]
    satisfy conditions S1--S3 of Proposition~\ref{prop:the structure theorem of bicovariant bimodule}, and thus define a one-dimensional bicovariant bimodule over $\Cf^\infty(K_q)$.
\end{prop}

\begin{proof}
The identities $\hat{\Delta}(K_\zeta) = K_\zeta \otimes K_\zeta$ and $(K_\zeta , 1 ) = q^{(\zeta , 0 )} = 1$ verify S1. Condition S2 is immediate.

Since $\zeta \in i \hbar^{-1} \Pbf^\lor$, we have $q^{(\zeta , \alpha_j)} = e^{\hbar (\zeta, \alpha_j)} = 1$ for all $1 \leq j \leq N$. As every weight of $V(\mu)$ is of the form $\mu - \sum_{j=1}^N n_j \alpha_j$ for $n_j \in \Nbb$, it follows that
\begin{equation*}
    \pi_\mu(K_\zeta) e_j^\mu = q^{(\zeta , \epsilon_j^\mu)} e_j^\mu = q^{(\zeta , \mu)} e_j^\mu.
\end{equation*}
Hence, for all $\mu \in \Pbf^+$ and $1 \leq i,j \leq n_\mu$,
\[
K_\zeta \triangleright u^\mu_{ij} = q^{(\zeta, \epsilon_j^\mu)} u^\mu_{ij} = q^{(\zeta, \mu)} u^\mu_{ij} = q^{(\zeta, \epsilon_i^\mu)} u^\mu_{ij} = u^\mu_{ij} \triangleleft K_\zeta,
\]
which implies
\begin{equation}\label{eq:module action of Kzeta}
    K_\zeta \triangleright f = f \triangleleft K_\zeta
\end{equation}
for all $f \in \Cf^\infty(K_q)$, verifying S3.
\end{proof}

\begin{rmk}\label{rmk:center inside the maximal torus}
Let $\zeta \in i \hbar^{-1} \Pbf^\lor$ and $X \in U_q^\Rbb(\kf)$. From \eqref{eq:module action of Kzeta}, we find that for all $f \in \Cf^\infty(K_q)$,
\begin{align*}
    (X K_\zeta , f) = \epsilon(X K_\zeta \triangleright f) = \epsilon \big( X \triangleright (f \triangleleft K_\zeta) \big) = (X , f \triangleleft K_\zeta) = (K_\zeta X , f).
\end{align*}
By nondegeneracy of the pairing, we conclude
\begin{equation}\label{eq:Kzeta centrality}
    X K_\zeta = K_\zeta X, \quad X \in U_q^\Rbb(\kf).
\end{equation}
Hence, by Proposition~\ref{prop:ad-invariance criterion for a linear functional} with $\Ucal = U_q^\Rbb(\kf)$, we see that $K_\zeta$ is $\ad$-invariant.

More conceptually, regarding $\{ K_\zeta \mid \zeta \in \tf^* / i \hbar^{-1} \Qbf^\lor \}$ as the maximal torus $T$ of $K$ embedded into $\Cf^\infty(K_q)^\circ$ (cf.~\cite[Section~6.1.1]{VoigtYuncken}), the subset $\{ K_\zeta \mid \zeta \in i \hbar^{-1} \Pbf^\lor / i \hbar^{-1} \Qbf^\lor \}$ corresponds to the center $Z$ of $K$, see the final paragraph of \cite[Section~4.3]{VoigtYuncken} and the second paragraph of \cite[Section~4.4.2]{VoigtYuncken}.

Accordingly, from now on, we denote
\[
\Zcal = i \hbar^{-1} \Pbf^\lor / i \hbar^{-1} \Qbf^\lor.
\]
\end{rmk}

Recall that the nondegenerate skew-pairing \eqref{eq:skew-pairing between Uq(K) and Cinfty(Kq)} gives an embedding of Hopf $*$-algebras
\(
    U_q^\Rbb(\kf) \hookrightarrow \Cf^\infty(K_q)^\circ.
\)

\begin{prop}\label{prop:construction of gamma and gamma^c}
Fix $\mu \in \mathbf{P}^+$. The two families of elements
\[
    \big(l^-(u^\mu_{ji})\big)_{1 \leq i,j \leq n_\mu} \subseteq U_q^\Rbb(\kf), \quad \big(u^\mu_{ij}\big)_{1 \leq i,j \leq n_\mu} \subseteq \Cf^\infty(K_q)
\]
satisfy the conditions S1--S3 of Proposition~\ref{prop:the structure theorem of bicovariant bimodule}.

The same holds for the families
\[
    \big(\hat{S}l^+(u^\mu_{ij})\big)_{1 \leq i,j \leq n_\mu} \subseteq U_q^\Rbb(\kf), \quad \big(S(u^\mu_{ji})\big)_{1 \leq i,j \leq n_\mu} \subseteq \Cf^\infty(K_q).
\]
\end{prop}

\begin{proof}
Although a proof can be found in \cite{Jurco1991}, we provide a detailed argument for the reader’s convenience.

In both cases, conditions S1 and S2 follow directly from the fact that $(u^\mu_{ij})_{1 \leq i,j \leq n_\mu}$ is a corepresentation and $l^\pm$ are Hopf algebra homomorphisms.

For S3, consider the following identities, which follow from \eqref{eq:the universal R-matrix braiding}: for any $f \in \Cf^\infty(K_q)$ and $1 \leq i,j \leq n_\mu$,
\begin{align*}
    \sum_{k} u^\mu_{ki} (f \triangleleft l^-(u^\mu_{jk}))
    &= \sum_{k} \big(l^-(u^\mu_{jk}), f_{(1)}\big) u^\mu_{ki} f_{(2)} \\
    &= \big(\Rcal^{-1}, (u^\mu_{ji})_{(1)} \otimes f_{(1)}\big) (u^\mu_{ji})_{(2)} f_{(2)} \\
    &= f_{(1)} (u^\mu_{ji})_{(1)} \big(\Rcal^{-1}, (u^\mu_{ji})_{(2)} \otimes f_{(2)}\big) \\
    &= \sum_{k} f_{(1)} u^\mu_{jk} \big(\Rcal^{-1}, u^\mu_{ki} \otimes f_{(2)}\big) \\
    &= \sum_{k} \big(l^-(u^\mu_{ki}) \triangleright f\big) u^\mu_{jk},
\end{align*}
and similarly,
\begin{align*}
    \sum_{k} S(u^\mu_{ik}) (f \triangleleft \hat{S}l^+(u^\mu_{kj}))
    &= \sum_{k} \big(l^+(S(u^\mu_{kj})), f_{(1)}\big) S(u^\mu_{ik}) f_{(2)} \\
    &= \big(\Rcal, f_{(1)} \otimes S(u^\mu_{ij})_{(1)}\big) S(u^\mu_{ij})_{(2)} f_{(2)} \\
    &= f_{(1)} S(u^\mu_{ij})_{(1)} \big(\Rcal, f_{(2)} \otimes S(u^\mu_{ij})_{(2)}\big) \\
    &= \sum_{k} f_{(1)} S(u^\mu_{kj}) \big(\Rcal, f_{(2)} \otimes S(u^\mu_{ik})\big) \\
    &= \sum_{k} \big(\hat{S}l^+(u^\mu_{ik}) \triangleright f\big) S(u^\mu_{kj}).
\end{align*}
\end{proof}

\begin{defn}\label{defn:definition of Omega zeta mu}
Let $(\zeta,\mu) \in \Zcal \times \Pbf^+$ with $(\zeta,\mu) \neq (0,0)$. By iteratively applying Proposition~\ref{prop:tensor product of structure maps} to the three bicovariant bimodules from Propositions~\ref{prop:one dimensional BCFODCs} and \ref{prop:construction of gamma and gamma^c}, we construct a bicovariant bimodule $\Omega_{\zeta\mu}$ with structure representations
\begin{equation}\label{eq:structure representation-Kq}
    \big(K_\zeta l^-(u^\mu_{ji}) \hat{S}l^+(u^\mu_{kl})\big)_{i,j,k,l} \subseteq \Cf^\infty(K_q)^\circ,\quad \big(u^\mu_{ij} S(u^\mu_{lk})\big)_{i,j,k,l} \subseteq \Cf^\infty(K_q)
\end{equation}
with respect to a fixed invariant basis $\{\omega^{\zeta\mu}_{ik} \mid 1 \leq i,k \leq n_\mu\}$. Note that $\dim \inv \Omega_{\zeta\mu} = n_\mu^2$. We define $\Omega_{00} = 0$, the zero bicovariant $*$-bimodule.
\end{defn}

Define a conjugate-linear map $*:\Omega_{\zeta\mu} \to \Omega_{\zeta\mu}$ by
\begin{equation}\label{eq:involution of Gamma^mu}
    \left(\sum_{i,k} f_{ik} \omega^{\zeta\mu}_{ik} \right)^* = - \sum_{i,k} \omega^{\zeta\mu}_{ki} f_{ik}^*, \quad f_{ik} \in \Cf^\infty(K_q).
\end{equation}

\begin{prop}\label{prop:Gamma^mu is a bicovariant *- bimodule}
Let $(\zeta, \mu) \in \Zcal \times \Pbf^+$ with $(\zeta, \mu) \neq (0,0)$. The bimodule $\Omega_{\zeta\mu}$ becomes a bicovariant $*$-bimodule under the map \eqref{eq:involution of Gamma^mu} if and only if $\zeta \in \frac{i}{2} \hbar^{-1} \Qbf^\lor$.
\end{prop}

\begin{proof}
We first show that $*$ is an involution if and only if $\zeta \in \frac{i}{2} \hbar^{-1} \Qbf^\lor$. Using the identity
\begin{equation}\label{eq:left module action and involution}
    (X \triangleright (f^*) )^* = f_{(1)} \overline{( X, f_{(2)} ^* )} = f_{(1)} \big( \hat{S}^{-1} (X)^* , f_{(2)} \big) = \hat{S}(X^*) \triangleright f,
\end{equation}
valid for all $X \in U_q^\Rbb(\kf)$ and $f \in \Cf^\infty(K_q)$, we compute for $f \in \Cf^\infty(K_q)$ and $1 \leq i,k \leq n_\mu$:
\begin{align*}
    \big((f \omega_{ik}^{\zeta\mu})^*\big)^* 
    &= (-\omega_{ki}^{\zeta\mu} f^*)^* \\
    &= \bigg(- \sum_{1 \leq j,l \leq n_\mu} \Big( \big(K_\zeta l^- (u^\mu_{lk}) \hat{S}l^+ (u^\mu_{ij})\big) \triangleright f^* \Big) \omega_{lj}^{\zeta\mu} \bigg)^* \\
    &= \sum_{1 \leq j,l \leq n_\mu} \omega_{jl}^{\zeta\mu} \bigg( \hat{S}\Big( \big(K_\zeta l^- (u^\mu_{lk}) \hat{S}l^+ (u^\mu_{ij})\big)^* \Big) \triangleright f \bigg) \\
    &= \sum_{1 \leq j,l \leq n_\mu} \omega_{jl}^{\zeta\mu} \bigg( \hat{S}\Big( \hat{S}^{-1} l^- \big((u^\mu_{ij})^*\big) l^+ \big((u^\mu_{lk})^*\big) K_{-\overline{\zeta}} \Big) \triangleright f \bigg) \\
    &= \sum_{1 \leq j,l \leq n_\mu} \omega_{jl}^{\zeta\mu} \bigg( \hat{S}\Big( \hat{S}^{-1} l^- \big( S(u^\mu_{ji}) \big) l^+ \big( S(u^\mu_{kl}) \big) K_{-\zeta} \Big) \triangleright f \bigg) \\
    &= \sum_{1 \leq j,l \leq n_\mu} \omega_{jl}^{\zeta\mu} \Big( \hat{S}\big( l^- (u^\mu_{ji}) \hat{S}l^+ (u^\mu_{kl}) K_{-\zeta} \big) \triangleright f \Big) \\
    &= \sum_{1 \leq j,l \leq n_\mu} \omega_{jl}^{\zeta\mu} \Big( \hat{S}\big( K_{-\zeta} l^- (u^\mu_{ji}) \hat{S}l^+ (u^\mu_{kl}) \big) \triangleright f \Big),
\end{align*}
by \eqref{eq:Kzeta centrality}, which, using \eqref{eq:definition of f_ij} and the nondegeneracy of \eqref{eq:skew-pairing between Uq(K) and Cinfty(Kq)}, equals $f \omega_{ik}^{\zeta\mu} = \sum_{1 \leq j,l \leq n_\mu} \omega_{jl}^{\zeta \mu} \big( \hat{S}( K_{\zeta} l^- (u^\mu_{ji}) \hat{S}l^+ (u^\mu_{kl}) ) \triangleright f \big)$ for all $f \in \Cf^\infty(K_q)$ if and only if
\[
    K_{-\zeta} l^- (u^\mu_{ji}) \hat{S} l^+ (u^\mu_{kl}) = K_\zeta l^- (u^\mu_{ji}) \hat{S} l^+ (u^\mu_{kl}), \quad 1 \leq i,j,k,l \leq n_\mu,
\]
which holds if and only if $K_\zeta = K_{-\zeta}$, i.e., $\zeta \in \frac{i}{2} \hbar^{-1} \Qbf^\lor$, since the matrices $l^\pm(u^\mu)$ are invertible.

Now assume $\zeta \in \frac{i}{2} \hbar^{-1} \Qbf^\lor$. We check that $\Omega_{\zeta\mu}$ becomes a bicovariant $*$-bimodule with the involution \eqref{eq:involution of Gamma^mu}. By definition,
\[
    (f \omega)^* = \omega^* f^*
\]
for all $f \in \Cf^\infty(K_q)$ and $\omega \in \Omega_\mu$, and hence
\[
    (\omega f)^* = \big( (f^* \omega^*)^{*} \big)^* = f^* \omega^*,
\]
proving condition B4 of Definition~\ref{defn:bicovariant *-bimodule}.

To verify condition B5, choose $f \in \Cf^\infty(K_q)$ and $1 \leq i,k \leq n_\mu$. Then:
\begin{align*}
    \Phi_\Omega \big( (f \omega_{ik}^{\zeta\mu})^* \big) 
    = \Phi_\Omega(-\omega_{ki}^{\zeta\mu} f^*) 
    = -(1 \otimes \omega_{ki}^{\zeta\mu}) \Delta(f)^* 
    = \big(\Delta(f)(1 \otimes \omega_{ik}^{\zeta\mu})\big)^* = \Phi_\Omega(f \omega_{ik}^{\zeta\mu})^*
\end{align*}
and
\begin{align*}
    {}_{\Omega} \Phi \big( (f \omega_{ik}^{\zeta\mu})^* \big) 
    = {}_{\Omega} \Phi(-\omega_{ki}^{\zeta\mu} f^*) 
    &= - \sum_{1 \leq j,l \leq n_\mu} \Big( \omega_{lj}^{\zeta\mu} \otimes \big(u^\mu_{lk} S(u^\mu_{ij}) \big) \Big) \Delta(f)^* \\
    &= \Big( \sum_{1 \leq j,l \leq n_\mu} (\omega_{jl}^{\zeta\mu})^* \otimes \big( u^\mu_{ji} S(u^\mu_{kl}) \big)^* \Big) \Delta(f)^* \\
    &= {}_\Omega \Phi \big( (\omega_{ik}^{\zeta\mu})^* \big) \Delta(f)^* 
    = {}_\Omega \Phi(f \omega_{ik}^{\zeta\mu})^*.
\end{align*}
\end{proof}

\begin{rmk}\label{rmk:index set for *-FODCs}
In general, neither $\Pbf^\lor \subseteq \frac{1}{2} \Qbf^\lor$ nor $\frac{1}{2} \Qbf^\lor \subseteq \Pbf^\lor$. 

For instance, in the case $K = E_6$, we have $\varpi_3^\lor = \cdots + \frac{10}{3} \alpha_3^\lor + \cdots$ (cf.~\cite[Appendix~C]{Knapp}), showing that $\Pbf^\lor \nsubseteq \frac{1}{2} \Qbf^\lor$. On the other hand, for $K = SU(3)$, we have $\alpha_1^\lor = 2\varpi_1^\lor - \varpi_2^\lor$, so $\frac{1}{2} \Qbf^\lor \nsubseteq \Pbf^\lor$.

Thus, not every bicovariant bimodule over $\Cf^\infty(K_q)$ admits the structure of a $*$-bimodule, nor can the set $\big( \frac{i}{2} \hbar^{-1} \Qbf^\lor / i \hbar^{-1} \Qbf^\lor \big) \times \Pbf^+$ be regarded as a natural index set for the finite-dimensional bicovariant $*$-bimodules over $\Cf^\infty(K_q)$.
\end{rmk}

\begin{prop}\label{prop:Gamma^mu is a *-FODC}
    Let $(0,0)\neq (\zeta, \mu) \in \Zcal \times \weights^+$. Define $\omega^{\zeta\mu}= \sum_{1 \leq i \leq n_\mu} \omega_{ii} ^{\zeta\mu} \in \inv \Omega_{\zeta\mu}$. Then, the linear map $ d_{\zeta\mu} : \Cf^\infty(K_q) \rightarrow \Omega_{\zeta\mu}$ defined by
    \begin{equation}\label{eq:differential of Gamma^mu}
        d_{\zeta\mu} f = \omega ^{\zeta\mu} f - f\omega^{\zeta\mu} = \sum_{1 \leq i,k \leq n_\mu} \Big( \big(K_\zeta I (u^\mu _{ik}) - \epsilon (u^\mu _{ik}) \big) \triangleright f \Big) \omega_{ik} ^{\zeta\mu}
    \end{equation}
    is a differential that makes $(\Omega_{\zeta\mu} , d_{\zeta\mu})$ a bicovariant FODC on $K_q$. Its Quantum germs map is given by
    \begin{equation}\label{eq:quantum germs map of Gamma^mu}
        Q_{\zeta \mu}(f) = \sum_{1 \leq i,k \leq n_\mu} \big(K_\zeta I(u^\mu _{ik}) - \epsilon(u^\mu _{ik}) , f \big) \omega_{ik} ^{\zeta \mu}.
    \end{equation}
    Thus, the right ideal corresponding to $(\Omega_{\zeta \mu}, d_{\zeta \mu} )$ is
    \begin{equation}\label{eq:the right ideal for Gamma^mu}
        R_{\zeta\mu} = \big\{ f \in \Ker \epsilon \;\big|\; 1 \leq {}^\forall i,k \leq n_\mu , \, \big( K_\zeta I(u^\mu _{ik} ) , f  \big) = 0 \big\}
    \end{equation}
    and the space of left-invariant vector fields for this FODC is
    \begin{equation}\label{eq:LIVF for Gamma^mu}
        \Xcal_{\zeta \mu} = \Span_\Cbb \big\{ K_{\zeta} I(u^{\mu} _{ik}) - \epsilon (u^{\mu} _{ik} ) \;\big|\; 1 \leq i,k \leq n_{\mu} \big\},
    \end{equation}
    for which the set inside the span sign is a linear basis.
    
    The FODC $(\Omega_{\zeta \mu} , d_{\zeta \mu})$ can be made a bicovariant $*$-FODC if and only if $\zeta \in \frac{i}{2} \hbar^{-1} \Qbf^\lor$, in which case the involution is given by \eqref{eq:involution of Gamma^mu}.
\end{prop}
\begin{proof}
    Throughout the proof, we will suppress all the sub/super-scripts ``$\zeta\mu$" for simplicity.
    Leibniz's rule holds since, for $f, g \in \Cf^\infty(K_q)$,
    \begin{align*}
        d (fg) = \omega (fg) - (fg) \omega = (\omega f - f \omega )g + f(\omega g - g \omega ) = (d f)g + f d g.
    \end{align*}
    
    We need to prove the second equality of \eqref{eq:differential of Gamma^mu} and that every element of $\Omega$ can be written in a standard form. Observe, for $f \in \Cf^\infty (K_q)$,
    \begin{align*}
        \omega f = \sum_{1 \leq i \leq n_\mu} \omega_{ii}  f = \sum_{1 \leq i,j,l \leq n_\mu} \Big( \big( K_\zeta l^- (u^\mu _{ji}) \hat{S} l^+ (u^\mu _{il}) \big) \triangleright f \Big) \omega_{jl} \\
        = \sum_{1 \leq j,l \leq n_\mu} \big( K_\zeta I(u^\mu _{jl}) \triangleright f \big) \omega_{jl}
    \end{align*}
    and hence
        \[
        d f = \omega f - f\omega = \sum_{1 \leq i,k \leq n_\mu} \Big( \big( K_\zeta I(u^\mu _{ik}) - \epsilon (u^\mu _{ik}) \big) \triangleright f \Big) \omega_{ik} ,
        \]
    proving \eqref{eq:differential of Gamma^mu}.
    
    Define $Q : \Cf^\infty(K_q) \rightarrow \Omega$ by $Q(f) = S(f_{(1)}) df_{(2)}$ and let $R = \Ker \epsilon \cap \Ker Q$. Then, \eqref{eq:differential of Gamma^mu} shows that
        \begin{equation*}
        Q(f) = \sum_{1 \leq i,k \leq n_\mu} \big(K_\zeta I(u^\mu _{ik}) - \epsilon(u^\mu _{ik}) , f \big) \omega_{ik}
        \end{equation*}
    for all $f \in \Cf^\infty(K_q)$. Hence, $Q(\Cf^\infty(K_q)) \subseteq \inv \Omega$ and
        \[
        R = \big\{ f \in \Ker \epsilon \;\big|\; 1 \leq {}^\forall i,k \leq n_\mu, \, \big(K_\zeta I(u^\mu _{ik} ) -\epsilon(u^\mu _{ik}) , f  \big) = 0\big\}.
        \]
    Since $I$ is injective, $K_\zeta$ is invertible, and $(\zeta,\mu) \neq (0,0)$, the family $\{ \epsilon, \, K_\zeta I(u_{ik} ^\mu) - \epsilon(u^\mu _{ik}) \mid 1 \leq i,k \leq n_\mu \} \subseteq \Cf^\infty (K_q)^\circ$ is linearly independent, which implies $\dim \Ker \epsilon / R = n_\mu ^2 = \dim \inv \Omega$. Therefore, the injective linear map
        \[
        \Ker \epsilon / R \ni \pi_{R} (f) \longmapsto Q(f) = S(f_{(1)}) df_{(2)} \in \inv \Omega
        \]
    is an isomorphism, showing that every element of $\Omega$ can be written in a standard form. Thus, $(\Omega, d)$ is an FODC on $K_q$.
    
    By the left-invariance of $\omega$ and Lemma~\ref{lem:right invariance of omega} proved below, we have
    \begin{align*}
        &\Phi_\Omega d f = \Phi_\Omega (\omega f - f \omega) = (1 \otimes \omega ) \Delta(f) - \Delta(f) (1 \otimes \omega)
        = (\id \otimes d ) \Delta(f) \\
        &{}_\Omega \Phi d f = {}_\Omega \Phi ( \omega f - f \omega ) = (\omega \otimes 1 ) \Delta(f) - \Delta(f) (\omega \otimes 1)
        = (d \otimes \id) \Delta(f)
    \end{align*}
    for all $f \in \Cf^\infty(K_q)$, proving that $(\Omega, d)$ is a bicovariant FODC.
    
    Note that \eqref{eq:quantum germs map of Gamma^mu} is the quantum germs map for this FODC by \eqref{eq:differential of Gamma^mu}. Thus, by \eqref{eq:right ideal corresponding to BCFODC} and Proposition~\ref{prop:LIVF basis}, we see \eqref{eq:the right ideal for Gamma^mu} is the $\ad$-invariant right ideal corresponding to $(\Omega, d)$ and that \eqref{eq:LIVF for Gamma^mu} is the set of left invariant vector field for $(\Omega, d)$, for which the subset inside the span sign is a linear basis.
    
Now, assume $\zeta \in \frac{i}{2} \hbar^{-1} \Qbf^\lor$ so that $K_\zeta^* = K_\zeta$, and that $\Omega$ is a bicovariant $*$-bimodule with involution given by \eqref{eq:involution of Gamma^mu}. Using Eqs.~\eqref{eq:definition of f_ij}, \eqref{eq:I and involution}, and \eqref{eq:left module action and involution}, one calculates:
\begin{align*}
    \big(d (f^*) \big)^* &= - \sum_{1 \leq i, k \leq n_\mu} \omega_{ki} \Big( \big( K_\zeta I(u^\mu_{ik}) - \epsilon(u^\mu_{ik}) \big) \triangleright f^* \Big)^* \\
    &= - \sum_{1 \leq i, k \leq n_\mu} \omega_{ki} \bigg( \hat{S} \Big( \big( K_\zeta I(u^\mu_{ik}) - \epsilon(u^\mu_{ik}) \big)^* \Big) \triangleright f \bigg) \\
    &= - \sum_{1 \leq i, k \leq n_\mu} \omega_{ki} \Big( \hat{S} \big( K_\zeta I(u^\mu_{ki}) - \epsilon(u^\mu_{ki}) \big) \triangleright f \Big) \\
    &= - \sum_{1 \leq i \leq n_\mu} (f \omega_{ii} - \omega_{ii} f) = df,
\end{align*}
where in the second to the last equality we used the identity
\begin{align*}
    f \omega = \sum_{1 \leq i \leq n_\mu} f \omega_{ii} 
    &= \sum_{1 \leq i, j, l \leq n_\mu} \omega_{jl} \Big( \hat{S} \big( K_\zeta l^-(u^\mu_{ji}) \hat{S}l^+(u^\mu_{il}) \big) \triangleright f \Big) \\
    &= \sum_{1 \leq j, l \leq n_\mu} \omega_{jl} \Big( \hat{S} \big( K_\zeta I(u^\mu_{jl}) \big) \triangleright f \Big).
\end{align*}
Thus, we see that $(\Omega, d)$ defines a bicovariant $*$-FODC on $K_q$.

Conversely, assume that $(\Omega, d)$ is a bicovariant $*$-FODC with respect to some involution $*$. We must show that $\zeta \in \frac{i}{2} \hbar^{-1} \Qbf^\lor$ and that the involution is given by \eqref{eq:involution of Gamma^mu}. By Proposition~\ref{prop:*-structure encoded in the dual space} and Eq.~\eqref{eq:LIVF for Gamma^mu}, the existence of such a $*$-structure implies that
\[
K_{-\zeta} I(u^\mu_{ki}) - \delta_{ki} = \big( K_\zeta I(u^\mu_{ik}) - \epsilon(u^\mu_{ik}) \big)^* \in \Xcal
\]
for all $1 \leq i, k \leq n_\mu$. However, by Corollary~\ref{cor:Kzeta independence} and \eqref{eq:LIVF for Gamma^mu}, this holds if and only if $K_{-\zeta} = K_\zeta$, that is, $\zeta \in \frac{i}{2} \hbar^{-1} \Qbf^\lor$.

Now, since \eqref{eq:quantum germs map of Gamma^mu} defines the quantum germs map of $(\Omega, d)$, it follows from \eqref{eq:involution and quantum germs map} that for all $f \in \Cf^\infty(K_q)$,
\begin{align*}
    \sum_{1 \leq i, k \leq n_\mu} \overline{\big( K_\zeta I(u^\mu_{ik}) - \epsilon(u^\mu_{ik}), f \big)} \omega_{ik}^* 
    &= Q(f)^* = -Q(S(f)^*) \\
    &= - \sum_{1 \leq i, k \leq n_\mu} \big( K_\zeta I(u^\mu_{ik}) - \epsilon(u^\mu_{ik}), S(f)^* \big) \omega_{ik} \\
    &= - \sum_{1 \leq i, k \leq n_\mu} \overline{\big( K_\zeta^* I(u^\mu_{ki}) - \epsilon(u^\mu_{ki}), f \big)} \omega_{ik} \\
    &= - \sum_{1 \leq i, k \leq n_\mu} \overline{\big( K_\zeta I(u^\mu_{ik}) - \epsilon(u^\mu_{ik}), f \big)} \omega_{ki},
\end{align*}
which, by the linear independence of the functionals $\{K_\zeta I(u^\mu_{ik}) - \epsilon(u^\mu_{ik}) \mid 1 \leq i,k \leq n_\mu\}$ (since $(\zeta, \mu) \neq (0,0)$), implies that the involution must be given by \eqref{eq:involution of Gamma^mu}.
\end{proof}

\begin{lem}\label{lem:right invariance of omega}
    For any $0 \neq (\zeta, \mu) \in \Zcal \times \weights^+$, we have
    \[
    {}_\Omega \Phi (\omega^{\zeta\mu}) = \omega^{\zeta\mu} \otimes 1.
    \]
\end{lem}
\begin{proof}
    \[
    \sum_{1 \leq i \leq n_\mu} {}_\Omega \Phi (\omega_{ii} ^{\zeta\mu}) 
    = \sum_{1 \leq i , j,l \leq n_\mu} \omega_{jl} ^{\zeta\mu} \otimes u^\mu _{ji} S (u^\mu _{il}) 
    = \sum_{1 \leq j,l \leq n_\mu} \omega_{jl} ^{\zeta\mu} \otimes \delta_{jl} 1 
    = \omega^{\zeta\mu} \otimes 1.
    \]
\end{proof}

For $(\zeta,\mu) = (0,0)$, we let $(\Omega_{00}, d_{00}) = 0$, the zero bicovariant $*$-FODC. The quantum germs map, the right ideal, and the space of left-invariant vector fields for this FODC are given by $Q_{00} = 0$, $R_{00} = \Ker\epsilon$, and $\Xcal_{00} = 0$, respectively. Hence, \eqref{eq:quantum germs map of Gamma^mu}--\eqref{eq:LIVF for Gamma^mu} still hold with $(\zeta,\mu) = (0,0)$.

\begin{prop}\label{prop:direct sum of FODCs-Kq}
    Let $(\zeta_1 , \mu_1) , \cdots , (\zeta_m, \mu_m) \in \Zcal \times \weights^+$ be mutually distinct pairs. Then, the map
    \[
    Q_{\boldsymbol{\zeta\mu}}: \Cf^\infty(K_q) \ni f \longmapsto \big( Q_{\zeta_1 \mu_1} (f), \cdots , Q_{\zeta_m \mu_m} (f) \big) \in \inv \Omega_{\zeta_1 \mu_1} \oplus \cdots \oplus \inv \Omega_{\zeta_m \mu_m}
    \]
    is surjective. Here, $\boldsymbol{\zeta} = (\zeta_1 , \cdots , \zeta_m)$ and $\boldsymbol{\mu} = (\mu_1 , \cdots , \mu_m)$. Hence, the direct sum
    \begin{equation*}
        (\Omega_{\boldsymbol{\zeta}\boldsymbol{\mu}} , d_{\boldsymbol{\zeta}\boldsymbol{\mu}}) = (\Omega_{\zeta_1 \mu_1} , d_{\zeta_1 \mu_1}) \oplus \cdots \oplus (\Omega_{\zeta_m \mu_m} , d_{\zeta_m \mu_m})
    \end{equation*}
    is a bicovariant FODC on $K_q$ whose quantum germs map is $Q_{\boldsymbol{\zeta\mu}}$. The right ideal corresponding to this FODC is
    \begin{equation}\label{eq:direct sum right ideal-Kq}
        R_{\boldsymbol{\zeta \mu}} = \bigcap_{1 \leq l \leq m} R_{\zeta_l \mu_l}
    \end{equation}
    and its space of left-invariant vector fields is
    \begin{equation}\label{eq:direct sum LIVF-Kq}
        \Xcal_{\boldsymbol{\zeta \mu}} = \bigoplus_{1 \leq l \leq m} \Xcal_{\zeta_l \mu_l}.
    \end{equation}
\end{prop}
\begin{proof}
    Throughout the proof, we abbreviate $(\Omega_{\zeta_l \mu_l} , d_{\zeta_l \mu_l})$ by $(\Omega_l , d_l)$, $Q_{\zeta_l \mu_l}$ by $Q_l$ ($1 \leq l \leq m$), $(\Omega_{\boldsymbol{\zeta\mu}} , d_{\boldsymbol{\zeta\mu}})$ by $(\Omega,d)$, and $Q_{\boldsymbol{\zeta\mu}}$ by $Q$, respectively. By Proposition~\ref{prop:direct sum of FODCs}, we only need to show that the map
    \[
    Q: \Cf^\infty(K_q) \ni f \longmapsto \big( Q_1 (f), \cdots , Q_m (f) \big) \in \inv \Omega_1 \oplus \cdots \oplus \inv \Omega_m
    \]
    is surjective. Without loss of generality, we may assume that $(\zeta_l, \mu_l) \neq (0, 0)$ for all $1 \leq l \leq m$ since $\Omega_{00} = 0$ contributes nothing.

    Fix $1 \leq l \leq m$ and let $f \in \Cf^\infty(K_q)$. By \eqref{eq:quantum germs map of Gamma^mu},
    \[
    Q_l (f) = \sum_{1 \leq i,k \leq n_{\mu_l}} \big( K_{\zeta_l} I(u^{\mu_l} _{ik}) - \delta_{ik} 1 , f \big) \omega^{\zeta_l \mu_l} _{ik}.
    \]
    Hence, the map $Q$ is equivalent to the linear map
    \begin{equation*}
    \Cf^\infty(K_q) \ni f \longmapsto \Big( \big( K_{\zeta_l} I(u^{\mu_l} _{ik}) - \delta_{ik} 1 , f \big) \Big)_{1 \leq l \leq m, \, 1 \leq i,k \leq n_{\mu_l}} \in \Cbb^{\sum_{1 \leq l \leq m} n_{\mu_l} ^2}.
    \end{equation*}
    Therefore, if we can prove that the linear functionals
    \begin{equation}\label{eq:components of Q}
    \big\{ K_{\zeta_l} I(u^{\mu_l} _{ik}) - \delta_{ik}1 \;\big|\; 1 \leq l \leq m, \, 1 \leq i,k \leq n_{\mu_l} \big\} \subseteq \Cf^\infty(K_q)^\circ
    \end{equation}
    are linearly independent, then the surjectivity of $Q$ will follow.

    By the assumption that $(\zeta_l, \mu_l) \neq (0,0)$ for all $1\leq l \leq m$, the injectivity of the linear map $I$, and Corollary~\ref{cor:Kzeta independence}, we have
    \[
    \Span_{\Cbb}\big\{ K_{\zeta_l} I(u^{\mu_l} _{ik}) \;\big|\; 1 \leq l \leq m, \, 1 \leq i,k \leq n_{\mu_l} \big\} \cap \Cbb 1 = 0.
    \]
    Hence, the linear independence of \eqref{eq:components of Q} is equivalent to the linear independence of
    \[
    \big\{ K_{\zeta_l} I(u^{\mu_l} _{ik}) \;\big|\; 1 \leq l \leq m, \, 1 \leq i,k \leq n_{\mu_l} \big\},
    \]
    which follows from Corollary~\ref{cor:Kzeta independence} and the injectivity of $I$.
\end{proof}

As shown in the next subsection, the bicovariant FODCs found in the preceding proposition exhaust all finite-dimensional bicovariant FODCs on \( K_q \) up to isomorphism. However, to obtain a similar classification for \( * \)-FODCs, we still need the following two propositions.

\begin{prop}\label{prop:bicovariant *-FODC given by a pair}
    Let \( \zeta \in \Zcal \) be such that \( \zeta \notin \frac{i}{2} \hbar^{-1} \Qbf^\lor \). Then, for \( \mu \in \weights^+ \), the direct sum of FODCs
    \begin{equation}\label{eq:bicovariant *-FODC given by a pair}
    ( \boldsymbol{\Omega}_{\zeta \mu} , \boldsymbol{d}_{\zeta \mu} ) = (\Omega_{\zeta \mu} , d_{\zeta \mu}) \oplus (\Omega_{-\zeta \mu} , d_{-\zeta \mu} )
    \end{equation}
    is a bicovariant \( * \)-FODC on \( K_q \) whose involution is given by
    \begin{equation}\label{eq:involution given by a pair}
        \Big( \sum_{1 \leq i , k \leq n_\mu} f_{ik} \omega_{ik} ^{ \pm \zeta \mu}\Big)^* = - \sum_{1 \leq i, k \leq n_\mu} \omega_{ki} ^{\mp \zeta \mu} f_{ik} ^* , \quad f_{ik} \in \Cf^\infty(K_q).
    \end{equation}
    Its quantum germs map will be denoted by \( \boldsymbol{Q}_{\zeta\mu} \).
\end{prop}

\begin{proof}
    Since \( \zeta \notin \frac{i}{2} \hbar^{-1} \Qbf^\lor \), we have \( \zeta \neq - \zeta \) in \( \Zcal \). So, by Proposition~\ref{prop:direct sum of FODCs-Kq}, \eqref{eq:bicovariant *-FODC given by a pair} is a bicovariant FODC.

    Since \( \{ \omega_{ik} ^{\pm \zeta \mu} \mid 1\leq i, k \leq n_\mu \} \) is a left \( \Cf^\infty(K_q) \)-basis of \( \Omega_{\pm \zeta \mu} \), respectively, we see \eqref{eq:involution given by a pair} is a well-defined conjugate linear map.

    Observe that by \eqref{eq:left module action and involution}, we have, for \( f \in \Cf^\infty(K_q) \) and \( 1 \leq i,k \leq n_\mu \),
    \begin{align*}
        \big((f \omega_{ik} ^{\pm\zeta\mu} )^*\big)^* &= \big(-\omega_{ki} ^{\mp\zeta\mu} f^* \big)^* \\
 &= \bigg(- \sum_{1 \leq j,l \leq n_\mu} \Big( \big(K_{\mp \zeta} l^- (u^\mu _{lk}) \hat{S}l^+ (u^\mu _{ij} )\big) \triangleright f^* \Big) \omega_{lj} ^{\mp\zeta\mu} \bigg) ^* \\
        &= \sum_{1 \leq j, l \leq n_\mu} \omega_{jl} ^{\pm\zeta\mu} \bigg(  \hat{S}\Big( \big( K_{\mp\zeta} l^- (u^\mu _{lk}) \hat{S}l^+ (u^\mu _{ij} )\big)^*\Big) \triangleright f \bigg) \\
        &= \sum_{1 \leq j, l \leq n_\mu} \omega_{jl} ^{\pm\zeta\mu} \bigg( \hat{S}\Big( \hat{S}^{-1} l^- \big((u^\mu _{ij}) ^* \big) l^+ \big ((u^\mu _{lk}) ^* \big) K_{\pm\overline{\zeta}}  \Big) \triangleright f \bigg) \\
        &= \sum_{1 \leq j, l \leq n_\mu} \omega_{jl} ^{\pm\zeta\mu} \bigg( \hat{S}\Big( \hat{S}^{-1} l^- \big( S (u^\mu _{ji}) \big) l^+ \big( S (u^\mu _{kl}) \big) K_{\pm\zeta} \Big) \triangleright f \bigg) \\
        &= \sum_{1 \leq j, l \leq n_\mu} \omega_{jl} ^{\pm\zeta\mu} \hat{S}\big( l^- ( u^\mu _{ji} ) \hat{S}l^+ ( u^\mu _{kl}) K_{\pm \zeta}  \big) \triangleright f \\
        &= \sum_{1 \leq j, l \leq n_\mu} \omega_{jl} ^{\pm\zeta\mu} \Big( \hat{S}\big( K_{\pm \zeta}  l^- ( u^\mu _{ji} ) \hat{S}l^+ ( u^\mu _{kl}) \big) \triangleright f \Big)
    \end{align*}
    by \eqref{eq:Kzeta centrality}, which is equal to \( f \omega_{ik} ^{\pm\zeta\mu} \) by \eqref{eq:definition of f_ij}. Thus, we see \eqref{eq:involution given by a pair} defines an involution on \( \boldsymbol{\Omega}_{\zeta \mu} \).

    Now, choose \( f \in \Cf^\infty (K_q) \) and \( 1 \leq i, k \leq n_\mu \). Observe that
        \begin{align*}
            \Phi_{\boldsymbol{\Omega}_{\zeta\mu}} \big( ( f \omega_{ik} ^{\pm\zeta\mu} )^* \big) = \Phi_{\boldsymbol{\Omega}_{\zeta\mu}} (- \omega_{ki} ^{\mp\zeta\mu} f^* ) = -(1 \otimes \omega_{ki} ^{\mp\zeta\mu} ) \Delta(f)^*\\
            = \big(\Delta(f) (1 \otimes \omega_{ik} ^{\pm\zeta\mu}) \big)^* 
            = \Phi_\Omega ( f \omega_{ik} ^{\pm\zeta\mu} ) ^*
        \end{align*}
        and
        \begin{align*}
            {}_{\boldsymbol{\Omega}_{\zeta \mu}} \Phi \big( (f \omega_{ik} ^{\pm\zeta\mu} )^* \big) = {}_{\boldsymbol{\Omega}_{\zeta \mu}} \Phi (- \omega_{ki} ^{\mp\zeta\mu} f^* ) = - \sum_{1 \leq j,l \leq n_\mu} \Big( \omega_{lj} ^{\mp\zeta\mu} \otimes \big(u^\mu _{lk} S(u^\mu _{ij}) \big) \Big) \Delta(f)^* \\
            = \Big( \sum_{1 \leq j,l \leq n_\mu} (\omega_{jl} ^{\pm\zeta\mu}) ^* \otimes \big( u^\mu _{ji} S(u^\mu _{kl} ) \big) ^* \Big) \Delta(f)^*\\
            = {}_{\boldsymbol{\Omega}_{\zeta \mu}} \Phi (\omega_{ik} ^{\pm\zeta\mu})^* \Delta(f)^* = {}_{\boldsymbol{\Omega}_{\zeta \mu}} \Phi ( f \omega_{ik} ^{\pm\zeta\mu} )^*,
        \end{align*}
        which proves that \( \boldsymbol{\Omega}_{\zeta \mu} \) becomes a bicovariant \( * \)-bimodule with the involution \eqref{eq:involution given by a pair}.

        Finally, since \( *( \Omega_{\pm\zeta \mu}) = \Omega_{\mp \zeta \mu} \) by definition, we have
        \begin{align*}
        (\boldsymbol{d}_{\zeta \mu} f)^* = \big(d_{\zeta\mu} f , d_{-\zeta \mu} f \big)^* = \Big( \omega^{\zeta \mu} f - f \omega^{\zeta \mu} , \omega^{-\zeta \mu} f - f \omega^{-\zeta \mu} \Big)^* \\
        =\Big( \big(\omega^{-\zeta \mu} f - f \omega^{-\zeta \mu} \big)^* , \big(\omega^{\zeta \mu} f - f \omega^{\zeta \mu} \big)^* \Big) \\
        =\Big( - f^* \omega^{\zeta \mu} + \omega^{\zeta \mu} f^* , - f^* \omega^{-\zeta \mu} + \omega^{-\zeta \mu} f^* \Big) \\
        =\big( d_{\zeta \mu} (f^*) , d_{-\zeta \mu} (f^*) \big) = \boldsymbol{d}_{\zeta \mu} (f^*)
        \end{align*}
        for $f \in \Cf^\infty(K_q)$, which completes the proof of the proposition.
\end{proof}

\begin{prop}\label{prop:direct sum of FODCs-Kq-involution}
A direct sum of FODCs given in Proposition~\ref{prop:direct sum of FODCs-Kq} can be made a bicovariant $*$-FODC on $K_q$ if and only if it is up to permutation given by
    \begin{align}\label{eq:direct sum of FODCs-Kq-involution}
        (\boldsymbol{\Omega_{\zeta\mu}}, \boldsymbol{d_{\zeta\mu}}) &= (\Omega_{\zeta_1 \mu_1}, d_{\zeta_1 \mu_1} ) \oplus \cdots \oplus (\Omega_{\zeta_p \mu_p}, d_{\zeta_p \mu_p}) \nonumber \\ &\oplus (\boldsymbol{\Omega}_{\zeta_{p+1} \mu_{p+1}} , \boldsymbol{d}_{\zeta_{p+1} \mu_{p+1}} ) \oplus \cdots \oplus (\boldsymbol{\Omega}_{\zeta_{m} \mu_{m}} , \boldsymbol{d}_{\zeta_{m} \mu_{m}} ),
    \end{align}
    where $(\zeta_1, \mu_1), \cdots , (\zeta_m, \mu_m) \in \Zcal \times \Pbf^+$ are distinct pairs with $\zeta_1 , \cdots , \zeta_p \in \frac{i}{2} \hbar^{-1} \Qbf^\lor$ and $\zeta_{p+1}, \cdots , \zeta_m \notin \frac{i}{2}\hbar^{-1} \Qbf^\lor$, in which case the involution is given by the product of the involutions on each direct summand in \eqref{eq:direct sum of FODCs-Kq-involution}.
\end{prop}
\begin{proof}
    Note that, by Proposition~\ref{prop:Gamma^mu is a *-FODC} and Proposition~\ref{prop:bicovariant *-FODC given by a pair}, each summand in \eqref{eq:direct sum of FODCs-Kq-involution} is a bicovariant $*$-FODC. Thus, by Proposition~\ref{prop:direct sum of FODCs}, we see an FODC of the form \eqref{eq:direct sum of FODCs-Kq-involution} becomes a bicovariant $*$-FODC when equipped with the product $*$-structure.

    Now, choose distinct pairs $(\xi_1, \nu_1), \cdots , (\xi_n , \nu_n) \in \Zcal \times \Pbf^+$ and let
    \[
    (\Omega, d) = (\Omega_{\xi_1 \nu_1}, d_{\xi_1 \nu_1}) \oplus \cdots \oplus (\Omega_{\xi_n \nu_n} , d_{\xi_n, \nu_n} ).
    \]
    Suppose that $(\Omega, d)$ is a bicovariant $*$-FODC with a certain involution $*: \Omega \rightarrow \Omega$. Note that, by \eqref{eq:LIVF for Gamma^mu} and \eqref{eq:direct sum LIVF-Kq}, the space of left-invariant vector fields for this FODC is given by
    \begin{equation*}
    \Xcal = \Span_\Cbb \big\{ K_{\xi_r} I(u^{\nu_r} _{ik}) - \epsilon (u^{\nu_r} _{ik}) \mid 1 \leq r \leq n, \,  1 \leq i, k  \leq n_{\nu_r} \big\}
    \end{equation*}
    for which the elements inside the span sign form a linear basis. Thus, by Proposition~\ref{prop:*-structure encoded in the dual space}, we must have
    \[
    K_{- \xi_r} I(u^{\nu_r} _{ik}) - \delta_{ik} = \big(K_{\xi_r} I(u^{\nu_r} _{ki}) - \epsilon(u^{\nu_r} _{ki}) \big)^* \in \Xcal
    \]
    for all $1 \leq r \leq n$ and $1 \leq i,k \leq \nu_r$.
    But, by Corollary~\ref{cor:Kzeta independence}, this holds only if for each $1 \leq r \leq n$, either $\xi_r = - \xi_r$ or there exists $1 \leq s\leq n$ such that $\xi_r \neq - \xi_r = \xi_s $ in $\Zcal$ and $\nu_r = \nu_s$. This implies that, as a bicovariant FODC, $(\Omega, d)$ must be of the form \eqref{eq:direct sum of FODCs-Kq-involution}.

    Now, it only remains to show that the involution of $(\Omega,d)$ is given by the product of the involutions on each direct summand in \eqref{eq:direct sum of FODCs-Kq-involution}. If $(\zeta_l, \mu_l) = (0,0)$ for some $1 \leq l \leq m$, then $\Omega_{\zeta_l \mu_l} = 0$, which contributes nothing to $(\Omega,d)$. So, without loss of generality, we may assume that $(\zeta_l,\mu_l) \neq 0$ for all $1 \leq l \leq m$. Let $Q$ be the quantum germs map of this FODC. Then, by \eqref{eq:quantum germs map of Gamma^mu} and \eqref{eq:involution and quantum germs map}, we have, for any $f \in \Cf^\infty(K_q)$,
    \begin{align*}
    &\sum_{1 \leq l \leq p} \sum_{1 \leq i,k \leq n_{\mu_l}} \overline{\big(K_{\zeta_l} I(u^{\mu_l} _{ik}) - \epsilon(u^{\mu_l} _{ik}) , f \big)} (\omega_{ik} ^{\zeta_l \mu_l}) ^* +
    \sum_{p+1 \leq l \leq m} \sum_{1 \leq i,k \leq n_{\mu_l}} \\
    &\quad \quad \Big(\overline{\big(K_{\zeta_l} I(u^{\mu_l} _{ik}) - \epsilon(u^{\mu_l} _{ik}) , f \big)} (\omega_{ik} ^{\zeta_l \mu_l}) ^* + \overline{\big(K_{-\zeta_l} I(u^{\mu_l} _{ik}) - \epsilon(u^{\mu_l} _{ik}) , f \big)} (\omega_{ik} ^{-\zeta_l \mu_l}) ^* \Big) \\
    &\quad = Q(f)^* = - Q(S(f)^*) \\
    &\quad = - \sum_{1 \leq l \leq p} \sum_{1 \leq i,k \leq n_{\mu_l}} \big(K_{\zeta_l} I(u^{\mu_l} _{ik}) - \epsilon(u^{\mu_l} _{ik}) , S(f)^* \big) \omega_{ik} ^{\zeta_l \mu_l} -
    \sum_{p+1 \leq l \leq m} \sum_{1 \leq i,k \leq n_{\mu_l}} \\
    &\quad \quad \quad \Big(\big(K_{\zeta_l} I(u^{\mu_l} _{ik}) - \epsilon(u^{\mu_l} _{ik}) , S(f)^* \big) \omega_{ik} ^{\zeta_l \mu_l} + \big(K_{-\zeta_l} I(u^{\mu_l} _{ik}) - \epsilon(u^{\mu_l} _{ik}) , S(f)^* \big)\omega_{ik} ^{-\zeta_l \mu_l} \Big) \\
    &\quad = - \sum_{1 \leq l \leq p} \sum_{1 \leq i,k \leq n_{\mu_l}} \overline{\big(K_{-\overline{\zeta_l}} I(u^{\mu_l} _{ki}) - \epsilon(u^{\mu_l} _{ki}) , f \big)} \omega_{ik} ^{\zeta_l \mu_l} -
    \sum_{p+1 \leq l \leq m} \sum_{1 \leq i,k \leq n_{\mu_l}} \\
    &\quad \quad \quad \Big(\overline{\big(K_{-\overline{\zeta_l}} I(u^{\mu_l} _{ki}) - \epsilon(u^{\mu_l} _{ki}) , f \big)} \omega_{ik} ^{\zeta_l \mu_l} + \overline{\big(K_{\overline{\zeta_l}} I(u^{\mu_l} _{ki}) - \epsilon(u^{\mu_l} _{ki}) , f \big)} \omega_{ik} ^{-\zeta_l \mu_l} \Big) \\
    &\quad = - \sum_{1 \leq l \leq p} \sum_{1 \leq i,k \leq n_{\mu_l}} \overline{\big(K_{\zeta_l} I(u^{\mu_l} _{ik}) - \epsilon(u^{\mu_l} _{ik}) , f \big)} \omega_{ki} ^{\zeta_l \mu_l} -
    \sum_{p+1 \leq l \leq m} \sum_{1 \leq i,k \leq n_{\mu_l}} \\
    &\quad \quad \quad \Big(\overline{\big(K_{-\zeta_l} I(u^{\mu_l} _{ik}) - \epsilon(u^{\mu_l} _{ik}) , f \big)} \omega_{ki} ^{\zeta_l \mu_l} + \overline{\big(K_{\zeta_l} I(u^{\mu_l} _{ik}) - \epsilon(u^{\mu_l} _{ik}) , f \big)} \omega_{ki} ^{-\zeta_l \mu_l} \Big).
    \end{align*}
    But, Corollary~\ref{cor:Kzeta independence} implies the linear independence of the functionals $\{K_{\pm\zeta_l} I(u^{\mu_l} _{ik}) - \epsilon (u^{\mu_l} _{ik}) \mid 1 \leq l \leq m,\, 1 \leq i,k \leq n_{\mu_l} \}$ (recall $(\zeta_l, \mu_l) \neq (0,0)$ for all $1\leq l\leq m$). Thus, the preceding equality forces
    \[
    (\omega_{ik} ^{\zeta_l \mu_l})^* = - \omega_{ki} ^{\zeta_l \mu_l} \text{ for $ 1 \leq l \leq p$} \;\; \text{and}\;\; (\omega_{ik} ^{\pm \zeta_l \mu_l})^* = - \omega_{ki} ^{\mp \zeta_l \mu_l} \text{ for $p+1 \leq l \leq m$},
    \]
    which shows that the involution $*:\Omega \rightarrow \Omega$ is given by the product of the involutions on each direct summand in \eqref{eq:direct sum of FODCs-Kq-involution}.
\end{proof}

\subsection{Classification of finite-dimensional bicovariant (\texorpdfstring{$*$}{TEXT}-)FODCs on \texorpdfstring{$K_q$}{TEXT}}\label{subsec:Classification of bicovariant FODCs on Kq}

This subsection is devoted to the proof of the following theorem.

\begin{thm}\label{thm:The classification of bicovariant FODCs over Cinfty(Kq)}
    Every finite-dimensional bicovariant FODC on $K_q$ is isomorphic to $(\Omega_{\boldsymbol{\zeta \mu}}, d_{\boldsymbol{\zeta \mu}})$ for some $\boldsymbol{\zeta} = (\zeta_1 , \cdots, \zeta_m)$ and $\boldsymbol{\mu} = (\mu_1, \cdots , \mu_m)$ ($m \in \Nbb$), where $(0,0) \neq (\zeta_1, \mu_1), \cdots , (\zeta_m, \mu_m) \in \Zcal \times \Pbf^+$ are mutually distinct.

    If $\boldsymbol{\xi} = (\xi_1, \cdots , \xi_n)$ and $\boldsymbol{\nu} = (\nu_1 , \cdots , \nu_n)$ are another such pair of indices and
    \[
    (\Omega_{\boldsymbol{\zeta \mu}} , d_{\boldsymbol{\zeta \mu}}) \cong (\Omega_{\boldsymbol{\xi \nu}} , d_{\boldsymbol{\xi \nu}}),
    \]
    then the pair $(\boldsymbol{\zeta}, \boldsymbol{\mu})$ is equal to $(\boldsymbol{\xi}, \boldsymbol{\nu})$ up to a permutation of indices.

    The same conclusions hold for bicovariant $*$-FODCs if we replace $(\Omega_{\boldsymbol{\zeta\mu}} , d_{\boldsymbol{\zeta\mu}})$ with $(\boldsymbol{\Omega_{\zeta\mu}}, \boldsymbol{d_{\zeta\mu}})$ in the statements.
\end{thm}

The following is \cite[Proposition~4]{Baumann1998} phrased in our convention.

\begin{lem}\label{lem:J-ideal correspondence}
    Let $R \subseteq \Cf^\infty(K_q)$ be a linear subspace. Then, $R$ is an $\ad$-invariant right ideal if and only if $J(R)$ is a two-sided ideal of $F_r U_q(\gf)$ that is $U_q(\gf)$-invariant with respect to the right adjoint action. The correspondence $R \mapsto J(R)$ preserves codimension.
\end{lem}

\begin{proof}
    Note that, for $f \in \Cf^\infty(K_q)$ and $X \in U_q(\gf)$,
    \begin{align*}
        f_{(2)}\big(\hat{S}(X), S(f_{(1)}) f_{(3)}\big) 
        &= f_{(2)}(X_{(1)}, f_{(1)}) \big(\hat{S}(X_{(2)}), f_{(3)}\big) 
        = f \leftarrow X.
    \end{align*}
    Hence,
    \begin{align*}
        \text{$R$ is $\ad$-invariant} 
        \Leftrightarrow (\id \otimes \hat{S}(X)) \ad(R) \subseteq R \text{ for all } X \in U_q(\gf) \\
        \Leftrightarrow R \leftarrow U_q(\gf) \subseteq R.
    \end{align*}
    Thus, by \eqref{eq:J preserves adjoint actions}, the $\ad$-invariance of $R$ is equivalent to the $U_q(\gf)$-invariance of $J(R)$ with respect to the right adjoint action on $U_q(\gf)$.
    
    For the remainder of the proof, assume that $R \subseteq \Cf^\infty(K_q)$ is an $\ad$-invariant subspace. Let $f, g \in \Cf^\infty(K_q)$. Then,
    \begin{equation*}
        J(fg) = \hat{S}l^+(f_{(1)} g_{(1)}) l^-(f_{(2)} g_{(2)}) 
        = \hat{S}l^+(g_{(1)}) J(f) l^-(g_{(2)}).
    \end{equation*}
    Therefore, we have
    \begin{align}\label{eq:right linearity of J-2}
        J(f g_{(1)}) \leftarrow \hat{S}l^-(g_{(2)}) 
        &= \hat{S}^2 l^-(g_{(4)}) \hat{S}l^+(g_{(1)}) J(f) l^-(g_{(2)}) \hat{S}l^-(g_{(3)}) \nonumber \\
        &= \hat{S} \big(l^+(g_{(1)}) \hat{S}l^-(g_{(2)})\big) J(f) 
        = \hat{S} I'(g) J(f)
    \end{align}
    and
    \begin{align}\label{eq:right linearity of J-3}
        J(f g_{(2)}) \leftarrow \hat{S}^{-1} l^+(g_{(1)}) 
        &= l^+(g_{(2)}) \hat{S}l^+(g_{(3)}) J(f) l^-(g_{(4)}) \hat{S}^{-1} l^+(g_{(1)}) \nonumber \\
        &= J(f) \hat{S}^{-1} \big(l^+(g_{(1)}) \hat{S}l^-(g_{(2)})\big) 
        = J(f) \hat{S}^{-1} I'(g).
    \end{align}
    Thus, when $R$ is a right ideal, putting $f \in R$ and $g \in \Cf^\infty(K_q)$ into \eqref{eq:right linearity of J-2}--\eqref{eq:right linearity of J-3} shows that $J(R)$ is a two-sided ideal of $F_r U_q(\gf)$ (note that $\hat{S}^{\pm 1} I'(\Cf^\infty(K_q)) = \hat{S}^{\pm 1}(F_l U_q(\gf)) = F_r U_q(\gf)$).

    Conversely, assume that $J(R)$ is a two-sided ideal of $F_r U_q(\gf)$. Let $f \in R$ and $g \in \Cf^\infty(K_q)$. Then, the right-hand side of \eqref{eq:right linearity of J-2} is contained in $J(R)$, which is invariant under the right adjoint action on $U_q(\gf)$. Acting with $\leftarrow l^-(g_{(3)})$ on the left-hand side of \eqref{eq:right linearity of J-2}, we obtain
    \[
        J(fg) \in J(R),
    \]
    which, by the injectivity of $J$, implies that $R$ is a right ideal in $\Cf^\infty(K_q)$.

    Finally, since $J\colon \Cf^\infty(K_q) \to F_r U_q(\gf)$ is a linear isomorphism, the correspondence $R \mapsto J(R)$ preserves codimension.
\end{proof}

\begin{lem}\label{lem:the subalgebra V}
    Let $V$ be the subalgebra of $U_q(\gf)$ generated by
    \[
    F_r U_q(\gf) \cup \{K_{2 \lambda} \mid \lambda \in \weights^+\}.
    \]
    Also, let $V'$ be the subalgebra of $U_q(\gf)$ generated by
    \[
    \{E_j, \, K_j F_j, \, K_{2 \lambda} \mid 1 \leq j \leq N, \, \lambda \in \weights\}.
    \]
    Then, $V = V'$.
\end{lem}

\begin{proof}
    By \cite[Theorem~3.113]{VoigtYuncken}, \eqref{eq:antipode and locally finite parts}, and \eqref{eq:antipodes and adjoint actions}, we have
    \begin{equation}\label{eq:right locally finite part decomposed}
        F_r U_q(\gf) = \bigoplus_{\mu \in \weights^+} \big(K_{-2\mu} \leftarrow U_q(\gf)\big).
    \end{equation}
    Note that for each $1 \leq j \leq N$,
    \begin{align*}
        K_{-2 \alpha_j} \leftarrow F_j 
        = - K_j F_j K_{-2 \alpha_j} + K_j K_{-2 \alpha_j} F_j
        = K_{-2 \alpha_j} \big(1 - q^{-2(\alpha_j, \alpha_j)}\big) K_j F_j,
    \end{align*}
    and hence, by \eqref{eq:right locally finite part decomposed},
    \[
    K_j F_j \in K_{2 \alpha_j} \cdot F_r U_q(\gf) \subseteq V.
    \]
    Likewise, since
    \begin{align*}
        K_{-2 \alpha_j} \leftarrow E_j = - E_j K_j^{-1} K_{-2 \alpha_j} K_j + K_{-2\alpha_j} E_j = K_{-2 \alpha_j} \big(1 - q^{2(\alpha_j, \alpha_j)}\big) E_j,
    \end{align*}
    we conclude, again by \eqref{eq:right locally finite part decomposed},
    \[
    E_j \in K_{2 \alpha_j} \cdot F_r U_q(\gf) \subseteq V.
    \]
    Thus, $V$ contains $\{E_j, \, K_j F_j, \, K_{2 \lambda} \mid 1 \leq j \leq N, \, \lambda \in \weights\}$, i.e., $V' \subseteq V$.
    
    Conversely, note that for each $\lambda \in \weights$, $V'$ is invariant under the algebra homomorphism $\leftarrow K_\lambda$ since its generators are preserved by this map up to multiplication by a scalar. Thus, for any $X \in V'$ and $1 \leq j \leq N$,
    \[
    X \leftarrow F_j = - (K_j F_j) X + (K_j X K_j^{-1}) (K_j F_j) \in V',
    \]
    and
    \[
    X \leftarrow E_j = - E_j (K_j^{-1} X K_j) + X E_j \in V'.
    \]
    This shows that $V'$ is invariant under the right adjoint action on $U_q(\gf)$. Since $\{K_{2\mu} \mid \mu \in \weights\} \subseteq V'$, this and \eqref{eq:right locally finite part decomposed} together imply
    \[
    F_r U_q(\gf) \subseteq V',
    \]
    which shows that $V \subseteq V'$.
\end{proof}

For each $\zeta \in \frac{i}{2} \hbar^{-1} \Pbf^\lor$, define $\hat{\epsilon}_\zeta : U_q(\gf) \rightarrow \Cbb$ by
\[
\hat{\epsilon}_\zeta(K_{\lambda}) = q^{(\zeta,\lambda)}, \quad \hat{\epsilon}_\zeta(E_j) = \hat{\epsilon}_\zeta(F_j) = 0, \quad \lambda \in \weights, \, 1 \leq j \leq N.
\]
These are characters of $U_q(\gf)$, see \cite[Section~3.4]{VoigtYuncken}. Note that $\hat{\epsilon}_\zeta \hat{S} = \hat{\epsilon}_{-\zeta}$ for all $\zeta \in \frac{i}{2} \hbar^{-1} \Pbf^\lor$, and $\hat{\epsilon}_\zeta = \hat{\epsilon}$ on $V$ if and only if $\zeta \in \frac{i}{2} \hbar^{-1} \Qbf^\lor$.

\begin{lem}\label{lem:irreducible representations of V}
    The following exhaust all equivalence classes of irreducible representations of $V$ and are mutually inequivalent:
    \begin{equation}\label{eq:irreducible representations of V}
        \Big\{ (\hat{\epsilon}_\zeta \otimes \pi_\mu )|_{V} \;\Big|\; \zeta \in \frac{i}{2} \hbar^{-1} \Pbf^\lor \Big/ \frac{i}{2} \hbar^{-1} \Qbf^\lor , \, \mu \in \weights^+ \Big\}.
    \end{equation}
    Also, finite-dimensional representations of $V$ are completely reducible.

    Moreover, these representations restrict to mutually inequivalent irreducible representations of the subalgebra $F_r U_q(\gf) \subseteq V$.
\end{lem}

\begin{proof}
    Since $V$ contains $E_j$ and $K_j F_j$ ($1 \leq j \leq N$), which act as weight raising and lowering operators (cf. \eqref{eq:raising and lowering actions of root vectors}) on the representation spaces in the list \eqref{eq:irreducible representations of V}, it follows that each such representation remains irreducible when restricted to $V$.

    Conversely, let $\pi$ be any finite-dimensional representation of $V$. Fix $1 \leq j \leq N$. Then the actions of the operators
    \[
    \big\{ \pi(E_j), \, \pi(K_j F_j), \, \pi(K_{2\alpha_j}) \big\} \subseteq \End(V)
    \]
    can be analyzed similarly to the case of $U_q(\mathfrak{sl}(2))$. In particular, the analysis in \cite[Section~VI.3]{Kassel} shows that the matrix $\pi(K_{2\alpha_j})$ is diagonalizable with strictly positive eigenvalues, since $0 < q$.

    As the matrices $\{\pi(K_{2\alpha_j}) \mid 1 \leq j \leq N\}$ commute, they are simultaneously diagonalizable. Thus, for each $1 \leq j \leq N$ and $r \in \Rbb$, we can define $\pi(K_{2 \alpha_j})^r$ as the diagonal matrix whose entries are the $r$-th powers of the strictly positive eigenvalues of $\pi(K_{2 \alpha_j})$ in this basis.

    Let $\lambda \in \weights$ and $r_1, \dots, r_N \in \Rbb$ be such that $\lambda = r_1 \alpha_1 + \cdots + r_N \alpha_N$. Define
    \[
    \pi^\lambda = \pi(K_{2\alpha_1})^{\frac{r_1}{2}} \cdots \pi(K_{2\alpha_N})^{\frac{r_N}{2}},
    \]
    which is a diagonal matrix with strictly positive entries. One can then check that the matrices
    \[
    \big\{ \pi^\lambda, \, \pi(E_j), \, (\pi^{\alpha_j})^{-1} \pi(K_j F_j) \;|\; \lambda \in \weights, \, 1 \leq j \leq N \big\}
    \]
    satisfy the defining relations of $U_q(\gf)$ in Definition~\ref{defn:the quantized universal enveloping algebra}. Thus, $\pi$ extends to a representation of $U_q(\gf)$.
    
    By \cite[Corollary~3.100]{VoigtYuncken}, this implies that $\pi$ decomposes into a direct sum of representations from the list \eqref{eq:irreducible representations of V}. This simultaneously proves both the complete reducibility of $\pi$ and the fact that \eqref{eq:irreducible representations of V} exhausts all equivalence classes of irreducible representations of $V$.

    To show that these representations are mutually inequivalent, let $\zeta, \xi \in \frac{i}{2} \hbar^{-1} \Pbf^\lor \Big/ \frac{i}{2} \hbar^{-1} \Qbf^\lor$ and $\mu, \nu \in \weights^+$ such that
    \[
    (\hat{\epsilon}_\zeta \otimes \pi_\mu)|_V \cong (\hat{\epsilon}_\xi \otimes \pi_\nu)|_V.
    \]
    Since $E_j \in V$ for $1 \leq j \leq N$, the highest weight vectors in the two $U_q(\gf)$-representations must correspond under an equivalence. Because the actions of $\{K_{2\lambda} \mid \lambda \in \weights\} \subseteq V$ on these vectors must agree, it follows that
    \[
    q^{2 (\zeta , \lambda)} q^{2(\mu ,\lambda)} = q^{2 (\xi, \lambda)} q^{2 (\nu, \lambda)}, \quad \lambda \in \weights.
    \]
    Equating absolute values, and then phases, shows that this holds if and only if $\zeta - \xi \in \frac{i}{2} \hbar^{-1} \Qbf^\lor$ and $\mu = \nu$.

    These representations restrict to irreducible representations of $F_r U_q(\gf)$ because it contains the weight raising and lowering operators
    \[
    K_{-2 \alpha_j} \big( 1 - q^{-2 (\alpha_j, \alpha_j)} \big) K_j F_j, \quad K_{-2 \alpha_j} \big( 1 - q^{2(\alpha_j, \alpha_j)} \big) E_j, \quad 1 \leq j \leq N,
    \]
    as shown in the proof of Lemma~\ref{lem:the subalgebra V}. To see that the restrictions are inequivalent, note that $\{K_{-2 \lambda} \mid \lambda \in \weights^+\} \subseteq F_r U_q(\gf)$ and argue as in the last paragraph.
\end{proof}

\begin{proof}[Proof of Theorem~\ref{thm:The classification of bicovariant FODCs over Cinfty(Kq)}]
    According to Theorem~\ref{thm:BCFODC}, it suffices to classify finite-codimensional $\ad$-invariant right ideals of $\Cf^\infty(K_q)$ contained in $\Ker \epsilon$. Let $R$ be one such ideal.

    By Lemma~\ref{lem:J-ideal correspondence}, $J(R)$ is a finite-codimensional two-sided ideal of $F_r U_q (\gf)$ that is invariant under the right adjoint action and contained in $\Ker \hat{\epsilon}$, since $\hat{\epsilon}J = \epsilon$. By \cite[Proposition~9]{Baumann1998}, phrased in our convention, this latter set is given by
    \begin{equation}\label{eq:identity of J(R)}
        J(R) = \{ X \in F_r U_q (\gf) \mid \hat{\epsilon}(X) = 0, \, \pi(X) = 0 \}
    \end{equation}
    for some finite-dimensional representation $\pi$ of $V$. By Lemma~\ref{lem:irreducible representations of V}, we have
    \begin{equation*}
        \pi \cong \left( \hat{\epsilon}_{\zeta_1} \otimes \pi_{\mu_1} \right)\big|_V \oplus \cdots \oplus \left( \hat{\epsilon}_{\zeta_m} \otimes \pi_{\mu_m} \right)\big|_V
    \end{equation*}
    for some $\zeta_1 , \ldots , \zeta_m \in \frac{i}{2} \hbar^{-1} \Pbf^\lor / \frac{i}{2} \hbar^{-1} \Qbf^\lor$ and $\mu_1 , \ldots , \mu_m \in \weights^+$. Hence,
    \begin{align}\label{eq:identity of R}
        R &= \{ f \in \Cf^\infty (K_q) \mid J(f) \in J(R) \} \\
        &= \Big\{ f \in \Ker \epsilon \;\Big|\; \left( \hat{\epsilon}_{\zeta_1} \otimes \pi_{\mu_1} \right) (J(f)) = \cdots = \left( \hat{\epsilon}_{\zeta_m} \otimes \pi_{\mu_m} \right)(J(f)) = 0 \Big\} \nonumber.
    \end{align}
    Note that if $(\zeta_l , \mu_l) = (0,0)$ or $(\zeta_l , \mu_l) = (\zeta_k , \mu_k)$ for some $1 \leq l, k \leq m$, then we may remove $(\hat{\epsilon}_{\zeta_k} \otimes \pi_{\mu_k})$ from \eqref{eq:identity of R} without changing $R$. Thus, without loss of generality, we assume that $(\zeta_1, \mu_1), \ldots, (\zeta_m , \mu_m)$ are nonzero and pairwise distinct.
    
    By the explicit form of the universal $R$-matrix \eqref{eq:the universal R-matrix} and the actions of the root vectors \eqref{eq:raising and lowering actions of root vectors}, we compute, for $\zeta \in \frac{i}{2} \hbar^{-1} \Pbf^\lor$ and $\mu \in \weights^+$,
    \begin{align*}
        &\hat{\epsilon}_\zeta l^- (u^\mu _{ij}) = \big(\Rcal^{-1} , u^{\mu} _{ij} \otimes \hat{\epsilon}_\zeta \big) = \hat{\epsilon}_\zeta (K_{-\epsilon_j ^\mu}) \delta_{ij} = q^{-(\zeta, \epsilon_j ^\mu)} \delta_{ij} = (K_{-\zeta} , u^\mu _{ij}) \\
        &\hat{\epsilon}_\zeta \hat{S}l^+ (u^\mu _{ij}) = \big(\Rcal, \hat{\epsilon}_{-\zeta} \otimes u^\mu _{ij} \big) = \hat{\epsilon}_{-\zeta} (K_{\epsilon_j ^\mu}) \delta_{ij} = q^{-(\zeta , \epsilon_j ^\mu)} \delta_{ij} = (K_{-\zeta} , u^\mu _{ij}),
    \end{align*}
    and hence, for all $f \in \Cf^\infty(K_q)$,
    \begin{align*}
        (\hat{\epsilon}_{\zeta} \otimes \pi_\mu ) J(f) &= (\hat{\epsilon}_{\zeta} \otimes \pi_\mu) \big( \hat{S}l^+ (f_{(2)}) l^- (f_{(3)}) \otimes \hat{S}l^+ (f_{(1)}) l^- (f_{(4)}) \big) \\
        &= (K_{-\zeta}, f_{(2)} ) (K_{-\zeta}, f_{(3)}) \pi_\mu \big( \hat{S}l^+ (f_{(1)}) l^- (f_{(4)}) \big) \\
        &= \pi_\mu \big( \hat{S}l^+ ( K_{-2 \zeta} \triangleright f_{(1)} ) l^- (f_{(2)}) \big) \\
        &= \pi_\mu \big( \hat{S}l^+ ( f_{(1)} \triangleleft K_{-2 \zeta} ) l^- (f_{(2)}) \big) \\
        &= \pi_\mu \big( J( f \triangleleft K_{-2 \zeta} ) \big),
    \end{align*}
    where we used \eqref{eq:module action of Kzeta} in the second to the last equality. Thus, \eqref{eq:identity of R} becomes
    \begin{align*}
        R &= \Big\{ f \in \Ker \epsilon \;\Big|\; \pi_{\mu_1} \big( J( f \triangleleft K_{-2 \zeta_1} ) \big) = \cdots = \pi_{\mu_m} \big( J( f \triangleleft K_{-2 \zeta_m}) \big) = 0 \Big\} \\
        &= \Big\{ f \in \Ker \epsilon \;\Big|\; \big( J( f \triangleleft K_{-2 \zeta_l} ) , u^{\mu_l} _{ij}  \big) = 0 \;\; \text{for } 1 \leq l \leq m, \, 1\leq i,j \leq n_{\mu_{l}} \Big\} \\
        &= \Big\{ f \in \Ker \epsilon \;\Big|\; \big( I( u^{\mu_l} _{ij} ), f \triangleleft K_{-2 \zeta_l} \big) = 0 \;\; \text{for } 1 \leq l \leq m, \, 1\leq i,j \leq n_{\mu_{l}} \Big\} \\
        &= \Big\{ f \in \Ker \epsilon \;\Big|\; \big( K_{-2 \zeta_l} I( u^{\mu_l} _{ij} ), f  \big) = 0 \;\; \text{for } 1 \leq l \leq m, \, 1\leq i,j \leq n_{\mu_{l}} \Big\} \\
        &= \bigcap_{1 \leq l \leq m} R_{ (-2 \zeta_l) \mu_l}
    \end{align*}
    by \eqref{eq:an identity regarding I} and \eqref{eq:the right ideal for Gamma^mu}. The first assertion then follows from \eqref{eq:direct sum right ideal-Kq} in Proposition~\ref{prop:direct sum of FODCs-Kq}.

    For the second assertion, the uniqueness part of Theorem~\ref{thm:BCFODC} implies $R_{\boldsymbol{\zeta \mu}} = R_{\boldsymbol{\xi \nu}}$. Tracing the first part of the proof backward up to \eqref{eq:identity of J(R)}, we see that $J(R_{\boldsymbol{\zeta \mu}}) = J(R_{\boldsymbol{\xi \nu}})$ entails
    \begin{align*}
        &\Big\{ X \in F_r U_q (\gf) \;\Big|\; \hat{\epsilon}(X) = (\hat{\epsilon}_{-\frac{1}{2}\zeta_l} \otimes \pi_{\mu_l})(X) =0 , \, 1\leq {}^\forall l \leq m \Big\} \\
        &= \Big\{ X \in F_r U_q (\gf) \;\Big|\; \hat{\epsilon}(X) = (\hat{\epsilon}_{-\frac{1}{2}\xi_l} \otimes \pi_{\nu_l})(X) =0 , \, 1\leq {}^\forall l \leq n \Big\}.
    \end{align*}
    By the final assertion of Lemma~\ref{lem:irreducible representations of V}, we obtain
    \[
        (\pi_1 \oplus \cdots \oplus \pi_k)(F_r U_q (\gf)) = \End(V(\pi_1)) \oplus \cdots \oplus \End(V(\pi_k))
    \]
    for any family of pairwise non-equivalent irreducible representations $\pi_1 , \ldots , \pi_k$ of $V$. Therefore, the above identity implies
    \begin{equation*}
        (\hat{\epsilon}_{-\frac{1}{2}\zeta_1} \otimes \pi_{\mu_1}) \oplus \cdots \oplus (\hat{\epsilon}_{-\frac{1}{2}\zeta_m} \otimes \pi_{\mu_m}) 
        \cong 
        (\hat{\epsilon}_{-\frac{1}{2}\xi_1} \otimes \pi_{\nu_1}) \oplus \cdots \oplus (\hat{\epsilon}_{-\frac{1}{2}\xi_n} \otimes \pi_{\nu_n})
    \end{equation*}
    on $F_r U_q (\gf)$. The second assertion then follows from the inequivalence part of Lemma~\ref{lem:irreducible representations of V} for $F_r U_q(\gf)$.

    The corresponding statements for bicovariant $*$-FODCs follow from the above conclusions and Proposition~\ref{prop:direct sum of FODCs-Kq-involution}.
\end{proof}

By removing the condition $R \subseteq \Ker \epsilon$ from the previous proof, we obtain
\[
J(R) = \{X \in F_r U_q(\gf) \mid \pi(X) = 0 \}
\]
in place of \eqref{eq:identity of J(R)}, as stated in \cite[Proposition~9]{Baumann1998}. By following the remainder of the argument, we derive the following result, which will be used in Section~\ref{subsec:The center of Cf(Kq)circ} to compute the center of $\Cf^\infty(K_q)^\circ$.

\begin{thm}\label{thm:finite codimensional ad-invarant right ideals of Cf(Kq)}
    Let $R \subseteq \Cf^\infty(K_q)$ be a finite codimensional $\ad$-invariant right ideal of $\Cf^\infty(K_q)$. Then, there exist distinct $(\zeta_1, \mu_1), \dots, (\zeta_m, \mu_m) \in \Zcal \times \Pbf^+$ such that
    \begin{align*}
    R = \Big\{ f \in \Cf^\infty(K_q) \; \Big| \; \big( K_{\zeta_l} I(u^{\mu_l}_{ij}) , f \big) = 0 \text{ for } 1 \leq l \leq m,\; 1 \leq i,j \leq n_{\mu_l} \Big\} \\
    = \bigcap_{\substack{1 \leq l \leq m \\ 1 \leq i,j \leq n_{\mu_l}}} \Ker \big( K_{\zeta_l} I(u^{\mu_l}_{ij}) \big).
    \end{align*}
\end{thm}

\section{Laplacians on \texorpdfstring{$K_q$}{TEXT}}\label{sec:Laplacians on Kq}

In this section, we apply the construction from Section~\ref{sec:Main construction} to $K_q$ to identify all linear functionals on $K_q$ that induce the finite-dimensional bicovariant FODCs listed in Theorem~\ref{thm:The classification of bicovariant FODCs over Cinfty(Kq)}. This leads to a classification of all Laplacians on $K_q$ that are associated with finite-dimensional bicovariant FODCs and arise from Theorem~\ref{thm:linear functionals as Laplacians}. We explicitly compute their eigenvalues and explore some of their operator-theoretic properties.

\subsection{FODCs induced by quantum Casimir elements}\label{subsec:FODCs induced by quantum Casimir elements}

For each $\mu \in \Pbf^+$, define
\begin{equation*}
t_\mu = \Tr_{V(\mu)} \Big((\,\cdot\,) K_{- 2 \rho} \Big) \in \Lbb(V(\mu))^* \subseteq \Cf^\infty (K_q).
\end{equation*}
Then, the element
\begin{equation*}
z_\mu = I(t_\mu) \in U_q (\gf)
\end{equation*}
is called \textit{the quantum Casimir element of $U_q (\gf)$ corresponding to $ \mu $}, and $ (z_\mu)_{\mu \in \Pbf^+} $ forms a linear basis of $ ZU_q (\gf) $, the center of $ U_q (\gf) $, see \cite[Theorem~3.120]{VoigtYuncken}.

\begin{prop}\label{prop:self-adjointness and ad-invariance of Casimir elements}
    For each $\mu \in \weights^+$, $z_\mu$ is self-adjoint, i.e., $(z_\mu)^* = z_\mu$. Also, as a linear functional on $\Cf^\infty(K_q)$, it is $\ad$-invariant.
\end{prop}
\begin{proof}
    Using \eqref{eq:I and involution}, we compute
    \begin{align*}
        z_\mu ^* = I( t_\mu ) ^* = \sum_{1 \leq i \leq n_\mu} I( K_{-2\rho} \triangleright u^\mu _{ii} )^* = \sum_{1 \leq i \leq n_\mu} q^{-2 (\rho, \epsilon^\mu _i)} I(u^\mu _{ii})^* \\
        = \sum_{1 \leq i \leq n_\mu} q^{-2 (\rho, \epsilon^\mu _i)} I(u^\mu _{ii}) = I(t_\mu) = z_\mu.
    \end{align*}
    
    For the second assertion, note that $U_q ^\Rbb (\kf) \subseteq \Cf^\infty(K_q)^*$ separates $\Cf^\infty(K_q)$. Since $z_\mu$ is central in $U_q ^\Rbb(\kf)$, Proposition~\ref{prop:ad-invariance criterion for a linear functional} implies that $z_\mu$ defines an $\ad$-invariant linear functional on $\Cf^\infty(K_q)$.
\end{proof}

Recall that, for $\zeta \in \Zcal$, $K_\zeta$ is an $\ad$-invariant linear functional with eigenvalues $(q^{(\zeta,\mu)})_{\mu \in \weights^+}$, see Remark~\ref{rmk:center inside the maximal torus}.

\begin{cor}\label{cor:linear functional for BCFODC-Kq}
    Let $\zeta_1 , \cdots , \zeta_m \in \Zcal$ and $\mu_1, \cdots , \mu_m \in \weights^+$. Then, for any $\boldsymbol{a} = (a_1, \cdots, a_m) \in \Cbb^m$, the linear functional
    \begin{equation*}
    z_{\boldsymbol{\zeta\mu}} ^{\abf} = a_1 K_{\zeta_1} z_{\mu_1} + \cdots + a_m K_{\zeta_m} z_{\mu_m} \in \Cf^\infty(K_q)^\circ
    \end{equation*}
    is $\ad$-invariant.

    If $\zeta_1, \cdots , \zeta_p \in \frac{i}{2}\hbar^{-1} \Qbf^\lor$ and $\zeta_{p+1} , \cdots , \zeta_{m} \notin \frac{i}{2} \hbar^{-1} \Qbf^\lor$, then, for any $\abf = (a_1, \cdots, a_m) \in \Rbb^p \times \Cbb^{m - p}$, the functional
    \begin{align}\label{eq:linear functional for BCFODC-Kq-self-adjoint}
    \boldsymbol{z_{\zeta\mu}^{\abf}}  = &a_1 K_{\zeta_1} z_{\mu_1} + \cdots + a_p K_{\zeta_p} z_{\mu_p} \\
    &+ \frac{1}{2}( a_{p+1} K_{\zeta_{p+1}} + \overline{a_{p+1}} K_{-\zeta_{p+1}}) z_{\mu_{p+1}} + \cdots + \frac{1}{2} (a_m K_{\zeta_m} + \overline{a_m} K_{-\zeta_m}) z_{\mu_m} \nonumber
    \end{align}
    is self-adjoint and $\ad$-invariant.
\end{cor}
\begin{proof}
    Since all the terms appearing on the right-hand sides are central in $\Cf^\infty(K_q)^\circ$, so are $z_{\boldsymbol{\zeta\mu}} ^{\abf}$ and $\boldsymbol{z_{\zeta\mu} ^a}$.

    The self-adjointness of \eqref{eq:linear functional for BCFODC-Kq-self-adjoint} follows from the fact that $(K_{\zeta})^* = K_{-\overline{\zeta}} = K_{- \zeta}$ for all $\zeta \in i \hbar^{-1} \Pbf^\lor \subseteq \tf^*$.
\end{proof}

Thus, according to Corollary~\ref{cor:BCFODC induced by linear functionals} and Proposition~\ref{prop:finite-dimensional BCFODC induced by linear functionals}, the $\ad$-invariant linear functional $z_{\boldsymbol{\zeta \mu}} ^{\abf}$ induces a finite-dimensional bicovariant FODC on $K_q$, and the self-adjoint, $\ad$-invariant linear functional $\boldsymbol{z_{\zeta\mu}^{\abf}}$ induces a finite-dimensional bicovariant $*$-FODC on $K_q$. The following theorem identifies the corresponding FODCs.

\begin{thm}\label{thm:FODCs induced by Casimir elements}
    Let $(\zeta_1, \mu_1), \cdots, (\zeta_m, \mu_m) \in \Zcal \times \weights^+$ be mutually distinct pairs and $\abf = (a_1, \cdots , a_m) \in \Cbb^m$ be such that $a_l \neq 0$ for all $1 \leq l \leq m$. Then, the FODC $(\Omega_{\boldsymbol{\zeta\mu}}, d_{\boldsymbol{\zeta \mu}})$ is induced by the linear functional $z_{\boldsymbol{\zeta \mu}}^{\abf} \in \Cf^\infty(K_q)^\circ$. That is,
        \begin{equation}\label{eq:right ideal generated by Casimir element}
            R_{\boldsymbol{\zeta\mu}} = \big\{ f \in \Ker \epsilon \;\big|\; \forall g \in \Ker \epsilon, \, (z_{\boldsymbol{\zeta\mu}} ^{\abf} , fg) =0 \big\}.
        \end{equation}
    In particular, if in addition $\zeta_1, \cdots, \zeta_p \in \frac{i}{2} \hbar^{-1} \Qbf^\lor$, $\zeta_{p+1}, \cdots, \zeta_m \notin \frac{i}{2} \hbar^{-1} \Qbf^\lor$, and $\abf \in \Rbb^m$, then the $*$-FODC $(\boldsymbol{\Omega_{\zeta\mu}}, \boldsymbol{d_{\zeta\mu}})$ is induced by $\boldsymbol{z_{\zeta \mu}^{\abf}} $.

    Therefore, up to isomorphism, every finite-dimensional bicovariant FODC on $K_q$ is induced by an $\ad$-invariant linear functional on $\Cf^\infty(K_q)$.
\end{thm}

This theorem is a simple consequence of the following lemma, which is of interest in its own right.

\begin{lem}\label{lem:comultiplication of Casimir element}
    Let $\mu \in \weights^+$. Then,
    \begin{equation*}
        \hat{\Delta}(z_\mu) = \sum_{1 \leq i,j \leq n_\mu} q^{-2 (\rho, \epsilon^\mu _i)} \, \hat{S} I'(u^\mu _{ji}) \otimes I(u^\mu _{ij}).
    \end{equation*}
\end{lem}
\begin{proof}
    Let $f, g \in \Cf^\infty(K_q)$. Using \eqref{eq:an identity regarding I}, we calculate
    \begin{align*}
        (z_\mu, fg) &= (I(t_\mu), fg) = \Big( I \big( S(fg) \big), S(t_\mu) \Big) \nonumber \\
        &= \Big( l^- \big( S(f_{(2)} g_{(2)}) \big) \hat{S} l^+ \big( S(f_{(1)} g_{(1)}) \big), S(t_\mu) \Big) \nonumber \\
        &= \Big( \hat{S} l^+ (f_{(1)} g_{(1)}) l^- (f_{(2)} g_{(2)}), t_\mu \Big) \nonumber \\
        &= \Big( \hat{S} l^+ (g_{(1)}) \hat{S} l^+ (f_{(1)}) l^- (f_{(2)}) l^- (g_{(2)}) K_{-2 \rho}, \operatorname{Tr}_{V(\mu)} \Big) \nonumber \\
        &= \Big( \hat{S} l^+ (f_{(1)}) l^- (f_{(2)}) l^- (g_{(2)}) K_{- 2\rho} \hat{S} l^+ (g_{(1)}), \sum_{1 \leq i \leq n_\mu} u^\mu _{ii} \Big) \nonumber \\
        &= \sum_{1 \leq i,j \leq n_\mu} \Big( \hat{S} l^+ (f_{(1)}) l^- (f_{(2)}), u^\mu _{ij} \Big) \Big( l^- (g_{(2)}) K_{-2\rho} \hat{S} l^+ (g_{(1)}), u^\mu _{ji} \Big) \nonumber \\
        &= \sum_{1 \leq i,j \leq n_\mu} \Big( \hat{S}^{-1} \Big( l^- \big(S(f_{(2)})\big) \hat{S} l^+ \big( S(f_{(1)}) \big) \Big), u^\mu _{ij} \Big) \nonumber \\
        &\hspace{5.5cm} \times \Big( l^- (g_{(2)}) \hat{S}^{-1} l^+ (g_{(1)}) K_{-2\rho}, u^\mu _{ji} \Big) \nonumber \\
        &= \sum_{1 \leq i,j \leq n_\mu} \Big( \hat{S}^{-1} \Big( I \big(S(f)\big) \Big), u^\mu _{ij} \Big) \Big( \hat{S}^{-1} \big( K_{2\rho} l^+ (g_{(1)}) \hat{S} l^- (g_{(2)}) \big), u^\mu _{ji} \Big) \nonumber \\
        &= \sum_{1 \leq i,j \leq n_\mu} \Big( I(u^\mu _{ij}), f \Big) \Big( K_{2 \rho} I'(g), S(u^\mu _{ji}) \Big).
    \end{align*}
    However, for any $X \in U_q ^\Rbb (\kf)$ and $x \in \Cf^\infty (K_q)$,
    \begin{align*}
        \big( K_{2\rho} X, S(x) \big) = \big(K_{2\rho}, S(x_{(2)}) \big) \big( X, S(x_{(1)}) \big) = (K_{-2 \rho}, x_{(2)}) \big(X, S(x_{(1)}) \big) \\
        = \big(X, S(K_{-2\rho} \triangleright x) \big)
    \end{align*}
    and hence, using \eqref{eq:an identity regarding I'},
    \begin{align*}
        (z_\mu, fg) &= \sum_{1 \leq i,j \leq n_\mu} \Big( I(u^\mu _{ij}), f \Big) \Big( I'(g), S(K_{-2\rho} \triangleright u^\mu _{ji}) \Big) \\
        &= \sum_{1 \leq i,j \leq n_\mu} \Big( I(u^\mu _{ij}), f \Big) q^{- 2 (\rho, \epsilon^\mu _i)} \Big( I'(g), S(u^\mu _{ji}) \Big) \\
        &= \sum_{1 \leq i,j \leq n_\mu} q^{-2 (\rho, \epsilon^\mu _i)} \Big( I(u^\mu _{ij}), f \Big) \Big( \hat{S} I'(u^\mu _{ji}), g \Big),
    \end{align*}
    which proves the lemma.
\end{proof}

\begin{proof}[Proof of Theorem~\ref{thm:FODCs induced by Casimir elements}]

If $(\zeta_m, \mu_m) = (0,0)$, then $1 = I(u^0_{11}) = z_0$, and hence
\[
z_{\boldsymbol{\zeta \mu}}^{\abf} - a_m K_{\zeta_m} z_{\mu_m} = z_{\boldsymbol{\zeta \mu}}^{\abf} - a_m
\]
induces the same FODC via the formula \eqref{eq:right ideal generated by Casimir element}. Also, since $\Omega_{\zeta_m \mu_m} = 0$, we have
\[
\Omega_{\boldsymbol{\zeta\mu}} = \Omega_{(\zeta_1, \cdots, \zeta_{m-1}) \, (\mu_1, \cdots, \mu_{m-1})}.
\]
These two facts show that we can exclude $(\zeta_m, \mu_m)$ from the beginning.

Thus, without loss of generality, we may assume that $(\zeta_l, \mu_l) \neq (0,0)$ for all $1 \leq l \leq m$. By Lemma~\ref{lem:comultiplication of Casimir element}, we have
\[
\hat{\Delta}(K_{\zeta_l} z_{\mu_l}) = \sum_{1 \leq i,j \leq n_\mu} q^{-2 (\rho, \epsilon^{\mu_l}_i)} K_{\zeta_l} \hat{S}I'(u^{\mu_l}_{ji}) \otimes K_{\zeta_l} I(u^{\mu_l}_{ij})
\]
for each $1 \leq l \leq m$. Hence, for each $f \in \Ker \epsilon$, the functional
\begin{equation*}
L_f : \Cf^\infty(K_q) \ni g \longmapsto \big( z_{\boldsymbol{\zeta\mu}}^{\abf}, f(g - \epsilon(g)1) \big) \in \Cbb
\end{equation*}
is equal to
\begin{equation*}
L_f = \sum_{1 \leq l \leq m} a_l \sum_{1 \leq i,j \leq n_{\mu_l}} q^{-2 (\rho, \epsilon_i^{\mu_l})} \big( K_{\zeta_l} I(u^{\mu_l}_{ij}), f \big) \big( K_{\zeta_l} \hat{S}I'(u^{\mu_l}_{ji}) - \delta_{ji} \big) \in \Cf^\infty(K_q)^\circ.
\end{equation*}

But since $(\zeta_l, \mu_l) \neq (0,0)$ for all $1 \leq l \leq m$, Corollary~\ref{cor:Kzeta independence} and the injectivity of the linear maps $\hat{S}$ and $I'$ imply that the functionals
\[
\big\{ K_{\zeta_l} \hat{S} I'(u^{\mu_l}_{ji}) - \delta_{ji} \;\big|\; 1 \leq l \leq m,\, 1 \leq i,j \leq n_{\mu_l} \big\} \subseteq \Cf^\infty(K_q)^\circ
\]
are linearly independent. Therefore, by the assumption that $a_l \neq 0$ for all $1 \leq l \leq m$, we have
\[
L_f = 0 \Leftrightarrow \big( K_{\zeta_l} I(u^{\mu_l}_{ij}), f \big) = 0 \text{ for all } 1 \leq l \leq m \text{ and } 1 \leq i,j \leq n_{\mu_l}.
\]
But, by definition, $L_f(1) = 0$, and hence
\[
L_f = 0 \Leftrightarrow L_f|_{\Ker \epsilon} = 0 \Leftrightarrow (z_{\boldsymbol{\zeta\mu}}^{\abf}, fg) = 0 \text{ for all } g \in \Ker \epsilon.
\]
Thus, by \eqref{eq:the right ideal for Gamma^mu}, \eqref{eq:direct sum right ideal-Kq}, and the preceding two equations, we have
\begin{align*}
R_{\boldsymbol{\zeta\mu}} &= \bigcap_{1 \leq l \leq m} \big\{ f \in \Ker \epsilon \mid \forall 1 \leq i,j \leq n_{\mu_l},\, \big( K_{\zeta_l} I(u^{\mu_l}_{ij}), f \big) = 0 \big\} \\
&= \big\{ f \in \Ker \epsilon \mid \forall g \in \Ker \epsilon,\, (z_{\boldsymbol{\zeta\mu}}^{\abf}, fg) = 0 \big\}.
\end{align*}
\end{proof}

\subsection{The center of \texorpdfstring{$\Cf^\infty(K_q)^\circ$}{TEXT}}\label{subsec:The center of Cf(Kq)circ}

Corollary~\ref{cor:linear functional for BCFODC-Kq} and Proposition~\ref{prop:ad-invariance criterion for a linear functional} imply
\begin{equation}\label{eq:center of the restricted dual-preliminary}
    \Span_{\Cbb} \{ K_\zeta z_\mu \mid \zeta \in \Zcal ,\, \mu \in \weights^+ \} \subseteq Z ( \Cf^\infty(K_q)^\circ ).
\end{equation}
Thus, Proposition~\ref{prop:finite-dimensional BCFODC induced by linear functionals} and Theorems~\ref{thm:The classification of bicovariant FODCs over Cinfty(Kq)} and \ref{thm:FODCs induced by Casimir elements} together imply that the correspondence
\[
\phi \longmapsto (\Omega_\phi , d_\phi)
\]
defines a surjection from $Z ( \Cf^\infty(K_q)^\circ )$ onto the set of equivalence classes of finite-dimensional bicovariant FODCs on $K_q$.

Naturally, this leads us to ask under what conditions two linear functionals in $Z ( \Cf^\infty(K_q)^\circ )$ induce the same FODC. To answer this, we require an explicit description of $Z ( \Cf^\infty(K_q)^\circ )$.

\begin{thm}\label{thm:center of the Hopf dual of Cinfty(Kq)}
    The multiplication map of $\Cf^\infty(K_q)^\circ$ yields an isomorphism
    \begin{equation}\label{eq:center of the Hopf dual of Cinfty(Kq)}
    \Span_\Cbb \{ K_\zeta \mid \zeta \in \Zcal \} \otimes Z U_q (\gf) \cong Z ( \Cf^\infty(K_q)^\circ ).
    \end{equation}
    Thus, \eqref{eq:center of the restricted dual-preliminary} is in fact an equality, and the set $\{ K_\zeta z_\mu \mid \zeta \in \Zcal ,\, \mu \in \weights^+ \}$ forms a linear basis for $Z ( \Cf^\infty(K_q)^\circ )$.
\end{thm}

Theorems~\ref{thm:FODCs induced by Casimir elements} and \ref{thm:center of the Hopf dual of Cinfty(Kq)} immediately imply:

\begin{cor}\label{cor:when two ad-invariant functionals induce the same FODC}
Define an equivalence relation $\sim$ on $Z ( \Cf^\infty(K_q)^\circ )$ by
\begin{align*}
    &\Big( \sum_{ ( \zeta, \mu) \in \Zcal \times \Pbf^+ } a_{(\zeta, \mu)} K_\zeta z_\mu \Big) \sim \Big( \sum_{ ( \zeta, \mu) \in \Zcal \times \Pbf^+ } b_{(\zeta, \mu)} K_\zeta z_\mu \Big) \quad \text{if and only if} \nonumber \\
    &\quad \big\{ (\zeta, \mu) \in \Zcal \times \Pbf^+ \setminus \{ (0,0) \} \; \big| \; a_{(\zeta,\mu)} \neq 0 \big\} = \big\{ (\zeta, \mu) \in \Zcal \times \Pbf^+ \setminus \{ (0,0) \} \; \big| \; b_{(\zeta,\mu)} \neq 0 \big\}. \nonumber
\end{align*}
Then, two linear functionals in $Z ( \Cf^\infty(K_q)^\circ )$ induce the same FODC if and only if they belong to the same equivalence class under $\sim$.
\end{cor}

\begin{proof}[Proof of Theorem~\ref{thm:center of the Hopf dual of Cinfty(Kq)}]
    Let $\phi \in Z ( \Cf^\infty(K_q)^\circ )$, and consider
    \[
    R_\phi' = \{ f \in \Cf^\infty(K_q) \mid \forall g \in \Cf^\infty(K_q),\, ( \phi , fg ) = 0 \},
    \]
    which is an $\ad$-invariant right ideal of finite codimension in $\Cf^\infty(K_q)$, by Proposition~\ref{prop:ad-invariance criterion for a linear functional} and Lemmas~\ref{lem:a variation of the induction construction}--\ref{lem:cofactorization of linear functionals}. Then, by Theorem~\ref{thm:finite codimensional ad-invarant right ideals of Cf(Kq)}, there exist distinct $(\zeta_1, \mu_1), \dots, (\zeta_m, \mu_m) \in \Zcal \times \Pbf^+$ such that
    \begin{equation}\label{eq:equality of right ideals}
    R_\phi' = \bigcap_{\substack{1 \leq l \leq m \\ 1 \leq i,j \leq n_{\mu_l}}} \Ker \big( K_{\zeta_l} I( u^{\mu_l}_{ij} ) \big).
    \end{equation}

    Let $\hat{\Delta}(\phi) = \sum_{1 \leq i \leq n} X_i \otimes Y_i$, where $\{ X_i \mid 1 \leq i \leq n \}$ is linearly independent and $Y_i \neq 0$ for all $1 \leq i \leq n$. Then, by Lemma~\ref{lem:cofactorization of linear functionals}, we have
    \[
    R_\phi' = \bigcap_{1 \leq i \leq n} \Ker Y_i.
    \]
    Combining this with \eqref{eq:equality of right ideals} yields
    \[
    \Span_\Cbb \{ Y_i \mid 1 \leq i \leq n \} = \Span_\Cbb \{ K_{\zeta_l} I ( u^{\mu_l}_{ij} ) \mid 1 \leq l \leq m,\; 1 \leq i,j \leq n_{\mu_l} \},
    \]
    because if $\bigcap_{1 \leq i \leq k} \Ker f_i \subseteq \Ker f$ for linear functionals $f, f_1, \dots, f_k$ on a vector space $V$, then the linear map
    \[
    \Cbb^k \supseteq \Span_\Cbb \Big\{ \big( f_i(w) \big)_{1 \leq i \leq k} \mid w \in V \Big\} \ni \big( f_i(v) \big)_{1 \leq i \leq k} \longmapsto f(v) \in \Cbb, \quad v \in V
    \]
    is well-defined, and hence $f \in \Span_\Cbb \{ f_i \mid 1 \leq i \leq k \}$.

    Therefore, we conclude
    \[
    \phi = \sum_{1 \leq i \leq n} \hat{\epsilon}(X_i)\, Y_i \in \Span_\Cbb \{ K_{\zeta_l} I(u^{\mu_l}_{ij}) \mid 1 \leq l \leq m,\; 1 \leq i,j \leq n_{\mu_l} \}.
    \]
    This implies that there exist $Z_1, \dots, Z_k \in U_q(\gf)$ and distinct $\zeta_1', \dots, \zeta_k' \in \Zcal$ such that
    \[
    \phi = K_{\zeta_1'} Z_1 + \cdots + K_{\zeta_k'} Z_k.
    \]

    Since $\phi$ lies in the center of $\Cf^\infty(K_q)^\circ$, we have
    \[
    0 = [X, \phi] = K_{\zeta_1'} [X, Z_1] + \cdots + K_{\zeta_k'} [X, Z_k]
    \]
    for all $X \in U_q(\gf)$. By Corollary~\ref{cor:Kzeta independence}, this implies
    \[
    [X, Z_1] = \cdots = [X, Z_k] = 0, \quad X \in U_q(\gf),
    \]
    i.e., $Z_1, \dots, Z_k \in ZU_q(\gf)$, which proves the claim.
\end{proof}

Theorem~\ref{thm:center of the Hopf dual of Cinfty(Kq)} has a classical analogue. Let $G$ be the complexification of $K$. Then, each point $x \in G$ gives rise to an element $\ev_x \in \Cf^\infty(K)^\circ$ given by
\[
\Cf^\infty(K) \ni f \longmapsto \tilde{f}(x) \in \Cbb
\]
where $\tilde{f}$ is the unique holomorphic extension of $f$ to $G$, see \cite[Chapter VI]{Chevalley}. Also, we have $U(\gf) \subseteq \Cf^\infty(K)^\circ$ via \eqref{eq:classical-skew-pairing}.

\begin{prop}\label{prop:center of the Hopf dual of Cinfty(K)}
    The multiplication in $\Cf^\infty(K)^\circ$ gives an isomorphism
    \begin{equation}\label{eq:center of the Hopf dual of Cinfty(K)}
        \Span_\Cbb \{ \ev_z \mid z \in Z \} \otimes ZU(\gf) \cong Z ( \Cf^\infty(K)^\circ )
    \end{equation}
    where $Z$ denotess the center of $K$ and $Z( \Cf^\infty(K)^\circ )$ the center of $\Cf^\infty(K)^\circ$.
\end{prop}

\begin{proof}
    By \cite[Theorem~2.1.8]{Joseph} and \cite[Section~III.8]{tomDieck}, the multiplication in $\Cf^\infty(K)^\circ$ gives an isomorphism
    \begin{equation}\label{eq:the Hopf dual of Cinfty(K)}
        \Span_\Cbb \{ \ev_x \mid x \in G \} \otimes U(\gf) \cong \Cf^\infty(K)^\circ.
    \end{equation}
    Note that for $x \in G$, $X \in \gf$, and $f \in \Cf^\infty(K)$,
    \begin{align*}
        (\ev_x X , f ) = \left. \frac{d}{dt} \right|_{t = 0} \tilde{f}( x \exp(tX) ) = \tilde{f} \Big( \exp \big(t \Ad(x) X \big) x \Big) = ( \Ad(x) X \, \ev_x , f ).
    \end{align*}
    
    Let $ \sum_{ 1 \leq i \leq n } \ev_{z_i} \otimes Z_i \in Z ( \Cf^\infty(K)^\circ ) $, the center of $\Cf^\infty(K)^\circ$, with $z_1, \cdots , z_n$ mutually distinct. Applying the conjugation $\ev_x ( \, \cdot \, ) \ev_x ^{-1}$ and using the above commutation relation, we obtain
    \[
        \sum_{ 1 \leq i \leq n } \ev_{z_i} \otimes Z_i = \sum_{1 \leq i \leq n} \ev_{x z_i x^{-1}} \otimes \Ad(x) Z_i, \quad \forall x \in G.
    \]
    By Artin’s theorem on the linear independence of characters, this implies
    \[
        \{ z_i \mid 1 \leq i \leq n \} = \{ x z_i x^{-1} \mid 1 \leq i \leq n \}, \quad \forall x \in G,
    \]
    which, by the connectedness of $G$, holds if and only if
    \[
        z_i = x z_i x^{-1}, \quad 1 \leq i \leq n, \, \forall x \in G.
    \]
    That is, $z_i \in Z$, the center of $G$, which coincides with the center of $K$ by \cite[Theorem~6.31.(e)]{Knapp}. Now, let $X \in \gf$ and apply $[ X , \, \cdot \, ]$ to $\sum_{1 \leq i \leq n } \ev_{z_i} \otimes Z_i$, yielding
    \[
        0 = \sum_{1 \leq i \leq n} \ev_{z_i} \otimes [X, Z_i].
    \]
    Again by Artin’s theorem, this implies $[X, Z_i] = 0$, i.e., $Z_i \in ZU(\gf)$ for all $1 \leq i \leq n$. This proves \eqref{eq:center of the Hopf dual of Cinfty(K)}.
\end{proof}

\begin{rmk}\label{rmk:remark on computation of Z(Cf(Kq)circ)-1}
    Recalling that $\{K_\zeta \mid \zeta \in \Zcal \}$ is the set of homomorphisms
    \[
        \Cf^\infty(K_q) \rightarrow \Cf^\infty(Z) \xrightarrow{\ev_z} \Cbb , \quad z \in Z,
    \]
    where the first map is the canonical projection identifying $Z$ as a closed quantum subgroup of $K_q$ (cf. Remark~\ref{rmk:center inside the maximal torus}), we see that \eqref{eq:center of the Hopf dual of Cinfty(Kq)} is the exact quantum analogue of the classical computation \eqref{eq:center of the Hopf dual of Cinfty(K)}. Note that the first tensor component remains unchanged under $q$-deformation.
    
    Since $ZU_q(\gf)$ and $ZU(\gf)$ are both isomorphic to the algebra of complex polynomials in $N$ generators by \cite[Theorem~3.120]{VoigtYuncken} and \cite[Ch.~VIII, \S~8, Proposition~4.(ii)]{Bourbaki_Lie_7-9}, respectively, it follows from \eqref{eq:center of the Hopf dual of Cinfty(Kq)} and \eqref{eq:center of the Hopf dual of Cinfty(K)} that
    \[
        Z( \Cf^\infty(K_q)^\circ ) \cong Z( \Cf^\infty(K)^\circ )
    \]
    as algebras for any $0 < q < 1$, which may not be the case for $\Cf^\infty(K_q)^\circ$ and $\Cf^\infty(K)^\circ$ in view of the descriptions given in \eqref{eq:the Hopf dual of Cinfty(Kq)} and \eqref{eq:the Hopf dual of Cinfty(K)}, respectively.
\end{rmk}

In this subsection, we have seen that the commutative algebra $Z ( \Cf^\infty(K_q)^\circ )$ provides a complete description of all finite-dimensional bicovariant FODCs on $K_q$ via the construction given in Corollary~\ref{cor:BCFODC induced by linear functionals} (cf. Theorem~\ref{thm:FODCs induced by Casimir elements} and Corollary~\ref{cor:when two ad-invariant functionals induce the same FODC}). In light of this, we raise the following question.

\begin{ques}\label{ques:conjecturs}
    Are all finite-dimensional bicovariant FODCs on $K$ induced by a linear functional in $Z ( \Cf^\infty(K)^\circ )$? If so, under what conditions do two such functionals induce the same FODC?
\end{ques}

\subsection{Laplacians on \texorpdfstring{$K_q$}{TEXT}}\label{subsec:Laplacians on Kq}

\begin{thm}\label{thm:classification of Laplacian functionals on Kq-preliminary}
    Let $\phi \in \Cf^\infty(K_q)^\circ$ be self-adjoint and $\ad$-invariant. Then there exist mutually distinct pairs $(\zeta_1, \mu_1), \cdots, (\zeta_m, \mu_m) \in \Zcal \times \weights^+$, with $\zeta_1, \cdots, \zeta_p \in \frac{i}{2} \hbar^{-1} \Qbf^\lor$ and $\zeta_{p+1}, \cdots, \zeta_{m} \notin \frac{i}{2} \hbar^{-1} \Qbf^\lor$, and a tuple $\abf = (a_1, \cdots, a_m) \in \Rbb^p \times \Cbb^{m-p}$ such that $a_l \neq 0$ for all $1 \leq l \leq m$, for which
    \[
    \phi = \boldsymbol{z _{\zeta \mu} ^a }.
    \]
    It is Hermitian if and only if the set $\{\mu_1, \dots, \mu_m\}$ is invariant under the involution $-w_0$, and for every pair $l, k$ such that $-w_0 \mu_l = \mu_k$, we have either:
    \begin{itemize}
        \item $\zeta_l = \zeta_k$ and $\overline{a_l} = a_k$, or
        \item $\zeta_l = -\zeta_k$ and $a_l = a_k$.
    \end{itemize}
\end{thm}

\begin{proof}
The first assertion follows from Theorem~\ref{thm:center of the Hopf dual of Cinfty(Kq)} and Corollary~\ref{cor:linear functional for BCFODC-Kq}.

For the second assertion, we apply Corollary~\ref{cor:self-adjointness and Hermiticity for ad-invariant functional}. Recall that $K_{\zeta} \circ S = K_{-\zeta}$ for $\zeta \in \hf^*$. Moreover, by Proposition~\ref{prop:antipode and Casimir element} proved below, we have $z_{\mu} S = z_{-w_0 \mu}$ for $\mu \in \weights^+$. Therefore, by the linear independence of the elements $\{K_\zeta z_\mu \mid \zeta \in \Zcal, \, \mu \in \weights^+\}$, we conclude that $\boldsymbol{z_{\zeta \mu}^{\abf}} = \boldsymbol{z_{\zeta \mu}^{\abf}} S$ if and only if the stated conditions hold.
\end{proof}

\begin{thm}\label{thm:Casimir elements as Laplacians}
    Let $(\zeta_1, \mu_1), \cdots, (\zeta_m, \mu_m) \in \Zcal \times \weights^+$ and $\abf = (a_1, \cdots, a_m) \in \Rbb^p \times \Cbb^{m-p}$ satisfy all the conditions in Theorem~\ref{thm:classification of Laplacian functionals on Kq-preliminary}. Then, for
    \begin{equation}\label{eq:Casimir elements as Laplacian functionals}
        Z_{\boldsymbol{\zeta\mu}}^{\abf} = \frac{2}{(q^{-1} - q)^2} \big( \boldsymbol{z_{\zeta \mu}^{\abf}} - \hat{\epsilon}(\boldsymbol{z_{\zeta \mu}^{\abf}}) \big) \in \Cf^\infty(K_q)^\circ,
    \end{equation}
    the operator
    \begin{equation}\label{eq:Casimir elements as Laplacians}
        Z_{\boldsymbol{\zeta\mu}}^{\abf} \triangleright : \Cf^\infty(K_q) \rightarrow \Cf^\infty(K_q)
    \end{equation}
    becomes the Laplacian associated with $\big( \boldsymbol{\Omega_{\zeta\mu}}, \boldsymbol{d_{\zeta\mu}}, \la \cdot, \cdot \ra_{\boldsymbol{\zeta \mu}} \big)$, where $\la \cdot, \cdot \ra_{\boldsymbol{\zeta \mu}} : \boldsymbol{\Omega_{\zeta\mu}} \times \boldsymbol{\Omega_{\zeta\mu}} \rightarrow \Cf^\infty(K_q)$ is a strongly nondegenerate right $\Cf^\infty(K_q)$-sesquilinear form given, for $f, g \in \Cf^\infty(K_q)$ and $x, y \in \Ker \epsilon$, by
    \begin{equation}\label{eq:strongly ndg form for q-deformed Laplacian}
        \big\la \boldsymbol{Q_{\zeta \mu}}(x) f , \boldsymbol{Q_{\zeta\mu}}(y) g \big\ra_{\boldsymbol{\zeta \mu}}^{\abf} = - \frac{1}{2} f^* g \big( Z_{\boldsymbol{\zeta \mu}}^{\abf}, S(x)^* S(y) \big).
    \end{equation}
    We will also refer to the functional $Z_{\boldsymbol{\zeta\mu}}^{\abf}$ as a Laplacian; see Remark~\ref{rmk:linear functional and operator correspondence}.
\end{thm}

\begin{proof}
It follows from Theorems~\ref{thm:linear functionals as Laplacians} and \ref{thm:classification of Laplacian functionals on Kq-preliminary}.
\end{proof}

The reason for the appearance of the scalar $\frac{2}{(q^{-1} - q)^2}$ in \eqref{eq:Casimir elements as Laplacian functionals} will be made clear in Corollary~\ref{cor:eigenvalues of Casimir Laplacians}.

\begin{rmk}\label{rmk:lifting the Hermiticity to gain generality}
By the argument in Remark~\ref{rmk:linear functionals as Laplacians-non-Hermitian case}, by dropping the Hermiticity condition, one need not impose the additional conditions stated in the second sentence of Theorem~\ref{thm:Casimir elements as Laplacians} on FODCs; however, the resulting Laplacians are still required to satisfy those conditions.

Therefore, one can assert that operators of the form \eqref{eq:Casimir elements as Laplacians} exhaust \emph{all} Laplacians associated with \emph{any} finite-dimensional bicovariant $*$-FODC on $K_q$ that can be constructed by Theorem~\ref{thm:linear functionals as Laplacians}, see also Proposition~\ref{prop:finite-dimensional BCFODC induced by linear functionals}.
\end{rmk}

Now, we compute the eigenvalues of some of the Laplacians in \eqref{eq:Casimir elements as Laplacians}.

\begin{prop}\label{prop:Casimir element multiplier}
Fix $\mu \in \weights^+$. Then, the eigenvalues of the $\ad$-invariant functional $z_\mu$ are given by
\begin{equation*}
C_{z_\mu}(\lambda) = \sum_{1 \leq j \leq n_\mu} q^{-2 ( \lambda + \rho, \epsilon_j^\mu)}, \quad \lambda \in \weights^+.
\end{equation*}
\end{prop}

\begin{proof}
Fix $\lambda \in \weights^+$ and let $v_{-\lambda} \in V(-w_0 \lambda)$ be a lowest weight unit vector. Then, by \eqref{eq:an evaluation of I},
\[
I\big(\la v_{-\lambda} \mid \cdot \mid v_{- \lambda} \ra \big) = K_{2 \lambda}.
\]
Note also that, since $\la v_{-\lambda} \mid \cdot \mid v_{- \lambda} \ra \in \End(V(-w_0 \lambda))^*$, we have
\[
S^{-1}\big(\la v_{-\lambda} \mid \cdot \mid v_{- \lambda} \ra\big) \in \End(V(\lambda))^*.
\]
Skew-pairing this with $F_{i_1} \cdots F_{i_m} K_\nu E_{j_1} \cdots E_{j_n}$ for any $1 \leq i_1, \ldots, i_m, j_1, \ldots, j_n \leq N$ and $\nu \in \weights$ shows that
\[
S^{-1}\big(\la v_{-\lambda} \mid \cdot \mid v_{-\lambda} \ra\big) = \la v_\lambda \mid \cdot \mid v_\lambda \ra \in \End(V(\lambda))^*,
\]
where $v_\lambda \in V(\lambda)$ is a highest weight unit vector. Hence, by \eqref{eq:an identity regarding I},
\begin{align*}
c_{z_\mu}(\lambda) &= \big(z_\mu, \la v_{\lambda} \mid \cdot \mid v_{\lambda} \ra \big) = \Big( I(t_\mu), S^{-1}\big(\la v_{-\lambda} \mid \cdot \mid v_{- \lambda} \ra\big) \Big) \\
&= \Big(I\big(\la v_{-\lambda} \mid \cdot \mid v_{-\lambda} \ra \big), S(t_\mu) \Big) = \big(K_{2 \lambda}, S(t_\mu)\big) \\
&= \Tr_{V(\mu)} (K_{-2 \lambda} K_{-2 \rho}) = \sum_{1 \leq j \leq n_\mu} q^{-2 (\lambda + \rho, \epsilon_j^\mu)},
\end{align*}
as required.
\end{proof}

\begin{prop}\label{prop:antipode and Casimir element}
    For $\mu \in \weights^+$, we have
    \[
    \hat{S}^{\pm 1} (z_\mu) = z_{-w_0 \mu}.
    \]
\end{prop}

\begin{proof}
    First, note that $\hat{S}^2(z_\mu) = K_{2\rho} z_\mu K_{-2\rho} = z_\mu$, and hence $\hat{S}(z_\mu) = \hat{S}^{-1}(z_\mu)$. Thus, it suffices to consider $\hat{S}(z_\mu)$ in the proof. Also observe that $\hat{S}(z_\mu) = z_\mu S^{-1}$ is an $\ad$-invariant linear functional by Proposition~\ref{prop:self-adjointness and ad-invariance are preserved under antipode}, and therefore admits eigenvalues with respect to the Peter–Weyl decomposition of $\Cf^\infty(K_q)$.
    
    Fix $\lambda \in \weights^+$ and let $v_{w_0\lambda} \in V(\lambda)$ be a lowest weight unit vector. Then, by \eqref{eq:an identity regarding I} and \eqref{eq:an evaluation of I},
    \begin{align*}
        C_{\hat{S}(z_\mu)}(\lambda) &= \big( \hat{S}(z_\mu), \la v_{w_0 \lambda} \mid \cdot \mid v_{w_0 \lambda} \ra \big) 
        = \Big( I(t_\mu), S^{-1} \big( \la v_{w_0 \lambda} \mid \cdot \mid v_{w_0 \lambda} \ra \big) \Big) \\
        &= \Big( I\big( \la v_{w_0 \lambda} \mid \cdot \mid v_{w_0 \lambda} \ra \big), S(t_\mu) \Big) 
        = \big( K_{-2 w_0 \lambda}, S(t_\mu) \big) \\
        &= \Tr_{V(\mu)} (K_{2 w_0 \lambda} K_{-2\rho}) 
        = \sum_{1 \leq j \leq n_\mu} q^{2 (w_0 \lambda - \rho, \epsilon_j^\mu)} \\
        &= \sum_{1 \leq j \leq n_\mu} q^{-2 (\lambda + \rho, -w_0 \epsilon_j^\mu)} 
        = \sum_{1 \leq j \leq n_{-w_0 \mu}} q^{-2 (\lambda + \rho, \epsilon_j^{-w_0 \mu})} = C_{z_{-w_0 \mu}}(\lambda)
    \end{align*}
    by Proposition~\ref{prop:Casimir element multiplier}, where in moving from the third to the fourth line, we used the facts that $w_0^2 = \id$, $w_0$ preserves the Killing form, and $w_0 \rho = -\rho$. In the second to the last equality, we used the identity
    \[
    -w_0 \{ \epsilon_j^\mu \mid 1 \leq j \leq n_\mu \} = -\{ \epsilon_j^\mu \mid 1 \leq j \leq n_\mu \} = \{ \epsilon_j^{-w_0 \mu} \mid 1 \leq j \leq n_\mu \},
    \]
    see \cite[Theorem~5.5 and Problem~5.1]{Knapp}.
\end{proof}

\begin{thm}\label{thm:eigenvalues of Casimir Laplacians}
Let $\mu_1, \cdots, \mu_m \in \weights^+$ and $\abf = (a_1, \cdots, a_m) \in \Rbb^m$ be such that $\mu_1, \cdots, \mu_m$ are mutually distinct, the set $\{ \mu_1, \cdots, \mu_m \}$ is invariant under the transformation $-w_0$, and $a_l = a_k$ whenever $-w_0 \mu_l = \mu_k$. Then, for $\lambda \in \weights^+$, the eigenvalues of the Laplacian $Z_{\boldsymbol{\mu}} ^{\abf} :=  Z_{\boldsymbol{0 \mu}}^{\abf} $ are given by
\begin{align}\label{eq:eigenvalues of q-Laplacian}
    C_{Z_{\boldsymbol{\mu}}^{\abf}}(\lambda) 
    = \sum_{1 \leq l \leq m} a_l \sum_{1 \leq j \leq n_{\mu_l}} \Big( \big[(\lambda + \rho , \epsilon_j^{\mu_l})\big]_q^2 - \big[(\rho, \epsilon_j^{\mu_l})\big]_q^2 \Big) .
\end{align}
\end{thm}

\begin{proof}
    Note that, by assumption, we have $Z_{\boldsymbol{\mu}}^{\abf} = Z_{\boldsymbol{\mu}}^{\abf} S$, and hence
    \begin{align*}
        Z_{\boldsymbol{\mu}}^{\abf} 
        &= \frac{1}{2} \left( Z_{\boldsymbol{\mu}}^{\abf} + Z_{\boldsymbol{\mu}}^{\abf} S \right) 
        = \frac{ \boldsymbol{z}_{\boldsymbol{\mu}}^{\abf} + \boldsymbol{z}_{-w_0 \boldsymbol{\mu}}^{\abf} - \hat{\epsilon}(\boldsymbol{z}_{\boldsymbol{\mu}}^{\abf}) - \hat{\epsilon}(\boldsymbol{z}_{(-w_0 \boldsymbol{\mu})}^{\abf}) }{(q - q^{-1})^2}
    \end{align*}
    by Proposition~\ref{prop:antipode and Casimir element}, where $\boldsymbol{z_{\mu}} := \boldsymbol{z_{0 \mu} }$ and $-w_0 \boldsymbol{\mu} = (-w_0 \mu_1, \cdots, -w_0 \mu_m)$.

    However, since $\Pbf(-w_0 \mu_l) = -\Pbf(\mu_l)$ including multiplicities, this implies, by Proposition~\ref{prop:Casimir element multiplier},
    \begin{align*}
        &C_{Z_{\boldsymbol{\mu}} ^{\abf} }(\lambda) \\
        &= \frac{1}{(q^{-1} - q)^2} \sum_{1 \leq l \leq m} a_l \sum_{1 \leq j \leq n_{\mu_l}}  
        \Big( \big( q^{-2 (\lambda + \rho, \epsilon_j^{\mu_l})} + q^{2 (\lambda + \rho, \epsilon_j^{\mu_l})} \big)
        - \big( q^{-2 (\rho, \epsilon_j^{\mu_l})} + q^{2 (\rho, \epsilon_j^{\mu_l})} \big) \Big) \\
        &= \sum_{1 \leq l \leq m} a_l \sum_{1 \leq j \leq n_{\mu_l}} \Big( \big[(\lambda + \rho , \epsilon_j^{\mu_l})\big]_q^2 - \big[(\rho, \epsilon_j^{\mu_l})\big]_q^2 \Big) .
    \end{align*}
\end{proof}

The eigenvalues of any other Laplacians in \eqref{eq:Casimir elements as Laplacians} can be computed in a similar manner. However, we computed them only for this simple case, due to the following corollary, which may not hold for more general Laplacians, as the parameters in $\Zcal$ depend on $q$.

Recall that $\Cf^\infty(K_q)$ and $\Cf^\infty(K)$ admit the same Peter–Weyl decomposition, see \eqref{eq:classical matrix coefficients identification-semisimple case} and \eqref{eq:quantized algebra of functions}.

\begin{cor}\label{cor:eigenvalues of Casimir Laplacians}
    If $0 < a_1, \cdots, a_m < \infty$ and the representation
    \[
    \pi_{\boldsymbol{\mu}} := \pi_{\mu_1} \oplus \cdots \oplus \pi_{\mu_m} : \gf \rightarrow \End (V(\mu_1) \oplus \cdots \oplus V(\mu_m) )
    \]
    is faithful, in addition to the assumptions of Theorem~\ref{thm:eigenvalues of Casimir Laplacians}, then there exists a classical Laplacian $\square$ on $K$ such that
    \[
    C_\square (\lambda) = \lim_{q \rightarrow 1} C_{Z_{\boldsymbol{\mu}} ^{\abf}} (\lambda), \quad \lambda \in \weights^+.
    \]
    In other words, the Laplacian $Z_{\boldsymbol{\mu}} ^{\abf} \triangleright : \Cf^\infty(K_q) \rightarrow \Cf^\infty(K_q)$ converges to the classical Laplacian $\square$ as $q \rightarrow 1$.
\end{cor}
\begin{proof}
    This follows from Theorems~\ref{thm:eigenvalues of Casimir Laplacians} and \ref{thm:Eigenvalues of classical Laplacians on K}.
\end{proof}

Corollary~\ref{cor:eigenvalues of Casimir Laplacians} leads us to the following definition:
\begin{defn}\label{defn:q-deformed Laplacians}
    Operators of the form $Z_{\boldsymbol{\mu}}^{\abf} \triangleright : \Cf^\infty(K_q) \rightarrow \Cf^\infty(K_q)$, for $\mu_1, \cdots, \mu_m \in \weights^+$ and $\abf = (a_1, \cdots, a_m) \in (0,\infty)^m$, such that $\mu_1, \cdots , \mu_m$ are mutually distinct, the representation $\pi_{\boldsymbol{\mu}}$ is faithful, the set $\{ \mu_1, \cdots, \mu_m \}$ is invariant under the transformation $-w_0$, and $a_l = a_k$ whenever $-w_0 \mu_l = \mu_k$, are called \textbf{$q$-deformed Laplacians}.
\end{defn}

The $q$-deformed Laplacians are not merely formal analogues of classical Laplacians (Theorem~\ref{thm:Casimir elements as Laplacians}), but genuine $q$-deformations of the latter (Corollary~\ref{cor:eigenvalues of Casimir Laplacians}).

\subsection{Eigenvalues of classical Laplacians on \texorpdfstring{$K$}{TEXT}}\label{subsec:Eigenvalues of classical Laplacians on K}

This subsection is devoted to the proof of the following theorem.

\begin{thm}\label{thm:Eigenvalues of classical Laplacians on K}
    Let $\mu_1, \cdots , \mu_m \in \weights^+$ and $\abf = ( a_1 , \cdots , a_m ) \in (0, \infty)^m$ be such that $\mu_1, \cdots , \mu_m$ are mutually distinct, the representation $\pi_{\boldsymbol{\mu}}$ is faithful, the set $\{ \mu_1, \cdots, \mu_m \}$ is invariant under the transformation $-w_0$, and $a_l = a_k$ whenever $-w_0 \mu_l = \mu_k$. Then, there exists a classical Laplacian $\square$ on $K$ such that
    \[
    C_\square (\lambda) = \sum_{1 \leq l \leq m} a_l \sum_{1 \leq j \leq n_{\mu_l}} \big( (\lambda + \rho , \epsilon_j^{\mu_l})^2 - (\rho, \epsilon_j^{\mu_l})^2 \big) .
    \]
\end{thm}

Throughout this subsection, we assume for simplicity that $K$ is simple. In this case, the assumption that $\pi_{\boldsymbol{\mu}}$ is faithful becomes redundant---it suffices to require that $\mu_l \neq 0$ for some $1 \leq l \leq m$. For the general case, see \cite[Section~8.4.2]{lee2024}, the first version of this manuscript.

Since terms with $\mu_k = 0$ contribute nothing, we assume from the outset that $m \geq 1$ and that $\mu_l \neq 0$ for all $1 \leq l \leq m$.

A complex bilinear form $B : \gf \times \gf \rightarrow \Cbb$ is called \textit{invariant} if
\begin{equation*}
    B\big( \ad X (Y), Z \big) = - B\big( Y, \ad X (Z) \big), \quad X, Y, Z \in \gf.
\end{equation*}
In particular, every complex bilinear form on $\gf$ extended from an $\Ad$-invariant inner product on $\kf$ is invariant; see \eqref{eq:ad invariance of inner product}.

\begin{prop}\label{prop:representations induce invariant forms}
    Let $\pi : K \rightarrow \Lbb(V)$ be a unitary representation of $K$ on a finite-dimensional Hilbert space $V$ whose induced Lie algebra representation, also denoted by $\pi : \kf \rightarrow \Lbb (V)$, is faithful. Then the symmetric bilinear form $B_\pi : \gf \times \gf \rightarrow \Cbb$ defined by
\begin{equation*}
    B_\pi(X, Y) = \Tr_V \big( \pi(X) \pi(Y) \big), \quad X, Y \in \gf
\end{equation*}
is an invariant, nondegenerate, symmetric bilinear form on $\gf$.

Moreover, the restriction $- B_\pi |_{\kf \times \kf} : \kf \times \kf \rightarrow \Rbb$ is an $\Ad$-invariant inner product on $\kf$.
\end{prop}

\begin{proof}
    Throughout the proof, we write $B = B_\pi$. Observe that, for all $X, Y, Z \in \gf$, we have
\begin{align*}
    B\big( \ad X (Y), Z \big) &= \Tr_V \big( \pi([X, Y]) \pi(Z) \big) = \Tr_V \big( \big( \pi(X) \pi(Y) - \pi(Y) \pi(X) \big) \pi(Z) \big) \\
    &= - \Tr_V \big( \pi(Y) \big( \pi(X) \pi(Z) - \pi(Z) \pi(X) \big) \big) = - B\big( Y, \ad X (Z) \big),
\end{align*}
proving the invariance of $B$.

Thus, by \eqref{eq:exp ad is Ad}, we have
\[
B\big( \Ad(\exp X) Y, Z \big) = B\big( Y, \Ad(-\exp X) Z \big), \quad Y, Z \in \kf,
\]
which proves the $\Ad$-invariance of $-B|_{\kf \times \kf}$ since $K$ is connected.

Also, since $\pi(X) = \left. \frac{d}{dt} \right|_{t=0} \pi(\exp(tX))$ is a skew-adjoint operator on $V$ for $X \in \kf$, we have, by the faithfulness of $\pi$,
\[
    -B(X, X) = -\Tr_V\big( \pi(X) \pi(X) \big) > 0, \quad 0 \neq X \in \kf.
\]
This simultaneously proves the nondegeneracy of $B$ and the fact that $-B : \kf \times \kf \rightarrow \Rbb$ is an inner product on $\kf$.
\end{proof}

Recall that a \textit{root vector with root $\alpha \in \boldsymbol{\Delta}$} is an element $e_\alpha \in \gf$ satisfying
\[
[H, e_\alpha] = \alpha(H) e_\alpha, \quad \forall H \in \hf.
\]

\begin{lem}\label{lem:invariant inner products on compact semisimple Lie algebras}
    Let $B$ be an invariant, nondegenerate, symmetric bilinear form on $\gf$. Then:
    \begin{enumerate}
        \item Given a root vector $e_\alpha \in \gf$ with root $\alpha \in \boldsymbol{\Delta}$, we have, for $H \in \hf$,
        \[
        B(e_\alpha, H) = 0,
        \]
        and for any other root vector $e_\beta$ with $-\alpha \neq \beta \in \boldsymbol{\Delta}$,
        \[
        B(e_\alpha, e_\beta) = 0.
        \]
        
        \item $B|_{\hf \times \hf}$ is nondegenerate.

        \item $B(e_\alpha, e_{-\alpha}) \neq 0$ for any nonzero root vectors $e_\alpha, e_{-\alpha}$ with roots $\alpha, -\alpha \in \boldsymbol{\Delta}$, respectively.
    \end{enumerate}

    Using (2), we identify $\hf^* \cong \hf$, and for each $\zeta \in \hf^*$, denote the corresponding element in $\hf$ by $h_\zeta$, i.e.,
    \begin{equation}\label{eq:definition of hzeta}
        h_\zeta \in \hf \quad \text{and} \quad B(h_\zeta, H) = \zeta(H) \quad \text{for all } H \in \hf.
    \end{equation}
    Then, for any root vectors $e_\alpha, e_{-\alpha}$ with roots $\alpha, -\alpha \in \boldsymbol{\Delta}$, respectively,
    \begin{equation}\label{eq:hzeta as commutator}
        [e_\alpha, e_{-\alpha}] = B(e_{\alpha}, e_{-\alpha}) h_\alpha.
    \end{equation}
\end{lem}
\begin{proof}
    See \cite[Section~II.4]{Knapp}.
\end{proof}

\begin{prop}\label{prop:eigenvalues of classical Casimir elements}
    Consider the invariant bilinear form $B_\pi$ from Proposition~\ref{prop:representations induce invariant forms}. For each $\zeta \in \hf^*$, define $h_\zeta^\pi$ as in \eqref{eq:definition of hzeta} using $B_\pi$.
    Let $Z_\pi \in U^\Rbb(\kf)$ be the classical Casimir element associated with the $\Ad$-invariant inner product $-B_\pi|_{\kf \times \kf}$. Then, its eigenvalues over the Peter-Weyl decomposition are given by
    \[
    C_{Z_\pi} (\lambda) = \Tr_V\big( \pi(h_{\lambda + \rho}^\pi)^2 - \pi(h_\rho^\pi)^2 \big), \quad \lambda \in \weights^+.
    \]
\end{prop}
\begin{proof}
    Throughout the proof, we write $B = B_\pi$ and $h_\zeta = h_\zeta^\pi$.
    For each $\alpha \in \boldsymbol{\Delta}^+$, we choose a nonzero root vector $e_\alpha \in \gf$. Then $f_\alpha = -\overline{e_\alpha} \in \gf$ is a root vector with root $-\alpha$. Let
\[
x_\alpha = \frac{e_\alpha - f_\alpha}{2}, \quad y_\alpha = \frac{e_\alpha + f_\alpha}{2i} \in \kf.
\]
Since $B(e_\alpha, \overline{e_\alpha}) = B(x_\alpha, x_\alpha) + B(y_\alpha, y_\alpha) < 0$, we can normalize $e_\alpha$ so that $-B(e_\alpha, f_\alpha) = B(e_\alpha, \overline{e_\alpha}) = -2$. Then, by \eqref{eq:hzeta as commutator}, we have
\begin{equation}\label{eq:halpha as commutator}
    [e_\alpha, f_\alpha] = 2 h_\alpha.
\end{equation}
Also, by Lemma~\ref{lem:invariant inner products on compact semisimple Lie algebras}~(1),
\begin{equation}\label{eq:orthonormality of root vectors}
    -B(x_\alpha, x_\beta) = -B(y_\alpha, y_\beta) = \delta_{\alpha \beta}, \quad -B(x_\alpha, y_\beta) = 0, \quad \alpha, \beta \in \boldsymbol{\Delta}^+.
\end{equation}

We further fix an orthonormal basis $\{h_1, \dots, h_N\}$ of $\tf = \hf \cap \kf$ with respect to the inner product $-B|_{\tf \times \tf}$ (cf. Lemma~\ref{lem:invariant inner products on compact semisimple Lie algebras}~(2)). Then, again by Lemma~\ref{lem:invariant inner products on compact semisimple Lie algebras}~(1) and \eqref{eq:orthonormality of root vectors}, we see that
\[
\{ h_j, \, x_\alpha, \, y_\alpha \mid 1 \leq j \leq N, \, \alpha \in \boldsymbol{\Delta}^+ \} \subseteq \kf
\]
is an orthonormal basis of $(\kf, -B|_{\kf \times \kf})$.

Therefore, $Z_\pi$ is given by
\begin{align*}
    Z_\pi = -\Big( \sum_{1 \leq j \leq N} h_j^2 + \sum_{\alpha \in \boldsymbol{\Delta}^+} (x_\alpha^2 + y_\alpha^2) \Big) = -\Big( \sum_{1 \leq j \leq N} h_j^2 - \frac{1}{2} \sum_{\alpha \in \boldsymbol{\Delta}^+} (e_\alpha f_\alpha + f_\alpha e_\alpha) \Big),
\end{align*}
which, by \eqref{eq:halpha as commutator}, is equal to
\[
- \Big( \sum_{1 \leq j \leq N} h_j^2 - \sum_{\alpha \in \boldsymbol{\Delta}^+} h_\alpha - \sum_{\alpha \in \boldsymbol{\Delta}^+} f_\alpha e_\alpha \Big).
\]
To compute its eigenvalues, let $\lambda \in \weights^+$ and $v_\lambda \in V(\lambda)$ be a highest weight vector. Then $e_\alpha v_\lambda = 0$ for any $\alpha \in \boldsymbol{\Delta}^+$ by definition. Hence,
\begin{align*}
    -C_{Z_\pi}(\lambda) v_\lambda = -\pi_\lambda(Z_\pi) v_\lambda &= \pi_\lambda\Big( \sum_{1 \leq j \leq N} h_j^2 - \sum_{\alpha \in \boldsymbol{\Delta}^+} h_\alpha - \sum_{\alpha \in \boldsymbol{\Delta}^+} f_\alpha e_\alpha \Big) v_\lambda \\
    &= \Big( \sum_{1 \leq j \leq N} \lambda(h_j)^2 - \sum_{\alpha \in \boldsymbol{\Delta}^+} \lambda(h_\alpha) \Big) v_\lambda \\
    &= \Big( \sum_{1 \leq j \leq N} B(h_\lambda, h_j)^2 - \sum_{\alpha \in \boldsymbol{\Delta}^+} B(h_\lambda, h_\alpha) \Big) v_\lambda \\
    &= \Big( -\sum_{1 \leq j \leq N} \big( -B(ih_\lambda, h_j) \big)^2 - 2 B(h_\lambda, h_\rho) \Big) v_\lambda \\
    &= \Big( -\big( -B(ih_\lambda, ih_\lambda) \big) - 2 B(h_\lambda, h_\rho) \Big) v_\lambda \\
    &= \big( -B(h_\lambda, h_\lambda) - 2 B(h_\lambda, h_\rho) \big) v_\lambda \\
    &= \big( -B(h_{\lambda + \rho}, h_{\lambda + \rho}) + B(h_\rho, h_\rho) \big) v_\lambda \\
    &= -\Tr_V\big( \pi(h_{\lambda + \rho})^2 - \pi(h_\rho)^2 \big) v_\lambda.
\end{align*}
\end{proof}

So far, we have not needed the simplicity assumption. However, it becomes necessary for the following.

\begin{prop}\label{prop:uniqueness of invariant Laplacians}
    Let $Z_K \in U^\Rbb(\kf)$ be the classical Casimir element associated with $-(\cdot,\cdot) : \kf \times \kf \rightarrow \Rbb$, the negative of the Killing form. If $Z \in U^\Rbb(\kf)$ is any other classical Casimir element associated with an $\Ad$-invariant inner product $\la \cdot , \cdot \ra : \kf \times \kf \rightarrow \Rbb$, then there exists a positive constant $b>0$ such that
    \[
    Z = b Z_K, \quad \la \cdot, \cdot \ra = -b ( \cdot, \cdot ).
    \]
    In particular, any summation of classical Laplacians is also a classical Laplacian.
\end{prop}
\begin{proof}
    We extend $-\la \cdot , \cdot \ra$ to a complex bilinear form $B : \gf \times \gf \rightarrow \Cbb$.
    Since $\gf$ is simple, by \cite[Exercise~6.6]{Humphreys}, there exists $0 \neq b \in \Cbb$ such that $B = b(\cdot,\cdot)$, which implies
    \[
    \la \, \cdot \, , \, \cdot \, \ra = -b ( \, \cdot \, , \, \cdot \, ).
    \]
    However, since both $\la \cdot, \cdot \ra$ and $-(\cdot,\cdot)$ are positive definite on $\kf$, $b$ must be positive. Now, observe that, by Proposition~\ref{prop:classical Laplacian characterization},
    \[
    \hcal \big( \overline{f}(Z \triangleright g) \big) = \hcal \big( \la df, dg \ra \big) = -b \hcal \big( (df,dg) \big) = b \hcal \big( \overline{f}(Z_K \triangleright g) \big)
    \]
    for all $f, g \in \Cf^\infty(K)$, where $\hcal$ is the Haar state of $\Cf^\infty(K)$. By the faithfulness of $\hcal$, this implies $Z = b Z_K$.
\end{proof}

Let $0 \neq \mu \in \weights^+$. Then, since $\gf$ is simple, by Propositions~\ref{prop:representations induce invariant forms} and \ref{prop:uniqueness of invariant Laplacians}, there exists a constant $b_\mu > 0$ such that
\begin{equation*}
\Tr_{V(\mu)} \big( \pi_\mu (X) \pi_\mu (Y) \big) = b_\mu ( X, Y ), \quad X, Y \in \kf.
\end{equation*}

For each $\zeta \in \hf^*$, let $H_\zeta \in \hf$ be the unique element satisfying
\[
(H_\zeta , H) = \zeta(H), \quad \forall H \in \hf.
\]
Then, we have
\begin{equation*}
    b_\mu^{-1} \Tr_{V(\mu)} \big( \pi_\mu( H_\zeta ) \pi_\mu( H ) \big) = (H_\zeta , H) = \zeta(H), \quad H \in \hf,
\end{equation*}
and hence $h_\zeta^{\pi_\mu}$ in Proposition~\ref{prop:eigenvalues of classical Casimir elements} is given by
\begin{equation}\label{eq:Hzeta and  hzeta}
h_\zeta^{\pi_\mu} = b_\mu^{-1} H_\zeta.
\end{equation}

Proposition~\ref{prop:uniqueness of invariant Laplacians} also implies that the classical Casimir element $Z_{\pi_\mu}$ associated with the $\Ad$-invariant inner product
\[
\kf \times \kf \ni (X, Y) \longmapsto -\Tr_{V(\mu)} \big( \pi_\mu(X) \pi_\mu(Y) \big) \in \Rbb
\]
is given by
\begin{equation}\label{eq:irreducible Laplacians}
    Z_{\pi_\mu} = b_\mu Z_K.
\end{equation}

\begin{proof}[Proof of Theorem~\ref{thm:Eigenvalues of classical Laplacians on K}]
\begin{align*}
\sum_{1 \leq l \leq m} a_l \sum_{1 \leq j \leq n_{\mu_l}} \big( (\lambda + \rho, \epsilon_j^{\mu_l})^2 - (\rho, \epsilon_j^{\mu_l})^2 \big) 
= \sum_{1 \leq l \leq m} a_l \Tr_{V(\mu_l)} \big( \pi_{\mu_l}(H_{\lambda + \rho})^2 - \pi_{\mu_l}(H_\rho)^2 \big) \\
= \sum_{1 \leq l \leq m} a_l \, b_{\mu_l}^2 \Tr_{V(\mu_l)} \big( \pi_{\mu_l}(h_{\lambda + \rho}^{\pi_{\mu_l}})^2 - \pi_{\mu_l}(h_\rho^{\pi_{\mu_l}})^2 \big) \\
= \sum_{1 \leq l \leq m} a_l \, b_{\mu_l}^2 \, C_{Z_{\pi_{\mu_l}}}(\lambda) = \sum_{1 \leq l \leq m} a_l \, b_{\mu_l}^3 \, C_{Z_K}(\lambda)
\end{align*}
by \eqref{eq:Hzeta and  hzeta}, Proposition~\ref{prop:eigenvalues of classical Casimir elements}, and \eqref{eq:irreducible Laplacians}.
However, according to the proof of Proposition~\ref{prop:uniqueness of invariant Laplacians}, this is the eigenvalue at $\lambda \in \weights^+$ of the classical Laplacian associated with the $\Ad$-invariant inner product
\[
- \sum_{1 \leq l \leq m} a_l \, b_{\mu_l}^3 (\,\cdot\,,\,\cdot\,) : \kf \times \kf \to \Rbb.
\]
\end{proof}

\begin{rmk}\label{rmk:q->1 limits of q-deformed Laplacians exhaust all classical Laplacians}
Proposition~\ref{prop:uniqueness of invariant Laplacians} and the proof of Theorem~\ref{thm:Eigenvalues of classical Laplacians on K} show that, given any classical Laplacian on $K$, there exist infinitely many $q$-deformed Laplacians on $K_q$ that converge to it as $q \to 1$; this can be achieved by choosing any $0 \neq \mu, -w_0 \mu \in \Pbf^+$ and appropriately controlling the constants $\abf$.
\end{rmk}

\section{The \texorpdfstring{$q$}{TEXT}-deformed Laplacians}\label{sec:The q-deformed Laplacians}

Throughout this section, we fix $\mu_1, \cdots, \mu_m \in \weights^+$ and $0 < a_1, \cdots, a_m < \infty$, such that $\mu_1, \cdots, \mu_m$ are mutually distinct, the representation $\pi_{\boldsymbol{\mu}}$ is faithful, the set $\{ \mu_1, \cdots, \mu_m \}$ is invariant under the transformation $-w_0$, and $a_l = a_k$ whenever $-w_0 \mu_l = \mu_k$; that is, we fix a $q$-deformed Laplacian (Definition~\ref{defn:q-deformed Laplacians})
\[
Z := Z_{\boldsymbol{\mu}}^{\abf}
\]
and explore its properties. For simplicity, we assume that $K$ is simple. For the general case, see \cite[Section~8.5]{lee2024}, the first version of this paper.

\subsection{Spectrum}\label{subsec:Spectrum}

\begin{prop}\label{prop:lower semiboundedness of Laplacian functionals}
    The eigenvalues \eqref{eq:eigenvalues of q-Laplacian} are lower semibounded; that is, there exists $L \in \Rbb$ such that
    \[
    L \leq C_{Z}(\lambda), \quad \forall \lambda \in \weights^+.
    \]
    Moreover, $C_{Z}(\lambda) \rightarrow \infty$ as $(\lambda, \lambda) \rightarrow \infty$.
\end{prop}
\begin{proof}
    Observe that, for any $\lambda \in \weights^+$,
    \begin{align*}
        C_{Z}(\lambda) = \sum_{1 \leq l \leq m} a_l \sum_{1 \leq j \leq n_{\mu_l}} \Big( \big[(\lambda + \rho, \epsilon_j^{\mu_l})\big]_q^2 - \big[(\rho, \epsilon_j^{\mu_l})\big]_q^2 \Big) 
        \geq - \sum_{1 \leq l \leq m} a_l \sum_{1 \leq j \leq n_{\mu_l}} \big[(\rho, \epsilon_j^{\mu_l})\big]_q^2,
    \end{align*}
    which proves the lower semiboundedness.

    For the second assertion, note that the bilinear form
    \[
    \gf \times \gf \ni (X,Y) \longmapsto \sum_{1 \leq l \leq m} a_l \Tr_{V(\mu_l)} \big( \pi_{\mu_l}(X) \pi_{\mu_l}(Y) \big) \in \Cbb
    \]
    is $\Ad$-invariant and negative definite on $\kf$ by Proposition~\ref{prop:representations induce invariant forms}. Since $\gf$ is simple, by \cite[Exercise~6.6]{Humphreys}, there exists $b > 0$ such that
    \begin{equation*}
        \sum_{1 \leq l \leq m} a_l \Tr_{V(\mu_l)} \big( \pi_{\mu_l}(X) \pi_{\mu_l}(Y) \big) = b (X, Y), \quad X, Y \in \gf.
    \end{equation*}

    Now, let $\lambda \in \weights^+$, and note that for any $1 \leq l \leq m$ and $1 \leq j \leq n_{\mu_l}$,
    \begin{align*}
        q^{-2 (\lambda + \rho, \epsilon_j^{\mu_l})} + q^{2 (\lambda + \rho, \epsilon_j^{\mu_l})} 
        = e^{-2h (\lambda + \rho, \epsilon_j^{\mu_l})} + e^{2h (\lambda + \rho, \epsilon_j^{\mu_l})}
        \geq \big( 2h (\lambda + \rho, \epsilon_j^{\mu_l}) \big)^2,
    \end{align*}
    as can be seen by expanding the two series. Moreover,
    \begin{align*}
        \sum_{1 \leq l \leq m} a_l \sum_{1 \leq j \leq n_{\mu_l}} (\lambda + \rho, \epsilon_j^{\mu_l})^2 
        = \sum_{1 \leq l \leq m} a_l \Tr_{V(\mu_l)} \big( \pi_{\mu_l}(H_{\lambda+\rho})^2 \big) 
        = b (H_{\lambda+\rho}, H_{\lambda+\rho}) \\
        = b (\lambda + \rho, \lambda + \rho) 
        \geq b (\lambda + \rho, \lambda) \geq b(\lambda, \lambda).
    \end{align*}

    Combining the two preceding calculations, we obtain
    \begin{align*}
        C_{Z}(\lambda) &= \frac{1}{(q^{-1} - q)^2} \sum_{1 \leq l \leq m} a_l \sum_{1 \leq j \leq n_{\mu_l}} \Big( q^{-2 (\lambda + \rho, \epsilon_j^{\mu_l})} + q^{2 (\lambda + \rho, \epsilon_j^{\mu_l})} - \big( q^{-2(\rho, \epsilon_j^{\mu_l})} + q^{2(\rho, \epsilon_j^{\mu_l})} \big) \Big) \\
        &\geq \frac{4h^2 b}{(q^{-1} - q)^2} (\lambda, \lambda) - \frac{1}{(q^{-1} - q)^2} \sum_{1 \leq l \leq m} a_l \sum_{1 \leq j \leq n_{\mu_l}} \big( q^{-2(\rho, \epsilon_j^{\mu_l})} + q^{2(\rho, \epsilon_j^{\mu_l})} \big),
    \end{align*}
    from which the second assertion follows.
\end{proof}

\subsection{Heat semigroups on \texorpdfstring{$K_q$}{TEXT}}\label{subsec:Heat semigroups on Kq}

\begin{defn}\label{defn:q-deformed heat semigroups}
    We call the semigroup of operators $\big(e^{-tZ \triangleright} : \Cf^\infty(K_q) \rightarrow \Cf^\infty(K_q) \big)_{t \geq 0}$ \textbf{the heat semigroup on $K_q$ generated by $Z$}.
\end{defn}

Proposition~\ref{prop:complete positivity of semigroup} tells us that this semigroup is a quantum Markov semigroup if and only if $-Z$ is conditionally positive. We will presently prove that this is never the case.

Let $\gamma \in \boldsymbol{\Delta}^+$ be the highest root; that is, $\gamma$ is the highest weight corresponding to the adjoint representation
\[
\ad_{\gf}: \gf \rightarrow \End(\gf).
\]
Then, since
\[
-w_0 \boldsymbol{\Delta} = -\boldsymbol{\Delta} = \boldsymbol{\Delta},
\]
we see that $-w_0 \gamma$ is also the highest root, i.e., $\gamma = -w_0 \gamma$.

\begin{thm}\label{thm:conditional nonpositivity}
    Let $f = t_{\gamma} - \epsilon(t_{\gamma}) \in \Ker \epsilon$. Then, for any $0 \neq \mu \in \weights^+$, we have
    \[
    \big( z_\mu , S^{-1}(f) S^{-1}(f)^* \big) > 0.
    \]
    Thus, $\displaystyle - Z = \frac{-2}{(q^{-1}-q)^2} \sum_{1 \leq l \leq m} a_l \big(z_{\mu_l} - \hat{\epsilon}(z_{\mu_l}) \big)$ is not conditionally positive.
\end{thm}

\begin{cor}\label{cor:non-Markovian process}
    The heat semigroups generated by $q$-deformed Laplacians do not form quantum Markov semigroups.
\end{cor}

\begin{proof}[Proof of Theorem~\ref{thm:conditional nonpositivity}]
    Let $0 \neq \mu \in \weights^+$ and $g \in \Cf^\infty(K_q)$. Then, by Lemma~\ref{lem:comultiplication of Casimir element}, we have
    \begin{align*}
    ( z_\mu , g g^* ) &= \sum_{1 \leq i,j \leq n_\mu} q^{-2 (\rho, \epsilon_{i}^\mu)} \big( I(u^\mu_{ij}), g \big) \big( \hat{S} I'(u^\mu_{ji}), g^* \big)  \\
    &= \sum_{1 \leq i,j \leq n_\mu} q^{-2 (\rho, \epsilon_i^\mu)} \big( I(u^\mu_{ij}), g \big) \overline{\big( I'(u^\mu_{ji})^*, g \big)} \\
    &= \sum_{1 \leq i,j \leq n_\mu} q^{-2 (\rho, \epsilon_i^\mu)} \big( I(u^\mu_{ij}), g \big) \overline{\big( I'(u^\mu_{ij}), g \big)} \\
    &= \sum_{1 \leq i,j \leq n_\mu} q^{-2 (\rho, \epsilon_i^\mu)} \Big( I\big( S(g) \big), S(u^\mu_{ij}) \Big) \overline{\Big( I'\big( S(g) \big), S(u^\mu_{ij}) \Big)}.
    \end{align*}
    Thus, substituting $g = S^{-1}(f)$, we obtain
    \begin{align*}
    \big( z_\mu , S^{-1}(f) S^{-1}(f)^* \big) = \sum_{1 \leq i,j \leq n_\mu} q^{-2 (\rho, \epsilon_i^\mu)} \big( I(f), S(u^\mu_{ij}) \big) \overline{\big( I'(f), S(u^\mu_{ij}) \big)}.
    \end{align*}
    By Lemma~\ref{lem:I and I'} proved below, we have $I'(t_{\gamma}) = z_{-w_0 \gamma} = z_{\gamma} = I(t_{\gamma})$. Hence,
    \[
    I'(f) = I'\big(t_{\gamma} - \epsilon(t_{\gamma})\big) = z_{\gamma} - \epsilon(t_{\gamma}) = I(f).
    \]
    Therefore,
    \begin{align}\label{eq:identity involving zmu-2}
    \big( z_\mu , S^{-1}(f) S^{-1}(f)^* \big) &= \sum_{1 \leq i,j \leq n_\mu} q^{-2 (\rho, \epsilon_i^\mu)} \big| \big( z_{\gamma} - \hat{\epsilon}(z_{\gamma}), S(u^\mu_{ij}) \big) \big|^2 \\
    &= \sum_{1 \leq i,j \leq n_\mu} q^{-2 (\rho, \epsilon_i^\mu)} \Big| \big( C_{z_{\gamma}}(-w_0 \mu) - \hat{\epsilon}(z_{\gamma}) \big) \big(1, S(u^\mu_{ij})\big) \Big|^2 \nonumber \\
    &= \sum_{1 \leq i \leq n_\mu} q^{-2 (\rho, \epsilon_i^\mu)} \big| C_{z_{\gamma}}(-w_0 \mu) - \hat{\epsilon}(z_{\gamma}) \big|^2. \nonumber
    \end{align}
    However, $\Pbf(\gamma) = \{0\} \cup \boldsymbol{\Delta}$. Thus, by Proposition~\ref{prop:Casimir element multiplier},
\begin{align*}    
&= \sum_{1 \leq i \leq n_\mu} q^{-2 (\rho, \epsilon_i^\mu)} \Big| \sum_{\alpha \in \boldsymbol{\Delta}} \big( q^{-2(-w_0 \mu + \rho, \alpha)} - q^{-2(\rho, \alpha)} \big) \Big|^2 \nonumber \\
    &= \sum_{1 \leq i \leq n_\mu} q^{-2 (\rho, \epsilon_i^\mu)} \Big| \sum_{\alpha \in \boldsymbol{\Delta}^+} \big( q^{-2(-w_0 \mu + \rho, \alpha)} + q^{2(-w_0 \mu + \rho, \alpha)} - (q^{-2(\rho, \alpha)} + q^{2(\rho, \alpha)}) \big) \Big|^2. \nonumber
    \end{align*}
    Note that for each $\alpha \in \boldsymbol{\Delta}^+$, we have
    \begin{align}\label{eq:each summand}
    q^{-2(-w_0 \mu + \rho, \alpha)} &+ q^{2(-w_0 \mu + \rho, \alpha)} - \big( q^{-2(\rho, \alpha)} + q^{2(\rho, \alpha)} \big) \\
    &= q^{-2(\rho, \alpha)} \big( q^{-2(-w_0 \mu, \alpha)} - 1 \big) + q^{2(-w_0 \mu + \rho, \alpha)} \big( 1 - q^{-2(-w_0 \mu, \alpha)} \big) \nonumber \\
    &= \big( q^{-2(-w_0 \mu, \alpha)} - 1 \big) \big( q^{-2(\rho, \alpha)} - q^{2(-w_0 \mu + \rho, \alpha)} \big) \geq 0, \nonumber
    \end{align}
    since $(-w_0 \mu, \alpha)$, $(-w_0 \mu + \rho, \alpha)$, and $(\rho, \alpha)$ are all nonnegative, and $0 < q < 1$. Moreover, since $-w_0 \mu \neq 0$, there exists at least one $\alpha \in \boldsymbol{\Delta}^+$ such that $(-w_0 \mu, \alpha) > 0$, in which case \eqref{eq:each summand} is strictly positive.

    Therefore, we conclude that \eqref{eq:identity involving zmu-2} is positive, as claimed.
\end{proof}

\begin{lem}\label{lem:I and I'}
    For $\lambda \in \weights^+$, we have
    \[
    I' (t_\lambda) = z_{-w_0 \lambda}.
    \]
\end{lem}
\begin{proof}
    Note that, by \eqref{eq:I preservers adjoint actions}, we have, for all $X \in U_q(\gf)$,
    \[
    I (X \rightarrow t_\lambda) = X \rightarrow z_\mu = \hat{\epsilon}(X) z_\mu = I \big( \hat{\epsilon}(X) t_\lambda \big),
    \]
    and hence $X \rightarrow t_\lambda = \epsilon(X) t_\lambda$. Thus, by \eqref{eq:I' preserves adjoint actions}, we obtain
    \[
    X \rightarrow I' (t_\lambda) = I' (X \rightarrow t_\lambda) = \hat{\epsilon}(X) I'(t_\lambda),
    \]
    for all $X \in U_q(\gf)$. Hence, $I'(t_\lambda)$ is a central element of $U_q(\gf)$, which implies that $I'(t_\lambda)$ is an $\ad$-invariant linear functional on $\Cf^\infty(K_q)$ by Proposition~\ref{prop:ad-invariance criterion for a linear functional}.
    
    Let $\nu \in \weights^+$ and let $v_\nu \in V(\nu)$ be a highest weight unit vector. Then,
    \[
    I' \big( \langle v_\nu \mid \cdot \mid v_\nu \rangle \big) = K_{2 \nu}
    \]
    by \eqref{eq:an evaluation of I'}. Notice also that, since $\langle v_\nu \mid \cdot \mid v_\nu \rangle \in \End(V(\nu))^*$, we have
    \[
    S^{-1} \big( \langle v_\nu \mid \cdot \mid v_\nu \rangle \big) \in \End(V(-w_0 \nu))^*,
    \]
    and, in fact, skew-pairing this with $E_{i_1} \cdots E_{i_m} K_\eta F_{j_1} \cdots F_{j_n}$ for any $1 \leq i_1, \dots, i_m, j_1, \dots, j_n \leq N$ and $\eta \in \weights$ shows that
    \[
    S^{-1} \big( \langle v_\nu \mid \cdot \mid v_\nu \rangle \big) = \langle v_{-\nu} \mid \cdot \mid v_{-\nu} \rangle \in \End(V(-w_0 \nu))^*,
    \]
    where $v_{-\nu} \in V(-w_0 \nu)$ is a lowest weight unit vector. Hence, by \eqref{eq:an evaluation of I'}--\eqref{eq:an identity regarding I'}, we obtain
    \begin{align*}
        C_{I'(t_\mu)}(-w_0 \nu) &= \big( I'(t_\mu), \langle v_{-\nu} \mid \cdot \mid v_{-\nu} \rangle \big) = \Big( I'(t_\mu), S^{-1} \big( \langle v_\nu \mid \cdot \mid v_\nu \rangle \big) \Big) \\
        &= \Big( I \big( \langle v_\nu \mid \cdot \mid v_\nu \rangle \big), S(t_\mu) \Big) = \big( K_{2 \nu}, S(t_\mu) \big) = \Tr_{V(\mu)} (K_{-2 \nu} K_{-2 \rho}) \\
        &= \sum_{1 \leq j \leq n_\mu} q^{-2 ( \nu + \rho, \epsilon_j^\mu )} = \sum_{1 \leq j \leq n_\mu} q^{-2 (-w_0 \nu + \rho, -w_0 \epsilon_j^\mu)} \\
        &= \sum_{1 \leq j \leq n_{-w_0 \mu}} q^{-2 (-w_0 \nu + \rho, \epsilon_j^{-w_0 \mu})} = C_{z_{-w_0 \mu}}(-w_0 \nu),
    \end{align*}
    by Proposition~\ref{prop:Casimir element multiplier}, which proves $I'(t_\mu) = z_{-w_0 \mu}$.
\end{proof}

\begin{rmk}\label{rmk:non-Markovian process}
    We have seen that the $q$-deformed Laplacian $Z \triangleright$ is intimately related to the differential structure of the compact quantum group $K_q$ (Theorem~\ref{thm:Casimir elements as Laplacians}). Therefore, Corollary~\ref{cor:non-Markovian process} suggests that, on $K_q$, the stochastic processes most relevant to the geometry of $K_q$ may be the non-quantum-Markovian ones, rather than the quantum Markov processes, see \cite{Franz2014} for an extensive discussion on stochastic processes on compact quantum groups.

    Moreover, in view of Corollary~\ref{cor:heat semigroups are Markov semigroups} we conclude that \emph{the $q$-deformation removes the complete positivity of the operators in the heat semigroups on $K$.}
\end{rmk}

\subsection{Strongly nondegenerate sesquilinear form}\label{subsec:Strongly nondegenerate sesquilinear form}

Write the FODC induced by $Z$ by
\(
(\Omega_{\boldsymbol{\mu}} ,d_{\boldsymbol{\mu}}) := (\Omega_{\boldsymbol{0 \mu}} ,d_{\boldsymbol{0 \mu}} ).
\)
Recall that, with respect to the strongly nondegenerate right $\Cf^\infty(K_q)$-sesquilinear form $\la \cdot, \cdot \ra_{\boldsymbol{\mu}} ^{\abf} := \la \cdot, \cdot \ra_{\boldsymbol{0 \mu}} ^{\abf} : \Omega_{\boldsymbol{\mu}} \times \Omega_{\boldsymbol{\mu}} \rightarrow \Cf^\infty(K_q)$ given in \eqref{eq:strongly ndg form for q-deformed Laplacian}, the $q$-deformed Laplacian $Z$ is characterized by
\[
\hcal \big( f^* (Z \triangleright g) \big) = \hcal \big( \la d_{\boldsymbol{\mu}} f, d_{\boldsymbol{\mu}} g \ra_{\boldsymbol{\mu}} ^{\abf} \big).
\]
The following theorem highlights a key difference between the classical and quantum cases.

\begin{thm}\label{thm:non positive definiteness of the strong sesquilinear form}
    The nondegenerate sesquilinear form
    \begin{equation}\label{eq:scalar nondegenerate form on FODC}
    \Omega_{\boldsymbol{\mu}} \times \Omega_{\boldsymbol{\mu}} \ni (\omega, \eta) \longmapsto \hcal \big( \la \omega, \eta \ra_{\boldsymbol{\mu}} ^{\abf} \big) \in \Cbb
    \end{equation}
    is neither positive definite nor negative definite. The same is true for $\la \cdot, \cdot \ra_{\boldsymbol{\mu}} ^{\abf}$.
\end{thm}

\begin{proof}
    In view of Proposition~\ref{prop:positive definiteness and conditional positivity}, Theorem~\ref{thm:conditional nonpositivity} implies that \eqref{eq:scalar nondegenerate form on FODC} is not positive definite.

    Moreover, Proposition~\ref{prop:lower semiboundedness of Laplacian functionals} implies that, for $\lambda \in \weights^+$ with sufficiently large $(\lambda, \lambda)$, we have
    \begin{align*}
    \hcal \big( \la d_{\boldsymbol{\mu}} u^{\lambda}_{ij}, d_{\boldsymbol{\mu}} u^{\lambda}_{ij} \ra_{\boldsymbol{\mu}} ^{\abf} \big) &= \hcal \big( (u^{\lambda}_{ij})^* ( Z \triangleright u^{\lambda}_{ij} ) \big) \\
    &= C_{Z} (\lambda) \hcal \big( (u^{\lambda}_{ij})^* u^{\lambda}_{ij} \big) > 0
    \end{align*}
    for any $1 \leq i,j \leq n_\lambda$, which proves that \eqref{eq:scalar nondegenerate form on FODC} is not negative definite either.

    The statement about $\la \cdot , \cdot \ra_{\boldsymbol{\mu}} ^{\abf}$ follows from this.
\end{proof}

\begin{rmk}\label{rmk:non positive definiteness of the strong sesquilinear form}
In the context of differential calculus on a CQG, the necessity of working with nondegenerate forms that are neither positive definite nor negative definite was already observed in \cite{Heckenberger1997}. These forms were subsequently used extensively in \cite{Heckenberger_Laplace-Beltrami_2000, Heckenberger_Spin_2003}, see also \cite{BhowmickMukhopadhyay_Pseudo-Riemannian_2020} for a recent exposition on this subject. The preceding theorem confirms that this phenomenon is generic in the sense that, for all finite-dimensional bicovariant $*$-FODCs on $K_q$ induced by $q$-Laplacians, such forms arise naturally.
\end{rmk}

\subsection{Positive spectrum}\label{subsec:Positive spectrum}

The eigenvalues of classical Laplacians on compact smooth manifolds are always nonnegative \cite{Jost}. However, due to Theorem~\ref{thm:non positive definiteness of the strong sesquilinear form}, we cannot as readily conclude as in the classical case whether the eigenvalues of $q$-deformed Laplacians are nonnegative. We were only able to establish this for the following particular $q$-deformed Laplacians:

\begin{prop}\label{prop:positive quantum Laplacians}
    If $m =1$ and $\mu_1 = \gamma$, the highest root in $\boldsymbol{\Delta}$, then the eigenvalues of $Z \triangleright$ are given by
    \begin{equation*}
        C_{Z}(\lambda) = a_1 \sum_{\alpha \in \boldsymbol{\Delta}} \Big( \big[ (\lambda + \rho, \alpha) \big]_q^2 - \big[ (\rho, \alpha) \big]_q^2 \Big), \quad \lambda \in \weights^+,
    \end{equation*}
    which is positive for $0 \neq \lambda \in \weights^+$.
\end{prop}
\begin{proof}
Let $0 \neq \lambda \in \weights^+$. Since $\Pbf(\gamma) = \{ 0 \} \cup \boldsymbol{\Delta}$, \eqref{eq:eigenvalues of q-Laplacian} becomes
\begin{align*}
C_{Z_{K_q}}(\lambda) &= a_1 \sum_{\alpha \in \boldsymbol{\Delta}} \Big( \big[ (\lambda + \rho, \alpha) \big]_q^2 - \big[ (\rho, \alpha) \big]_q^2 \Big) \\
&= \frac{a_1}{(q^{-1}-q)^2} \sum_{\alpha \in \boldsymbol{\Delta}} \Big( q^{-2(\lambda+\rho, \alpha)} + q^{2(\lambda+\rho, \alpha)} - \big( q^{-2(\rho, \alpha)} + q^{2(\rho, \alpha)} \big) \Big) \\
&= \frac{2a_1}{(q^{-1}-q)^2} \sum_{\alpha \in \boldsymbol{\Delta}^+} \Big( q^{-2(\lambda+\rho, \alpha)} + q^{2(\lambda+\rho, \alpha)} - \big( q^{-2(\rho, \alpha)} + q^{2(\rho, \alpha)} \big) \Big).
\end{align*}
However, in \eqref{eq:each summand}, we have already observed that each summand in the final expression is nonnegative, and that at least one of them is positive (we just need to replace $-w_0 \mu$ by $\lambda$ there).
\end{proof}

We leave the following as a conjecture.

\begin{ques}\label{ques:are the eigenvalues of q-deformed Laplacians nonnegative}
    Are the eigenvalues of other $q$-deformed Laplacians nonnegative?
\end{ques}

\section{The \texorpdfstring{$q \rightarrow 1$}{TEXT} limit of FODCs associated with \texorpdfstring{$q$}{TEXT}-deformed Laplacians}\label{sec:the q->1 limit of FODCs}

Corollary~\ref{cor:eigenvalues of Casimir Laplacians} showed that $q$-deformed Laplacians serve as $q$-deformations of classical Laplacians. In view of Theorem~\ref{thm:Casimir elements as Laplacians}, it is natural to ask what happens to their associated FODCs $(\Omega_{\boldsymbol{\mu}}, d_{\boldsymbol{\mu}})$ in the $q \rightarrow 1$ limit. In this section, we show that they likewise converge to the classical FODC.

\subsection{More on semisimple Lie algebras}\label{subsec:More on semisimple Lie algebras}

Recall that $\hf \cong \hf^*$ via the Killing form $(\cdot,\cdot)$. For $\lambda \in \hf^*$, let $H_\lambda$ denote the corresponding element in $\hf$ under this identification. For $\alpha \in \boldsymbol{\Delta}^+$, define
\begin{equation*}
    H'_\alpha = \frac{2}{(\alpha, \alpha)} H_\alpha.
\end{equation*}

Let $\overline{(\cdot)} : \gf \rightarrow \gf$ denote the complex conjugation with respect to the real form $\kf$. For each $\alpha \in \boldsymbol{\Delta}^+$, choose a root vector $E_\alpha$ for $\alpha$ and define $F_\alpha = -\overline{E_\alpha}$, which is a root vector for $-\alpha$, since $\boldsymbol{\Delta} \subseteq (i\kf)^*$. Define
\begin{equation*}
    X_\alpha = \frac{1}{2} (E_\alpha - F_\alpha), \quad Y_\alpha = \frac{1}{2i} (E_\alpha + F_\alpha).
\end{equation*}
Then $X_\alpha, Y_\alpha \in \kf$, and we have $E_\alpha = X_\alpha + i Y_\alpha$ and $F_\alpha = -X_\alpha + i Y_\alpha$. Since the Killing form is negative definite on $\kf$, we obtain
\[
- (E_\alpha , F_\alpha) = (E_\alpha, \overline{E_\alpha}) = (X_\alpha, X_\alpha) + (Y_\alpha, Y_\alpha) < 0.
\]
Thus, we can normalize $E_\alpha$ (and hence $F_\alpha = -\overline{E_\alpha}$) so that
\[
(E_\alpha, F_\alpha) = \frac{2}{(\alpha,\alpha)}.
\]
Because $[E_\alpha, F_\alpha] = (E_\alpha, F_\alpha) H_\alpha$, we deduce
\begin{equation}\label{eq:relations between E F H}
    [H'_\alpha, E_\alpha] = 2 E_\alpha, \quad [H'_\alpha, F_\alpha] = -2 F_\alpha, \quad [E_\alpha, F_\alpha] = H'_\alpha, \quad \alpha \in \boldsymbol{\Delta}^+.
\end{equation}

For $ 1 \leq j \leq N$, define
\[
H'_j = H'_{\alpha_j}, \quad E_j = E_{\alpha_j}, \quad F_j = F_{\alpha_j}.
\]
Then, the set $\{ H'_j, E_\alpha, F_\alpha \mid 1 \leq j \leq N, \, \alpha \in \boldsymbol{\Delta}^+ \}$ forms a linear basis of $\gf$. The following is \cite[Theorem~2.98]{Knapp}.

\begin{thm}\label{thm:Lie algebra presentation}
The Lie algebra $\gf$ is a universal complex Lie algebra generated by $\{ H'_j, E_j, F_j \mid 1 \leq j \leq N \}$ with the following relations:
\begin{enumerate}
    \item $[H'_i, H'_j] = 0$
    \item $[H'_i, E_j] = A_{ij} E_j$ and $[H'_i, F_j] = -A_{ij} F_j$
    \item $[E_i, F_j] = \delta_{ij} H'_i$
    \item For $i \neq j$, $(\ad E_i)^{1 - A_{ij}} E_j = 0$ and $(\ad F_i)^{1 - A_{ij}} F_j = 0$.
\end{enumerate}
\end{thm}

The last condition is equivalent to the following condition for $i \neq j$:
    \[
    \sum_{k=0}^{1 - A_{ij}} \binom{1 - A_{ij}}{k} E_i^{1 - A_{ij} - k} E_j E_i^k = \sum_{k=0}^{1 - A_{ij}} \binom{1 - A_{ij}}{k} F_i^{1 - A_{ij} - k} F_j E_i^k = 0
    \]
in $U(\gf)$, where $\binom{*}{*}$ denotes the binomial coefficient.

\begin{prop}\label{prop:the classical differential in terms of H E F}
    Consider $\Cf^\infty(K)$, the space of matrix coefficients of the compact Lie group $K$, and let $f \in \Cf^\infty(K)$. Then, under the identification $\gf^* \cong \gf$ via the Killing form, we have
    \begin{equation*}
        df = \sum_{1 \leq j \leq N} (H_j' \triangleright f) H_{\varpi_j} + \sum_{\alpha \in \boldsymbol{\Delta}^+} \frac{(\alpha,\alpha)}{2} \Big( (E_\alpha \triangleright f) F_\alpha + (F_\alpha \triangleright f) E_\alpha \Big).
    \end{equation*}
\end{prop}
\begin{proof}
Since $\{ H'_j, \, E_\alpha, \, F_\alpha \mid 1 \leq j \leq N, \, \alpha \in \boldsymbol{\Delta}^+ \}$ is a $\Cbb$-linear basis of $\gf$, we have
\begin{equation}\label{eq:linear basis of kf}
\{ i H'_j, \, X_\alpha, \, Y_\alpha \mid 1 \leq j \leq N, \, \alpha \in \boldsymbol{\Delta}^+ \} \subseteq \kf
\end{equation}
as an $\Rbb$-linear basis of $\kf$. Let $\alpha \neq \beta \in \boldsymbol{\Delta}^+$. Since $(E_\alpha, E_\beta) = (F_\alpha, F_\beta) = 0$ by Lemma~\ref{lem:invariant inner products on compact semisimple Lie algebras}~(1), we have
\begin{equation}\label{eq:norms of Xalpha and Yalpha}
    (X_\alpha, X_\alpha) = (Y_\alpha, Y_\alpha) = - \frac{1}{2} (E_\alpha, F_\alpha) = \frac{1}{(\alpha,\alpha)}, \quad
    (X_\alpha, X_\beta) = 0 = (Y_\alpha, Y_\beta).
\end{equation}
Also, by the same lemma,
\begin{equation*}
    (X_\alpha, Y_\alpha) = \frac{1}{4i} \big( (E_\alpha, E_\alpha) - (F_\alpha, F_\alpha) \big) = 0, \quad (X_\alpha, Y_\beta) = 0.
\end{equation*}
Thus, we see that
    \[
    \kf = \tf \oplus \big( \bigoplus_{\alpha \in \Delta^+} \Rbb X_\alpha \big) \oplus \big( \bigoplus_{\alpha \in \Delta^+} \Rbb Y_\alpha \big)
    \]
is an orthogonal decomposition with respect to the Killing form. Now, using \eqref{eq:norms of Xalpha and Yalpha} and the defining relations $\frac{2(\varpi_i,\alpha_j)}{(\alpha_j, \alpha_j)} = \delta_{ij}$ ($1\leq i,j \leq N$) for the fundamental weights, we can check that
    \[
    \{ -i H_{\varpi_j}, \, -(\alpha,\alpha) X_\alpha, \, -(\alpha,\alpha) Y_\alpha \mid 1 \leq j \leq N, \, \alpha \in \boldsymbol{\Delta}^+ \}
    \]
is a dual basis of \eqref{eq:linear basis of kf} with respect to the Killing form.

Therefore, \eqref{eq:classical differential} and Proposition~\ref{prop:differential operators from linear functionals} imply
    \begin{align*}
    df = \sum_{1 \leq j \leq N} (i H'_j \triangleright f) (-i H_{\varpi_j} ) + \sum_{\alpha \in \boldsymbol{\Delta}^+} \Big( (X_\alpha \triangleright f ) ( - (\alpha,\alpha) X_\alpha) + (Y_\alpha \triangleright f ) ( - (\alpha, \alpha) Y_\alpha) \Big) \\
    = \sum_{1 \leq j \leq N} (H'_j \triangleright f) H_{\varpi_j} - \sum_{\alpha \in \boldsymbol{\Delta}^+} (\alpha, \alpha) \Big( (X_\alpha \triangleright f ) X_\alpha + (Y_\alpha \triangleright f ) Y_\alpha \Big)
    \end{align*}
    under the identification $\gf^* \cong \gf$ via the Killing form.
However, for each $\alpha \in \boldsymbol{\Delta}^+$, we have
    \begin{align*}
         (E_\alpha \triangleright f) F_\alpha + (F_\alpha \triangleright f) E_\alpha = - 2(X_\alpha \triangleright f) X_\alpha - 2(Y_\alpha \triangleright f) Y_\alpha.
    \end{align*}
Combining the two preceding identities, we obtain
\begin{equation*}
    df = \sum_{1 \leq j \leq N} (H'_j \triangleright f) H_{\varpi_j} + \sum_{\alpha \in \boldsymbol{\Delta}^+} \frac{(\alpha,\alpha)}{2} \Big( (E_\alpha \triangleright f) F_\alpha + (F_\alpha \triangleright f) E_\alpha \Big).
\end{equation*}
\end{proof}

\subsection{The \texorpdfstring{$q \rightarrow 1$}{TEXT} behavior of irreducible representations}\label{subsec:The q->1 behavior of irreducible representations}

When considering the $q \rightarrow 1$ limits of objects defined on $K_q$, it is convenient to disregard the $*$-structures of $U_q^\Rbb(\kf)$ for $0 < q < 1$ and view them purely as Hopf algebras. We denote these by $U_q(\gf)$ when doing so.

We distinguish objects depending on $q$ by adding a superscript $q$ on the left. For example, the irreducible representation of the algebra $U_q(\gf)$ corresponding to $\mu \in \weights^+$ is denoted by ${}^q\pi_\mu$. We also set $U_1(\gf) := U(\gf)$, $\Cf^\infty(K_1) := \Cf^\infty(K)$, and ${}^1\pi_\mu := \pi_\mu$, the irreducible representation of $\gf$ (and its extension to $U(\gf)$).

According to \cite[Section~5.1.1]{VoigtYuncken}, the sets of weights for $\pi_\lambda$ and ${}^q\pi_\lambda$ are identical, including multiplicities, for any $0 < q < 1$. We denote this common set by $\Pbf(\lambda)$ and let $\epsilon_1^\lambda, \dots, \epsilon_{n_\lambda}^\lambda$ be an enumeration of the weights counted with multiplicity.

\begin{thm}\label{thm:collecting all q}
    For each $\lambda \in \gf$, there exists a finite-dimensional vector space $V(\lambda)$, on which all the irreducible representations $( {}^q \pi_\lambda )_{0 < q \leq 1}$ are realized, such that the following hold:
    \begin{enumerate}
    \item For all $\mu \in \weights$ and $\alpha \in \boldsymbol{\Delta}^+$,
    \begin{gather*}
        \id_{V(\lambda)} = \lim_{q \rightarrow 1} {}^q \pi_\lambda ({}^q K_\mu), \\
        {}^1 \pi_\lambda (H_\mu) = \lim_{q \rightarrow 1} {}^q \pi_\lambda \bigg( \frac{{}^q K_\mu - {}^q K_\mu^{-1}}{q - q^{-1}} \bigg), \\
        {}^1 \pi_\lambda (c_\alpha^+ E_\alpha) = \lim_{q \rightarrow 1} {}^q \pi_\lambda ({}^q E_\alpha) = - \lim_{q \rightarrow 1} {}^q \pi_\lambda \, {}^q \hat{S}({}^q E_\alpha), \\
        {}^1 \pi_\lambda (c_\alpha^- F_\alpha) = \lim_{q \rightarrow 1} {}^q \pi_\lambda ({}^q F_\alpha) = - \lim_{q \rightarrow 1} {}^q \pi_\lambda \, {}^q \hat{S}({}^q F_\alpha)
    \end{gather*}
    for some $c_\alpha^\pm \in \Cbb$ with $c_\alpha^+ c_\alpha^- = 1$, which do not depend on $\lambda$.
    
    \item There exists a continuous family of inner products $\big( \langle \cdot , \cdot \rangle_q \big)_{0 < q \leq 1}$ on $V(\lambda)$ such that, for each $0 < q \leq 1$, ${}^q \pi_\lambda$ is a $*$-representation of $U_q^\Rbb(\kf)$ on the Hilbert space $\big(V(\lambda), \langle \cdot, \cdot \rangle_q\big)$.
    
    \item There exists an orthonormal basis $\{ e_j^\lambda \mid 1 \leq j \leq n_\lambda \}$ of $\big(V(\lambda), \langle \cdot, \cdot \rangle_1\big)$ such that, for each $1 \leq j \leq n_\lambda$, $e_j^\lambda$ is a weight vector of ${}^q \pi_\lambda$ with weight $\epsilon_j^\lambda$ for any $0 < q \leq 1$.
    \end{enumerate}
\end{thm}

\begin{proof}
    See Appendix~\ref{sec:proof of collecting all q}.
\end{proof}

Fix $\lambda \in \weights^+$. For each $0 < q \leq 1$, we apply the Gram–Schmidt orthonormalization to $\{ e_j^\lambda \mid 1 \leq j \leq n_\lambda \}$ to obtain an orthonormal basis $\{ {}^q e_j^\lambda \mid 1 \leq j \leq n_\lambda \}$ for $\big(V(\lambda), \langle \cdot, \cdot \rangle_q\big)$ such that, for each $1 \leq j \leq n_\lambda$, ${}^q e_j^\lambda$ depends continuously on $0 < q \leq 1$. Also, since two weight vectors of ${}^q \pi_\lambda$ having different weights are orthogonal with respect to $\langle \cdot, \cdot \rangle_q$ by Theorem~\ref{thm:collecting all q}~(2), the Gram–Schmidt process does not alter the weights of the vectors $\{ e_j^\lambda \mid 1 \leq j \leq n_\lambda \}$ by Theorem~\ref{thm:collecting all q}~(3). Thus, ${}^q e_j^\lambda$ remains a weight vector of ${}^q \pi_\lambda$ with weight $\epsilon_j^\lambda$ for each $1 \leq j \leq n_\lambda$.

Note that the elements
\begin{equation}\label{eq:Peter-Weyl basis for all q}
    {}^q u_{ij}^\lambda = \langle {}^q e_i^\lambda , (\,\cdot\,) \, {}^q e_j^\lambda \rangle_q \in \End(V(\lambda))^*, \quad 1 \leq i,j \leq n_\lambda
\end{equation}
depend continuously on $0 < q \leq 1$ and form a unitary corepresentation of $\Cf^\infty(K_q)$. Also, $\{ {}^q u_{ij}^\lambda \mid \lambda \in \weights^+,\, 1 \leq i,j \leq n_\lambda \}$ is a Peter–Weyl basis of $\Cf^\infty(K_q)$. For $q = 1$, we sometimes simply write $u_{ij} ^\lambda := {}^1 u_{ij}^\lambda $.

\begin{prop}\label{prop:the q->1 behavior of I}
    Let $\lambda, \mu \in \weights^+$. Then, for each $1 \leq i, j \leq n_\mu$,
    \begin{equation*}
    \lim_{q \rightarrow 1 }{}^q \pi_\lambda \bigg( \frac{{}^q I ( {}^q u^\mu _{ij} ) - \delta_{ij} }{q^{-1} - q} \bigg) = \begin{cases}
        \pi_\lambda ( H_{\epsilon_i ^\mu} ) & \text{if } i = j, \\
        \frac{(\alpha, \alpha)}{2} (E_\alpha , u^\mu _{ij} ) \pi_\lambda(F_\alpha) & \text{if } \epsilon_i ^\mu - \epsilon_j ^\mu = \alpha \in \boldsymbol{\Delta}^+, \\
        \frac{(\alpha, \alpha)}{2} (F_\alpha , u^\mu _{ij} ) \pi_\lambda(E_\alpha) & \text{if } \epsilon_j ^\mu - \epsilon_i ^\mu = \alpha \in \boldsymbol{\Delta}^+, \\
        0 & \text{otherwise}.
    \end{cases}
    \end{equation*}
\end{prop}

\begin{proof}
    Note that, by \eqref{eq:the universal R-matrix}, \eqref{eq:inverse of q-exponential}, and \eqref{eq:l-functionals}, the expression
    \[
    {}^q l^- ({}^q u^\mu _{ij}) = \big({}^q \Rcal^{-1}, {}^q u^\mu _{ij} \otimes (\,\cdot\,) \big)
    \]
    vanishes when $\epsilon_i ^\mu - \epsilon_j ^\mu \in - \Qbf^+$ with $i \neq j$, is equal to ${}^q K_{- \epsilon_j ^\mu}$ when $i = j$, and is a nonconstant polynomial on $\big\{ (q_\alpha^{-1} - q_\alpha) \, {}^q F_\alpha \mid \alpha \in \boldsymbol{\Delta}^+ \big\}$ multiplied by ${}^q K_{- \epsilon_j ^\mu}$ from the right, whose coefficients converge as $q \rightarrow 1$ by Theorem~\ref{thm:collecting all q}~(1) and the continuity of the family \eqref{eq:Peter-Weyl basis for all q}, otherwise. Note that, for any $\alpha \in \boldsymbol{\Delta}^+$, the term ${}^q \pi_\lambda\big((q_\alpha^{-1} - q_\alpha)\, {}^q F_\alpha\big)$ has order $(q^{-1} - q)$ as $q \rightarrow 1$ by Theorem~\ref{thm:collecting all q}~(1).
    
    Analogously, the expression
    \[
    {}^q \hat{S} \, {}^q l^+ ({}^q u^\mu _{ij}) = {}^q \hat{S} \big({}^q \Rcal, (\,\cdot\,) \otimes {}^q u^\mu _{ij} \big)
    \]
    vanishes when $\epsilon_i ^\mu - \epsilon_j ^\mu \in \Qbf^+$ with $i \neq j$, is equal to ${}^q K_{- \epsilon_i ^\mu}$ when $i = j$, and is a nonconstant polynomial on $\big\{ (q_\alpha^{-1} - q_\alpha)\, {}^q \hat{S}({}^q E_\alpha) \mid \alpha \in \boldsymbol{\Delta}^+ \big\}$ multiplied by ${}^q K_{- \epsilon_i ^\mu}$ from the right, whose coefficients converge as $q \rightarrow 1$ by Theorem~\ref{thm:collecting all q}~(1) and the continuity of the family \eqref{eq:Peter-Weyl basis for all q}, otherwise. For the same reason, the term ${}^q \pi_\lambda\big((q_\alpha^{-1} - q_\alpha)\, {}^q \hat{S}({}^q E_\alpha)\big)$ has order $(q^{-1} - q)$ as $q \rightarrow 1$ for any $\alpha \in \boldsymbol{\Delta}^+$.
    
    Now, let $i = j$ and note that, in the summation
    \[
    {}^q I({}^q u^\mu _{ii}) = \sum_{1 \leq k \leq n_\mu} {}^q l^- ({}^q u^\mu _{ik}) \, {}^q \hat{S}\big({}^q l^+ ({}^q u^\mu _{ki})\big),
    \]
    the summands corresponding to $k \neq i$ must be either zero or contain a factor
    \[
    \cdots \Big((q_\alpha^{-1} - q_\alpha)\, {}^q F_\alpha {}^q K_{- \epsilon_k ^\mu}\Big) \Big((q_\beta^{-1} - q_\beta)\, {}^q \hat{S}({}^q E_\beta) {}^q K_{- \epsilon_k ^\mu}\Big) \cdots
    \]
    for some $\alpha, \beta \in \boldsymbol{\Delta}^+$ with convergent coefficients, all of which vanish when taking $\frac{1}{q^{-1} - q}\lim_{q \rightarrow 1} \pi_\lambda(\,\cdot\,)$ by the two preceding paragraphs. Therefore, again by Theorem~\ref{thm:collecting all q}~(1), we have
    \begin{align*}
    \lim_{q \rightarrow 1} {}^q \pi_\lambda \bigg( \frac{{}^q I({}^q u^\mu _{ii}) - 1}{q^{-1} - q} \bigg) &= \lim_{q \rightarrow 1} {}^q \pi_\lambda \bigg( \frac{{}^q l^-({}^q u^\mu _{ii}) \hat{S}l^+({}^q u^\mu _{ii}) - 1}{q^{-1} - q} \bigg) \\
    &= \lim_{q \rightarrow 1} {}^q \pi_\lambda \bigg( \frac{{}^q K_{-2\epsilon_i^\mu} - 1}{q^{-1} - q} \bigg) = \pi_\lambda(H_{\epsilon_i^\mu}),
    \end{align*}
    proving the first case.

    For the second case, let $\epsilon_i^\mu - \epsilon_j^\mu = \alpha$ for some $\alpha \in \boldsymbol{\Delta}^+$. Then, in the summation
    \[
    {}^q I({}^q u^\mu_{ij}) = \sum_{1 \leq k \leq n_\mu} {}^q l^-({}^q u^\mu_{ik})\, {}^q \hat{S}\big({}^q l^+({}^q u^\mu_{kj})\big),
    \]
    all the nonzero summands except the one corresponding to $k = j$ consist of terms containing at least two factors of the form
    \[
    (q_\beta^{-1} - q_\beta)\, {}^q F_\beta {}^q K_{-\epsilon_k^\mu} \quad \text{or} \quad (q_\beta^{-1} - q_\beta)\, {}^q \hat{S}({}^q E_\beta) {}^q K_{-\epsilon_k^\mu}, \quad \beta \in \boldsymbol{\Delta}^+,
    \]
    with convergent coefficients, since the weight raising and lowering processes producing a nonzero term in other summands cannot succeed in ``one shot." Also, in the summand corresponding to $k = j$, the only term having exactly one such factor is $({}^q E_\alpha, {}^q u^\mu_{ij})(q_\alpha^{-1} - q_\alpha)\, {}^q F_\alpha\, {}^q K_{-\epsilon_j^\mu}\, {}^q \hat{S}({}^q K_{\epsilon_j^\mu})$. Therefore, again by Theorem~\ref{thm:collecting all q}~(1),
    \begin{align*}
    \lim_{q \rightarrow 1} {}^q \pi_\lambda \bigg( \frac{{}^q I({}^q u^\mu_{ij}) - \delta_{ij}}{q^{-1} - q} \bigg) &= \lim_{q \rightarrow 1} {}^q \pi_\lambda \bigg( \frac{{}^q l^-({}^q u^\mu_{ij})\, {}^q \hat{S}\big({}^q l^+({}^q u^\mu_{jj})\big)}{q^{-1} - q} \bigg) \\
    &= \lim_{q \rightarrow 1} {}^q \pi_\lambda \bigg( \frac{({}^q E_\alpha, {}^q u^\mu_{ij})(q_\alpha^{-1} - q_\alpha)\, {}^q F_\alpha {}^q K_{-\epsilon_j^\mu}\, {}^q K_{-\epsilon_j^\mu}}{q^{-1} - q} \bigg) \\
    &= (c_\alpha^+ E_\alpha, u^\mu_{ij}) \frac{(\alpha, \alpha)}{2} c_\alpha^- \pi_\lambda(F_\alpha) = (E_\alpha, u^\mu_{ij}) \frac{(\alpha, \alpha)}{2} \pi_\lambda(F_\alpha).
    \end{align*}

    A similar reasoning as in the second case proves the third case.

    Finally, if $i \neq j$ and $\epsilon_i^\mu - \epsilon_j^\mu$ is not contained in $\boldsymbol{\Delta}$, then either $\epsilon_i^\mu = \epsilon_j^\mu$ or $\epsilon_i^\mu - \epsilon_j^\mu$ is ``big." In both cases, every nonzero term in the summation
    \[
    {}^q I({}^q u^\mu_{ij}) = \sum_{1 \leq k \leq n_\mu} {}^q l^-({}^q u^\mu_{ik})\, {}^q \hat{S}\big({}^q l^+({}^q u^\mu_{kj})\big)
    \]
    contains at least two factors of the form
    \[
    (q_\alpha^{-1} - q_\alpha)\, {}^q F_\alpha {}^q K_{-\epsilon_k^\mu} \quad \text{or} \quad (q_\alpha^{-1} - q_\alpha)\, {}^q \hat{S}({}^q E_\alpha) {}^q K_{-\epsilon_k^\mu}, \quad \alpha \in \boldsymbol{\Delta}^+,
    \]
    with convergent coefficients, all of which vanish under $\frac{1}{q^{-1} - q}\lim_{q \rightarrow 1} {}^q \pi_\lambda(\,\cdot\,)$.
\end{proof}

Because of Theorem~\ref{thm:collecting all q}, the family of algebras $\big(U_q(\gf)\big)_{0 < q \leq 1}$ can be embedded into the fixed algebra $\prod_{\lambda \in \weights^+} \End(V(\lambda))$. Also, note that the vector spaces in the family $\big( \Cf^\infty(K_q) \big)_{0 < q \leq 1}$ can all be identified with the fixed space $\bigoplus_{\lambda \in \weights^+} \End(V(\lambda))^*$, by definition for $0 < q < 1$ and by \eqref{eq:classical matrix coefficients identification-semisimple case} for $q = 1$. Under these identifications, the natural pairing
\begin{equation*}
\Big(\prod_{\lambda \in \weights^+} \End(V(\lambda))\Big) \times \Big(\bigoplus_{\lambda \in \weights^+} \End(V(\lambda))^*\Big) \ni (x, f) \longmapsto \sum_{\lambda \in \weights^+} f_\lambda(x_\lambda) \in \Cbb
\end{equation*}
comprises all the skew-pairings in \eqref{eq:skew-pairing between Uq(K) and Cinfty(Kq)} and Proposition~\ref{prop:classical skew-pairing}. Therefore, the comultiplications of $\Cf^\infty(K_q)$ are all equal regardless of $0 < q \leq 1$, which we denote by
\begin{equation*}
\Delta : \bigoplus_{\lambda \in \weights^+} \End(V(\lambda))^* \rightarrow \Big( \bigoplus_{\lambda \in \weights^+} \End(V(\lambda))^* \Big) \otimes \Big( \bigoplus_{\lambda \in \weights^+} \End(V(\lambda))^* \Big).
\end{equation*}
I.e., the family of coalgebras $\big(\Cf^\infty(K_q)\big)_{0 < q \leq 1}$ are all isomorphic to the coalgebra $\bigoplus_{\lambda \in \weights^+} \End(V(\lambda))^*$. Note that this gives rise to left and right $\prod_{\lambda \in \weights^+} \End(V(\lambda))$-module multiplications on $\bigoplus_{\lambda \in \weights^+} \End(V(\lambda))^*$ via
\begin{equation*}
x \triangleright f = (\id \otimes x) \Delta(f), \quad f \triangleleft x = (x \otimes \id) \Delta(f)
\end{equation*}
for $x \in \prod_{\lambda \in \weights^+} \End(V(\lambda))$ and $f \in \bigoplus_{\lambda \in \weights^+} \End(V(\lambda))^*$, respectively.

Fix $\lambda \in \weights^+$. Let $\{ e_{ij}^\lambda \mid 1 \leq i,j \leq n_\lambda \}$ be the matrix units of $\End(V(\lambda))$ with respect to the basis $\{ e_j^\lambda \mid 1 \leq j \leq n_\lambda \}$ in Theorem~\ref{thm:collecting all q}~(3), i.e.,
\begin{equation*}
e_{ij}^\lambda e_k^\lambda = \delta_{jk} e_i^\lambda
\end{equation*}
for all $1 \leq i,j,k \leq n_\lambda$. Then, $\{ u_{ij}^\lambda = {}^1 u_{ij}^\lambda \mid 1 \leq i,j \leq n_\lambda \} \subseteq \End(V(\lambda))^*$, constructed in \eqref{eq:Peter-Weyl basis for all q}, becomes its dual basis and satisfies
\begin{equation*}
\Delta(u_{ij}^\lambda) = \sum_{1 \leq k \leq n_\lambda} u_{ik}^\lambda \otimes u_{kj}^\lambda, \quad 1 \leq i,j \leq n_\lambda.
\end{equation*}

\subsection{The \texorpdfstring{$q\rightarrow 1$}{TEXT} limit of FODCs}\label{subsec:the q->1 limit of FODCs}

Let $\mu_1, \cdots, \mu_m \in \Pbf^+$ be pairwise distinct elements, and let $\pi_{\boldsymbol{\mu}} := \pi_{\mu_1} \oplus \cdots \oplus \pi_{\mu_m} : \gf \rightarrow \End(V(\mu_1)) \oplus \cdots \oplus \End(V(\mu_m))$.

For each $0 < q < 1$, identify
\[
\prescript{q}{\operatorname{inv}} \Omega_{\boldsymbol{\mu}} \cong \End(V(\mu_1)) \oplus \cdots \oplus \End(V(\mu_m))
\]
via the linear map sending ${}^q \omega_{ij}^{\mu_l}$ to $e_{ij}^{\mu_l}$ for $1 \leq l \leq m$ and $1 \leq i,j \leq n_{\mu_l}$.

Recall that the FODC $({}^q \Omega_{\boldsymbol{\mu}}, {}^q d_{\boldsymbol{\mu}})$ is isomorphic to $({}^q \Omega_{\boldsymbol{\mu}}, \frac{1}{q^{-1} - q} {}^q d_{\boldsymbol{\mu}})$, see Remark~\ref{rmk:constant multiple of FODC}.

\begin{thm}\label{thm:the q->1 behavior of FODCs}
    Let $f \in \bigoplus_{\lambda \in \weights^+} \End(V(\lambda))^*$. Then, in the vector space
    \[
    \Big(\bigoplus_{\lambda \in \weights^+} \End(V(\lambda))^* \Big) \otimes \Big( \End(V(\mu_1)) \oplus \cdots \oplus \End(V(\mu_m)) \Big),
    \]
    which is isomorphic to ${}^q \Omega_{\boldsymbol{\mu}}$ for all $0<q<1$ as vector spaces, we have
    \begin{equation}\label{eq:q->1 limit of differential}
    \lim_{q \rightarrow 1} \frac{{}^q d_{\boldsymbol{\mu}} f}{q^{-1} - q} = (\id \otimes \pi_{\boldsymbol{\mu}} ) df,
    \end{equation}
    where $d: \bigoplus_{\lambda \in \weights^+} \End(V(\lambda))^* \rightarrow \bigoplus_{\lambda \in \weights^+} \End(V(\lambda))^* \otimes \gf$ is the classical differential, and $\gf^*$ has been identified with $\gf$ via the Killing form.
\end{thm}

\begin{proof}
    Let $0<q<1$. By \eqref{eq:differential of Gamma^mu} and \eqref{eq:direct sum of differentials},
    \[
    {}^q d_{\boldsymbol{\mu}} f = \sum_{1\leq l \leq m} \sum_{1 \leq i,j \leq n_{\mu_l}} \Big( \big( {}^q I( {}^q u^{\mu_l}_{ij} ) - \delta_{ij} \big) \triangleright f \Big) \otimes e_{ij}^{\mu_l}
    \]
    for all $f \in \Cf^\infty(K_q) = \bigoplus_{\lambda \in \weights^+} \End(V(\lambda))^*$. Thus, by Proposition~\ref{prop:the q->1 behavior of I}, we have, for all $\lambda \in \weights^+$ and $1 \leq r,s \leq n_\lambda$,
    \begin{align*}
    \lim_{q \rightarrow 1} \frac{{}^q d_{\boldsymbol{\mu}} u^\lambda_{rs}}{q^{-1} - q} &= \lim_{q \rightarrow 1} \sum_{1\leq l \leq m} \sum_{1 \leq i,j \leq n_{\mu_l}} \bigg( \frac{{}^q I( {}^q u^{\mu_l}_{ij} ) - \delta_{ij}}{q^{-1} - q} \triangleright u^\lambda_{rs} \bigg) \otimes e_{ij}^{\mu_l} \\
    &= \begin{multlined}[t] \sum_{1\leq l \leq m} \bigg( \sum_{1\leq i \leq n_{\mu_l}} (H_{\epsilon_i^{\mu_l}} \triangleright u^\lambda_{rs}) \otimes e_{ii}^{\mu_l} + \sum_{\alpha \in \boldsymbol{\Delta}^+} \frac{(\alpha,\alpha)}{2} \\
    \sum_{\substack{1 \leq i,j \leq n_{\mu_l} \\ \epsilon_i^{\mu_l} - \epsilon_j^{\mu_l} = \alpha}} \Big( (F_\alpha \triangleright u^\lambda_{rs}) \otimes (E_\alpha, u^{\mu_l}_{ij}) e_{ij}^{\mu_l} + (E_\alpha \triangleright u^\lambda_{rs}) \otimes (F_\alpha, u^{\mu_l}_{ji}) e_{ji}^{\mu_l} \Big) \bigg) \end{multlined} \\
    &= \begin{multlined}[t] \sum_{1\leq l \leq m} \bigg( \sum_{1\leq i \leq n_{\mu_l}} (H_{\epsilon_i^{\mu_l}} \triangleright u^\lambda_{rs}) \otimes e_{ii}^{\mu_l} \\
    + \sum_{\alpha \in \boldsymbol{\Delta}^+} \frac{(\alpha,\alpha)}{2} \Big( (F_\alpha \triangleright u^\lambda_{rs}) \otimes \pi_{\mu_l}(E_\alpha) + (E_\alpha \triangleright u^\lambda_{rs}) \otimes \pi_{\mu_l}(F_\alpha) \Big) \bigg). \end{multlined}
    \end{align*}
    However, since $(\eta, \nu) = \sum_{1\leq j\leq N} (\alpha_j^\lor, \nu)(\varpi_j, \eta)$ for any $\nu, \eta \in \Pbf$, we have
    \begin{align*}
    \sum_{1 \leq i \leq n_{\mu_l}} (H_{\epsilon_i^{\mu_l}} \triangleright u^\lambda_{rs}) \otimes e_{ii}^{\mu_l} &= \sum_{1\leq i \leq n_{\mu_l}} (\epsilon_i^{\mu_l}, \epsilon_s^\lambda) u^\lambda_{rs} \otimes e_{ii}^{\mu_l} \\
    &= \sum_{1\leq i \leq n_{\mu_l}} \sum_{1\leq j\leq N} (\alpha_j^\lor, \epsilon_s^\lambda) u^\lambda_{rs} \otimes (\varpi_j, \epsilon_i^{\mu_l}) e_{ii}^{\mu_l} \\
    &= \sum_{1\leq j\leq N} (H_j' \triangleright u^\lambda_{rs}) \otimes \pi_{\mu_l}(H_{\varpi_j})
    \end{align*}
    for all $1\leq l \leq m$. Hence, for all $\lambda \in \weights^+$ and $1\leq r,s \leq n_\lambda$,
    \begin{align*}
    \lim_{q \rightarrow 1} \frac{{}^q d_{\boldsymbol{\mu}} u^\lambda_{rs}}{q^{-1} - q} 
    &= \begin{multlined}[t] \sum_{1\leq l\leq m} \bigg( \sum_{1\leq j\leq N} (H_j' \triangleright u^\lambda_{rs}) \otimes \pi_{\mu_l}(H_{\varpi_j}) \\
    + \sum_{\alpha \in \boldsymbol{\Delta}^+} \frac{(\alpha,\alpha)}{2} \Big( (F_\alpha \triangleright u^\lambda_{rs}) \otimes \pi_{\mu_l}(E_\alpha) + (E_\alpha \triangleright u^\lambda_{rs}) \otimes \pi_{\mu_l}(F_\alpha) \Big) \bigg) \end{multlined} \\
    &= \sum_{1\leq l\leq m} (\id \otimes \pi_{\mu_l}) d u^\lambda_{rs} = (\id \otimes \pi_{\boldsymbol{\mu}}) d u^\lambda_{rs}
    \end{align*}
    by Proposition~\ref{prop:the classical differential in terms of H E F}. As $\{ u^\lambda_{ij} \mid \lambda \in \weights^+,\, 1\leq i,j\leq n_\lambda \}$ forms a linear basis of $\bigoplus_{\lambda \in \weights^+} \End(V(\lambda))^*$, this proves the claim.
\end{proof}

\begin{rmk}\label{rmk:q->1 limit of FODCs}
In addition to the hypothesis of Theorem~\ref{thm:the q->1 behavior of FODCs}, assume further that $\pi_{\boldsymbol{\mu}}$ is faithful, which is always the case if $\gf$ is simple and $\boldsymbol{\mu} \neq 0$. Then, the right-hand side of \eqref{eq:q->1 limit of differential} becomes a matrix realization of the classical differential.

In this matrix realization, the right coaction formula \eqref{eq:right coaction on Omega(K)-invariant part-matrix group} for the classical FODC $(\Omega_K, d)$ becomes
\[
e_{jl}^{\mu_p} \longmapsto \sum_{1 \leq i,k \leq n_{\mu_p}} e_{ik}^{\mu_p} \otimes \big( {}^1 u^{\mu_p}_{ij} \, {}^1 S({}^1 u^{\mu_p}_{lk}) \big)
\]
on the matrix units $\{e_{jl}^{\mu_p} \mid 1 \leq p \leq m, \, 1 \leq j,l \leq n_{\mu_p} \}$. However, because of \eqref{eq:definition of R_ij}, \eqref{eq:structure representation-Kq}, and the identification ${}^q \omega_{jl}^{\mu_p} \cong e_{jl}^{\mu_p}$, the right coaction for the FODC $({}^q \Omega_{\boldsymbol{\mu}}, {}^q d_{\boldsymbol{\mu}})$ is given by
\[
e_{jl}^{\mu_p} \longmapsto \sum_{1 \leq i,k \leq n_{\mu_p}} e_{ik}^{\mu_p} \otimes \big( {}^q u^{\mu_p}_{ij} \, {}^q S({}^q u^{\mu_p}_{lk}) \big)
\]
for each $0 < q < 1$. Hence, we see that the right coaction on ${}^q \Omega_{\boldsymbol{\mu}}$ also converges to the classical right coaction in a precise sense. On the other hand, the left coactions for all these FODCs are equal by definition.

Moreover, as $q \rightarrow 1$, the elements of the first family
\[
\big( {}^q l^-({}^q u^\mu_{ji}) \, {}^q \hat{S} \, {}^q l^+({}^q u^\mu_{kl}) \big)_{1 \leq i,j,k,l \leq n_\mu} \subseteq \Cf^\infty(K_q)^\circ \subseteq \prod_{\lambda \in \weights^+} \End(V(\lambda))
\]
in the structure representations of ${}^q \Omega_{\mu}$ (Definition~\ref{defn:definition of Omega zeta mu}) converge to
\[
(\delta_{ji} \delta_{kl})_{1 \leq i,j,k,l \leq n_\mu} \subseteq \prod_{\lambda \in \weights^+} \End(V(\lambda))
\]
in each component $\End(V(\lambda))$. Thus, as $q \rightarrow 1$, the right multiplication by an element $f \in \bigoplus_{\lambda \in \weights^+} \End(V(\lambda))^*$ on $\inv {}^q \Omega_{\boldsymbol{\mu}}$
\[
e_{ik}^{\mu_p} f = \sum_{1 \leq j,l \leq n_{\mu_p}} \big( {}^q l^-({}^q u^{\mu_p}_{ji}) \, {}^q \hat{S} \, {}^q l^+({}^q u^{\mu_p}_{kl}) \triangleright f \big) e_{jl}^{\mu_p}
\]
converges to
\[
e_{ik}^{\mu_p} f = f e_{ik}^{\mu_p},
\]
which gives the right multiplication by $f$ on $\inv \Omega_K \cong \gf^* \cong \gf$ in the matrix realization $\pi_{\boldsymbol{\mu}}$. On the other hand, the left multiplications by an element $f \in \bigoplus_{\lambda \in \weights^+} \End(V(\lambda))^*$ on $\inv {}^q \Omega_{\boldsymbol{\zeta\mu}}$ are all equal for any $0 < q \leq 1$ by definition.

Therefore, as long as $\pi_{\boldsymbol{\mu}}$ is a matrix realization of the Lie algebra $\gf$, the bicovariant FODC $(\Omega_{\boldsymbol{\mu}}, d_{\boldsymbol{\mu}}) \cong \big(\Omega_{\boldsymbol{\mu}}, \frac{1}{q^{-1} - q} d_{\boldsymbol{\mu}}\big)$ in its entirety converges to the classical FODC $(\Omega_K, d)$ in the matrix realization $\pi_{\boldsymbol{\mu}}$ as $q \rightarrow 1$.

However, whereas the $\Cf^\infty(K)$-dimension of $\Omega_K$ is always equal to $d = \dim \gf$ independent of the matrix realization $\pi_{\boldsymbol{\mu}}$ of $\gf$, the $\Cf^\infty(K_q)$-dimension of ${}^q \Omega_{\boldsymbol{\mu}}$ for any $0 < q < 1$ is given by
\[
n_{\mu_1}^2 + \cdots + n_{\mu_m}^2.
\]
It may be phrased as follows: \textit{The $q$-deformation makes the classical FODC ``fill up" each irreducible block of the matrix realization $\pi_{\boldsymbol{\mu}}$.}
\end{rmk}

\section*{Funding}
This work began while H.~Lee was affiliated with Seoul National University, during which he was supported by the Basic Science Research Program through the National Research Foundation of Korea (NRF) under Grant No.~NRF-2022R1A2C1092320, the National Research Foundation of Korea (NRF) grant funded by the Ministry of Science and ICT (MSIT) (No.~2020R1C1C1A01009681), and the Samsung Science and Technology Foundation under Project Number SSTF-BA2002-01.

This work was completed while H.~Lee was affiliated with the Institute for Advanced Study in Mathematics of Harbin Institute of Technology, during which he was supported by the National Natural Science Foundation of China, Grant No.~12371138.

This work underwent significant improvement while the author was visiting the Institut Henri Poincaré. The author acknowledges support of the Institut Henri Poincaré (UAR 839 CNRS-Sorbonne Université), and LabEx CARMIN (ANR-10-LABX-59-01).

\section*{Acknowledgements}
First and foremost, I wish to express my deepest gratitude to my God, Jesus Christ, for His endless grace and unceasing love shown to me throughout my entire life.

I am deeply thankful to Christian Voigt, who devoted a great deal of time to discussing the early ideas of this work with me, helping to improve those that were originally at a rudimentary stage, and even introducing me to researchers who had worked on related topics. I would also like to thank Ulrich Krähmer for offering many valuable suggestions regarding this work.

I am grateful as well to many of the participants of the trimester Representation Theory and Noncommutative Geometry, held at the Institut Henri Poincaré from 6 January 2025 to 4 April 2025. Their interest and the questions they raised during the presentation of this work have been very helpful in revising it.
\appendix

\section{Proof of Theorem~\ref{thm:collecting all q}}\label{sec:proof of collecting all q}

\renewcommand{\theequation}{\thesection.\arabic{equation}}
\renewcommand{\thethm}{\thesection.\arabic{thm}}

\subsection*{Preliminaries}

To prove the theorem, we need to consider the quantized universal enveloping algebra of $\gf$ over the field of rational polynomials $\Qbb(s)$ and an integral form inside it. All the necessary materials can be found in \cite[Chapter~3]{VoigtYuncken}.

Let $1 \leq L \in \Nbb$ be such that $L\frac{(\varpi_i, \varpi_j)}{2} \in \Zbb$ for all $1 \leq i,j \leq N$. Let $\qbf = s^L \in \Qbb(s)$ and define, for each $1 \leq j \leq N$, $\qbf_j = \qbf^{\frac{(\alpha,\alpha)}{2}}$, which are well-defined elements of $\Qbb(s)$ by our choice of $L$ and $\qbf$. Let $\Ubf_\qbf(\gf)$ be the quantized universal enveloping algebra of $\gf$ over the field $\Qbb(s)$, whose generators will be denoted by $\Kbf_\lambda$, $\Ebf_j$, and $\Fbf_j$ ($\lambda \in \weights$, $1 \leq j \leq N$).

Then, except for those that involve the $*$-structure of $U_q^\Rbb(\kf)$, all the statements of Section~\ref{sec:The q-deformation} also hold for $\Ubf_\qbf(\gf)$. In particular, the irreducible finite-dimensional integrable representations are classified by $\Pbf^+$ via the correspondence that associates to each such representation its highest weight. We denote the irreducible representation corresponding to $\lambda \in \weights^+$ by $\boldsymbol{\pi}_\lambda : \Ubf_\qbf(\gf) \rightarrow \End_{\Qbb(s)}(\Vbf(\lambda))$. Also, we denote the root vectors of $\Ubf_\qbf(\gf)$ by $\Ebf_\alpha$ and $\Fbf_\alpha$ for $\alpha \in \boldsymbol{\Delta}^+$.

Let $\Acal = \Zbb[s, s^{-1}] \subseteq \Qbb(s)$ and define $\Ubf_\qbf^\Acal(\gf)$ as the $\Acal$-subalgebra of $\Ubf_\qbf(\gf)$ generated by
\[
\Kbf_\lambda, \quad \frac{\Kbf_j - \Kbf_j^{-1}}{\qbf_j - \qbf_j^{-1}}, \quad \frac{1}{[r]_{\qbf_j}!} \Ebf_j^r, \quad \frac{1}{[r]_{\qbf_j}!} \Fbf_j^r
\]
for $\lambda \in \weights$, $1 \leq j \leq N$, and $r \in \Nbb$.

Fix $\lambda \in \weights^+$. Let $v_\lambda \in \Vbf(\lambda)$ be a highest weight vector. Define
\[
\Vbf(\lambda)_\Acal = \Ubf_\qbf^\Acal(\gf) v_\lambda \subseteq \Vbf(\lambda).
\]
Then, $\Vbf(\lambda)_\Acal$ is a free $\Acal$-module and $\Qbb(s) \otimes_\Acal \Vbf(\lambda)_\Acal = \Vbf(\lambda)$. Fix $0 < q < 1$ and let $\Cbb_q$ be the space $\Cbb$ equipped with the $\Acal$-module structure provided by the ring homomorphism $\ev_q : \Acal \rightarrow \Cbb$ sending $s$ to $q^{\frac{1}{L}}$. Consider the $\Cbb$-vector space
\begin{equation*}
{}^q V(\lambda) = \Cbb_q \otimes_\Acal \Vbf(\lambda)_\Acal
\end{equation*}
which is also a left $\Acal$-module. Then, the following map is $\Acal$-linear:
\begin{equation*}
    \Ubf_\qbf^\Acal(\gf) \ni \Xbf \longmapsto \id_{\Cbb_q} \otimes \boldsymbol{\pi}_\lambda (\Xbf) \in \End({}^q V(\lambda))
\end{equation*}
The defining relations for $\Ubf_\qbf(\gf)$ satisfied by the operators
\[
\boldsymbol{\pi}_\lambda (\Kbf_\mu), \,\, \boldsymbol{\pi}_\lambda (\Ebf_j), \,\, \boldsymbol{\pi}_\lambda (\Fbf_j), \quad \mu \in \weights, \, 1 \leq j \leq N
\]
become the defining relations for $U_q(\gf)$ satisfied by
\[
\id_{\Cbb_q} \otimes \boldsymbol{\pi}_\lambda (\Kbf_\mu), \,\, \id_{\Cbb_q} \otimes \boldsymbol{\pi}_\lambda (\Ebf_j), \,\, \id_{\Cbb_q} \otimes \boldsymbol{\pi}_\lambda (\Fbf_j), \quad \mu \in \weights, \, 1 \leq j \leq N.
\]
Hence, there is a representation ${}^q \pi_\lambda : U_q(\gf) \rightarrow \End({}^q V(\lambda))$ given by
\begin{equation*}
{}^q \pi_\lambda({}^q K_\mu) = \id \otimes \boldsymbol{\pi}_\lambda (\Kbf_\lambda), \, {}^q \pi_\lambda({}^q E_j) = \id \otimes \boldsymbol{\pi}_\lambda (\Ebf_j), \, {}^q \pi_\lambda({}^q F_j) = \id \otimes \boldsymbol{\pi}_\lambda (\Fbf_j)
\end{equation*}
for $\mu \in \weights$ and $1 \leq j \leq N$. The discussion between \cite[Theorems~3.145--3.146]{VoigtYuncken} shows that ${}^q \pi_\lambda$ is in fact the irreducible representation of $U_q(\gf)$ corresponding to $\lambda \in \weights^+$.

According to \cite[Section~5.1.1]{VoigtYuncken}, the sets of weights counted with multiplicity for the representations $\boldsymbol{\pi}_\lambda$ and ${}^q \pi_\lambda$ for any $0 < q < 1$ are all equal. We denote this set by $\Pbf(\lambda)$ and let $\epsilon_1^\lambda, \cdots , \epsilon_{n_\lambda}^\lambda$ be an enumeration of the weights counted with multiplicity.

Fix an $\Acal$-linear basis $\{ \gbf_j \mid 1 \leq j \leq n_\lambda \} \subseteq \Vbf(\lambda)_\Acal$ given in \cite[Theorems~3.145--3.146]{VoigtYuncken}, called the \textit{global basis}. Then,
\[
\{ 1_{\Cbb_q} \otimes \gbf_j \mid 1 \leq j \leq n_\lambda \}
\]
is a $\Cbb$-basis of ${}^q V(\lambda)$ for each $0 < q < 1$.

On the other hand, by definition, $\Vbf(\lambda)_\Acal$ contains an $\Acal$-basis consisting of weight vectors for $\boldsymbol{\pi}_\lambda$. Let $\{\vbf_j \mid 1 \leq j \leq n_\lambda \} \subseteq \Vbf(\lambda)_\Acal$ be an $\Acal$-basis such that for each $1 \leq j \leq n_\lambda$, $\vbf_j$ is a weight vector for $\boldsymbol{\pi}_\lambda$ with weight $\epsilon_j^\lambda$. Let $T \in M_{n_\lambda}(\Acal)$ be an invertible matrix defined by
\[
\sum_{1 \leq i \leq n_\lambda} T_{ij} \gbf_i = \vbf_j \quad \text{for} \quad 1 \leq j \leq n_\lambda.
\]
Note that for each $0 < q < 1$, the matrix $\ev_q(T) \in M_{n_\lambda}(\Cbb)$ is also invertible, and
\[
\sum_{1 \leq i \leq n_\lambda} \ev_q(T_{ij}) (1_{\Cbb_q} \otimes \gbf_i) = 1_{\Cbb_q} \otimes \sum_{1 \leq i \leq n_\lambda} T_{ij} \gbf_i = 1_{\Cbb_q} \otimes \vbf_j \quad \text{for} \quad 1 \leq j \leq n_\lambda.
\]
Hence,
\[
\{ 1_{\Cbb_q} \otimes \vbf_j \mid 1 \leq j \leq n_\lambda \}
\]
is a $\Cbb$-basis for ${}^q V(\lambda)$ such that $1_{\Cbb_q} \otimes \vbf_j$ is a weight vector of ${}^q \pi_\lambda$ with weight $\epsilon_j^\lambda$ for each $1 \leq j \leq n_\lambda$.

\subsection*{Proof of (1)}

Choose an $n_\lambda$-dimensional $\Cbb$-vector space $V(\lambda)$ with a fixed basis $\{v_j \mid 1 \leq j \leq n_\lambda\}$ and identify the representation space ${}^q V(\lambda)$ for any $0 < q < 1$ with this one via the isomorphism that maps $1_{\Cbb_q} \otimes \vbf_j$ to $v_j$ for $1 \leq j \leq n_\lambda$. Hence, all the representations $\{{}^q \pi_\lambda \mid 0 < q < 1 \}$ are defined on $V(\lambda)$ and, for each $1 \leq j \leq n_\lambda$, $v_j$ is a weight vector of ${}^q \pi_\lambda$ with weight $\epsilon_j^\lambda$ for all $0 < q < 1$.

\begin{lem}\label{lem:q->1 limits of defining relations}
    Let $\lambda \in \weights^+$. Then, $\lim_{q \rightarrow 1} {}^q \pi_\lambda ({}^q K_\mu ) = \id_{V(\lambda)}$ for all $\mu \in \weights$, and also the following operators converge in $\End(V(\lambda))$ as $q \rightarrow 1$:
    \begin{equation*}
        \bigg\{ {}^q \pi_\lambda \Big( \frac{{}^q K_\mu - {}^q K_\mu^{-1}}{q - q^{-1}} \Big), \;\; {}^q \pi_\lambda ({}^q E_\alpha), \;\; {}^q \pi_\lambda ({}^q F_\alpha) \; \bigg| \; \mu \in \weights , \, \alpha \in \boldsymbol{\Delta}^+ \bigg\}
    \end{equation*}
\end{lem}
\begin{proof}
Let $\mu \in \weights$. By our choice of $\{ v_j \mid 1 \leq j \leq n_\lambda \}$, ${}^q \pi_\lambda ({}^q K_\mu )$ is given in this basis by
\[
\diag \Big(q^{(\mu, \epsilon_1^\lambda)} , \cdots , q^{(\mu, \epsilon_{n_\lambda}^\lambda )} \Big)
\]
where $\diag(a_1 , \cdots , a_m )$ denotes a diagonal matrix with entries $a_1, \cdots , a_m$ on the diagonal. Hence, we have
\[
\lim_{q \rightarrow 1} {}^q \pi_\lambda ({}^q K_\mu) = \id_{V(\lambda)}.
\]

Likewise, since the operator ${}^q \pi_\lambda \Big( \frac{{}^q K_\mu - {}^q K_\mu^{-1}}{q - q^{-1}} \Big)$ is represented by
\begin{equation*}
\diag \bigg( \frac{ q^{(\mu , \epsilon_1^\lambda)} - q^{- (\mu, \epsilon_1^\lambda)}}{q - q^{-1}} \,,\,\cdots\,,\, \frac{ q^{(\mu , \epsilon_{n_\lambda}^\lambda)} - q^{- (\mu, \epsilon_{n_\lambda}^\lambda)}}{q - q^{-1}} \bigg)
\end{equation*}
in the basis $\{ v_j \mid 1 \leq j \leq n_\lambda \}$, it converges to
\begin{equation}\label{eq:K - K in global basis}
    \diag\big( (\mu, \epsilon_1^\lambda ) , \cdots , (\mu, \epsilon_{n_\lambda}^\lambda ) \big).
\end{equation}

For each $0 < q < 1$, the $\Cbb$-valued matrix entries of the operators
\[
{}^q \pi_\lambda ({}^q E_\alpha), \, {}^q \pi_\lambda ({}^q F_\alpha), \quad \alpha \in \boldsymbol{\Delta}^+
\]
in the basis $\{ v_j \mid 1 \leq j \leq n_\lambda \}$ are by definition given by the evaluations at $q$ of the corresponding $\Acal$-valued matrix entries of
\[
\boldsymbol{\pi}_\lambda (\Ebf_\alpha), \,\boldsymbol{\pi}_\lambda (\Fbf_\alpha), \quad \alpha \in \boldsymbol{\Delta}^+
\]
in the basis $\{ \vbf_j \mid 1 \leq j \leq n_\lambda \}$.

Let $\alpha \in \boldsymbol{\Delta}^+$. Note that $\qbf_\alpha = \qbf^{\frac{(\alpha,\alpha)}{2}} = s^{L \frac{(\alpha,\alpha)}{2}} \in \Qbb(s)$. The relations \eqref{eq:relations for the root vectors} imply that the subalgebra of $\Ubf_\qbf (\gf)$ generated by the three elements $\Kbf_\alpha$, $\Ebf_\alpha$, and $\Fbf_\alpha$ is isomorphic to the subalgebra of $\Ubf_{\qbf_\alpha} (\mathfrak{sl}_2)$ generated by $\Kbf^{\pm 2}$, $\Ebf$, and $\Fbf$. Thus, the restriction of $\pibf_\lambda$ to this subalgebra decomposes into irreducible integrable representations of it, which, according to \cite[Section~3.6.1]{VoigtYuncken}, are of the form given in \cite[Lemma~3.38]{VoigtYuncken}. Since the weight vectors $\{\vbf_j \mid 1 \leq j \leq n_\lambda \}$ are also weight vectors for these irreducible integrable representations, we see that the matrix entries of the operators $\boldsymbol{\pi}_\lambda (\Ebf_\alpha)$ and $\boldsymbol{\pi}_\lambda (\Fbf_\alpha)$ in the basis $\{ \vbf_j \mid 1 \leq j \leq n_\lambda \}$ are given by $\Acal$-linear combinations of monomials of the form $[k_1]_{\qbf_\alpha} \cdots [k_r]_{\qbf_\alpha}$ with $k_1, \cdots, k_r \in \Nbb$. However, the evaluations of these quantities at $0 < q < 1$ all converge as $q \rightarrow 1$.
\end{proof}

For the next proposition, we need to introduce the operators
\begin{equation}\label{eq:c alpha pm}
\tilde{s}_i^{\ad} = \big[ \exp (\ad E_{i}) \exp( - \ad F_{i} ) \exp(\ad E_{i}) \big]^{-1} \in \End(\gf), \quad 1 \leq i \leq N.
\end{equation}
These operators are Lie algebra automorphisms, and by \cite[Section~21.2~(6)]{Humphreys}, $\tilde{s}_i^{\ad}$ maps a root vector with root $\beta \in \boldsymbol{\Delta}$ to a root vector with root $s_i^{-1} \beta = s_i \beta$. Thus, if $\alpha = s_{i_1} \cdots s_{i_{r-1}} \alpha_{i_r}$ for some $ 1 \leq r \leq t$ (cf. \eqref{eq:the longest Weyl group element}), then
\[
\tilde{s}_{i_1}^{\ad} \cdots \tilde{s}_{i_{r-1}}^{\ad} E_{i_r} = c_\alpha^+ E_{\alpha}, \quad \tilde{s}_{i_1}^{\ad} \cdots \tilde{s}_{i_{r-1}}^{\ad} F_{i_r} = c_\alpha^- F_{\alpha}
\]
for some nonzero constants $c_\alpha^\pm \in \Cbb$.

\begin{prop}\label{prop:Lie algebra homomorphism as q->1 limit}
    For each $\lambda \in \weights^+$, there exists an irreducible complex Lie algebra homomorphism $\pi_\lambda = {}^1 \pi_\lambda : \gf \rightarrow \End(V(\lambda))$ that has $\lambda$ as its highest weight and satisfies
    \begin{gather*}
    \pi_\lambda (H_\mu ) = \lim_{q \rightarrow 1} {}^q \pi_\lambda \bigg( \frac{{}^q K_\mu - {}^q K_\mu^{-1}}{q - q^{-1}} \bigg) \\
    c_\alpha^+ \pi_\lambda (E_\alpha) = \lim_{q \rightarrow 1} {}^q \pi_\lambda ({}^q E_\alpha), \quad c_\alpha^- \pi_\lambda (F_\alpha) = \lim_{q \rightarrow 1} {}^q \pi_\lambda ({}^q F_\alpha)
    \end{gather*}
    for all $\mu \in \weights$ and $\alpha \in \boldsymbol{\Delta}^+$. Also, we have $c_\alpha^+ c_\alpha^- = 1$ for all $\alpha \in \boldsymbol{\Delta}^+$ and $c_{\alpha_j}^\pm = 1$ for all $ 1 \leq j \leq N$.
\end{prop}
\begin{proof}
    We prove that the operators
    \begin{equation*}
        \bigg\{ \lim_{q \rightarrow 1} {}^q \pi_\lambda \Big( \frac{{}^q K_j - {}^q K_j ^{-1}}{q_j - q_j ^{-1}} \Big), \, \lim_{q \rightarrow 1} {}^q \pi_\lambda ({}^q E_j), \, \lim_{q \rightarrow 1} {}^q \pi_\lambda ({}^q F_j) \, \bigg| \, 1 \leq j \leq N \bigg\}
    \end{equation*}
    in $\End(V(\lambda))$, which are well-defined by Lemma~\ref{lem:q->1 limits of defining relations}, satisfy the relations (1)--(4) of Theorem~\ref{thm:Lie algebra presentation}.
    
    First, since $\{K_\mu \mid \mu \in \weights \}$ are mutually commuting, (1) is satisfied.
    
    By U2 of Definition~\ref{defn:the quantized universal enveloping algebra}, we have
    \begin{align*}
    &{}^q \pi_\lambda \bigg( \frac{ ({}^q K_i - {}^q K_i ^{-1} )\, {}^qE_j - {}^qE_j ( {}^qK_i - {}^qK_i ^{-1})}{q_i - q_i ^{-1}} \bigg) \\
    &\hspace{3cm} = {}^q \pi_\lambda \bigg( \frac{  {}^qE_j {}^q K_i (q^{(\alpha_i, \alpha_j)} - 1) - {}^q K_i ^{-1} \, {}^qE_j ( 1 - q^{(\alpha_i, \alpha_j)} )}{q_i - q_i ^{-1}} \bigg) \\
    &\hspace{3cm}\longrightarrow 2 \frac{(\alpha_i, \alpha_j)}{(\alpha_i, \alpha_i)} \lim_{q \rightarrow 1} {}^q \pi_\lambda ( {}^q E_j ) = a_{ij} \lim_{q \rightarrow 1} {}^q \pi_\lambda ({}^q E_j)
    \end{align*}
    since $\lim_{q \rightarrow 1} {}^q \pi_\lambda ({}^q K_\mu ) = \id$. The same reasoning with ${}^q F_j$ will show
    \begin{align*}
    &{}^q \pi_\lambda \bigg( \frac{ ({}^q K_i - {}^q K_i ^{-1} )\, {}^qF_j - {}^qF_j ( {}^qK_i - {}^qK_i ^{-1})}{q_i - q_i ^{-1}} \bigg)
    \longrightarrow - a_{ij} \lim_{q \rightarrow 1} {}^q \pi_\lambda ( {}^q F_j ).
    \end{align*}
    These two prove (2) of Theorem~\ref{thm:Lie algebra presentation}.
    
    The two conditions (3)--(4) are simple consequences of U3--U4 of Definition~\ref{defn:the quantized universal enveloping algebra} placed inside $\lim_{q \rightarrow 1}{}^q \pi_\lambda (\,\cdot\,)$ and the identity
    \[
    \lim_{q \rightarrow 1} \left[ \begin{array}{c} n \\ k \end{array} \right]_{q_i} = \begin{pmatrix} n \\ k \end{pmatrix},
    \]
    which holds for all $ n, k \in \Nbb $ and $ 1 \leq i \leq N$.
    
    Thus, by Theorem~\ref{thm:Lie algebra presentation}, there exists a unique complex Lie algebra homomorphism $\pi_\lambda : \gf \rightarrow \End(V(\lambda))$ satisfying
    \begin{gather*}
    \pi_\lambda (H_j ' ) = \lim_{q \rightarrow 1} {}^q \pi_\lambda \Big( \frac{{}^q K_j - {}^q K_j ^{-1}}{q_j - q_j ^{-1}} \Big) \\
    \pi_\lambda (E_j) = \lim_{q \rightarrow 1} {}^q \pi_\lambda ({}^q E_j), \quad \pi_\lambda (F_j) = \lim_{q \rightarrow 1} {}^q \pi_\lambda ({}^q F_j)
    \end{gather*}
    for all $1 \leq j \leq N$.
    
    Note that \eqref{eq:K - K in global basis} for $\mu = \alpha_1, \cdots , \alpha_N$ implies that
    \[
    \pi_\lambda (H_j ') = \diag \big( (\alpha_j ^\lor, \epsilon_1 ^\lambda) , \cdots , (\alpha_j ^\lor , \epsilon_{n_\lambda} ^\lambda) \big), \quad 1 \leq j \leq N
    \]
    in the basis $\{ v_i \mid 1 \leq i \leq n_\lambda \}$. Therefore, the set of weights counted with multiplicity for the Lie algebra representation $\pi_\lambda$ is equal to the set of weights counted with multiplicity for ${}^q \pi_\lambda$ for any $ 0 <q < 1$. In particular, their highest weights are equal, namely $\lambda$. Thus, $\pi_\lambda$ contains an irreducible Lie algebra subrepresentation whose highest weight is $\lambda$. However, the dimension of that subrepresentation is equal to the dimension of $V(\lambda)$ by \cite[Section~5.1.1]{VoigtYuncken}, which proves that $\pi_\lambda$ is the irreducible representation of $\gf$ with highest weight $\lambda$.
    
    Now, it only remains to check the identities of the proposition, which have already been checked for simple roots. Let $\mu = r_1 \alpha_1 ^\lor + \cdots r_N \alpha_N ^\lor \in \weights$ with $r_1, \cdots , r_N \in \Rbb$. By \eqref{eq:K - K in global basis}, we have
    \begin{align*}
    &\lim_{q \rightarrow 1} {}^q \pi_\lambda \bigg( \frac{{}^q K_\mu - {}^q K_\mu ^{-1}}{q - q^{-1}} \bigg) = \diag\big( (\mu, \epsilon_1 ^\lambda) , \cdots, (\mu, \epsilon_{n_\lambda} ^\lambda) \big) \\
    &\hspace{2cm} = \sum_{j=1} ^N r_j \diag\big( (\alpha_j ^\lor, \epsilon_1 ^\lambda) , \cdots, (\alpha_j ^\lor, \epsilon_{n_\lambda} ^\lambda) \big) = \sum_{j=1} ^N r_j \pi_\lambda (H_j ') = \pi_\lambda (H_\mu),
    \end{align*}
    proving the first identity.

    Now, let $\alpha \in \boldsymbol{\Delta}^+$. To check the identities for $E_\alpha$ and $F_\alpha$, we first need to look more closely at the definitions of ${}^q E_\alpha$ and ${}^q F_\alpha$. Note that the definition of the algebra automorphisms ${}^q \Tcal_1, \cdots, {}^q \Tcal_N$ given in \cite[p.76 and Theorem~3.58]{VoigtYuncken} shows that, for each $X \in U_q (\gf)$ and $1 \leq i \leq N$,
    \begin{align*}
        {}^q \pi_\lambda \big( {}^q \Tcal_i (X) \big) &= {}^q \Tcal_i ^\lambda \, {}^q \pi_\mu (X) ({}^q \Tcal_i ^\lambda) ^{-1}
    \end{align*}
    where $({}^q \Tcal_i ^\lambda)^{\pm 1} : V(\lambda) \rightarrow V(\lambda)$ are given by
    \begin{align}\label{eq:q-braid automorphisms-rep}
        {}^q \Tcal_i ^\lambda v &= \sum_{\substack{r, s, t \geq 0 \\ r-s + t = m}} (-1)^s q_i ^{s- rt} \frac{ {}^q \pi_\lambda( {}^q F_i) ^r}{[r]_{q_i} !} \, \frac{{}^q \pi_\lambda( {}^q E_i) ^s}{[s]_{q_i} !} \, \frac{{}^q \pi_\lambda({}^q F_i) ^t}{[t]_{q_i} !}v \nonumber \\
        \big({}^q \Tcal_i ^\lambda\big)^{-1} v &= \sum_{\substack{r, s, t \geq 0 \\ -r+s -t = m}} (-1)^s q_i ^{rt - s} \frac{ {}^q \pi_\lambda( {}^q E_i) ^r}{[r]_{q_i} !} \, \frac{{}^q \pi_\lambda( {}^q F_i) ^s}{[s]_{q_i} !} \, \frac{{}^q \pi_\lambda({}^q E_i) ^t}{[t]_{q_i} !}v
    \end{align}
    when $v \in V(\lambda)$ is a weight vector for the representation ${}^q \pi_\lambda$ with weight $\nu \in \Pbf(\lambda)$ and $m = ( \alpha_i ^\lor , \nu) \in \Zbb$ (see \cite[Corollary~3.50]{VoigtYuncken} for $({}^q \Tcal_i ^\lambda)^{-1}$). But, we have seen that, for each $1 \leq i \leq N$, ${}^q \pi_\lambda( F_i)$ and ${}^q \pi_\lambda (E_i)$ converge to $\pi_\lambda (E_i)$ and $\pi_\lambda(F_i)$, respectively. Thus,
    \begin{align*}
        \lim_{q \rightarrow 1}{}^q \Tcal_i ^\lambda v &= \sum_{\substack{r, s, t \geq 0 \\ r-s + t = m}} (-1)^s \frac{ \pi_\lambda(F_i) ^r}{r !} \frac{\pi_\lambda (E_i) ^s}{s !} \frac{\pi_\lambda ( F_i) ^t}{t !} v \\
        \lim_{q \rightarrow 1} \big({}^q \Tcal_i ^\lambda \big)^{-1} v &= \sum_{\substack{r, s, t \geq 0 \\ - r + s - t = m}} (-1)^s \frac{ \pi_\lambda(E_i) ^r}{r !} \frac{\pi_\lambda (F_i) ^s}{s !} \frac{\pi_\lambda ( E_i) ^t}{t !} v.
    \end{align*}
    Note that the latter summation is a part of the series
    \begin{align*}
        \exp(\pi_\lambda(E_i) ) \exp( - \pi_\lambda(F_i) ) \exp( \pi_\lambda (E_i)) v,
    \end{align*}
    which, by \cite[Section~21.2~(6)]{Humphreys} and the fact that $v$ is a weight vector for $\pi_\lambda$ with weight $\nu$, yields a weight vector with weight $s_i \nu$. Since the index $r-s + t = m$ in the summation of \eqref{eq:q-braid automorphisms-rep} was added only to ensure that we only get the terms whose weights are $s_i \nu$ (cf. \cite[Section~3.7.1]{VoigtYuncken}), this peculiar property of Lie algebra representation enables us to conclude
    \[
    \lim_{q \rightarrow 1} {}^q \Tcal_i ^\lambda = \big[ \exp(\pi_\lambda(E_i) ) \exp( - \pi_\lambda(F_i) ) \exp( \pi_\lambda (E_i)) \big]^{-1} \in \End(V(\lambda)),
    \]
    which will be denoted by $\tilde{s}_i ^\lambda$.

    Thus, if $\alpha = s_{i_1} \cdots s_{i_{r-1}} \alpha_{i_r} \in \boldsymbol{\Delta}^+$ for some $ 1 \leq r \leq t$, then by \eqref{eq:c alpha pm},
    \begin{align*}
        \lim_{q \rightarrow 1} {}^q\pi_\lambda({}^q E_\alpha ) = \lim_{q \rightarrow 1} {}^q \Tcal_{i_1} ^\lambda \cdots {}^q \Tcal_{i_{r-1}} ^\lambda  {}^q\pi_\lambda ({}^q E_{i_r} ) \big( {}^q \Tcal_{i_{r-1}} ^\lambda \big) ^{-1} \cdots \big( {}^q \Tcal_{i_{1}} ^\lambda \big) ^{-1} \\
        = \tilde{s}_{i_1} ^\lambda \cdots \tilde{s}_{i_{r-1}} ^\lambda \pi_\lambda (E_{i_r}) (\tilde{s}_{i_{r-1}} ^\lambda)^{-1} \cdots (\tilde{s}_{i_{1}} ^\lambda)^{-1}  \\
        = \pi_\lambda \big( \tilde{s}_{i_1} ^{\ad} \cdots \tilde{s}_{i_{r-1}} ^{\ad} E_{i_r} \big) = c_\alpha ^+ \pi_\lambda (E_\alpha).
    \end{align*}

    Exactly the same reasoning with $F_\alpha$ in place of $E_\alpha$ gives us
    \[
    \lim_{q \rightarrow 1} {}^q\pi_\lambda({}^q F_\alpha) = c_\alpha ^- \pi_\lambda (F_\alpha).
    \]
    Note that \cite[Lemma~3.61]{VoigtYuncken} implies $c_{\alpha_j} ^\pm = 1$ for any $ 1 \leq j \leq N$.

    Finally, \eqref{eq:relations between E F H} and the last identity of \eqref{eq:relations for the root vectors} show
    \begin{align*}
        c_\alpha ^+ c_\alpha ^- \pi_\lambda (H_\alpha ') = c_\alpha ^+ c_\alpha ^- \pi_\lambda \big( [E_\alpha , F_\alpha] \big) = c_\alpha ^+ \pi_\lambda (E_\alpha) c_\alpha ^- \pi_\lambda (F_\alpha) - c_\alpha ^+ \pi_\lambda (F_\alpha) c_\alpha ^- \pi_\lambda (E_\alpha) \\
        = \lim_{q \rightarrow 1} {}^q \pi_\lambda \big( [{}^q E_\alpha , {}^q F_\alpha] \big) = \lim_{q \rightarrow 1} {}^q \pi_\lambda \bigg( \frac{ {}^q K_\alpha - {}^q K_\alpha ^{-1} }{q_\alpha - q_\alpha ^{-1}} \bigg) = \pi_\lambda( H_\alpha ').
    \end{align*}
    Since this expression holds for all $\lambda \in \weights^+$, we see $c_\alpha ^+ c_\alpha ^- = 1$.
\end{proof}

To finish the proof of Theorem~\ref{thm:collecting all q}~(1), we need to show, for $\alpha \in \boldsymbol{\Delta}^+$,
\[
\lim_{q \rightarrow 1} {}^q \pi_\lambda {}^q \hat{S} ({}^q E_\alpha) = - \lim_{q \rightarrow 1} {}^q \pi_\lambda ({}^q E_\alpha), \quad \lim_{q \rightarrow 1} {}^q \pi_\lambda {}^q \hat{S} ({}^q F_\alpha) = - \lim_{q \rightarrow 1} {}^q \pi_\lambda ({}^q F_\alpha)
\]

For this, we use induction as follows. Note that, by \cite[Theorem~3.58]{VoigtYuncken}, we have, for each $ 1 \leq j \leq N$,
\begin{gather*}
{}^q \Tcal_j ( K_\mu ) = K_{s_j \mu} , \quad \mu \in \weights \\
    {}^q \Tcal_j ({}^q E_j) = - {}^q K_j \, {}^q F_j , \quad {}^q \Tcal_j (F_j) = - {}^q E_j \, {}^q K_j ^{-1} \\
{}^q \Tcal_i ({}^q E_j) = \sum_{k=0} ^{-a_{ij}} (-1)^k q_i ^k \frac{{}^qE_i ^{k}}{[k]_{q_i} !} {}^qE_j \frac{{}^qE_i ^{-a_{ij} - k}}{[-a_{ij} -k]_{q_i} !} , \quad i \neq j  \\
{}^q \Tcal_i ({}^q F_j) = \sum_{k=0} ^{-a_{ij}} (-1)^k q_i ^{-k} \frac{{}^qF_i ^{-a_{ij} - k}}{[-a_{ij} -k]_{q_i} !} {}^qF_j \frac{{}^qF_i ^{k}}{[k]_{q_i} !} , \quad i \neq j.
\end{gather*}
Thus, by Lemma~\ref{lem:q->1 limits of defining relations}, we see that, for any $1 \leq j \leq N$,
\begin{align*}
    \lim_{q \rightarrow 1} {}^q\pi_\lambda {}^q \hat{S}\, {}^q \Tcal_j ({}^q E_j) = \lim_{q \rightarrow 1} {}^q \pi_\lambda ( {}^q K_j {}^q F_j {}^q K_j ^{-1} ) = \lim_{q \rightarrow 1} {}^q \pi_\lambda ({}^q F_j) = - \lim_{q \rightarrow 1} {}^q \pi_\lambda ( {}^q \Tcal_j ({}^q E_j) )
\end{align*}
and similarly $\lim_{q \rightarrow 1} {}^q\pi_\lambda {}^q \hat{S} \, {}^q \Tcal_j ({}^q F_j) = - \lim_{q \rightarrow 1} {}^q \pi_\lambda ( {}^q \Tcal_j ({}^q F_j) )$. Also, for any $ 1 \leq  i \neq j \leq N$,
\begin{align*}
    \lim_{q \rightarrow 1} {}^q\pi_\lambda {}^q \hat{S} {}^q \Tcal_i ({}^q E_j) = \lim_{q \rightarrow 1} {}^q \pi_\lambda \Big( \sum_{k=0} ^{-a_{ij}} (-1)^k q_i ^{k} \frac{{}^q \hat{S} ({}^q E_i) ^{- a_{ij} - k} }{[-a_{ij} - k]_{q_i} !} {}^q \hat{S} ({}^q E_j) \frac{{}^q \hat{S} ({}^q E_i) ^{k} }{[k]_{q_i} !} \Big) \\
    = - \lim_{q \rightarrow 1} {}^q \pi_\lambda \Big( \sum_{k=0} ^{-a_{ij}} (-1)^{ - a_{ij} - k } q_i ^{k} \frac{{}^q E_i ^{- a_{ij} - k} }{[-a_{ij} - k]_{q_i} !} {}^q E_j \frac{{}^q E_i ^{k}}{[k]_{q_i} !} \Big) \\
    = - \lim_{q \rightarrow 1} {}^q \pi_\lambda {}^q \Tcal_i ( {}^q E_j )
\end{align*}
and similarly $\lim_{q \rightarrow 1} {}^q \pi_\lambda {}^q \hat{S} \, {}^q \Tcal_i ({}^q F_j) = - \lim_{q \rightarrow 1} {}^q \pi_\lambda {}^q \Tcal_i ( {}^q F_j ) $. Thus, we conclude that
\begin{align*}
    \lim_{q \rightarrow 1} {}^q \pi_\lambda {}^q \hat{S} \, {}^q \Tcal_i ({}^q E_j) &= - \lim_{q \rightarrow 1} {}^q \pi_\lambda {}^q \Tcal_i ({}^q E_j) \nonumber \\
    \lim_{q \rightarrow 1} {}^q \pi_\lambda {}^q \hat{S} \, {}^q \Tcal_i ({}^q F_j) &= - \lim_{q \rightarrow 1} {}^q \pi_\lambda {}^q \Tcal_i ({}^q F_j)
\end{align*}
for all $ 1 \leq i , j \leq N$. Using this as the base case, we fix $n \geq 2$ and assume
\begin{align*}
\lim_{q \rightarrow 1} {}^q \pi_\lambda {}^q \hat{S} \, {}^q \Tcal_{j_1} \cdots {}^q \Tcal_{j_{n-1}} (E_{j_n}) &= - \lim_{q \rightarrow 1} {}^q \pi_\lambda {}^q \Tcal_{j_1} \cdots {}^q \Tcal_{j_{n-1}} (E_{j_n}) \\
\lim_{q \rightarrow 1} {}^q \pi_\lambda {}^q \hat{S} \, {}^q \Tcal_{j_1} \cdots {}^q \Tcal_{j_{n-1}} (F_{j_n}) &= - \lim_{q \rightarrow 1} {}^q \pi_\lambda {}^q \Tcal_{j_1} \cdots {}^q \Tcal_{j_{n-1}} (F_{j_n})
\end{align*}
for all $1 \leq j_1 , \cdots , j_n \leq N$. Choose $ 1 \leq j_1 , \cdots , j_{n+1} \leq N$. If $j_n = j_{n+1}$, then by the identity ${}^q \Tcal_{j_n} ({}^q E_{j_n}) = - {}^q K_{j_n} \, {}^q F_{j_n} $ and the induction hypothesis,
\begin{align*}
    \lim_{q \rightarrow 1} {}^q \pi_\lambda {}^q \hat{S} \, {}^q \Tcal_{j_1} \cdots {}^q \Tcal_{j_{n}} ({}^q E_{j_{n+1}} ) = \lim_{q \rightarrow 1} {}^q \pi_\lambda {}^q \hat{S} \, {}^q \Tcal_{j_1} \cdots {}^q \Tcal_{j_{n-1}} ( - {}^q K_{j_{n+1}} {}^q F_{j_{n+1}} ) \\
    = \lim_{q \rightarrow 1} {}^q \pi_\lambda {}^q \hat{S} \, {}^q \Tcal_{j_1} \cdots {}^q \Tcal_{j_{n-1}} ( - {}^q F_{j_{n+1}} ) \\
    = - \lim_{q \rightarrow 1} {}^q \pi_\lambda {}^q \Tcal_{j_1} \cdots {}^q \Tcal_{j_{n-1}} ( - {}^q F_{j_{n+1}} ) \\
    = - \lim_{q \rightarrow 1} {}^q \pi_\lambda {}^q \Tcal_{j_1} \cdots {}^q \Tcal_{j_{n-1}} {}^q \Tcal_{j_n} ( {}^q E_{j_{n+1}} )
\end{align*}
and similarly
\[
\lim_{q \rightarrow 1} {}^q \pi_\lambda {}^q \hat{S} \, {}^q \Tcal_{j_1} \cdots {}^q \Tcal_{j_{n}} ({}^q F_{j_{n+1}} )
    = - \lim_{q \rightarrow 1} {}^q \pi_\lambda {}^q \Tcal_{j_1} \cdots {}^q \Tcal_{j_n} ( {}^q F_{j_{n+1}} ).
\]
Now, let $j_{n} \neq j_{n+1}$. Temporarily, we denote $j_n = i$ and $j_{n+1} = j$. Then, take $\lim_{q \rightarrow 1} {}^q \pi_\lambda {}^q \hat{S} {}^q \Tcal_{j_1} \cdots {}^q \Tcal_{j_{n-1}}$ on both sides of
\[
{}^q \Tcal_i ({}^q E_j) = \sum_{k=0} ^{-a_{ij}} (-1)^k q_i ^k \frac{{}^qE_i ^{k}}{[k]_{q_i} !} {}^qE_j \frac{{}^qE_i ^{-a_{ij} - k}}{[-a_{ij} -k]_{q_i} !}
\]
and apply the induction hypothesis to each of the three resulting factors inside the summation sign to obtain
\begin{align*}
    \lim_{q \rightarrow 1} {}^q \pi_\lambda {}^q &\hat{S} \, {}^q \Tcal_{j_1} \cdots {}^q \Tcal_{j_{n}} ({}^q E_{j_{n+1}} ) \\
    &= \lim_{q \rightarrow 1} {}^q \pi_\lambda {}^q \Tcal_{j_1} \cdots {}^q \Tcal_{j_{n-1}} \Big( \sum_{k=0} ^{- a_{ij}} (-1)^{k} q_i ^{k} \frac{ (- \, {}^qE_i) ^{-a_{ij} - k}}{[-a_{ij} - k]_{q_i} !} (- \, {}^qE_j) \frac{(-\, {}^qE_i)^k}{[k]_{q_i} !} \Big) \\
    &= - \lim_{q \rightarrow 1} {}^q \pi_\lambda {}^q \Tcal_{j_1} \cdots {}^q \Tcal_{j_{n-1}} \Big( \sum_{k=0} ^{- a_{ij}} (-1)^{-a_{ij} - k} q_i ^{k} \frac{{}^qE_i ^{-a_{ij} - k}}{[-a_{ij} - k]_{q_i} !} {}^qE_j \frac{{}^qE_i ^{k}}{[k]_{q_i} !} \Big) \\
    &= - \lim_{q \rightarrow 1} {}^q \pi_\lambda {}^q \Tcal_{j_1} \cdots {}^q \Tcal_{j_{n-1}} ( {}^q \Tcal_{j_n} {}^q E_{j_{n+1}} ).
\end{align*}
In the same way, we get
\begin{align*}
    \lim_{q \rightarrow 1} {}^q \pi_\lambda {}^q \hat{S} \, {}^q \Tcal_{j_1} \cdots {}^q \Tcal_{j_{n}} ({}^q F_{j_{n+1}} ) = - \lim_{q \rightarrow 1} {}^q \pi_\lambda {}^q \Tcal_{j_1} \cdots {}^q \Tcal_{j_{n}} ({}^q F_{j_{n+1}} ),
\end{align*}
completing the induction.

Now, let $\alpha = s_{i_1} \cdots s_{i_{r-1}} \alpha_{i_r} \in \boldsymbol{\Delta}^+$ for some $ 1 \leq r \leq t$. By what has just been proved, we conclude
\begin{align*}
    \lim_{q \rightarrow 1} {}^q \pi_\lambda {}^q \hat{S} ({}^q E_\alpha) &= \lim_{q \rightarrow 1} {}^q \pi_\lambda {}^q \hat{S} \, {}^q \Tcal_{i_1} \cdots {}^q \Tcal_{i_{r-1}} ({}^q E_{i_r} ) = - \lim_{q \rightarrow 1} {}^q \pi_\lambda ({}^q E_\alpha), \\
    \lim_{q \rightarrow 1} {}^q \pi_\lambda {}^q \hat{S} ({}^q F_\alpha) &= \lim_{q \rightarrow 1} {}^q \pi_\lambda {}^q \hat{S} \, {}^q \Tcal_{i_1} \cdots {}^q \Tcal_{i_{r-1}} ({}^q F_{i_r} ) = - \lim_{q \rightarrow 1} {}^q \pi_\lambda ({}^q F_\alpha).
\end{align*}

\subsection*{Proof of (2)}

This fact was asserted in the proof of \cite[Proposition~4.16]{VoigtYuncken}, and we supply here a proof for it.

Fix $0 < q \leq 1$ and $\lambda \in \weights^+$. Let $\epsilon_1^\lambda$ be the highest weight in $\Pbf(\lambda)$. Since the representation ${}^q \overline{\pi}_\lambda : U_q(\gf) \rightarrow \End\big( \overline{V(\lambda)}^* \big)$ (where $\overline{V(\lambda)}$ is the vector space $V(\lambda)$ equipped with a new scalar multiplication $\cdot$ given by $a \cdot v = \overline{a} v$ for $a \in \Cbb$ and $v \in V$) defined by, for $X \in U_q(\gf)$,
\[
{}^q \overline{\pi}_\lambda(X) f = f \circ {}^q \pi_\lambda(X^*), \quad f \in \overline{V(\lambda)}^*,
\]
is an irreducible representation with highest weight $\lambda$, we see there exists a unique nondegenerate sesquilinear form $\la \cdot, \cdot \ra_q : V(\lambda) \times V(\lambda) \rightarrow \Cbb$ such that, for all $X \in U_q(\gf)$ and $v, w \in V(\lambda)$,
\begin{equation}\label{eq:*-representation property}
\la {}^q \pi_\lambda(X) v , w \ra_q = \la v, {}^q \pi_\lambda(X^*) w \ra_q \quad \text{and} \quad \la v_1 , v_1 \ra_q = 1.
\end{equation}
Note that, since $v_1$ has a weight different from the weights of all the other $v_j$s, these two conditions imply
\begin{equation}\label{eq:orthogonality of e1}
\la v_1 , v_j \ra_q = \delta_{1j}, \quad 1 \leq j \leq n_\lambda.
\end{equation}
Since the representation $\pi_\lambda$ of $U_q(\gf) = U_q^\Rbb(\kf)$ can always be made into a $*$-representation on a Hilbert space, we see that $\la \cdot , \cdot \ra_q$ must be positive definite; see \cite[Proposition~4.16]{VoigtYuncken}.

To finish the proof of (2), we need to prove that the family $\big( \la \cdot, \cdot \ra_q \big)_{0 < q \leq 1}$ of inner products depends continuously on $0 < q \leq 1$. For that, we introduce a formal adjoint $*: \Ubf_\qbf(\gf) \rightarrow \Ubf_\qbf(\gf)$ defined by
\[
(\Kbf_\mu)^* = \Kbf_\mu, \quad (\Ebf_j)^* = \Kbf_j \Fbf_j, \quad (\Fbf_j)^* = \Ebf_j \Kbf_j^{-1}, \quad \mu \in \weights, \, 1 \leq j \leq N,
\]
which is a well-defined $\Qbb(s)$-linear antihomomorphism by \cite[Lemma~3.16]{VoigtYuncken}. Just as in the case of $U_q(\gf)$, the representation $\overline{\pibf}_\lambda : \Ubf_\qbf(\gf) \rightarrow \End_{\Qbb(s)}( \Vbf(\lambda)^* )$ defined by
\[
\overline{\pibf}_\lambda(\Xbf) f = f \circ \pibf_\lambda(\Xbf^*), \quad f \in \End_{\Qbb(s)}( \Vbf(\lambda)^* )
\]
for $\Xbf \in \Ubf_\qbf(\gf)$ is an irreducible representation with highest weight $\lambda$, and hence there exists a unique nondegenerate $\Qbb(s)$-bilinear map $\La \cdot , \cdot \Ra : \Vbf(\lambda) \times \Vbf(\lambda) \rightarrow \Qbb(s)$ such that for all $\Xbf \in \Ubf_\qbf(\gf)$ and $\vbf, \wbf \in \Vbf(\lambda)$,
\begin{equation}\label{eq:*-representation property for integral form}
\La \pibf_\lambda(\Xbf) \vbf , \wbf \Ra = \La \vbf, \pibf_\lambda(\Xbf^*) \wbf \Ra \quad \text{and} \quad \La \vbf_1 , \vbf_1 \Ra = 1.
\end{equation}
Likewise, we also have
\[
\La \vbf_1, \vbf_j \Ra = \delta_{1j}, \quad 1 \leq j \leq n_\lambda.
\]

Since $\Vbf(\lambda)_\Acal = \Ubf_\qbf^\Acal(\gf) \vbf_1$, we can find $\Xbf_j \in \Ubf_\qbf^\Acal(\gf)$ for each $1 \leq j \leq n_\lambda$ such that
\begin{equation*}
\pibf_\lambda(\Xbf_j) \vbf_1 = \vbf_j.
\end{equation*}
Thus, for all $1 \leq i, j \leq n_\lambda$, we have
\[
\La \vbf_i, \vbf_j \Ra = \La \vbf_1, \pibf_\lambda(\Xbf_i^* \Xbf_j) \vbf_1 \Ra,
\]
which is the $\vbf_1$-component of the expansion of $\pibf(\Xbf_i^* \Xbf_j) \vbf_1$ in the basis $\{ \vbf_j \mid 1 \leq j \leq n_\lambda \}$, and thus an element of $\Acal \subseteq \Qbb(s)$. Since $\{ \vbf_j \mid 1 \leq j \leq n_\lambda \}$ is an $\Acal$-basis of $\Vbf(\lambda)_\Acal$, we conclude
\[
\La \Vbf(\lambda)_\Acal, \Vbf(\lambda)_\Acal \Ra \subseteq \Acal.
\]
Hence, for each $0 < q < 1$, the map $\la \cdot, \cdot \ra_q' : \big( \Cbb_q \otimes_\Acal \Vbf(\lambda)_\Acal \big) \times \big( \Cbb_q \otimes_\Acal \Vbf(\lambda)_\Acal \big) \to \Cbb$ defined by
\[
\la a \otimes \vbf, b \otimes \wbf \ra_q' = \overline{a} b\, \ev_q\big( \La \vbf, \wbf \Ra \big)
\]
is a well-defined sesquilinear form on $\Cbb_q \otimes_\Acal \Vbf(\lambda)_\Acal$. Note that, by \eqref{eq:*-representation property for integral form}, we have
\[
\Big\la \big( \id_{\Cbb_q} \otimes \pibf_\lambda(\Xbf) \big) \xi, \eta \Big\ra_q' = \Big\la \xi, \big( \id_{\Cbb_q} \otimes \pibf_\lambda(\Xbf^*) \big) \eta \Big\ra_q'
\]
for all $\Xbf \in \Ubf_\qbf^\Acal(\gf)$ and $\xi, \eta \in \Cbb_q \otimes_\Acal \Vbf(\lambda)_\Acal$, and also,
\[
\la 1_{\Cbb_q} \otimes \vbf_1, 1_{\Cbb_q} \otimes \vbf_1 \ra_q' = 1.
\]
However, under the identifications $\Cbb_q \otimes_\Acal \Vbf(\lambda)_\Acal \cong V(\lambda)$ described in the paragraph preceding Lemma~\ref{lem:q->1 limits of defining relations}, the preceding two conditions precisely become \eqref{eq:*-representation property}, which implies that $\la \cdot, \cdot \ra_q = \la \cdot, \cdot \ra_q'$. Thus, for all $0 < q < 1$ and $1 \leq i, j \leq n_\lambda$, we have
\begin{equation}\label{eq:q-inner product as an evaluation}
\la v_i, v_j \ra_q = \la 1_{\Cbb_q} \otimes \vbf_i, 1_{\Cbb_q} \otimes \vbf_j \ra_q' = \ev_q \La \vbf_i, \vbf_j \Ra,
\end{equation}
which enables us to conclude that the family of inner products $\big( \la \cdot, \cdot \ra_q \big)_{0 < q < 1}$ on $V(\lambda)$ depends continuously on $0 < q < 1$.

Now, it remains to prove the continuity at $q = 1$. Note that, since $\Ubf_\qbf^\Acal(\gf)$ is generated by
\[
\Kbf_\lambda, \quad \frac{\Kbf_j - \Kbf_j^{-1}}{\qbf_j - \qbf_j^{-1}}, \quad \frac{1}{[r]_{\qbf_j}!} \Ebf_j^r, \quad \frac{1}{[r]_{\qbf_j}!} \Fbf_j^r
\]
for $\lambda \in \weights$, $1 \leq j \leq N$, and $r \in \Nbb$, Proposition~\ref{prop:Lie algebra homomorphism as q->1 limit} implies that
\begin{equation}\label{eq:q->1 limit of elements in integral form}
\lim_{q \to 1} \big( \id_{\Cbb_q} \otimes \pibf_\lambda(\Xbf) \big) = \pi_\lambda\big( p( 1_{\Cbb_1} \otimes \Xbf ) \big), \quad \Xbf \in \Ubf_\qbf^\Acal(\gf),
\end{equation}
where $\Cbb_1$ is the space $\Cbb$ equipped with the $\Acal$-module structure provided by the ring homomorphism $\ev_1 : \Acal \to \Cbb$ sending $s$ to $1$, and $p : \Cbb_1 \otimes_\Acal \Ubf_\qbf^\Acal(\gf) \to U(\gf)$ is the surjective $*$-preserving algebra homomorphism given in \cite[Proposition~3.25]{VoigtYuncken}, characterized by $p(1_{\Cbb_1} \otimes \Kbf_\mu) = 1$ for $\mu \in \weights$ and
\[
p\bigg( 1_{\Cbb_1} \otimes \frac{\Kbf_j - \Kbf_j^{-1}}{\qbf_j - \qbf_j^{-1}} \bigg) = H_j', \quad p(1_{\Cbb_1} \otimes \Ebf_j) = E_j, \quad p(1_{\Cbb_1} \otimes \Fbf_j) = F_j
\]
for $1 \leq j \leq N$. In particular, \eqref{eq:q->1 limit of elements in integral form} implies that, for each $1 \leq j \leq n_\lambda$,
\begin{equation*}
\pi_\lambda\big( p(1_{\Cbb_1} \otimes \Xbf_j) \big) v_1 = \lim_{q \to 1} \big( \id_{\Cbb_q} \otimes \pibf_\lambda(\Xbf_j) \big)(1_{\Cbb_q} \otimes \vbf_1) = \lim_{q \to 1}(1_{\Cbb_q} \otimes \vbf_j) = v_j.
\end{equation*}
Thus, for all $1 \leq i, j \leq n_\lambda$, we have
\begin{align*}
\la v_i, v_j \ra_1 = \Big\la \pi_\lambda\big( p(1_{\Cbb_1} \otimes \Xbf_i) \big) v_1, \pi_\lambda\big( p(1_{\Cbb_1} \otimes \Xbf_j) \big) v_1 \Big\ra_1 = \Big\la v_1, \pi_\lambda\big( p(1_{\Cbb_1} \otimes \Xbf_i^* \Xbf_j) \big) v_1 \Big\ra_1,
\end{align*}
which is equal to the $v_1$-component of $\pi_\lambda\big( p(1_{\Cbb_1} \otimes \Xbf_i^* \Xbf_j) \big) v_1$ in the basis $\{ v_j \mid 1 \leq j \leq n_\lambda \}$ by \eqref{eq:orthogonality of e1}. By \eqref{eq:q->1 limit of elements in integral form}, this is equal to the $q \to 1$ limit of the $v_1$-component of the expression $\big( \id_{\Cbb_q} \otimes \pibf_\lambda(\Xbf_i^* \Xbf_j) \big) v_1$ in the basis $\{ v_j \mid 1 \leq j \leq n_\lambda \}$, which, by \eqref{eq:orthogonality of e1} and \eqref{eq:q-inner product as an evaluation}, is
\begin{align*}
\Big\la v_1, \big( \id_{\Cbb_q} \otimes \pibf_\lambda(\Xbf_i^* \Xbf_j) \big) v_1 \Big\ra_q = \big\la 1_{\Cbb_q} \otimes \vbf_1, 1_{\Cbb_q} \otimes \pibf_\lambda(\Xbf_i^* \Xbf_j) \vbf_1 \big\ra_q' \\
= \ev_q \big\la \vbf_1, \pibf_\lambda(\Xbf_i^* \Xbf_j) \vbf_1 \big\ra \\
= \ev_q \La \vbf_i, \vbf_j \Ra = \la v_i, v_j \ra_q.
\end{align*}
That is, we have
\begin{equation*}
\la v_i, v_j \ra_1 = \lim_{q \to 1} \la v_i, v_j \ra_q, \quad 1 \leq i, j \leq n_\lambda.
\end{equation*}
This shows that $\la \cdot, \cdot \ra_q$ is continuous at $q = 1$ as well.

\subsection*{Proof of (3)}

We apply the Gram-Schmidt orthonormalization with respect to the inner product $\la \cdot, \cdot \ra_1$ to the basis $\{v_j \mid 1 \leq j \leq n_\lambda\}$ to obtain an orthonormal basis $\{e_j^\lambda \mid 1 \leq j \leq n_\lambda\}$ of $(V(\lambda), \la \cdot, \cdot \ra_1)$. Fix $1 \leq j \leq n_\lambda$. If $\epsilon_j^\lambda \neq \epsilon_k^\lambda$ for some $1 \leq k \leq n_\lambda$, then $\la v_j, v_k \ra_1 = 0$ by the first property of $\la \cdot, \cdot \ra_1$ in \eqref{eq:*-representation property}. Hence, in the formula
\[
c \, e_j^\lambda = v_j - \sum_{1 \leq i \leq j-1} \la v_i, v_j \ra_1 v_i,
\]
where $c \in (0, \infty)$ is the norm of the right-hand side with respect to $\la \cdot, \cdot \ra_1$, the vector $v_k$ does not appear. In other words, $e_j^\lambda$ is a linear combination of $\{v_i \mid \epsilon_i^\lambda = \epsilon_j^\lambda\}$, which implies that $e_j^\lambda$ is still a weight vector of ${}^q \pi_\lambda$ with weight $\epsilon_j^\lambda$ for any $0 < q \leq 1$.

\renewcommand{\theequation}{\thesubsection.\arabic{equation}}
\renewcommand{\thethm}{\thesubsection.\arabic{thm}}

\printbibliography

\end{document}